\documentclass[11pt, reqno]{amsart}
\usepackage{bm}
\usepackage{xcolor}

\usepackage{multirow}
\usepackage{tikz}
\usepackage{caption}
\usepackage{textgreek}
\usetikzlibrary{
 decorations.pathmorphing,%
 decorations.markings,%
 decorations.pathreplacing,%
 decorations.text,%
 calc,%
 patterns,%
 shapes.geometric,%
 arrows,
 decorations.shapes,
 positioning,
 plotmarks,
 shadings,
 math,
 intersections
}

\definecolor{jeffColor}{RGB}{102, 0, 204}
\definecolor{yaizaColor}{RGB}{0, 153, 153}
\definecolor{certainty}{RGB}{64, 228, 198}
\definecolor{hope}{RGB}{228, 194, 64}

\definecolor{periodColor}{RGB}{255, 167, 105}
\definecolor{dark-green}{RGB}{135, 194, 130}
\tikzset{>=latex} 
\tikzset{font=\small}
\tikzset{mark size=1.5pt, mark options=thin}
\tikzset{pin distance=4pt,
 every pin edge/.style={<-, thin, shorten <= -2pt}}
\usepackage{array,booktabs,tabularx}
\usepackage{graphicx}
\usepackage{amssymb}
\usepackage{arydshln}
\usepackage{enumerate}
\usepackage{color}
\usepackage{esint}
\usepackage{mathtools}
\usepackage{dsfont}
\usepackage{ esint }
\usepackage{placeins}
\usepackage[show]{ed}
\usepackage[final]{showkeys}
\usepackage{mathrsfs}
\usepackage{changepage}   

\newcommand{\smalleq}[1]{\scalebox{.9}{$#1$}}%
\usepackage{comment}
\usepackage{enumitem}
\usepackage{soul}
\usepackage{needspace}
\usepackage[normalem]{ulem}

\usepackage{multicol}
\usepackage{tabto}
\usepackage{hyperref}

\definecolor{uipoppy}{RGB}{221,128,71}
\definecolor{uipaleblue2}{RGB}{179,196,215}
\definecolor{uiviolet}{RGB}{86,86,99}
\definecolor{uiblack}{RGB}{0, 0, 0}
\definecolor{azul}{RGB}{0,128,255}
\definecolor{verde}{RGB}{50,180,50}
\definecolor{pale-verde}{RGB}{155,207,145}
\definecolor{uipaleblue}{RGB}{108,199,220}
\definecolor{light-gray}{RGB}{198,198,198}

\newtheorem{lemma}{Lemma}
\newtheorem{theorem}{Theorem}
\numberwithin{theorem}{section}

\newtheorem{corollary}[lemma]{Corollary}
\newtheorem{proposition}[lemma]{Proposition}

\theoremstyle{definition}
\newtheorem{definition}[lemma]{Definition}
\newtheorem{remark}[lemma]{Remark}

\usepackage{dsfont}

\newcommand{\B}{\mc{B}}

\newcommand{\R}{{\mathbb R}}
\newcommand{\Z}{{\mathbb Z}}

\newcommand{\G}{{\mathscr G}}

\newcommand{\sub}[1]{_{_{#1}}}

\newcommand{\ep}{\varepsilon}

\newcommand{\h}{h}

\newcommand{\cs}{$\clubsuit$}

\newcommand{\x}{\ensuremath{\times}}

\newcommand{\Id}{I}
\newcommand{\mc}[1]{\mathcal{#1}}
\newcommand{\re}{\mathbb{R}}

\renewcommand{\L}{\mathfrak{L}}
\newcommand{\Si}[1]{\bm{\Sigma}\sub{\!#1}}
\newcommand{\Gb}{G}
\newcommand{\CG}{C\sub{G}}
\newcommand{\dG}{\delta\sub{G}}

\newcommand{\red}[1]{{\color{purple}{#1}}}
\newcommand{\blue}[1]{#1}

\newcommand{\TM}{T^*\!M}

\newcommand{\SM}{S^*\!M}
\newcommand{\tSM}{\widetilde{\SM}}

\newcommand{\mS}{\m\sub{\R^\L}}

\newcommand{\cd}[3]{\Psi\sub{#2}^{#3}}

\newcommand{\w}{\omega}

\renewcommand{\a}{{\bf a}}

\def\XXint#1#2#3{{\setbox0=\hbox{$#1{#2#3}{\int}$} \vcenter{\hbox{$#2#3$}}\kern-.5\wd0}}

\newcommand{\ANDrate}{\mathfrak{f}}

\DeclareMathOperator{\vol}{vol}

\DeclareMathOperator{\supp}{supp}
\DeclareMathOperator{\inj}{inj}

\newcommand{\e}{\varepsilon}

\numberwithin{equation}{section}
\numberwithin{lemma}{section}

\newcommand{\fuInj}{\digamma}
\newcommand{\ga}{\mathfrak{y}}
\newcommand{\de}{\delta}
\newcommand{\ze}{\bm \zeta}

\newcommand{\Rtemp}{R}

\newcommand{\ms}[1]{\mathscr{#1}}
\newcommand{\n}{{\bm{\beta}}}

\newcommand{\Ti}{{\bf T}}

\newcommand{\Sym}{{{\operatorname{Sym}}}}
\newcommand{\Rec}{\ms{R}}
\newcommand{\Sim}{\ms{S}}
\newcommand{\ND}{\ms{N}\!}

\newcommand{\param}{{\bm \Sigma}}

\newcommand{\m}{\mathbf{m}}
\newcommand{\Sp}{\operatorname{Sp}}
\newcommand{\spec}{\operatorname{spec}}
\renewcommand{\sp}{\mathfrak{sp}}
\newcommand{\step}{\mathfrak{a}}
\newcommand{\gap}{\mathfrak{b}}
\newcommand{\remainder}{\mathfrak{r}}
\newcommand{\Rind}{\mathfrak{R}}

\newcommand{\HpProj}{\Pi}
\newcommand{\someLetter}{\Omega}
\newcommand{\sPower}{\blue{\mathfrak{J}}}
\newcommand{\rPower}{\blue{\mathfrak{K}}}
\newcommand{\RPower}{\vartheta_2}

\title[Predominant Weyl improvements and bounds on closed geodesics]{Logarithmic improvements in the Weyl law and Exponential bounds on the number of closed geodesics are predominant}

\author{Yaiza Canzani}
\address{Department of Mathematics, University of North Carolina at Chapel Hill, Chapel Hill, NC, USA}
\email{canzani@email.unc.edu}

\author{Jeffrey Galkowski}
\address{Department of Mathematics, University College London, London, UK}
\email{j.galkowski@ucl.ac.uk}

\begin{document}
\begin{abstract}Let $M$ be a smooth, compact manifold of dimension $d$ without boundary. We introduce the concept of predominance for Riemannian metrics on $M$, a notion analogous to full Lebesgue measure which, in particular, implies density. We show that there is $\someLetter>1$ such that, for a predominant $\mc{C}^\nu$ metric, the number, $\mathfrak{c}(T,g)$, of closed geodesics of length $\leq T$ satisfies
$$
 \log \mathfrak{c}(T,g)= O\big( T^{\someLetter}\big).
$$

In addition, for $g$ a Riemannian metric on $M$, let $0=\lambda_1^2(g)< \lambda_2^2(g)\leq\lambda_3^2(g) \dots$ be the eigenvalues of $-\Delta_g$. The Weyl law states that there is $c_d>0$ such that
$$\#\{j:\;\lambda_j(g)\leq \lambda\}=c_d\vol_g(M)\lambda^d+E(\lambda,g)$$
with $E(\lambda,g)=O(\lambda^{d-1})$ as $\lambda \to \infty$. We show that for $\nu>0$ large enough  there is $\someLetter>1$ such that for a predominant $\mc{C}^\nu$ metric 
$$
E(\lambda,g)=O(\lambda^{d-1}/(\log\lambda)^{1/\someLetter}).
$$ 

After an application of recent results of the authors in the case of the Weyl law~\cite{CG20Weyl}, these estimates follow from a study of the non-degeneracy properties of nearly closed orbits for predominant sets of metrics.

\end{abstract}
\maketitle

\section{Introduction}
We study properties of the geodesic flow and remainders in the Weyl law for `typical' metrics on a compact manifold {without boundary}. Since the space of Riemannian metrics, $\ms{G}$, on a manifold cannot be endowed with a non-trivial, translation invariant Borel measure, we introduce an analog of full Lebesgue measure in infinite dimensions called \emph{predominance}. We then study properties of the geodesic flow and remainders in the Weyl law for predominant sets of metrics.

The notion of predominance has three important properties: 1) any predominant set is dense, 2) a finite intersection of predominant sets is predominant, and 3) in finite dimensions, a predominant set has full Lebesgue measure. Heuristically, a set $G\subset \ms{G}$ is \emph{predominant}  if there is a family of submanifolds endowed with {finite} Borel measures $\{(\Gamma_g,\mu_g)\}_{g\in \ms{G}}$ such that $g\in \Gamma_g$, $\mu_g$ assigns a positive measure to any neighborhood of $g$, the map $g\mapsto \Gamma_g$ is $\mc{C}^1$, and $G\cap \Gamma_g$ has full $\mu_g$ measure for every $g\in\ms{G}$. For the careful definition of this concept, see Definition \ref{d:predominant} and Remark~\ref{r:heuristic}.

 \subsection*{Remainders in the Weyl law}
 
Let $\nu \geq \blue{2}$ and $M$ be a compact  $\mc{C}^\nu$ manifold without boundary, of dimension $d$. Let $\ms{G}^\nu$ denote the space of $\mc{C}^\nu$ Riemannian metrics on $M$ {with the topology induced from the $\mc{C}^\nu$ norm on symmetric tensors}. For $g\in \ms{G}^\nu$, let $-\Delta_g$ denote the (positive) Laplace--Beltrami operator with eigenvalues $0=\lambda_1^2(g)<\lambda_2^2(g)\leq \lambda_3^2(g)\leq\dots$. Then, for $\lambda>0$ define the eigenvalue counting function
$$
N(\lambda,g):=\#\{j\,:\,\lambda_j(g)\leq \lambda\}.
$$
 Our first theorem shows that the Weyl law  has a \emph{logarithmic} improvement for a predominant set of metrics.
 In what follows, $B_1$ denotes the ball of radius 1 in $\mathbb{R}^d$, $\vol\sub{\mathbb{R}^d}(B_1)$ denotes its volume, and $\vol_{g}(M)$ denotes the volume of $M$ as measured by the metric $g$. 

\begin{theorem}\label{t:predominantWeyl} {Let $d\geq 2$.} There is $\nu_0>0$ and for all $\nu\blue{\geq 2}$ there is $\someLetter_\nu>0$ such that the following holds. If $\nu\geq\nu_{0}$, $M$ is a compact $\mc{C}^{\nu}$-manifold of dimension $d$ without boundary, and $\someLetter>\someLetter_{\nu}$, then
there is a predominant set $G\sub{\someLetter}\subset\ms{G}^\nu$ (see Definition~\ref{d:predominant}) such that for every
$g\in G\sub{\someLetter}$
\[
N(\lambda,g)=(2\pi)^{-d}\vol_{g}(M)\vol\sub{\mathbb{R}^d}(B_1)\lambda^{d}+O_{g}\big(\lambda^{d-1}\big/(\log\lambda)^{\frac{1}{\someLetter}}\big),\qquad\lambda\to \infty.
\]
In particular, $G\sub{\someLetter}$ is dense in $\ms{G}^\nu$.
\end{theorem}

The constant $\someLetter_\nu$ in Theorem \ref{t:predominantWeyl} is explicit, and we can take
 \blue{\begin{equation}
\label{e:someLetter}
\someLetter_\nu:=1+
\log_2\left[
2(d-1)^2(4\max(\nu,4)
+2d^2+5d-6)
\right].
\end{equation}}

 \begin{remark}
The authors wish to stress that, although our original motivation was to study typical properties of remainders in the Weyl law, the analysis in this article is dynamical in nature and {studies} predominant properties of the geodesic flow. Once these dynamical properties are established in Theorem \ref{t:predominantR-ND} below, a direct application of the authors' work~\cite[Theorem 2]{CG20Weyl} produces the Weyl remainder estimate of Theorem \ref{t:predominantWeyl} as a corollary (See Section~\ref{s:pWeyl}).
 \end{remark}

Defining
$
E(\lambda,g):=N(\lambda, g)-(2\pi)^{-d}\vol_{g}(M)\vol\sub{\mathbb{R}^d}(B_1)\lambda^{d},
$
 the Weyl law states that, for sufficiently smooth metrics, $E(\lambda,g)=O_g(\lambda^{d-1})$. This estimate is sharp on the round sphere and has a long history dating back to the work of Weyl~\cite{We:12}, who proved (in a slightly different context) that $E(\lambda,g)={o(\lambda^{d})}$. The estimate $E(\lambda,g)=O_g(\lambda^{d-1})$ was proved by Levitan~\cite{Le:52} and Avakumovi\'c~\cite{Av:56} after which H\"ormander~\cite{Ho:68} provided a general framework for studying such remainders, reproving this estimate and making far reaching generalizations.
 
 Using this framework, B\'erard~\cite{Be:77} showed that $E(\lambda,g)=O_g(\lambda^{d-1}/\log \lambda)$ on both {surfaces} without conjugate points {and non-positively curved manifolds of any dimension}. Duistermaat--Guillemin~\cite{DuGu:75} showed that $E(\lambda,g)=o(\lambda^{d-1})$ provided that the measure of the set of closed geodesics in $S^*\!M$ is 0. Fifteen years later, Volovoy~\cite{Vo:90} provided estimates under dynamical conditions guaranteeing that $E(\lambda,g)=O_g(\lambda^{d-1}/\log \lambda)$ and verified these conditions for certain specific examples in~\cite{Vo:90b}.  
 {The recent work of Bonthenneau~\cite{Bo:16} improved a geometric estimate in B\'erard's work, thus generalizing his result to manifolds without conjugate points of any dimension.}
 Finally, in~\cite{CG20Weyl}, the authors provided estimates on $E(\lambda,g)$ under assumptions on the volume of nearly closed geodesics which improve the results of~\cite{DuGu:75, Be:77,Vo:90}.  For manifolds with boundary, the analog of~\cite{DuGu:75} was proved by Ivrii~\cite{Iv:80}. (For a more comprehensive history of the Weyl law, see~\cite{Iv:16}.)
 
  Manifolds where there are known polynomial improvements of the form $E(\lambda,g)=O_{g}(\lambda^{d-1-\e})$ are very rare. For instance, such estimates hold on the torus~\cite{Hu:03, BoWa:17}, products of spheres~\cite{IoWy:19}, and other special integrable systems~\cite{Vo:90b}. Nevertheless, it has long been expected that, for a `typical' metric $g$, there exists $\e>0$ such that $E(\lambda,g)=O_{g}(\lambda^{d-1-\e})$. However, until now, the best available result is that $E(\lambda,g)=o(\lambda^{d-1})$ for a Baire-generic set of $g$. This can be recovered from the work of Sogge--Zelditch~\cite{SoggeZelditch} or can be seen by combining the remainder estimates in~\cite{DuGu:75} with the bumpy metric theorem of Anosov and Abraham~\cite{An:82, Ab:70}. 
  
Theorem~\ref{t:predominantWeyl} improves on these bounds in two important ways. First, Baire genericity is replaced by the concept of predominance which is an analog of full Lebesgue measure in infinite dimensions. Just as in finite dimensions a full Lebesgue measure set is much more `typical' than a Baire generic one (indeed, a Baire generic set may have measure 0), a predominant set in infinite dimensions is much more `typical' than a Baire generic set. Second, although the change from $o(\lambda^{d-1})$ to $O(\lambda^{d-1}/(\log \lambda)^{1/\someLetter})$ may seem small, this improvement requires the development of new ideas and requires subtle dynamical estimates. In addition, it is the only quantitative remainder estimate available for typical metrics.


 \begin{remark}
We have not attempted to make the value of $\nu_0$  from {in Theorem \ref{t:predominantWeyl}} explicit. However, it is likely that, following the arguments in~\cite{CG20Weyl} carefully, $\nu_0$ can be taken to be $\nu_0=Cd$ for some $C>0$.
\end{remark}
\begin{remark}
Although we have kept careful track of the constant $\someLetter_\nu$ in~\eqref{e:someLetter}, we do not expect it to be optimal. Indeed, we conjecture that $\someLetter_\nu$ could be replaced by $1+\e$ for any $\e>0$. This would also allow us to obtain the same estimate for predominant sets in $\ms{G}^\infty$. At the moment, when working in $\ms{G}^\infty$, we obtain weaker remainder estimates for predominant sets of metrics (see Remark~\ref{r:almostDone}). 
\end{remark}

 \subsection*{Growth of the number of periodic geodesics} We next discuss the growth of the number of periodic geodesics of a given length. {We say that a geodesic $\gamma\subset M$ is a primitive periodic geodesic with length $T>0$, if there is a diffeomorphism $h:\mathbb{R}/T\mathbb{Z} \to \gamma$, such that $|\dot h|_{g(h(t))}=1$, 
 $$
  {(h (0),\dot h (0))= (h (T),\dot h (T))},\qquad ( h (0),\dot{ h }(0))\neq ( h (t),\dot h (t))\text{ for }t\in(0,T).
 $$
 That is, $\gamma$ is a periodic geodesic and $T$ is its minimal period.}
For $T>0$ and $g$ a Riemannian metric on $M$ let 
 $$
 \mathfrak{c}(T,g):=\#\{\gamma\,:\, \gamma\text{ is a primitive periodic geodesic for }g\text{ with length }\leq T\}.
 $$

We obtain the following theorem on the growth of $\mathfrak{c}(T,g)$. 
\begin{theorem}\label{t:predominantLengths} Let  $\blue{\nu\geq 3}$, $M$ be a $\mc{C}^{\nu}$- compact manifold of dimension $d$ without boundary, and $\someLetter_{\nu}$ as in~\eqref{e:someLetter}. Then, for each  $\someLetter>\someLetter_\nu$ there is a predominant set $G\sub{\someLetter}\subset\ms{G}^\nu$
such that for $g\in G\sub{\someLetter}$ there exists $C>0$ such that  for all $T>0$
\[
 \mathfrak{c}(T,g)\leq \exp\big(CT^{\someLetter}\big).
\]
In particular, $G\sub{\someLetter}$ is dense in $\ms{G}^\nu$.
\end{theorem}

Bounds on $\mathfrak{c}(T,g)$ have a long history in the literature. For Baire generic metrics, a complete picture of the \emph{non-quantitative} behavior of $\mathfrak{c}(T,g)$ is available.  The bumpy metric theorem~\cite{An:82,Ab:70} can be used to show that for Baire generic smooth metrics, $g$, and all $T>0$, $\mathfrak{c}(T,g)<\infty$ (see also~\cite{KlTa:72}). Furthermore, Hingston~\cite{Hi:84} showed that  $\lim_{T\to \infty}\mathfrak{c}(T,g)=\infty$ {for a Baire-generic set of metrics}. Petkov and Stojanov have studied similar properties of closed billiards in generic domains in $\re^d$~\cite{St:87, PeSt:87,PeSt:87b}.

As with remainders in the Weyl law, quantitative estimates on $\mathfrak{c}(T,g)$ are much more subtle. Only recently, Contreras~\cite{Co:10} showed that {for $g$ in an open dense subset of $\ms{G}^2$ there is $c>0$ such that}
\begin{equation}
\label{e:lowerCount}
\log \mathfrak{c}(T,g)\geq cT-c.
\end{equation}
{One can check that this lower bound is optimal for any dense set of metrics (so certainly for any reasonable notion of a typical metric, including predominance).} Indeed, in the case of manifolds with {negative curvature}, the works of Margulis \cite{Ma:69} and Bowen~\cite{Bo:72} show that, {for such a metric $g$, there is $\alpha>0$ such that}
$$
\lim_{T\to \infty}\frac{1}{T}\log\mathfrak{c}(T,g)=\alpha.
$$
 In particular, this shows that there are open sets of metrics such that $\mathfrak{c}(T,g)$ grows exactly exponentially and hence that~\cite{Co:10} gives a complete picture for lower bounds on $\mathfrak{c}(T,g)$ for {typical} $\mc{C}^2$ metrics. (See also~\cite{Kn:98} for the case of compact rank 1 manifolds.)

On the other hand, as far as the authors are aware, Theorem~\ref{t:predominantLengths} is the first quantitative upper bound on $\mathfrak{c}(T,g)$ for `typical' metrics. One reason that lower bounds on $\mathfrak{c}(T,g)$ are well understood, but upper bounds are not is that one can find a structure, called a hyperbolic basic set (see~\cite{Co:10}), which is stable  under perturbation and guarantees the lower bound~\eqref{e:lowerCount}. Moreover, the existence of such a set can be guaranteed by studying the Poincar\'e maps associated to genuinely periodic orbits. 

Unfortunately, no such structure for upper bounds exists and one must understand not only Poincar\'e maps of periodic orbits but also those of near periodic orbits. Indeed, {although there is no rigorous proof at present,} there is strong evidence that there is no quantitative upper bound on $\mathfrak{c}(T,g)$ which is Baire generic. For instance, in the case of diffeomorphims, Kaloshin~\cite{Ka:00} showed that there is no growth rate for the number of periodic points that is Baire generic. (See Section~\ref{s:predominant} for an additional heuristic discussion for lack of Baire genericity.)

\subsection*{Non-degeneracy of nearly periodic orbits}
Our main theorem, which implies both Theorems~\ref{t:predominantWeyl} and~\ref{t:predominantLengths}, controls how close two periodic orbits may be for a predominant set of metrics and can be used, for instance, to control the volume of nearly periodic orbits (see Section~\ref{s:nearlyClosed}). 
The result will be stated in terms of how close $d\varphi_t^g$ may be to the identity for this set of metrics, {where $\varphi_t^g$ is the geodesic flow for the metric $g$  at time $t$ acting on $S^*\!M$}.

To understand how this is connected to the distance between periodic orbits, let $H\sub{|\xi|_g}$ denote the Hamiltonian vector field associated to $|\xi|_g$, and $\varphi_t^g:=\exp(tH\sub{|\xi|_g}):S^*\!M\to S^*\!M$. Observe that if $\rho$ is a $t$ periodic point (i.e. $\varphi_t^g(\rho)=\rho$ for some $t>0$), and $v\in T_\rho S^*\!M/\mathbb{R}H\sub{|\xi|_g}(\rho)$ is not in the kernel of $I-d\varphi_t^g$, then any perturbation of the initial point $\rho$ in the direction of $v$ will not be periodic with period near $t$.  (See Figure~\ref{f:nonDegeneracy} for a schematic of such an orbit.)  In particular, if the map
$$
I-d\varphi_t^g: T_\rho S^*\!M/\mathbb{R}H\sub{|\xi|_g}(\rho)\to T_\rho S^*\!M/\mathbb{R}H\sub{|\xi|_g}(\rho)
$$
is invertible, then there are no vectors in this kernel and hence every small perturbation of $\rho$ other than those along $H\sub{|\xi|_g}(\rho)$ will produce a point which is not periodic with period near $t$.

Motivated by this, we say that a $t$ periodic point, $\rho$, is \emph{non-degenerate} if 
$$
I-d\varphi_t^g: T_\rho S^*\!M/\mathbb{R}H\sub{|\xi|_g}(\rho)\to T_\rho S^*\!M/\mathbb{R}H\sub{|\xi|_g}(\rho)
$$
is invertible. The famous bumpy metric theorem~\cite{Ab:70,An:82} states that, for a Baire generic set of smooth metrics, every periodic trajectory is non-degenerate. In particular, this implies the finiteness of the set of closed geodesics of any bounded length.

Since we are interested in a quantitative version of non-degeneracy for both periodic geodesics and nearly periodic geodesics, we need to introduce a few concepts to make precise statements about non-degeneracy.

\begin{definition}[{returning points}]\label{d:periodic-g} Let $g\in\ms{G}^\nu$
and $\beta>0$. 
For {$t\in \R$ and  $\rho\in S^*\!M$} we write 
\[
\rho\in\Rec(t,\beta,g) \qquad \text{if}\qquad d(\varphi_{t}^{g}(\rho),\rho)<\beta.
\]
In this case, we say that \emph{$\rho$ is $\beta$-returning for
$g$ at time $t$}.
\end{definition}


We also recall that the geodesic flow is a contact flow on $S^*\!M$ and thus there is a natural smooth decomposition of $T\!S^*\!M$ preserved by the geodesic flow (see e.g.~\cite{Pa:99}).  In particular, $\xi dx|_{TS^*M}$ is a contact form for $S^*\!M$ with the geodesic flow as its Reeb flow. Thus, for all $\rho \in S^*\!M$,
$$
T_\rho S^*\!M=\ms{H}(\rho)\oplus \mathbb{R}H\sub{|\xi|_g}(\rho),
$$
where $\ms{H}(\rho):=\ker (\xi dx|_{T_\rho S^*\!M})$ and $\oplus$ denotes the direct sum, and we have
$$
d\varphi^g_t (\ms{H}(\rho))= \ms{H}(\varphi^g_t(\rho)).
$$

Since we work with nearly periodic orbits, we need to identify the tangent spaces at $\rho$ and $\varphi_t^g(\rho)$ when they are close. 
Let $g \in \ms{G}^{{\nu}}$ and $U\subset S^*\!M\times S^*\!M$. We say that $\mc{W}^U=\{\mc{W}_{\rho_2,\rho_1}:\;(\rho_2,\rho_1)\in U\}$ is \emph{a family of transition maps for $g$ on $U$} if for each $(\rho_2,\rho_1)\in U$ the map $\mc{W}_{\rho_2,\rho_1}$ is an invertible linear transformation,
\begin{equation}\label{e:transition}
\begin{gathered}
\mathcal{W}_{\rho_{2},\rho_{1}}:T_{\rho_{1}}S^*\!M\to T_{\rho_{2}}S^*\!M,\qquad (\rho_2,\rho_1)\mapsto \mc{W}_{\rho_2,\rho_1}\; \text{is Lipschitz}\\
\mathcal{W}_{\rho_1,\rho_1}=I,\qquad \mathcal{W}_{\rho_{2},\rho_{1}}H\sub{|\xi|_g}(\rho_{1})=H\sub{|\xi|_g}(\rho_{2}),\qquad \mc{W}_{\rho_2,\rho_1}\ms{H}(\rho_1)=\ms{H}(\rho_2).
\end{gathered}
\end{equation}
Here, by asking that $(\rho_2,\rho_1)\mapsto\mc{W}_{\rho_2,\rho_1}$ be Lipschitz, we mean that for any choice of coordinates $\psi_i:W_i\to V_i\subset\mathbb{R}^{2d-1}$  near $\rho_i$ the map
$$
W_1\times W_2\ni (x_1,x_2)\mapsto d\psi_2|_{\rho=\psi_2^{-1}(x_2)}\circ \mc{W}_{\psi_2^{-1}(x_2),\psi_1(x_1)}\circ d(\psi_1^{-1}(x))|_{x=x_1}\in \mathbb{GL}(2d-1)
$$
is a Lipschitz family of matrices.


We say that a collection
\emph{$\mathcal{W}=\{\mc{W}^{U_i}\}_{i=1}^N$ is a family of transition maps {for $g$}} if $U_i\subset S^*\!M\times S^*\!M$ is open,
$$
\{(\rho,\rho)\,:\, \rho\in S^*\!M\}\subset \bigcup_{i=1}^NU_i,
$$
and, for each $i$, $\mc{W}^{U_i}$ is a family of transition maps for $g$ on $U_i$.
 We say $\mc{W}$ is a \emph{$\beta_0$-family of transitions maps} if for each pair $(\rho_2,\rho_1)$ with $d(\rho_2,\rho_1)<\beta_0$, there is $i$ such that $(\rho_2,\rho_1)\in U_i$.


\begin{remark}
\label{r:fixedNorm}
It will be convenient throughout the text to have a fixed reference metric on $M$. For this, we choose some $g\sub{f}\in \ms{G}^\nu$ and whenever we refer to a norm  $|\cdot|$, on $\TM$ or $TM$, it is the one induced by $g\sub{f}$. We will, in particular, use this metric to define $\mc{C}^{\nu'}$ norms for $\nu'\leq \nu$. 
\end{remark}

We are now in a position to define non-degeneracy of a nearly periodic orbit.
\begin{definition}[{non-degenerate points}]\label{d:nonDegenerate-g} Let
$g \in \ms{G}^{\nu}$, $\beta_{0}>0$ and $\mc{W}=\{\mathcal{W}^{U_{i}}\}_{i=1}^N$ be a $\beta_0$-family of transition maps for $g$.
Let $0<\beta<\beta_{0}$, $t\in\R$, 
$\rho\in\Rec(t,\beta,g)$, and $\alpha>0$. We write  $\rho\in\ND(t,\alpha,(g,\mathcal{W}))$ {if for every $i\in \{1,\dots N\}$ such that $(\varphi_t^g(\rho),\rho)\in U_i$,}
\[
|\HpProj\sub{\varphi_t^g(\rho)}\big(\mathcal{W}^{U_{i}}_{\varphi_{t}^{g}(\rho),\rho}-d\varphi_{t}^{g}\big)v|\geq \alpha|\HpProj_{\rho}v|, \quad v\in T_{\rho}\SM.
\]
Here, $\HpProj_{\rho}:T_{{\rho}}\SM\to T_{{\rho}}\SM/\mathbb{R}H\sub{p}(\rho)$
denotes the natural projection map and, by an abuse of notation, $|\cdot|$
denotes the norm induced by the metric $g\sub{f}$ fixed above in Remark \ref{r:fixedNorm}.
In this case we say \emph{$\rho$
is $\alpha$ non-degenerate for $(g,\mathcal{W})$ at time $t$}.
 (See Figure~\ref{f:nonDegeneracy} for an example of a non-degenerate orbit.) \end{definition}
 
 \begin{remark}
 Although Definition~\ref{d:nonDegenerate-g} depends on the choice of the metric $g_f$, note that the norm induced by any other $g_f'$ is comparable to that induce by $g_f$.
 \end{remark}



The main theorem of this article shows that there is a predominant set of metrics such that \emph{every} sufficiently returning geodesic is non-degenerate with the degree of non-degeneracy depending explicitly on the length of the trajectory. As far as the authors are aware, this theorem is the first quantitative estimate on non-degeneracy of orbits for typical metrics.

\begin{theorem}\label{t:predominantR-ND} Let $\nu\geq \blue{3}$, $M$
be a compact $\mc{C}^{\nu}$ manifold of dimension $d$ without boundary, and $\someLetter_{\nu}$ as in~\eqref{e:someLetter}.
Then for every $\someLetter>\someLetter_\nu$ there is a predominant set $G\sub{\someLetter}\subset\ms{G}^\nu$ such that for all $g {\in G\sub{\someLetter}}$ and every family of transition maps $\mc{W}$ for $g$ there are {$C,c>0$} such that 
\[
{\begin{gathered}\Rec\big(t,\n(t),g\big)\subset\ND\big(t,\n(t),(g,\mathcal{W})\big)\quad\text{for $t>c$}\end{gathered}}
\]
 where $\n(t):= C^{-C(t+1)^{\someLetter}-1}$, and 
 $$
d(\varphi_t^g(\rho),\rho)\geq c|t|\quad\text{ for $|t|\leq c$}.
$$
 In particular, $G\sub{\someLetter}$ is dense in $\ms{G}^\nu$.
\end{theorem}

\begin{remark}\label{r:almostDone}
The reason for the growth of $\someLetter_\nu$ as $\nu\to \infty$ in Theorems~\ref{t:predominantWeyl},~\ref{t:predominantLengths}, and~\ref{t:predominantR-ND} comes from the fact that we make perturbations to the metric at increasingly small scales as the length of trajectories goes to infinity. Because of this, the size of these perturbations in $\mc{C}^\nu$ grows as $\nu\to \infty$ and this in turn results in weaker non-degeneracy statements. Moreover, with $f(t):[0,\infty)\to [0,\infty)$ growing faster than any polynomial in $t$, if one replaces $\n(t)$ by $Ce^{Cf(t)}$ in Theorem~\ref{t:predominantR-ND}, then one can work in $\mc{C}^\infty$. That is, we obtain predominance in the $\ms{G}^\infty$ topology.
\end{remark}

\begin{remark}The notion of predominance (see Section~\ref{s:predominant}) involves using certain families of perturbations to probe the space of metrics. In this article, the predominance in Theorem~\ref{t:predominantR-ND} involves probing with the families of perturbations described in Section~\ref{s:perturbedMetrics}. However, one may wonder whether one can probe with other families of perturbations (e.g. conformal perturbations) and still obtain Theorem~\ref{t:predominantR-ND} for that family of probes.  In fact, in Sections~\ref{s:basicPerturb} and~\ref{s:theMadness}, we prove Proposition~\ref{p:thePredominantMeat} which gives a result analogous to Theorem~\ref{t:predominantR-ND} under the assumption that a family of perturbations of metrics satisfies some abstract assumptions (see Definitions~\ref{ass:1} and~\ref{ass:2}). The type of perturbations used to probe the space of metrics is then tied to the family of perturbations satisfying our abstract assumptions. It is only in Section~\ref{s:perturbedMetrics} that we construct such a family of metric perturbations and it is likely that many other families of perturbations suffice. However, we do not pursue this here. \end{remark}

\subsection*{Outline of the paper}
In Section~\ref{s:predominant}, {we define predominance and discuss the reasons for introducing this notion: predominance is in some sense more `typical' than Baire generic and our results are unlikely to hold for a  Baire generic set of metrics}. Then, in Section~\ref{s:theoremProofs}, we use Theorem~\ref{t:predominantR-ND} to prove Theorems~\ref{t:predominantWeyl} and~\ref{t:predominantLengths}. Before starting the proof of Theorem~\ref{t:predominantR-ND}, we give a detailed description of the ideas used in Section~\ref{s:outline}. 

To begin the proof of Theorem~\ref{t:predominantR-ND}, we study volumes of relevant sets of symplectic matrices in Section~\ref{s:symplectic}. In Section~\ref{s:iterates}, we review some basic estimates for returning points and introduce the notion of a chain of symplectomorphisms associated to a flow, as well as that of a well-separated set. 
The notion of a well-separated set replaces that of a Poincar\'e section when a global section for the flow is not available. Although this requires some technical work, it does not substantially change the proof of the main result and the reader may wish to first assume that there is a global section and replace chains of symplectomorphisms by the standard Poincar\'e map for that section.

Section~\ref{s:basicPerturb} defines sufficient assumptions on a family of metric perturbations to guarantee Theorem~\ref{t:predominantR-ND}. Under these assumptions, we study the volume of perturbations that produce degenerate periodic points of a given length. Section~\ref{s:theMadness} then proves an analog of Theorem~\ref{t:predominantR-ND} by implementing a delicate induction argument.

Finally, in Sections~\ref{s:perturbedMetrics} and~\ref{s:poincare}, we construct a family of metric perturbations which satisfies our technical assumptions and we prove Theorem~\ref{t:predominantR-ND} in Section~\ref{s:theProof}.

Appendix~\ref{s:control} contains some elementary control estimates from ODE theory used to construct the perturbations of metrics in Section~\ref{s:perturbedMetrics}.


\subsection*{Index of Notation} \ \\[1em]
\hspace{-1cm}
\begin{tabular*}{\textwidth}{|l@{\extracolsep{\fill}}l|ll|ll|}
\hline
$\Rec(t,\beta,g)$ 		&Def.~\ref{d:periodic-g}			&$\mc{W}_{\rho_2,\rho_1}$		&\eqref{e:transition}			&$\ND(t,\alpha(g,\mc{W}))$			&Def.~\ref{d:nonDegenerate-g} \\
probing map			&Def.~\ref{e:probingMaps}		&predominant					&Def.~\ref{d:predominant}		&$(\beta,q)$ non-degenerate		&\eqref{e:q-Nondegenerate}\\
$\mc{M}\sub{Y}(V,s)$	&Def.~\ref{d:coveringNumber}		&$\Rec\sub{\fuInj}(n,\delta,g)$		&Def.~\ref{d:returningKappa}&$\Sim\sub{\!\fuInj}(n,\alpha,g)$			&Def.~\ref{d:simple}\\
well-separated			&Def.~\ref{d:wellSeparated}		&$\mc{P}\sub{\mc{I}}^{(n)}[g]$		&Def~\ref{d:chain}			&$\bm{T}\sub{\mc{I}}^{(n)}[g](\rho)$		&Def.~\ref{d:tChain}\\
$\ND\sub{\fuInj}(n,\beta,g)$&Def.~\ref{d:ND}				&$\ND_{q,_{\fuInj}}(n,\beta,g)$		&Def.~\ref{d:NDq}			&$\L$							&\eqref{e:LDef}\\
Good perturbation		&Def.~\ref{ass:1}				&Admissible pairs				&Def.~\ref{ass:2}			&$F_J^{\bm\Rind_\e,\bm\delta_\e}$		&\eqref{e:F_J}\\
$F_\infty^{\bm\Rind_\e,\bm\delta_\e}$ &\eqref{e:F_infty}		&$\gap$						&\eqref{e:omega}			&$\gamma_j$						&\eqref{e:gamma_j}\\
$\beta_{j,j}$			&\eqref{e:betas}				&$\alpha_{j}, \alpha_{j,\ell}, \beta_{i,j,\ell},\tilde{\beta}_{j,\ell}, s_\ell$ &\eqref{e:alphas} &$\Phi_{\rho}^{g_\star}$		&\eqref{e:picky}\\
$g_\sigma$			&\eqref{e:perturbationDef}		&$\Delta_\sigma,\tilde{\Delta}_\sigma$&\eqref{e:deltaSigma}		&$\Psi_{\zeta_0}^{g_0}(\sigma),\,\Xi_i(\sigma,\zeta_0,g_0)$&\eqref{e:PsiMetric}\\
\hline
\end{tabular*}
\smallskip

\noindent\textsc{Acknowledgements.} The authors would like to thank Steve Zelditch, Maciej Zworski, and Leonid Parnovski for comments on an early draft of this paper, Gabriel Paternain for pointing them to the work of Kaloshin--Hunt~\cite{KaHu:07} and Kaloshin~\cite{Ka:00} at the beginning of their work on this article, and Peter Sarnak for helpful discussions on the problem. The authors are grateful to Gerhard Knieper, Amie Wilkinson, Keith Burns, and Nancy Hingston for comments on the existing literature for bounds on numbers of closed geodesics. \blue{Thanks also to an anonymous pointing an error in an earlier version as well as for careful reading and many detailed that improved exposition. We also acknowledge the use of ChatGPT 5.5Pro which helped us produce some of the results in Appendix~\ref{a:linear}. } Y.C. was supported by the Alfred P. Sloan Foundation, NSF CAREER Grant DMS-2045494, and NSF Grant DMS-1900519.  J.G. is grateful to the EPSRC for partial funding under Early Career Fellowship EP/V001760/1 and Standard Grant EP/V051636/1.

\section{Predominance in Banach Spaces}
\label{s:predominant}

The main goals of this article are to give upper bounds on the number of closed geodesics of length $T$ and upper bounds for remainders in the Weyl law for a predominant set of metrics on a compact manifold $M$. Before proceeding to define our notion of predominance on Banach spaces in Section~\ref{s:pred}, we discuss several other available notions and motivate our choice of definition in Section \ref{s:full}. Finally, in Section \ref{s:Baire} we explain why the notion of Baire genericity is not well-suited for our purposes.

\subsection{Existing notions of `full measure'}\label{s:full} 
The main difficulty {in defining a concept analogous full measure in an infinite dimensional space}, like the space of Riemannian metrics over a given manifold, is that  there are no non-trivial, translation invariant, Borel measures. Several possible fixes for this problem have been introduced in the literature. We mention here the concepts of prevalence~\cite{HuSaYo:92} which uses an underlying linear structure and metric prevalence~\cite{Ka:97} which does not require such a structure. These two notions have three important properties
\begin{equation}
\label{prevList}
\begin{split}
&\text{(1) A prevalent set is dense.}\\
&\text{(2) The intersection of prevalent sets is prevalent.}\\
&\text{(3) If $G\subset \mathbb{R}^n$ is prevalent, then $G$ has full Lebesgue measure.}
\end{split}
\end{equation}

Although quite flexible, as far as we are aware, the notion of metric prevalence has not proved useful in studying \emph{quantitative} statements such as the growth of the number of periodic orbits of length $T$. Because of this, we focus on the notion of prevalence from~\cite{HuSaYo:92} which has appeared before in this type of application. In fact, for $\nu \geq 1$, \cite{KaHu:07} proves that there is a prevalent set {$G \subset \mc{C}^{\nu}(I;I)$} of diffeomorphisms on the interval $I=[0,1]$, such that for all $\e>0$ and $f\in G$ there is $C>0$ such that 
$$
\#\{ x\in I\,:\, f^n(x)=x\}\leq Ce^{C n^{1+\e}}.
$$ 

Given a Banach space $\ms{G}$,  a Borel set $G\subset \ms{G}$ is said to be \emph{prevalent} if there is a Borel measure $\mu$ and a compact set $K\subset \ms{G}$ such that 
$$
0<\mu(K)<\infty,\qquad G+g\text{ has full }\mu\text{ measure for all }g\in \ms{G}.
$$
Usually, when one shows that a set $G$ is prevalent, it is convenient to construct a probe $\Sigma\subset \ms{G}$ which carries the measure $\mu$. In other words, if $G$ is prevalent, there are  $\Sigma\subset \ms{G}$,  a smooth map 
\begin{equation}
\label{e:oldF}
F:\ms{G}\times \Sigma\to \ms{G},\qquad F(g,\sigma):=g+\sigma
\end{equation}
and a Borel measure, $\mu\sub{\Sigma}$, on $\Sigma$ such that for all $g \in \ms{G}$
\begin{equation}
\label{e:prelimMeasure}
\mu\sub{\Sigma}(\sigma\in \Sigma:\; F(g,\sigma)\in G^c)=0.
\end{equation}

Because we will be working in an open subset of a Banach space (the space metrics inside the space of symmetric 2-tensors) we would like a notion which does not rely on the fact that the space is linear. To do this, we generalize the type of functions allowed in~\eqref{e:oldF} and slightly weaken~\eqref{e:prelimMeasure}. 
\begin{remark}
Under definition~\eqref{e:oldF}, the submanifolds $F(g,\Sigma)$ are simply translations of a fixed submanifold, $\Sigma$. This presents us with two issues. First, since we work in an open subset, $\ms{G}$, of a Banach space, for some $g\in \ms{G}$, $F(g,\Sigma)\nsubseteq \ms{G}$; i.e. we fall out of $\ms{G}$. Second,  we would like a more flexible notion where the probes $F(g,\Sigma)$, can be tailored, to some extent, to $g$ and hence, whose structure can be determined locally.  See also \eqref{e:fail}.
\end{remark}

We now define the notion of predominance on an open subset of a Banach space $\ms{G}$.

\subsection{Predominant sets} \label{s:pred}
In this section we introduce the notion of predominance.
Let $\ms{G}$ and $\ms{G}'$ be open subsets of the Banach spaces $(\ms{B}, \|\,\|\sub{\ms{B}})$ and $(\ms{B}', \|\,\|\sub{\ms{B}'})$ respectively, 
\begin{equation}\label{e:theMatrix}
 \ms{G}\subset \ms{B} , \qquad \ms{G}' \subset \ms{B}',  
\end{equation}
and such that $\ms{G}\subset \ms{G}'$ and $\ms{B}\subset \ms{B}'$
via a continuous embedding,  with $\ms{B}$ dense in $\ms{B}'$ and $\ms{G}$ dense in $\ms{G}'$. Let $\iota:\ms{G}\to \ms{G}'$ be the natural inclusion map.  

\begin{remark}
When $\ms{O}\subset \ms{B}$ is open, we will say that a subset $K\subset \ms{O}$ is \emph{bounded} if it is bounded as a subset of $\ms{B}$ and $d\sub{\ms{B}}(K,\partial \ms{O})>0$, where $d\sub{\ms{B}}$ is the distance induced by $\| \cdot\|\sub{\ms{B}}$.
\end{remark}

{The space ${\ms{G}}$ will be probed by perturbations indexed by a parameter $\bm \sigma \in \Si{N_\e}$ for $\Si{N_\e}$ as follows.} Given $\L\in\mathbb{N}$ and $\{N_{\e}\}_{\e>0}\subset\mathbb{N}\cup\{\infty\}$
 let 
 \begin{equation}\label{e:SigmaE}
 \bm{\Sigma}\sub{\!N_\e}:=\prod_{j=1}^{N_{\e}}B_{\mathbb{R}^{\L}}(0,1).
 \end{equation}
We endow this space with the sup-norm $\| \,\cdot\, \|_{\ell^\infty}$.

\begin{definition}[{probing maps}]\label{e:probingMaps}
We say that a collection $\ms F=\{(F_{\e},N_{\e})\}_{\e}$, with $\e\in(0,1)$,
is a family of $\ms{G}'$-\emph{probing maps {for $\ms{G}$}} if $N_\e\leq \dim \ms{B}$, {there is a collection of closed bounded sets $\ms{G}_\e\subset\ms{G}$ with $\cup_\e \ms{G}_\e=\ms{G}$,  { $\ms{G}_{\e_2}\subset\ms{G}_{\ep_1}$ for $\ep_1<\ep_2$}, and such that } $F_{\e}:\ms{G}_\e\times\Si{N_\e}\to\ms{G}$ 
is a continuous map with 
\begin{equation}\label{e:typicalness}
F_\e(g,0)=g, \qquad g\in \ms{G}_{{\e}},
\end{equation}
 and the following hold. For all $K\subset \ms{G}$ bounded, 
{ \begin{equation}
 \label{e:closeToId}
 \lim_{\e \to 0^+}\sup_{g\in K}\sup_{\bm \sigma {\in \Si{N_\e}}}\|F_\e(g,\bm\sigma)-g\|_{{\ms{B}}}=0,
 \end{equation}}
\begin{equation}\label{e:bddDiff}
\begin{gathered}
\lim_{\e\to 0^+}\sup_{g\in K}\sup_{\substack{\bm \sigma_1,\bm \sigma_2\in \Si{N_\e}\\\bm\sigma_1\neq\bm\sigma_2}}\frac{\|{F}_{\e}({g},\bm\sigma_1)-{F}_{\e}({g},\bm \sigma_2)\|\sub{\mathscr{B}'}}{ \|\bm\sigma_1-\bm\sigma_2\|_{\ell^\infty}}=0.
\end{gathered}
\end{equation}
In addition, the map $\tilde{F}_\e:=\iota\circ F_\e:\ms{G}_\e\times\Si{N_\e}\to\ms{G}'$ is Lipschitz,
and satisfies that for all $g\in \ms{G}_\e$ the Frechet derivative, $D_g\tilde{F}_\e$, of $\tilde{F}_\e$ in $g$ exists and for all $K\subset \mathscr{B}$ bounded 
 \begin{equation}
 \label{e:bddDiff0}
\lim_{\e\to 0^+}\sup_{g\in K,\bm\sigma\in\Si{N_\e}} \| D_g\tilde{F}_\e|_{(g,\bm \sigma)} -\Id\|\sub{\ms{B}'\to \ms{B}'}=0,
 \end{equation}

\end{definition}

\begin{remark}
We typically imagine that $N_\e$ is non-decreasing as $\e\to 0^+$ so that probing maps become more dispersed as $\e\to 0$. In fact, in our applications, $N_\e$ will often be identically equal to $\infty$. However, when $\ms{G}$ is finite dimensional it is reasonable to assume that $N_\e \leq \dim \ms{G}$  since, otherwise, one would be putting a very diffuse measure on a finite dimensional space. 

In addition, the spaces $\ms{G}_\e$ allow us to define probes only in bounded subsets of $\ms{G}$ provided that, as $\e\to 0^+$, these subsets exhaust $\ms{G}$. 

\end{remark}


 In our treatment, $\ms{G}$ and $\ms{G}'$ will be the spaces of $\mc{C}^\nu$ and $\mc{C}^{\nu-1}$ {Riemannian metrics on a given manifold, while $\ms{B}$ and $\ms{B}'$ will be the spaces of $\mc{C}^\nu$ and $\mc{C}^{\nu-1}$ symmetric two-tensors.} We need to refer to $\ms{G}'$ because our probing maps will typically \emph{not} be Frechet differentiable as a map from $\mc{C}^{\nu}$ to $\mc{C}^{\nu}$ and instead, the Frechet derivative of the map will make sense as a map from $\mc{C}^\nu$ to $\mc{C}^{\nu-1}$ and will extend as in~\eqref{e:bddDiff0} to a map from $\mc{C}^{\nu-1}$ to $\mc{C}^{\nu-1}$. 

  We next discuss briefly the roles of each piece of the definition of a family of probing maps. The assumption \eqref{e:typicalness} is crucial to know that $F_\e$ probes all of $\mathscr{G}_\e$ 
 and does not avoid any open sets.
To understand~\eqref{e:bddDiff0} and~\eqref{e:bddDiff}, recall~\cite{Mi:97} that one can construct an example of a foliation of the unit square by \emph{analytic} leaves $\{W_{g}\}_{g\in[0,1]}$, such that the map $g\mapsto W_{g}$ is continuous (but not differentiable), and there is a set $E\subset[0,1]^2$ with full measure such that $\#\{ W_{\bm\sigma}\cap E\}\leq 1$. Thus,  assumptions~\eqref{e:bddDiff0} and~\eqref{e:bddDiff}, which imply that the map $F_\e$ has reasonable regularity properties (both as a function of $g$ and $\e$), are crucial in proving that item (3) in list~\eqref{prevList} holds for predominant sets.

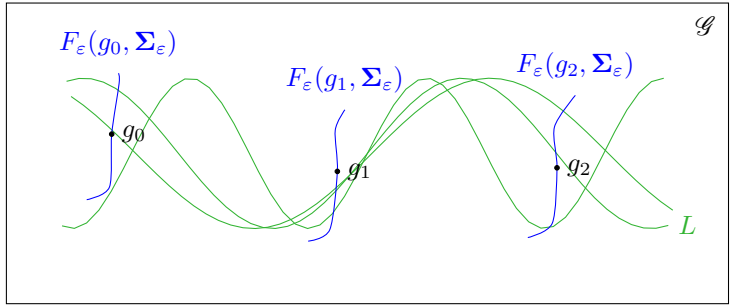
\begin{figure}
\begin{tikzpicture}
\def \w{4}
\def \shift{.3};
\def \eps{.9};
\def \del{.2};
\def \delb{.4};
\def \n{2};
\def \g{1.3};
\draw  (-1.2*\w,-.5*\w)--(1.2*\w,-.5*\w)--(1.2*\w,.5*\w)--(-1.2*\w,.5*\w)--cycle;

\def \a{.3};
\def \b{-.3};
\def \c{0};
\draw node at (1.2*\w-\shift,.5*\w-\shift){$\mathscr{G}$};

\def \xs{\del*\a};

\draw[verde, domain=-\w:\w,samples =50] plot (\x+\xs,{sin( (\x) r)}){};
\def \xs{\del*\b}
\draw[verde, domain=-\w:\w,samples =50] plot (\x+\xs,{sin( 2*(\x+\xs) r)}){};
\pgfmathparse{\del* (random()-1/2)}
\def \xs{\c}
\draw[verde, domain=-\w:\w,samples =50] plot (\x+\xs,{sin( 1.25*(\x+\xs) r)}){} node[right]{$L$};

\def \x{0};
\def \raa{.139};
\def \rba{.075};
\def \rca{.224};
\def \rda{-.45};
\def \xsa{-.35};
\def \ysa{.41};

{
\def \ra{\del*\raa};
\def \rb{\del*\rba};
\def \rc{\del*\rca};
\def \rd{\del*\rda};
\def \xs{\xsa};
\def \ys{\delb*\ysa};

\draw[color=blue,shift={({(\xs+(\g*\x-\n/2)/(\n))*\w},{(\ys-.1)*\w})}]   plot[smooth] coordinates {({-.4*\eps+\ra},{-\eps+\ra}) ({-.05*\eps+\rb},{-.8*\eps+\rb}) (0,0) ({.06*\eps+\rc},{.7*\eps+\rc}) ({.2*\eps+\rd},{\eps+\rd}) } ;
\draw [shift={({(\xs+(\g*\x-\n/2)/(\n))*\w},{(\ys-.1)*\w})},blue]({.2*\eps+\rd},{1.1*\eps+\rd}) node[above]{$F_\e(g_{\x},\bm \Sigma_\e)$};
\filldraw[shift={({(\xs+(\g*\x-\n/2)/(\n))*\w},{(\ys-.1)*\w})}] (0,0) circle(.03);
\draw[shift={({(\xs+(\g*\x-\n/2)/(\n))*\w},{(\ys-.1)*\w})}] (0,0) node[right]{$g_{\x}$};
};

\def \x{1};
\def \raa{-.14};
\def \rba{-.28};
\def \rca{-.41};
\def \rda{-.4};
\def \xsa{-.25};
\def \ysa{.1};
{
\def \ra{\del*\raa};
\def \rb{\del*\rba};
\def \rc{\del*\rca};
\def \rd{\del*\rda};
\def \xs{\xsa};
\def \ys{\delb*\ysa};

\draw[color=blue,shift={({(\xs+(\g*\x-\n/2)/(\n))*\w},{(\ys-.1)*\w})}]   plot[smooth] coordinates {({-.4*\eps+\ra},{-\eps+\ra}) ({-.05*\eps+\rb},{-.8*\eps+\rb}) (0,0) ({.06*\eps+\rc},{.7*\eps+\rc}) ({.2*\eps+\rd},{\eps+\rd}) } ;
\draw [shift={({(\xs+(\g*\x-\n/2)/(\n))*\w},{(\ys-.1)*\w})},blue]({.2*\eps+\rd},{1.1*\eps+\rd}) node[above]{$F_\e(g_{\x},\bm \Sigma_\e)$};
\filldraw[shift={({(\xs+(\g*\x-\n/2)/(\n))*\w},{(\ys-.1)*\w})}] (0,0) circle(.03);
\draw[shift={({(\xs+(\g*\x-\n/2)/(\n))*\w},{(\ys-.1)*\w})}] (0,0) node[right]{$g_{\x}$};
};

\def \x{2};
\def \raa{-.14};
\def \rba{-.09};
\def \rca{-.44};
\def \rda{.32};
\def \xsa{-.17};
\def \ysa{.13};

{
\def \ra{\del*\raa};
\def \rb{\del*\rba};
\def \rc{\del*\rca};
\def \rd{\del*\rda};
\def \xs{\xsa};
\def \ys{\delb*\ysa};

\draw[color=blue,shift={({(\xs+(\g*\x-\n/2)/(\n))*\w},{(\ys-.1)*\w})}]   plot[smooth] coordinates {({-.4*\eps+\ra},{-\eps+\ra}) ({-.05*\eps+\rb},{-.8*\eps+\rb}) (0,0) ({.06*\eps+\rc},{.7*\eps+\rc}) ({.2*\eps+\rd},{\eps+\rd}) } ;
\draw [shift={({(\xs+(\g*\x-\n/2)/(\n))*\w},{(\ys-.1)*\w})},blue]({.2*\eps+\rd},{1.1*\eps+\rd}) node[above]{$F_\e(g_{\x},\bm \Sigma_\e)$};
\filldraw[shift={({(\xs+(\g*\x-\n/2)/(\n))*\w},{(\ys-.1)*\w})}] (0,0) circle(.03);
\draw[shift={({(\xs+(\g*\x-\n/2)/(\n))*\w},{(\ys-.1)*\w})}] (0,0) node[right]{$g_{\x}$};
};

    \end{tikzpicture}
    \caption{\label{f:predominant} An example of a probing family $F_{\e}$ at a fixed $\e$ and a thin set $L$. One way to think of a probing map is that, to each point $g\in \ms{G}$, we attach a probe $F_\e(g,\Si{N_\e}^\aleph)$. These probes are asymptotically translates of one another in the limit $\e\to 0$ and, moreover, are contained in a small ball around $g$. A set $L$ is thin if its intersection with each of these probes has vanishing measure in the limit $\e\to 0$.}
\end{figure}
 
 {In what follows we work with the  measure $m\sub{\Si{N_\e}}$ on $\Si{N_\e}$
defined to be the product measure 
\begin{equation}
\label{e:stevesMeasure}
m\sub{\Si{N_\e}}:=\otimes_{j=1}^{N_{\e}}m,\qquad\qquad m:=\frac{m\sub{\mathbb{R}^{\L}}|_{B(0,1)}}{m\sub{\mathbb{R}^{\L}}\!(B(0,1))},
\end{equation}
where $m\sub{\mathbb{R}^{\L}}$ denotes the Lebesgue measure on $\mathbb{R}^{\L}$.} {Note that $m\sub{\bm \Sigma_\e}(\bm \Sigma_\e)=1$.} We choose to work with $m\sub{\Si{N_\e}}$ because we want the probability measure $\mu_{\e,g}=(F_\e(g,\cdot))_*m\sub{\Si{N_\e}}$ to uniformly sample metrics from the submanifold $F_\e(g,\Si{N_\e})$ (see also Remark~\ref{r:heuristic}).

\begin{definition}[{predominant sets}]\label{d:predominant} Let $\ms{G}$ be an open subset of a Banach
space $\ms{B}$, and $\ms{F}:=\{(F_{\e},N_{\e})\}_{\e>0}$ be a family of $\mathscr{G}'$ probing maps for $\ms{G}$. We say a set $L\subset\ms{G}$
is $\ms F$-\emph{thin} if {for all $K\subset \ms{G}$ bounded} there is a Borel subset {$L_0 \subset \ms{G}$} such that $L\subset { L_0}$
and for every $g\in {K}$, and $\e>0$, there exists  an $m\sub{\bm\Sigma_{\e}}$-measurable set $S_{g,\e}{\subset \Si{N_\e}}$ such that
\begin{equation}
{\{\bm \sigma \in \Si{N_\e}:\; F_\ep(g, \bm \sigma) \in L_0\} \subset S_{g,\e}}, \qquad \qquad
{\lim_{\e\to0^+}\sup_{g\in K}}\,m\sub{\Si{N_\e}}(S_{g,\e}){=0}.
\label{e:defShy}
\end{equation}
We say $G\subset\ms{G}$ is $\ms{F}$-\emph{predominant} if $\ms{G}\backslash G$
is $\ms{F}$-thin.
We say $G\subset \ms{G}$ is respectively \emph{predominant} or \emph{thin} if there exists {$\ms{G}'$ as above and} a family of $\mathscr{G}'$ probing maps for $\ms{G}$ such that $G$ is $\ms{F}$-predominant or $\ms{F}$-thin {respectively}.  Figure~\ref{f:predominant} shows a schematic of a family of probing maps and a thin set.
\end{definition}

\begin{remark}
\label{r:heuristic}
We note that if a  set $G\subset \ms{G}$ is {predominant} the family of submanifolds endowed with Borel measures, $\{(\Gamma_{\e,g},\mu_{\e,g})\}_{g\in \ms{G}}$, with $\Gamma_{\e,g}:=F_\e(g,\bm{\Sigma}\sub{N_\e})$ and $\mu_{\e,g}:=(F_\e(g,\cdot))_*m\sub{\bm{\Sigma}_\e}$ satisfies that $g\in \Gamma_{\e,g}$ (by \eqref{e:typicalness}), $\mu_{\e,g}$ assigns a positive measure to any neighborhood of $g$ (by~\eqref{e:closeToId}), the map $g\mapsto \Gamma_{\e,g}$ is a $\mc{C}^1$ family of Lipschitz submanifolds (by \eqref{e:bddDiff}, \eqref{e:bddDiff0}), and $\mu_{\e,g}(G^c\cap \Gamma_{\e,g})\leq \e$ for every $g\in\ms{G}$ (by \eqref{e:defShy}). \blue{Note that, by \eqref{e:bddDiff0}, we have that locally $\Gamma_{\e,g}$ is almost a translate of $\Gamma_{\e,g_0}$. Observe that, if $\Gamma_{\e, g}$ were exactly a translate of $\Gamma_{\e,g_0}$, then a set, $G$, being predominant would mean that, there is a probability measure on a subset of $\mathscr{G}$ such that, for each $g\in\mathscr{G}$, and almost every $\phi$ randomly chosen with respect to the probability measure, $g+\phi\in G$.}
\end{remark}

The direct analogy to~\eqref{e:prelimMeasure} would replace the condition~\eqref{e:defShy} with 
\begin{equation}\label{e:fail}
 \sup_{g\in K}\,m\sub{\Si{N_\e}}(S_{g,\e}){=0}.
\end{equation}
We are, however, not able to show this in our applications and, instead, relax the condition to~\eqref{e:defShy}.


\subsubsection{Verification of properties~\eqref{prevList} for predominance}

We now check that the notion of $\ms{F}$-predominance satisfies the properties listed in~\eqref{prevList}. {We first prove that predominant sets are dense.}

\begin{lemma}
Suppose $\ms{F}$ is a family of $\mathscr{B}'$ probing maps for $\ms{G}$ and that $G\subset \ms{G}$ is  $\ms{F}$-predominant. Then, $G$ is dense {in $\ms{G}$}.
\end{lemma}
\begin{proof} Let $G$ be $\ms{F}$-predominant.
Fix $g_0\in \ms{G}$ and $\delta>0$.  We will prove that $B\sub{\ms{G}}(g_0,\delta)\cap G\neq \emptyset$. Let {$\ms{F}:=\{(F_{\e},N_{\e})\}_{\e>0}$} {with $F_{\e}:\ms{G}_\e\times\Si{N_\e}\to\ms{G}$,  $\cup_\e \ms{G}_\e=\ms{G}$, and  $\ms{G}_{\e_2}\subset\ms{G}_{\ep_1}$ for $\ep_1<\ep_2$}. By {~\eqref{e:closeToId}}, there is $\e_0>0$ such that for $0<\e<\e_0$, {$g_0\in \ms{G}_\ep$ and}
\begin{equation}
\label{e:imVeryClose}
{\sup_{\bm \sigma {\in \Si{N_\e}}}}\|F_\e(g_0,\bm\sigma)-g_0\|_{\ms{B}}<\delta.
\end{equation}

Let $G^c\subset L_0$ with $L_0$ Borel and satisfying~\eqref{e:defShy}. Then, since $m\sub{\bm \Sigma_\e}(\bm \Sigma_\e)=1$, there is $\e_1>0$ such that
$
{\{\bm \sigma \in \Si{N_\e}:\; F_\ep(g_0, \bm \sigma) \in L_0\}}\neq \bm \Sigma_\e
$
for $0<\e<\min(\e_0,\e_1)$.
In particular, there is $\bm \sigma\in \bm \Sigma_\e$ such that 
$
F_\e(g_0,\bm\sigma)\in \ms{G}\setminus L_0\subset G.
$
By~\eqref{e:imVeryClose}, this implies that 
$
B\sub{\ms{G}}(g_0,\delta)\cap G\neq \emptyset.
$
\end{proof}

{Next, we check that finite intersections of predominant sets are predominant.}
\begin{lemma}
Suppose $\ms{F}$ is a family of $\mathscr{B}'$ probing maps for $\ms{G}$ and that $G_j\subset \ms{G}$, $j=1,2,\dots,J$ are  $\ms{F}$-predominant. Then, $\bigcap_{j=1}^J G_j$ is $\ms{F}$-predominant.
\end{lemma}
\begin{proof}
{Let $K \subset \ms{G}$ bounded, let {$\{L_{j,0}\}_{j=1}^J \subset \ms{G}$ be a collection of Borel sets} with  $G_j^c\subset L_{j,0}$, and for $\ep>0$ {and $g\in K$} let $S_{j,g,\e} {\subset \Si{N_\e}}$ be $m\sub{\bm \Sigma_\e}$-measurable sets satisfying~\eqref{e:defShy}.  Let $L_0:=\bigcup_{j=1}^J L_{j,0}$ and for each $\e>0$ {and $g\in K$} set $S_{g,\e}:=\bigcup_{j=1}^J S_{j,g,\e}$.
Then,  for all $g\in K$ and $\e>0$ 
$$
{\{\bm \sigma \in \Si{N_\e}:\; F_\ep(g, \bm \sigma) \in L_0\}}
= \bigcup_{j=1}^J  {\{\bm \sigma \in \Si{N_\e}:\; F_\ep(g, \bm \sigma) \in L_{j,0}\}}
\subset S_{g,\e}.
$$}
Finally,
$
{\lim_{\e\to0^+}\sup_{g\in K}}\, m\sub{\Si{N_\e}}\big(S_{g,\e}\big)\leq \sum_{j=1}^J {\lim_{\e\to 0^+}\sup_{g\in K
}}\,m\sub{\Si{N_\e}}\big(S_{j,g,\e}\big)=0,
$
as claimed.
\end{proof}

{We end this section by checking} that in finite dimensions predominant sets have full measure. 

Since we are working in finite dimensions, we assume that $N_\e<\infty$ and $\mathscr{G}\subset \mathbb{R}^n$ for some $n<\infty$. We also  take $\mathscr{B}=\mathbb{R}^n$. Now, the topology induced on $\mathbb{R}^n$ from $\mathscr{B}'$ for any Banach space $\mathscr{B}'$ is identical. Therefore, we may assume, without loss of generality, that $\mathscr{B}'=\mathscr{B}=\mathbb{R}^n$.

\begin{lemma}
Let $n<\infty$, {$N_0<\infty$,} and suppose $\ms{F}{=\{(F_{\e},N_{\e})\}_{\e>0}}$ is a family of $\mathbb{R}^n$ probing maps  for $\ms{G}:=\mathbb{R}^n$ with ${\sup_{\ep>0}}N_\e<N_0$. {If $G\subset \ms{G}$ is $\ms{F}$-predominant,  then} $G$ has full Lebesgue measure.
\end{lemma}
\begin{proof}
We will show that if $L$ is $\ms{F}$-thin, then $L$ has zero Lebesgue measure. Let ${K_0}\subset \ms{G}=\mathbb{R}^n$ be closed and bounded. 
Next, observe that by~\eqref{e:closeToId}, for $\e>0$ small enough,
$$
\pi\sub{\ms{G}}(F_\e^{-1}(K_0))\subset \{ g\in \ms{G}\,:\, d(g,K_0)\leq 1\}=:K.
$$
Since $K \subset  \ms{G}$ is bounded, we let $L_0\subset \ms{G}$ be a Borel set with $L\subset L_0$ and $L_0$ satisfying~\eqref{e:defShy} and we take $\e$ small enough such that $K\subset\ms{G}_\e$. Note that  $F_{\e}^{-1}(L_0\cap{K_0} )$ is measurable for all $\ep>0$.
Fubini's theorem yields
\begin{equation}\label{e:fubini}
\int_{\ms{G}\times \Si{N_\e}}1_{F_\e^{-1}(L_0\cap K_0)}d(m\sub{\mathbb{R}^n}\!\!\times\! m\sub{\Si{N_\e}})
=\int_{K}m\sub{\Si{N_\e}}\big(\{\bm \sigma: F_{\e}(g,\bm \sigma) \in L_0 \cap K_0\} \big)dm\sub{\mathbb{R}^n}(g).
\end{equation}

Next, by {~\eqref{e:bddDiff0}}, $DF_{\e}|_{(g,\bm\sigma)}:T_{(g,\bm \sigma)}(\mathbb{R}^m\times \bm \Sigma_\e)\to T_g\mathbb{R}^m$ is surjective  for $\e>0$ small enough,{ $g\in K$, and $\bm \sigma\in \bm \Sigma_\e$}. Therefore, by the coarea formula,
\begin{equation}\label{e:coarea}
\int_{\ms{G}\times \bm \Sigma_\e}1_{F_\e^{-1}(L_0\cap {K_0})}d(m\sub {\mathbb{R}^n}\!\!\times\! m\sub {\Si{N_\e}})=\int_{{K_0}}1_{L_0\cap F_\e(\ms{G}\times \Si{N_\e})}\Big(\int_{F_\e^{-1}(g)}\frac{1}{|J\!F_\e|}d\mc{H}_{g,\e}\Big)dm\sub{\mathbb{R}^n}(g),
\end{equation}
where $\mc{H}_{g,\e}$ denotes the $\dim(\bm \Sigma\sub{N_\e})$-Hausdorff measure on $F_{\e}^{-1}(g)$ {and $|J\!F_\e|:=\sqrt{\det DF_\e(DF_\e)^*}$}. 

Next, we will prove that there is $\e_0>0$ such that for all $0<\e<\e_0$ and  $g_0 \in K_0$
\begin{equation}\label{e:bananas}
\pi\sub{\bm \Sigma\sub{N_\e}}(F_\e^{-1}(g_0))=\bm \Sigma\sub{N_\e}.
\end{equation}
Once we have~\eqref{e:bananas}, observe that 
$$
\mc{H}_{g,\e}(F_\e^{-1}(g))\geq \mc{H}_{g,\e}(\pi\sub{\bm\Sigma\sub{N_\e}}(F_\e^{-1}(g))\geq c>0.
$$
Hence, since $|JF_\e|\leq C<\infty$, we have for all $g$
\begin{equation}\label{e:lbF}
\int_{F_\e^{-1}(g)}\frac{1}{|J\!F_\e|}d\mc{H}_{g,\e}>c>0.
\end{equation}

To see~\eqref{e:bananas}, fix $(g_0,\bm \sigma)\in {K_0}\times \Si{N_\e}$, and define $\Psi_\e:\R^n\to \R^n$ by 
$$
\Psi_\e(g_1):=g_1-(D_g F_\e|\sub{(g_0,\bm\sigma)})^{-1}(F_\e(g_1+g_0,\bm\sigma)-g_0).
$$
{We claim that there is $\e_0>0$ such that  $\Psi_\e:\bar B(0,1)\to \bar B(0,1)$ is a contraction for $0<\e<\e_0$ , where $\bar B(0,1)$ is the closed unit ball}. 
To prove the claim note that by {~\eqref{e:closeToId}},~\eqref{e:bddDiff} and ~\eqref{e:bddDiff0} there is $\e_0>0$ small enough such that for $0<\e\leq \e_0$, {$g_0\in K_0$}, and {$g_1 \in \bar B(0,1)$},
\begin{equation}\label{e:soup1}
\begin{gathered}
\|(D_gF|_{(g_0,\bm\sigma)})^{-1}(F_\e(g_0,\bm\sigma)-g_0)\|<{\tfrac{1}{2}},\qquad
\|(D_gF_\e|\sub{(g_0,\bm\sigma)})^{-1}\big(D_gF_{\e}|\sub{(g_1+g_0,\bm\sigma)}-D_gF_\e|\sub{(g_0,\bm\sigma)}\big)\|<{\tfrac{1}{4}}.
\end{gathered}
\end{equation}
Now, for $g_1,g_2 \in \bar B(0, 1)$,
\begin{align}
&\Big\|(D_gF_\e|\sub{(g_0,\bm\sigma)})^{-1}\big[F_\e(g_0+g_1,\bm\sigma)-F_\e(g_0+g_2,\bm \sigma)-D_gF_\e|\sub{(g_0+g_2,\bm\sigma)}(g_1-g_2)\big]\Big\| \notag\\
&=\Big\|\int_0^1(D_gF_\e|\sub{(g_0,\bm\sigma)})^{-1}\big( D_gF_\e|\sub{(g_0+(1-t)g_2+tg_1,\bm\sigma)}-D_gF_\e|\sub{(g_0+g_2,\bm\sigma)}\big)(g_1-g_2)dt\Big\|\leq {\tfrac{1}{2}}\|g_1-g_2\| \label{e:hot}.
\end{align}
Therefore,  using \eqref{e:soup1} and letting $g_2=0$ in \eqref{e:hot}, we have that for $g_1 \in \bar B(0,1)$
\begin{align*}
\|\Psi_\e(g_1)\|
&=\|(D_gF_\e|\sub{(g_0,\bm\sigma)})^{-1}\Big(D_gF_\e|\sub{(g_0,\bm\sigma)}g_1-F_\e(g_0+g_1,\bm \sigma)+F_\e(g_0,\bm\sigma)-F_\e(g_0,\bm\sigma)+g_0\Big)\|
<  1,
\end{align*}
 and so $\Psi_\e: \bar B(0,1)\to \bar B(0,1)$  for $\e<\e_0$. 
Next,  {again by \eqref{e:soup1} and \eqref{e:hot}, for $g_1,g_2 \in \bar B(0, 1)$} 
$$
\begin{aligned}
\|\Psi_\e(g_1)-\Psi_\e(g_2)\|
&= \|(D_gF_\e|\sub{(g_0,\bm\sigma)})^{-1}\big(D_gF_\e|\sub{(g_0,\bm\sigma)}(g_1-g_2)-F_\e(g_1+g_0,\bm \sigma)+F_\e(g_2+g_0,\bm \sigma)\big)\|\\
&\leq \tfrac{3}{4}\|g_1-g_2\|.
\end{aligned}
$$
Hence,  $\Psi_\e:\overline{ B(0,1)}\to \overline{ B(0,1)}$ is a contraction.
 For each $(g_0,\bm \sigma)\in \R^n\times \Si{N_\e}$ let $g_1=g_1(g_0,\bm \sigma)$ be the fixed point of $\Psi_\e$. Then, for each $(g_0,\bm \sigma)\in \R^n\times \Si{N_\e}$ there is $g_1$ with $F_\e(g_1+g_0,\bm \sigma)=g_0$  as claimed in \eqref{e:bananas}.
 
Finally, by~\eqref{e:bddDiff} {there is $C>0$ with $\sup_{\ep>0}\sup_{(g, \bm \sigma)\in K\times\bm \Sigma_\e}|J\!F_\e(g, \bm \sigma)|<C$,} 
{\eqref{e:fubini}, \eqref{e:coarea}, \eqref{e:defShy}, \eqref{e:lbF} yield}
$$
{0=\lim_{\e \to 0^+}}\int_{\ms{G}\times \Si{N_\e}}1_{F_\e^{-1}(L_0\cap {K_0})}d(m\sub{\mathbb{R}^n}\times m\sub {\bm \Sigma})\geq \tfrac{1}{C}\int_{{K_0}}1_{L_0\cap F_\e(\ms{G}\times \Si{N_\e})}d m\sub{\mathbb{R}^n}=\tfrac{1}{C}m\sub{\mathbb{R}^n}(L_0\cap {K_0}).
$$
Hence, $m\sub{\mathbb{R}^n}(L_0\cap {K_0})=0$ and, since ${K_0}$ is an arbitrary closed bounded set, $m\sub{\mathbb{R}^n}(L_0)=0$. In particular, by the completeness of the Lebesgue measure, $L$ is Lebesgue measurable with $m\sub{\mathbb{R}^n}(L)=0$ as claimed.
\end{proof}

\subsection{Heuristic explanation for lack of Baire-genericity}\label{s:Baire}

 A set is Baire generic if it contains the intersection of countably many open dense sets. In order to {explain} why we do not pursue this notion of genericity, we discuss one of the key features we require at periodic points for the geodesic flow. 
 Let $\Gamma\subset S^*\!M$ be a Poincar\'e section through the point $\rho$ and $\ms{P}:\Gamma\to \Gamma$ the corresponding Poincar\'e map. 
For simplicity, we will assume that $\ms{P}(\rho)=\rho$. 
Theorem~\ref{t:predominantR-ND} has consequences for the eigenvalues of $d\ms{P}|_{\rho}: T_\rho \Gamma \to T_\rho \Gamma$. Indeed, we have that there is $C>0$ such that for all $n$
$$
\|(d\ms{P}|_{\rho})^n-\Id)^{-1}\|=\|(d(\ms{P}^n)|_{\rho}-\Id)^{-1}\|\leq (Cn)^{Cn^{\alpha_{\nu}+\e}+1}.
$$
Therefore, the eigenvalues, $\{\lambda_j\}_{j=1}^{2d-2}$ of $d\ms{P}|_{\rho}$ must satisfy 
$$
\inf_{j,p}\big(1+|\lambda_j|\big)^{n-1}|\lambda_j - e^{i \frac{2\pi p}{n}}|\geq\inf_{j}|\lambda_j^n - 1|\geq (Cn)^{-Cn^{\alpha_{\nu}+\e}-1},\qquad n\geq 1.
$$
Since $d\ms{P}$ is a symplectic transformation, eigenvalues may be confined to the unit circle (see Section~\ref{s:evs}) and this becomes analogous to understanding the structure of the set real numbers which are poorly approximable by rational numbers.

To discuss the issues of genericity in a simpler setting, we forget now about Poincar\'e maps and instead discuss them in the context of approximation of real numbers. As explained above, {letting $h_\ell(q)=(\ell q)^{-\ell q^{\alpha}-1}$} we want to investigate the set 
\begin{equation}
\label{e:diophantine}
\mathfrak{D}:=\bigcup_{\ell>0}{\mathfrak{D}_\ell},\qquad \mathfrak{D}_\ell:=\{s\in [0,1]\,:\,  |s-\tfrac{p}{q}|>{h_\ell(q)},\text{ for all }p,q\in\mathbb{N}\}.
\end{equation}
To study $\mathfrak{D}$ we consider
$$
\mathfrak{U}:=\bigcap_{q_0=1}^\infty\bigcup_{q=q_0}^{\infty} \mathfrak{U}_q,\qquad\mathfrak{U}_q:=\bigcup_{p=0}^{q-1} (\tfrac{p}{q}-h_q(q),\tfrac{p}{q}+ h_q(q)).
$$
Observe that $\mathfrak{U}$ is Baire generic. However, $\mathfrak{U}\cap \mathfrak{D}_\ell=\emptyset$ for all $\ell>0$ which implies that $\mathfrak{D}\cap \mathfrak{U}=\emptyset$ and hence that $\mathfrak{D}$ is \emph{not} Baire generic. Since even the property we want for the eigenvalues of {$d\ms{P}$} is non-generic, it is unlikely that the set of metrics which produce such $\ms{P}$ is generic.

\blue{While $\mathfrak{D}$ in~\eqref{e:diophantine} is not Baire generic, one can see that, since $\lim_{n\to \infty}|[0,1]\setminus \mathfrak{D}_n|=0$, $\mathfrak{D}$ has full Lebesgue measure and hence is, in a much stronger sense, typical. Indeed, for many purposes, the notion of full Lebesgue measure is a better version of `typical' than Baire genericity. For example, a randomly chosen element of $[0,1]$ is almost surely in $\mathfrak{D}$ but may not be in a given Baire generic set. Also, a full Lebesgue measure set in $\mathbb{R}$ has Hausdorff dimension 1, while a Baire generic set in $\mathbb{R}$ may have 0-Hausdorff dimension.
Motivated by this discussion, the notion of `typicality' that we use, i.e. that of predominance (see Definition \ref{d:predominant}), is an analog of a full Lebesgue measure set of metrics $g\in \ms{G}^\nu$. }

For further evidence of lack of Baire genericity, we consider the space $\text{Diff}^\nu(\Gamma)$ of  $\mc{C}^\nu$ diffeomorphisms on a smooth manifold $\Gamma$ for $\nu\geq 2$. Kaloshin~\cite{Ka:00} showed that for any $\{a_n\}_{n=1}^\infty\subset [1,\infty)$, the set
$$
\Big\{ f\in \text{Diff}^\nu(\Gamma)\,:\, \limsup_n\frac{P_n(f)}{a_n}< \infty\Big\}
$$
is \emph{not} Baire generic, where $P_n(f):=\#\{\rho  \in \Gamma\; \text{isolated}:\; f^{n}(\rho)=\rho\}$. \blue{Furthermore}, when $\Gamma$ is a 3-manifold, this set is \emph{not} Baire generic even in the space of volume preserving maps~\cite{KaSa:06}.


\section{Counting closed geodesics and improvements for Weyl laws:\\ Proof of Theorems~\ref{t:predominantWeyl} and~\ref{t:predominantLengths}}
\label{s:theoremProofs}


In this Section, we use Theorem~\ref{t:predominantR-ND} to prove Theorems~\ref{t:predominantWeyl} and~\ref{t:predominantLengths}. Both of these theorems rely on volume estimates {on the set of} nearly closed geodesics. {We obtain these estimates in Section \ref{s:nearlyClosed} and prove Theorems~\ref{t:predominantWeyl} and~\ref{t:predominantLengths} in Sections \ref{s:plengths} and \ref{s:pWeyl} respectively. }

We start by letting $M$ be a $\mc{C}^\nu$ manifold \blue{$\nu\geq 3$}, {$g \in  \ms{G}^\nu$},   ${\beta_0}>0$, and $\mc{W}$ be a ${\beta_0}$-family of transition maps {for $g$}. Let $\n:[0, +\infty) \to (0, +\infty)$ be a continuous, decreasing function.
Throughout the section we will suppose that  
\begin{equation}
\label{e:imNondegenerate}
\Rec(t,\n(t),g)\subset \ND(t,\n(t), (g,\mc{W})), \quad {\text{for } t>c,\qquad d(\varphi_t^g(\rho),\rho)\geq c|t|,\quad\text{for } |t|\leq c}.
\end{equation}
In Theorem~\ref{t:predominantR-ND}, we show that there is a predominant {set of metrics $G \subset \ms{G}^\nu$  such that~\eqref{e:imNondegenerate} holds for $g\in G$ with $\n(t)= C^{-C(t+1)^{\someLetter}-1}$  and some $C>0$}.

\subsection{Volume of nearly closed geodesics}
\label{s:nearlyClosed}
We start by following~\cite[Appendix A]{DyZw:16}, replacing their assumption of hyperbolicity of the flow with~\eqref{e:imNondegenerate} to estimate the volume of nearly closed geodesics.
Recall that  there are $\tilde C, L>0$ such that for all   $f\in \mc{C}^2(S^*\!M)$ and $t\in \re$
\begin{equation}
\label{e:c2Bound}
\|f\circ \varphi_t^g\|_{\mc{C}^2(S^*\!M)}\leq \tilde Ce^{L|t|}{\|f\|_{\mc{C}^2(S^*\!M)}}.
\end{equation}
In particular, $d(\varphi_t^g(\rho), \varphi_t^g(\rho'))\leq  \tilde{C}e^{L|t|}d(\rho, \rho')$ for all $\rho, \rho' \in S^*\!M$ and $t \in \re$. \blue{The bound in \eqref{e:c2Bound} is stated in greater generality as part of Lemma \ref{l:PoincareDer}.}

The following lemma shows that two nearby orbits that return to their starting point after similar times are almost iterates of one another.
\begin{lemma}
\label{l:stability}
Let {$M, g, \beta_0, \mc{W}$} and $\n$ be such that ~\eqref{e:imNondegenerate} and~\eqref{e:transition} hold. Given $t_0>0$ there are $C,\delta>0$ such that the following holds. Let   $t\geq t_0$, $t'\geq t_0$ with $|t-t'|\leq \delta$, and for 
$0 \leq \e<\n(t)$ let $\rho, \rho' \in {S^*\!M}$ be such that 
\begin{gather*}
d(\rho,\varphi_t^g(\rho))\leq \e,\,\qquad  d(\rho',\varphi_{t'}^g(\rho'))\leq \e,\,\qquad \,d(\rho,\rho')\leq \delta \n(t) e^{-Lt}.
\end{gather*}
Then, $|t-t'|\leq C\e$ and there exists $s\in[-1,1]$ such that $\n(t) d(\rho,\varphi_s^g(\rho'))\leq C\e.$
\end{lemma}
\begin{remark}
    \blue{Note that Lemma~\ref{l:stability} is useful only if $\e\leq \delta \n(t)^2e^{-Lt}$ since otherwise the assumption $d(\rho,\rho')\leq \delta \n(t)e^{-Lt}$ is stronger than the conclusion about the distance between $\rho$ and the trajectory of length $2$ through $\rho'$. }
\end{remark}
\begin{proof}
First, observe that for any $\e_0>0$, we may increase $C$ enough so that the statement becomes trivial for $\e\geq \e_0$. Therefore, we need only work with $\e<\e_0{<\min(\frac{1}{4}\beta_0,\frac{1}{2})}$ small enough and we do so from now on.  Next, observe that we may shrink $\delta$  so that   $d(\varphi_t^g(\rho),\varphi_t^g(\rho'))\leq \tfrac{1}{4}{\beta_0}$ whenever $d(\rho,\rho')\leq \delta e^{-Lt}$, {$t\geq t_0$}. {We may also assume that} all of the relevant points are contained in a single coordinate chart and hence we may work in a small ball in {$\mathbb{R}^{2d-1}$}. 

We also assume that $\delta$ is small enough such that whenever $d(\rho, \rho')\leq \delta$ there exists $|s|\leq 1$ with  
$
\rho-\varphi_s^g(\rho')\in {\ms{H}(\rho)}
$
{and such that $d(\rho, \varphi_s^g(\rho')) \leq 2d(\rho, \rho')$.} 
Set {$\rho_0:=\varphi_s^g(\rho')$ and note
\begin{equation}\label{e:rhoZ}
\rho-\rho_0\in (H_p(\rho))^\perp, \qquad d(\rho,\rho_0)\leq 2\delta \n(t) e^{-Lt},\qquad d(\varphi_{t'}^g(\rho_0),\rho_0)\leq \tilde C e^L\e.
\end{equation}}
By Taylor expanding in $\rho$ and using~\eqref{e:c2Bound}, we have 
$$
\|\varphi_t^g(\rho_0)-\varphi_{t}^g(\rho)-d\varphi_t^g(\rho)(\rho_0-\rho)\|\leq \tilde{C}e^{Lt}d(\rho,\rho_0)^2\leq {2}\tilde{C}\delta \n(t) d(\rho,\rho_0).
$$
Next, Taylor expanding in $t$ and using that there is $C>0$ such that $\|\partial_t^2\varphi_t^g\|\leq C$, 
$$
\|\varphi_{t'}^g(\rho_0)-\varphi_{t}^g(\rho_0)-H_p(\varphi_t^g(\rho_0))(t'-t)\|\leq C|t'-t|^2\leq C\delta |t'-t|.
$$
Thus, we have {that there is $C>0$ with}
$$
\|\varphi_{t'}^g(\rho_0)-\varphi_t^g(\rho)-d\varphi_t^g(\rho)(\rho_0-\rho)-H_p(\varphi_{t}^g(\rho_0))(t'-t)\|\leq C\delta\Big(\n(t) d(\rho_0,\rho)+|t-t'|\Big).
$$
Next, since $d(\varphi_{t}(\rho),\rho)\leq \e$, {by \eqref{e:rhoZ} there is $C>0$ such that}  
$$
\|\big[I-d\varphi_t^g(\rho)\big](\rho_0-\rho)-H_p(\varphi_{t}^g(\rho_0))(t'-t)\|\leq C\delta \Big(\n(t) d(\rho_0,\rho)+|t-t'|\Big)+C\e.
$$
Now, {by \eqref{e:transition} there is $C>0$ such that}
$
\|\mc{W}_{\varphi_t^g(\rho),\rho}-I\|\leq Cd(\varphi_t^g(\rho),\rho)\leq C\e.
$
Therefore, there is $C>0$ such that
\begin{equation}\label{e:Norm}
\|\big[\mc{W}_{\varphi_t^g(\rho),\rho}-d\varphi_t^g(\rho)\big](\rho_0-\rho)-H_p(\varphi_{t}^g(\rho))(t'-t)\|\leq {C\delta\Big(\n(t) d(\rho_0,\rho)+|t-t'|\Big) +C\e}.
\end{equation}
Now, since $\rho_0-\rho \in \ms{H}(\rho)$, we have by \eqref{e:imNondegenerate}, Definition~\ref{d:nonDegenerate-g}, {and the fact that  that $\e\leq \n(t)$, that} there is $C>0$ such that
\begin{align}
\n(t) d(\rho_0,\rho)&\leq \|\HpProj_{\varphi_{t}(\rho)}\big[\mc{W}_{\rho,\varphi_t^g(\rho)}-d\varphi_t^g(\rho)\big](\rho_0-\rho)\|\notag\\
&\leq C\|\big[\mc{W}_{\rho,\varphi_t^g(\rho)}-d\varphi_t^g(\rho)\big](\rho_0-\rho)-H_p(\varphi_{t}^g(\rho))(t'-t)\|. \label{e:distance} 
\end{align}
On the other hand, since $\mc{W}_{\varphi_t^g(\rho),\rho}(H_p(\rho))=H_p(\varphi_t^g(\rho))$, { $\mc{W}_{\varphi_t(\rho),\rho}(\ms{H}(\rho))=\ms{H}(\varphi_t^g(\rho))$, and $\ms{H}(\rho)$ is transverse to $H_p(\rho)$ (uniformly in $\rho$)}, we have 
\begin{equation}\label{e:timesA}
|t-t'|\leq C\|\big[\mc{W}_{\rho,\varphi_t^g(\rho)}-d\varphi_t^g(\rho)\big](\rho_0-\rho)-H_p(\varphi_{t}^g(\rho))(t'-t)\|.
\end{equation}
{Combining \eqref{e:Norm}, \eqref{e:distance}, \eqref{e:timesA}}, and choosing $\delta\ll 1$ and $\e_0<\frac{1}{2C}$, we have
$
\n(t) d(\rho,\rho_0)+|t-t'|\leq C\e
$
as claimed.
\end{proof}

We proceed to bound the volume of nearly closed trajectories in $S^*\!M$. \blue{First, we make the following definition.
\begin{definition}
    A set $\mathcal{A}\subset M$ is \emph{$r$ separated} if 
    $$
    \rho_{1}\notin B(\rho_2,r)\text{ for all }\rho_1\neq \rho_2, \quad \,\rho_1,\rho_2\in\mathcal{A}
    $$
    and is a \emph{maximal $r$ separated set} if it is $r$ separated and for any $\rho_0\notin \mathcal{A}$, the set $\{\rho_0\}\cup \mathcal{A}$ is not $r$ separated.
\end{definition}}
\begin{lemma}
\label{l:vol}
{Let $M, g, \beta_0, \mc{W}$, and $\n$ be such that ~\eqref{e:imNondegenerate} holds.}
Let $t_0>0$, then there is $C>0$ such that for $T>t_0$ and $0<\e\leq  \n(T)$, we have
$$
\vol\sub{S^*\!M}\Big(\big\{ \rho\,:\,\exists t\in[t_0,T]\text{ such that }d(\varphi_t^g(\rho),\rho)\leq \e\big\}\Big)\leq C{\e^{2d-2}}{e^{CT}}{\n(T)^{-(4d-3)}}.
$$

\end{lemma}
\begin{proof}
Let $t_0>0$ and $T>t_0$. Then, let $\delta$ as in Lemma~\ref{l:stability} and fix $0<\e\leq \n(T)$. Next divide $[t_0,T]$ into intervals, $\{I_i\}_{i=1}^{N\sub{T}}$ with right endpoints at $\{T_i\}_{i=1}^{N\sub{T}}$ satisfying $|I_i|\leq \frac{\delta}{2}$.
{Next, for each $i=1, \dots, N\sub{T}$} let $\{\rho_j\}_{j=1}^{K\sub{T_i}}{\subset S^*\!M}$ be a maximal $\frac{\delta}{2}e^{-LT_i}\n(T_i)$ separated set. Note that 
\begin{equation}\label{e:nameless}
\begin{gathered}
\big\{ \rho\,:\,\exists t\in[t_0,T]\text{ such that }d(\varphi_t^g(\rho),\rho)\leq \e\big\} \subset \bigcup_{i=1}^{N\sub{T}} \bigcup_{j=1}^{K\sub{T_i}} \pi\sub{\SM}(\mathbf{P}_{i, \rho_j}),\\
\mathbf{P}_{i, \rho_j}:=\Big\{(\rho,t) \,:\; t\in I_i,\, d(\varphi_t^g(\rho),\rho)\leq \e,\; d(\rho,\rho_j)\leq \tfrac{\delta}{2}e^{-LT_i}\n(T_i)\Big\}.
\end{gathered}
\end{equation}
{To estimate $\vol\sub{S^*\!M}(\pi\sub{S^*\!M}(\mathbf{P}_{i,\rho_j}))$, fix $(\rho', t')\in \mathbf{P}_{i,\rho_j}$ and let $(\rho,t)\in \mathbf{P}_{i,\rho_j}$ with $(\rho,t) \neq (\rho', t')$.} 
By Lemma~\ref{l:stability}, 
$$
|t-t'|\leq C\e,\qquad
 d\Big(\rho\;, \;\underset{|s|\leq 1}{\bigcup}\varphi_s^g(\rho')\Big) \leq  {C\e}{\n(T_i)^{-1}}.
$$
Then,
$
\vol\sub{S^*\!M}({\pi\sub{\SM}}(\mathbf{P}_{i,\rho_j}))\leq C{\e^{2d-2}}{\n(T_i)^{-(2d-2)}}, 
$
and 
\begin{equation}\label{e:volPi}
\vol\sub{S^*\!M}\Big({\pi\sub{\SM}}\Big(\bigcup_{j=1}^{K\sub{T_i}} \pi\sub{\SM}(\mathbf{P}_{i, \rho_j})\Big)\Big)\leq {K\sub{T_i}C\e^{2d-2}{\n(T_i)^{-(2d-2)}} } \leq   C {\e^{2d-2}}{e^{(2d-1)LT_i}}{\n(T_i)^{-(4d-3)}},
\end{equation}
since there is $C_\delta>0$ such that
$
K\sub{T_i}\leq C_\delta [{e^{-LT_i}}\n(T_i)]^{-(2d-1)}.
$
The claim follows from combining \eqref{e:volPi} with \eqref{e:nameless}, and  using that $\n$ is decreasing and $N\sub{T}\leq CT$.
\end{proof}

{\begin{remark}
We note that if one is only interested in obtaining a volume estimate of the form 
$$
\vol_{S^*\!M}\Big(\big\{ \rho\,:\,\exists t\in[t_0,T]\text{ such that }d(\varphi_t^g(\rho),\rho)\leq \e\big\}\Big)\leq C\e^{2d-2-k}{e^{CT}}{\n(T)^{-(4d-3-k)}},
$$
for some $k\geq 1$,
then it is possible to use a weaker notion of non-degeneracy. Indeed, one can replace the definition of $\ND(t,\rho, (g,\mc{W}))$ by saying that $\rho\in \ND_k(t,\rho,(g,\mc{W}))$ if there is a $2d-2-k$  dimensional subspace of $v\in T_{\rho} S^*\!M$ such that 
$
|\HpProj\sub{\varphi_t^g(\rho)}(\mc{W}\sub{\rho,\varphi_t^g(\rho)}-d\varphi_t^g)v|>\alpha |\HpProj_\rho v|.
$
If implemented throughout the article, this would in turn lead to a smaller value of \blue{$\Omega_\nu$} in Theorem~\ref{t:predominantWeyl} but would not be sufficient for counting closed geodesics.
\end{remark}}

\subsection{Proof of Theorem \ref{t:predominantLengths}} 
\label{s:plengths}
Let $\someLetter>\someLetter_\nu$.
By Theorem~\ref{t:predominantR-ND} there is a predominant set $G\sub{\someLetter} \subset \ms{G}^\nu$ such that  for all $g \in G$ and  $\mc{W}$ a family of transition maps for $g$ there is $C>0$ such that \eqref{e:imNondegenerate} holds with $\n(t)={ C^{-C(t+1)^{\someLetter}-1}}$.

Let $g \in G\sub{\someLetter}$ {and $0<t_0<\inj_g(M)$ so that there are no closed geodesics with length $\leq t_0$}. Let $\delta$ be as in Lemma~\ref{l:stability}. First, we claim that {$\mathfrak{c}(T,g)<\infty$} {for each $T>0$}. 
Indeed, suppose that $\{\rho_i\}_{i=1}^\infty\subset S^*\!M$ and $\{t_i\}_{i=1}^\infty\subset [t_0,T]$ are sequences satisfying  $\varphi_{t_i}^g(\rho_i)=\rho_i$ and such that
$
\rho_j\notin\{\varphi_s^g(\rho_i):\; \blue{s\in\mathbb{R}}\}
$ 
for all $i \neq j$.
Then, {we may assume that} there are $ \rho\in S^*\!M$ and  $t\in[t_0,T]$ such that $\rho_i\to \rho$ and $t_i\to t$. By continuity, we conclude  $\varphi_t^g(\rho)=\rho$. Next, let $i_0>0$ be such that  $d(\rho_i,\rho)\leq \delta \n(T)e^{-LT}$ and $|t_i-t|\leq \delta$ for $i\geq i_0$. Then, by Lemma~\ref{l:stability} \blue{(applied with $\e=0$, $\rho'=\rho_i$, and $t'=t_i$)} we obtain that for $i>i_0$ there is $s_i\in [-1, 1]$ such that $\varphi_{s_i}^g(\rho_i)=\rho$ and $t_i=t$. In particular, for $i,j>i_0$, we have $\rho_j=\varphi_{s_i-s_j}^g(\rho_i)$ which is a contradiction.


Since $\mathfrak{c}(T,g)<\infty$, we let $\{\gamma_i\}_{i=1}^{\mathfrak{c}(T,g)}$ be the finite collection of primitive closed geodesics of length $\leq T$.
Then, there is $\delta_0={\delta_0(T)}>0$ such that if  $\gamma_i( \delta)\subset S^*\!M$ denotes the $\delta$ neighborhood of $\gamma_i$,  
\begin{equation}\label{e:disjoint}
\gamma_i(\delta)\cap \gamma_j(\delta)=\emptyset,\qquad i\neq j,\,\qquad  0\leq \delta<\delta_0.
\end{equation}
Letting $T_i$ be the length of $\gamma_i$, we have by~\eqref{e:c2Bound} that
$
\sup_{\rho\in \gamma_i(\delta)}d(\rho,\varphi\sub{T_i}^g(\rho))\leq \tilde{C}{\delta} e^{LT_i}.
$
Therefore, using Lemma~\ref{l:vol} with $\tilde{C}{\delta} e^{LT}$ in place of $\varepsilon    $, {there is $C>0$ such that},
\begin{equation}\label{e:volumize}
\sum_{1\leq i\leq \mathfrak{c}(T,g)}\!\!\! \vol\sub{S^*\!M}(\gamma_i(\delta))
=\vol\sub{S^*\!M}\Big(\bigcup_{1\leq i\leq \mathfrak{c}(T,g)} \!\!\!\!\gamma_i(\delta) \Big)
\leq 
C\delta^{2d-2}{e^{CT+(2d-2)LT}}{\n(T)^{-(4d-3)}}.
\end{equation}
In addition, since there is $c>0$ such that
$
\vol\sub{S^*\!M}(\gamma_i(\delta))\geq c{\delta}^{2d-2}t_0
$
{for all $1\leq i\leq \mathfrak{c}(T,g)$,}
we obtain, using~\eqref{e:disjoint} and~\eqref{e:volumize} that there is $C>0$ such that
$$
\mathfrak{c}(T,g)t_0
\leq C{e^{CT+(2d-2)LT}}{\n(T)^{-(4d-3)}}.
$$
Together with the fact that $\mathfrak{c}(T,g)=0$ for $0<T<t_0$, this implies Theorem~\ref{t:predominantLengths}.
%
\qed
\begin{remark}
    \blue{Notice that, any improvement in $\n(t)$, leads to a corresponding improvement in Theorem~\ref{t:predominantLengths} until $\n(t)\geq e^{-Ct}$ for some $C>0$ depending on finitely many seminorms of the metric.}
\end{remark}

\subsection{Proof of Theorem \ref{t:predominantWeyl}}
\label{s:pWeyl}
Let $g$ be a $\mc{C}^\nu$-Riemannian metric on a manifold $M$ of dimension $d$. For $R>0$ let $\mc{P}^R(t_0,T)\subset \SM$ be the set of directions that yield trajectories that are $R$ close to being periodic at some time with $t_0\leq |t|\leq T$. That is,
\begin{equation}
\label{e:sadWithoutNumber}
\mc{P}^R(t_0,T):=\bigg\{ \rho \in S^*\!M\, :\, \bigcup_{t_0\leq |t|\leq T}\varphi_t^g(B\sub{S^*\!M}(\rho,R))\cap B\sub{S^*\!M}\!(\rho,R)\neq \emptyset\bigg\}.
\end{equation}
In  \cite[Theorem 2]{CG20Weyl} it is proved that if $\nu\geq \nu_0$ and $\Ti(R)$ is a sub-logarithmic resolution function and there exist $C>0$ and $t_0$ such that
\begin{equation}\label{e:volCond}
\limsup_{R\to 0^+}\vol\sub{S^*\!M}\bigg(B\Big(\,\mc{P}^R\big(t_0,\Ti(R)\big)\,,\,R\,\Big)\bigg)\Ti(R)\leq  C ,
\end{equation}
then, as $\lambda \to \infty$,  
\begin{equation}\label{e:evss}
\#\{j:\;\lambda_j(g)\leq \lambda\} = \frac{\vol_{\mathbb{R}^d}(B^d)\vol_g(M)}{(2\pi )^d}\lambda^d +O\Big(\frac{\lambda^{d-1}}{\Ti(\lambda^{-1})}\Big).
\end{equation}

\begin{remark}
The application of~\cite[Theorem 2]{CG20Weyl} deserves a brief comment. Note that, as stated, the theorem is valid for $\mc{C}^\infty$ metrics. However, it is clear that all of the proofs rely only on the $\mc{C}^\nu$ norm of the metric for some sufficiently large $\nu$. Therefore, provided that we work with sufficiently smooth metrics, we may apply~\cite[Theorem 2]{CG20Weyl}. 
\end{remark}

The proof of Theorem \ref{t:predominantWeyl}, given \blue{below}, is then reduced to proving that {for any $\someLetter>\someLetter_\nu$ and some} $C_0>0$ 
\begin{equation}\label{e:ourT}
\Ti(R)= f^{-1}(R^{\gamma}),
\end{equation}
with 
$$
\gamma:=\frac{2d-2}{4d-3}, \qquad  f(t)=\frac{\n(t)}{\n(0)}e^{-C_0t}, \qquad \n(t):= C^{-C(t+1)^{\someLetter}-1} 
$$
is a  sub-logarithmic resolution function (\blue{as defined in \eqref{e:subLog}.})and that \eqref{e:volCond} holds.  Indeed, Theorem \ref{t:predominantWeyl} follows once \blue{we show that}
\begin{equation}\label{e:lBound}
{C_1}(\log \lambda)^{\frac{1}{\someLetter}}\geq \Ti(\lambda^{-1})=f^{-1}(\lambda^{-\gamma})\geq  \frac{1}{C_1}(\log\lambda)^{\frac{1}{\someLetter}},\qquad {\lambda\geq \lambda_0}
\end{equation}
for some $C_1>0$ and $\lambda_0>0$ large enough.

\blue{To see that~\eqref{e:lBound} holds, observe that since $f$ is decreasing, it is equivalent to 
\begin{gather}
\label{e:annoyingLogs}
C^{-C(C_1(\log \lambda)^{\frac{1}{\Omega}}+1)^{\Omega}+C}e^{-C_0C_1(\log \lambda)^{\frac{1}{\Omega}}}=f(C_1(\log \lambda)^{\frac{1}{\Omega}})\leq \lambda^{-\gamma},\\
\label{e:annoyingLogs2}
\lambda^{-\gamma}\leq f\Big(\tfrac{1}{C_1}(\log \lambda)^{\frac{1}{\Omega}}\Big)=C^{-C(\tfrac{1}{C_1}(\log \lambda)^{\frac{1}{\Omega}}+1)^{\Omega}+C}e^{-\frac{1}{C_1}C_0(\log \lambda)^{\frac{1}{\Omega}}}.
\end{gather}
Taking $C_1>0$ large enough, we can make the left-hand side of~\eqref{e:annoyingLogs} smaller than any negative power of $\lambda$ and the right hand side of~\eqref{e:annoyingLogs2} larger than any negative power of $\lambda$.  Hence~\eqref{e:lBound} holds.

We also observe for later use that
\begin{equation}\label{e:betaObs}
    \beta(\Ti(R))=R^\gamma \beta(0) e^{C_0\Ti(R)}.
\end{equation}
}

Before we proceed to the proof of Theorem~\ref{t:predominantWeyl}, we recall the notion of a sub-logarithmic resolution function from~\cite[Definition 1.1]{CG20Weyl}.
We say that $\Ti:(0,1)\to (0,\infty)$ is a \emph{resolution function} if it is continuous and decreasing. We say that $\Ti$ is \emph{sub-logarithmic} if it is differentiable and 
\begin{equation}\label{e:subLog}
(\log \Ti(R))'\geq -\frac{1}{R\log R^{-1}}\qquad 0<R<1.
\end{equation}
Next, we introduce a convenient class of sub-logarithmic rate functions.
\begin{lemma}
\label{l:subLog}
Suppose that $A(t):\mathbb{R}\to \mathbb{R}$ is \blue{$C^1$, strictly increasing, convex,} and satisfies $A(0)=0$.
Then, with $f(t)=e^{-A(t)}$, we have that $\Ti(R):=f^{-1}(R^\gamma)$ is a sub-logarithmic rate function for $\gamma>0$.
\end{lemma}
\begin{proof}
First, observe that $\Ti(R)$ is decreasing and differentiable. Moreover,
$$
(\log \Ti)'(R)=\frac{\gamma R^\gamma}{Rf'(f^{-1}(R^\gamma))f^{-1}(R^\gamma)}=-\frac{\gamma}{RA'(f^{-1}(R^\gamma))f^{-1}(R^\gamma)}.
$$
Next, note that \blue{by the mean value theorem there is $t_c\in (0,t)$ such that 
$$
A(t)=A'(t_c)t\leq A'(t)t
$$
where the last inequality follows from convexity of $A$.}
Therefore, 
$
A'(f^{-1}R^\gamma)f^{-1}(R^\gamma)\geq A(f^{-1}(R^\gamma))=\gamma\log R^{-1},
$
and hence \eqref{e:subLog} holds as claimed.
\end{proof}

\begin{remark}
By Lemma \ref{l:subLog}, the function $\Ti(R)$ as defined in \eqref{e:ourT} is a sub-logarithmic resolution function. Indeed, $f(t)=e^{-A(t)}$ with 
$ A(t)=C_0t+C((t+1)^{\someLetter}-1).$
Note that $A(0)=0$, and, since $\someLetter\geq 1$, it is easy to check that $A'(t)>0$, and \blue{$A$ is convex} for $t\geq 0$ as claimed. 
\end{remark}

\begin{proof}[Proof of Theorem~\ref{t:predominantWeyl}]
Let $\someLetter>\someLetter_\nu$. By Theorem~\ref{t:predominantR-ND} there is a predominant set $G\sub{\someLetter} \subset \ms{G}^\nu$ such that  for all $g \in G_{\someLetter}$ and  $\mc{W}$ a family of transition maps for $g$ there is $C>0$ such that \eqref{e:imNondegenerate} holds with {$\n(t)=C^{-C(t+1)^{\someLetter}-1}$}. Without loss of generality we assume $C>1$. In Lemma \ref{l:subLog} we proved that $\Ti$ as defined in \eqref{e:ourT} is a {sub-logarithmic resolution function}. Therefore, as explained above, the proof of Theorem~\ref{t:predominantWeyl} would follow from combining \eqref{e:evss} and \eqref{e:lBound} once we prove \eqref{e:volCond}.

{Let $0<t_0<T$. In order to prove \eqref{e:volCond}, we first claim that there is ${C}>0$ such that
\begin{equation}\label{e:squirrel}
B(\mc{P}^R(t_0,T),R) \subset \Big\{q \in \SM: \; \exists t\in(t_0,T)\cup(-T,-t_0)\; \text{s.t.}\;\; d(\varphi_t(q),q)\leq {C}e^{{C}T}R \Big\}.
\end{equation}}
Indeed, for  $q \in B(\mc{P}^R(t_0,{T}),R)$, there is $\rho\in \mc{P}^R(t_0,T)$ such that $d(\rho,q)<R$. Then, there are $t\in(t_0,{T}){\cup ({-T}, -t_0)}$ and $\rho_1\in S^*\!M$ such that $d(\rho,\rho_1)<R$ and $d(\varphi_t(\rho_1),\rho)<R$. Since $\varphi_t((x,\xi))=\varphi_{-t}(x,-\xi)$, we may assume without loss of generality that $t>0$. Therefore, {by \eqref{e:c2Bound} there is {$C>0$} such that} 
\begin{align*}
d(\varphi_t(q),q)\leq d(\varphi_t(q),\varphi_t(\rho_1))+d(\varphi_t(\rho_1),\rho)+d(\rho,q)
\leq \tilde Ce^{Lt}d(q,\rho_1)+R+R\leq {C}e^{{C}t}R,
\end{align*}
proving the claim in \eqref{e:squirrel}.

{By \eqref{e:squirrel} and} Lemma~\ref{l:vol} \blue{with $\e= Ce^{CT}R$}, for 
$
R\leq \n(T){ (Ce^{CT})^{-1}},
$
we have
\begin{equation} 
\label{e:oneMoreVolume}
\vol\sub{S^*\!M}\big(B(\mc{P}^R(t_0,T),R)\big)T\leq CT({C}e^{{C}T}R)^{2d-2}{e^{CT}}{\n(T)^{-(4d-3)}}\leq {Ce^{C_1T}R^{2d-2}\n(T)^{-(4d-3)}}.
\end{equation}

\blue{By~\eqref{e:betaObs}, the fact that $\gamma<1$, and~\eqref{e:lBound} for $R$ small enough 
$$
 \n(\Ti(R))(C e^{C\Ti(R)})^{-1}= R^\gamma \n(0)e^{(C_0-C)\Ti(R)}\geq C_{\delta}R^{\gamma+\delta} \beta(0)>R.
$$
Hence~\eqref{e:oneMoreVolume} holds with $T=\Ti(R)$. }

    \blue{Choosing {$C_0=\frac{C_1}{4d-3}$} and using~\eqref{e:betaObs} we also obtain 
    \begin{align*}
\vol\sub{S^*\!M}\big(B(\mc{P}^R(t_0,\Ti(R)),R)\big)\Ti(R)&\leq {Ce^{C_1T}R^{2d-2}\n(T)^{-(4d-3)}}\\
    &=Ce^{C_1\Ti(R)}R^{2d-2}(R^\gamma \beta(0)e^{C_0\Ti(R)})^{-(4d-3)}\\
    &=Ce^{(C_1-C_0(4d-3))\Ti(R)}\beta(0)\leq C,
    \end{align*}
    as claimed.}
\end{proof}



\section{Outline of the proof of Theorem~\ref{t:predominantR-ND}} \label{s:outline}


To simplify the exposition in this outline, we will {first} imagine that it is only necessary to understand exactly periodic points rather than returning points. We also assume that one can construct a global Poincar\'e section (see Section \ref{s:gP}). In addition, rather than providing the details to prove the result for a predominant set of metrics, we will outline a proof of the theorem for a dense set of metrics.  See the end of the section for remarks on how to drop these assumptions and prove predominance.

\subsection{The perturbation}
The first key point to understand when proving Theorem~\ref{t:predominantR-ND} is how to perturb away the degeneracy of a single periodic orbit with quantitative control on how large a perturbation is needed to produce non-degeneracy. When doing this, it is important to use a perturbation of the metric which interacts with the periodic geodesic exactly once. That is, the perturbation must be supported in a ball, {$B\subset M$}, over which the geodesic passes exactly once. Because of this, we will only be able to directly perturb away degeneracy for {primitive} closed geodesics.   The construction of one such perturbation is done in Section \ref{s:perturbedMetrics}. 
\begin{remark}{Our main inductive argument is actually proved under some general assumptions on the perturbation (see Definitions~\ref{ass:1} and~\ref{ass:2}). Because of this, we postpone this construction until after our main inductive argument, which appears in Section~\ref{s:theMadness}.}
\end{remark}


Restricting our attention to a single geodesic, $\gamma_{g_0}$, for a metric $g_0$, we find a finite dimensional family of perturbations $g_{\sigma}$ of the metric $g_0$. When $\gamma_{g_0}$ is a primitive periodic geodesic, these perturbations should produce non-degeneracy of the orbit $\gamma_{g_\sigma}$. The perturbations are a combination of those considered {by Anosov}~\cite{An:82} and Klingenberg~\cite[\S 3.3]{Kl:78} (see also Klingenberg--Takens~\cite{KlTa:72}). In our case, when $\gamma_{g_0}$ is primitive, we give careful estimates on how these perturbations affect the Poincar\'e map, $\ms{P}_{\gamma_{g_\sigma}}$, associated to $\gamma_{g_\sigma}$ and its derivative, $d\ms{P}_{\gamma_{g_\sigma}}$. In particular, estimating the size of the inverse of the {derivative of the} map  $\sigma \mapsto (\ms{P}_{\gamma_{g_\sigma}},d\ms{P}_{\gamma_{g_\sigma}})$. 

Once we have a family of perturbations for any given geodesic, we cover $S^*\!M$ by finitely many small balls, $B_i$, of radius $r$ and center $\rho_i$ and associate to each ball a family of perturbations modelled on those above; replacing the geodesic $\gamma_{g_0}$ by the geodesic, $\gamma_{g_0,\rho_i}$ for $g_0$ through $\rho_i$. After an approximation argument, this gives a finite (albeit very large) dimensional family of perturbations which can be used to perturb away degeneracy for primitive geodesics of a given length. In particular, for every such geodesic {$\gamma_{g_0}$}, we can find a ball, $B_i$, such that, with $g_{i,\sigma}$ the perturbation of $g_0$ associated to $\gamma_{g_0,\rho_i}$, the derivative of the map {$\sigma \mapsto (\ms{P}_{\gamma_{g_{i,\sigma},\rho_i}},d\ms{P}_{\gamma_{g_{i,\sigma},\rho_i}})$} is invertible at $\sigma$ for all $\sigma \in B_{\mathbb{R}^\L}(0,1)$.  Once we have this in place, it will be possible to control the volume of the set of perturbations for which there is a degenerate primitive closed geodesic of some length, with some quantitative control on how degenerate such an orbit may be. In particular, once this volume is small enough, there is at least one perturbation of $g_0$ for which all primitive orbits of a certain length are non-degenerate. Since $g_0$ is arbitrary, and we make our perturbations arbitrarily small, we will eventually obtain density after an induction on the length of the closed geodesics.

\begin{remark}
In order to obtain predominance in our main theorem, we use the control on the volume of bad perturbations more seriously. In particular, making it smaller than $\e>0$ for any chosen $\e$.
\end{remark}

\subsection{The Induction}\label{s:theInduction}
The proof of Theorem~\ref{t:predominantR-ND} relies on an induction on the length of an orbit. In order to do this, we follow a strategy motivated by that of Yomdin~\cite{Yo:85}. {The work in \cite{Yo:85} shows that every diffeomorphism $f_0$ on $M$ can be perturbed to a diffeomorphism $f_\e$ in which every $n$-periodic point $x$ is $\gamma_\e$-hyperbolic in the sense that the eigenvalues $\lambda_j$ of $df_\e^n(x)$ satisfy $||\lambda_j|-1|\geq \gamma_\e$ where $\gamma_\e$ is a power of $\e$ depending on $n, \e$ and the regularity of $f_0$.}  The idea there is to first use a perturbation to make primitive orbits hyperbolic and then to use the fact that hyperbolicity of an orbit passes to its multiples and, moreover, that {every multiple is more hyperbolic than the previous multiple}.  

{In contrast to what happens when working with diffeomorphisms on $M$,} one of the main difficulties to overcome in the case of geodesic flows (or indeed symplectic maps) is that, in general, it is not possible to perturb a closed orbit to turn it into a  hyperbolic closed orbit. Indeed, eigenvalues of the Poincar\'e map may be `trapped' on the unit circle {(see Section \ref{s:evs})} and hence there are so-called stable closed orbits which cannot be perturbed away. 
Because of this structure, one must find a different property which guarantees non-degeneracy and can be passed from a primitive closed orbit to its multiples.

Given a metric $g_0$, our objective is to find a metric $g_\infty$, that is arbitrarily close to $g_0$, with the property that for each $\ell$  every periodic trajectory of length $\leq 2^\ell$ is $\beta_\ell$ non-degenerate in the sense of Definition~\ref{d:nonDegenerate-g}.  Here, $\{\beta_\ell\}_\ell$ is a decreasing sequence of numbers.

We build $g_\infty$ by induction on the parameter $\ell$ which represents {(the logarithm of)} the maximum length of the orbits up to which non-degeneracy is controlled.
Note that we start with the most naive possible version of the induction and gradually add the details necessary to handle our situation.

\begin{adjustwidth}{1cm}{}
\noindent {\bf A.} \emph{Hypothesis}. Assume $g_\ell$ is a metric such that for $k\leq \ell$, all closed trajectories of length $ n\in(2^{k-1}, 2^k]$ are  $\beta_k$ non-degenerate in the sense of Definition~\ref{d:nonDegenerate-g}.

\noindent{\bf B.} \emph{Perturbation:} 
    Find a perturbation, $g_{\ell+1}$ of $g_\ell$ which satisfies
      \begin{itemize}
      \item primitive closed trajectories of length $n \in (2^\ell, 2^{\ell+1}]$ are $\beta_{\ell+1}$ non-degenerate, 
      \item closed trajectories of  length
      $n \in (2^{k-1},2^k]$ with $k\leq \ell$ remain $\beta_k$ non-degenerate.
      \end{itemize}
      
\noindent{\bf C.} \emph{Deal with non-primitive closed trajectories:} A non-primitive closed trajectory $\gamma$  for $g^{\ell+1}$ of length {$n\in (2^\ell,2^{\ell+1}]$} is a multiple of a primitive closed trajectory $\tilde{\gamma}$ of length $m \in (2^{k-1},2^k]$ with $k \leq \ell$. By the inductive hypothesis  $\tilde{\gamma}$ is {$\beta_{k}$} non-degenerate. Use this to show that the {$\beta_{k}$} non-degeneracy of $\tilde{\gamma}$ implies {$\beta_{\ell+1}$} non-degeneracy of $\gamma$.
\end{adjustwidth}
\smallskip 
\noindent{\emph{Why it fails:}} The {$\beta_{k}$} non-degeneracy of the orbit $\tilde{\gamma}$ does not imply {$\beta_{\ell+1}$} non-degeneracy of its multiple $\gamma$. Indeed, the Poincar\'e map associated to $\gamma$ may have a root of unity as an eigenvalue. That is, Step C cannot be completed. 

\noindent{\emph{Solution:}} We define the notion of $(\beta,q)$-nondegeneracy. For $\aleph>0$ and $\beta, q>0$  we say a matrix $\bm{A}\in\mathbb{M}(\aleph)$, the set of $\aleph\times \aleph$ matrices, is $(\beta,q)$
non-degenerate if 
\begin{equation}
\big\|(\Id-\bm{A}^{q})^{-1}\big\|\leq {(\tfrac{5}{2}q^2 \beta^{-1})}^{\aleph}(1+\|\bm{A}\|^{q})^{\aleph-1}.\label{e:q-Nondegenerate}
\end{equation}
The key observation here is that, if the derivative, $\bm{A}=d\ms{P}_{\gamma}$, of the Poincar\'e map associated to a closed orbit {$\gamma$} is $(\beta,q)$-nondegenerate, then the $q$ iterate of the orbit is $(\tfrac{5}{2}q^2 \beta^{-1})^{-\aleph}\blue{C^{-\text{length}(\gamma)}}$ non-degenerate \blue{for some $C>0$ depending on finitely many $\mathcal{C}^k$ norms of the metric} {(see Figure~\ref{f:nonDegeneracy} for a schematic of a $(\beta,2)$ non-degenerate orbit)}. Therefore, our goal will be to make perturbations of the metric, $g_\ell$, so that primitive orbits become $(\beta(q),q)$-non-degenerate for \emph{all} $q$ and some sequence $\{\beta(q)\}_{q=1}^\infty$. This will be possible provided that $\beta(q)=O(q^{-2-0})$ {as $q\to \infty$}. The main difficulty with the notion of $(\beta(q),q)$ non-degeneracy is that, although for each fixed $(\beta(q),q)$ the property of $(\beta(q),q)$ non-degeneracy is stable under small perturbations, the property of being $(\beta(q),q)$ non-degenerate for all $q$ is not. {Therefore, we only look for non-degeneracy for $q<Q$.}

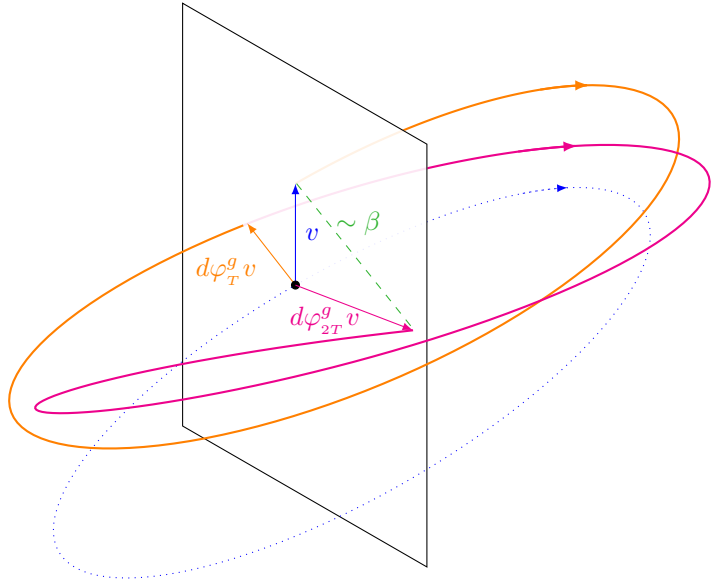
\begin{figure}
 \begin{tikzpicture}
        \def \rot{-30}
        \def \ecc{3};
        \def \eccb{3.2}
        \def \arrowTime{1};
        \def \vx{0}
        \def \vy{.9}
        \def \phase{1.1}
     \def \amp{.0025}
            \def \ampX{.0035}  
        \def \eTime{6.1}
      \def \eTimeb{12.5}
        \def \width{2.5}
        \def \cX{0}
        \def \cY{0}
        \def \shiftPad{.01}
        \def \textPlace{.75}

      
      \begin{scope}[scale=1.5]
            \draw[color=orange,thick, ->,domain=0: \arrowTime, samples =100]   plot ({(\vx+\cX+cos(\rot)*\eccb*sin((\x) r)+sin(\rot)*(cos( (\phase *\x)  r )-1)},{-sin(\rot)*\eccb*sin((\x) r)+cos(\rot)*(cos( (\phase*\x)  r )-1)+\vy+\cY})   node[right] {};
          \draw[color=magenta,thick, ->,domain=\eTime:{\eTime+ \arrowTime}, samples =100]   plot ({(\vx+\cX+cos(\rot)*\eccb*sin((\x) r)+sin(\rot)*(cos( (\phase *\x)  r )-1)+\ampX*(\x-\eTime)*(\x-\eTime)*(\x-\eTime)},{-sin(\rot)*\eccb*sin((\x) r)+cos(\rot)*(cos( (\phase*\x)  r )-1)+\vy+\cY-\amp*(\x-\eTime)*(\x-\eTime)*(\x-\eTime)})   node[right] {};
        \draw[color=blue,thin ,->,dotted,domain=0: \arrowTime, samples =100]   plot ({(cos(\rot)*\ecc*sin((\x) r)+sin(\rot)*(cos( (\x)  r )-1)+\cX},{-\ecc*sin(\rot)*sin((\x) r)+cos(\rot)*(cos( (\x)  r )-1)+\cY})   node[right] {};
      \filldraw[white, opacity=0.9]  ({-.4*\width},{-.5*\width})++(-30:{1*\width})-- ++(90:{1.5*\width})--++(150:{1*\width})-- ++(-90:{1.5*\width})--cycle;
            \draw  ({-.4*\width},{-.5*\width})++(-30:{1*\width})-- ++(90:{1.5*\width})--++(150:{1*\width})-- ++(-90:{1.5*\width})--cycle;
      \draw[->,blue] (\cX,\cY)--({\cX+\vx/2},{\cY+\vy/2})node[right]{$v$}--(\cX+\vx,\cY+\vy);
     \draw[color=blue,thin,dotted, domain= \arrowTime-.2:6.28, samples =100]   plot ({(cos(\rot)*\ecc*sin((\x) r)+sin(\rot)*(cos( (\x)  r )-1)+\cX},{-sin(\rot)*\ecc*sin((\x) r)+cos(\rot)*(cos( (\x)  r )-1)+\cY})   node[right] {};
      \draw[color=orange,thick, domain= \arrowTime-.2:{\eTime}, samples =100]   plot ({(\vx+cos(\rot)*\eccb*sin((\x) r)+sin(\rot)*(cos( (\phase *\x)  r )-1)},{-sin(\rot)*\eccb*sin((\x) r)+cos(\rot)*(cos( (\phase*\x)  r )-1)+\vy})   node[right] {};
             \draw[color=magenta, thick,domain={\eTime+ \arrowTime-.2}:{\eTimeb}, samples =100]   plot ({\vx+cos(\rot)*\eccb*sin((\x) r)+sin(\rot)*(cos( (\phase *\x)  r )-1)+\ampX*(\x-\eTime)*(\x-\eTime)*(\x-\eTime)},{-sin(\rot)*\eccb*sin((\x) r)+cos(\rot)*(cos( (\phase*\x)  r )-1)+\vy-\amp*(\x-\eTime)*(\x-\eTime)*(\x-\eTime)})   node[right] {};
      \draw[->,orange] (\cX,\cY)-- ({(\vx+\cX+cos(\rot)*\ecc*sin((\eTime) r)+sin(\rot)*(cos( (\phase *\eTime)  r )-1)},{-sin(\rot)*\ecc*sin((\eTime) r)+cos(\rot)*(cos( (\phase*\eTime)  r )-1)+\vy+\cY});
      
      \draw[orange] node at ({(\vx+\cX+1/2*cos(\rot)*\ecc*sin((\eTime) r)+1/2*sin(\rot)*(cos( (\phase *\eTime)  r )-1)-.4},{-1/2*sin(\rot)*\ecc*sin((\eTime) r)+1/2*cos(\rot)*(cos( (\phase*\eTime)  r )-1)+\vy+\cY-.6}){$d\varphi^g_{_{T}}v$};

      \filldraw (\cX,\cY) circle (1pt);
      
      \draw[->,magenta] (\cX,\cY)--({\cX*\textPlace+ (1-\textPlace)*(\vx+cos(\rot)*\eccb*sin((\eTimeb) r)+sin(\rot)*(cos( (\phase *\eTimeb)  r )-1)+\ampX*(\eTimeb-\eTime)*(\eTimeb-\eTime)*(\eTimeb-\eTime))},{\cY*\textPlace+(1-\textPlace)*(-sin(\rot)*\eccb*sin((\eTimeb) r)+cos(\rot)*(cos( (\phase*\eTimeb)  r )-1)+\vy-\amp*(\eTimeb-\eTime)*(\eTimeb-\eTime)*(\eTimeb-\eTime))})node[below]{$d\varphi_{_{2T}}^gv$}--({\vx+cos(\rot)*\eccb*sin((\eTimeb) r)+sin(\rot)*(cos( (\phase *\eTimeb)  r )-1)+\ampX*(\eTimeb-\eTime)*(\eTimeb-\eTime)*(\eTimeb-\eTime)},{-sin(\rot)*\eccb*sin((\eTimeb) r)+cos(\rot)*(cos( (\phase*\eTimeb)  r )-1)+\vy-\amp*(\eTimeb-\eTime)*(\eTimeb-\eTime)*(\eTimeb-\eTime)});
      
        \draw[verde,dashed] (\cX+\vx+\shiftPad,\cY+\vy)
        --({\vx+\shiftPad+cos(\rot)*\eccb*sin((\eTimeb) r)+sin(\rot)*(cos( (\phase *\eTimeb)  r )-1)+\ampX*(\eTimeb-\eTime)*(\eTimeb-\eTime)*(\eTimeb-\eTime)},{-sin(\rot)*\eccb*sin((\eTimeb) r)+cos(\rot)*(cos( (\phase*\eTimeb)  r )-1)+\vy-\amp*(\eTimeb-\eTime)*(\eTimeb-\eTime)*(\eTimeb-\eTime)});
        
        \draw[verde] 
        ({1/2*(\vx+\shiftPad+cos(\rot)*\eccb*sin((\eTimeb) r)+sin(\rot)*(cos( (\phase *\eTimeb)  r )-1)+\ampX*(\eTimeb-\eTime)*(\eTimeb-\eTime)*(\eTimeb-\eTime))+3*\shiftPad},{1/2*(\cY+\vy)+1/2*(\cX+\vx+\shiftPad)+1/2*(-sin(\rot)*\eccb*sin((\eTimeb) r)+cos(\rot)*(cos( (\phase*\eTimeb)  r )-1)+\vy-\amp*(\eTimeb-\eTime)*(\eTimeb-\eTime)*(\eTimeb-\eTime))+10*\shiftPad})node[above]{$\sim\beta$} ;
        \end{scope}
      
        \end{tikzpicture}
        
        \caption{\label{f:nonDegeneracy} {In the figure, we show a primitive periodic orbit of length $T$ such that the derivative {$d\ms{P}$} of the Poincar\'e map  is $(\beta,2)$ non-degenerate.  The primitive orbit is shown in the blue dotted line. Note that  $d\ms{P}$ satisfies $(d\ms{P})^k=d\varphi_{kT}^g$. Thus, the property of $(\beta,2)$ non-degeneracy implies that $\|d\varphi_{2T}^gv-v\|\sim \beta$.}}
\end{figure}

We next present the modified induction argument. The induction is done on the parameter $\ell$. One can find sequences of numbers $\{\beta_k\}_k$, $\{\beta_{k, \ell}\}_{k, \ell}$, and $\{Q_{k,\ell}\}_{k,\ell}$ so that the induction below can be completed.

\begin{adjustwidth}{1cm}{}
\noindent{\bf A.} \emph{Hypothesis}. Assume $g_\ell$ is a metric such that for $k \leq \ell$ 
 \begin{itemize}
      \item[--]  closed trajectories of length $n \in (2^{k-1}, 2^k]$  are $\beta_k$ non-degenerate,
      \item[--]  primitive closed trajectories of length $n \in (2^{k-1}, 2^k]$ are $(\beta_{k,\ell}q^{-3},q)$ non-degenerate for all $0<q<Q_{k,\ell}$.
      \end{itemize}

\noindent{\bf B.} \emph{Perturbation:} 
Find a perturbation, $g_{\ell+1}$, of $g_\ell$ which satisfies
      \begin{itemize}
      \item[--] primitive closed trajectories of length $n \in (2^\ell, 2^{\ell+1}]$ for $g_{\ell+1}$  are {$\beta_{\ell+1}$} non-degenerate
      \item[--] closed trajectories of length 
      $n \in (2^{k-1},2^k]$ for $g_{\ell+1}$  remain $\beta_k$ non-degenerate for all $k\leq \ell$.
      \item[--]  primitive closed trajectories of length 
      $n \in (2^{k-1},2^k]$ for $g_{\ell+1}$ are $(\beta_{k,\ell+1}q^{-3},q)$ non-degenerate for $0<q<Q_{k,\ell+1}$ and $k\leq {\ell+1}$,
      \end{itemize}

\noindent{\bf C.} \emph{Deal with non-primitive closed trajectories:}  Let $\gamma$ be a closed trajectory for $g_{\ell+1}$ of length {$n\in (2^\ell,2^{\ell+1}]$} that is non-primitive. Then $\gamma$  is a multiple of a primitive closed trajectory $\tilde{\gamma}$ of length $m \in (2^{k-1},2^k]$ with $k \leq \ell$.
Note that then $\tilde \gamma$ satisfies the hypothesis in Step A.
In particular, 
to show that $\gamma$ is $\beta_{\ell+1}$ non-degenerate,  
we would like to use that $\tilde \gamma$ is $(\beta_{k,\ell}q^{-3}, q)$ non degenerate with $q=n/m$, and for that we need  $n/m<Q_{k,\ell}$. We then show that ${2^{\ell+1-k}}<Q_{k,\ell}$ and that by the inductive hypothesis this implies {$\beta_{\ell+1}$} non-degeneracy of $\gamma$.
\end{adjustwidth}

After the steps A-C above (modulo the fact that we have only considered exactly closed orbits), the inductive argument yields a metric $g_\infty$ which is arbitrarily close to $g_0$ and has the desired non-degeneracy property that, for each $\ell$, every closed trajectory for $g_\infty$ of length $\leq 2^\ell$ is $\beta_\ell$ non-degenerate. This yields a discretized version of the statement in Theorem \ref{t:predominantR-ND}.


It is important to note that one cannot do a simplified version of this induction in which a sequence $\{Q_k\}_k$ is used in place of $\{Q_{k,\ell}\}_{k,\ell}$ and  $\{\beta_k\}_k$ is used in place of $\{\beta_{k,\ell}\}_{k,\ell}$. The problem would be that, when applying the inductive hypothesis to carry out Step C, one would need $n/m < Q_k$ so that  the non-degeneracy of $\tilde{\gamma}$ may  pass to that of $\gamma$. However, this may not be possible to arrange. For example, for a closed trajectory of length $n$ that is a multiple of a primitive trajectory of length $m=1$. One might wonder, at this point, why we do not simply take $Q_{k,\ell}=Q_\ell=2^{\ell+2}$. We will see below, in sections~\ref{s:almostClosed} and~\ref{s:stepCOutline}, that this is indeed not possible.

\subsection{How to deal with almost closed trajectories}
\label{s:almostClosed}
Next, we explain how to modify the argument to deal with trajectories that are not necessarily closed and primitive but are very close to being such. First of all,  instead of working with closed trajectories, one works with $(n, \alpha)$ returning trajectories. In this discussion a point $\rho$ is said to be $(n, \alpha)$ returning if $d(\rho, \ms{P}^n(\rho))\leq \alpha$, where we are still pretending that we can work with a globally defined Poincar\'e section $\ms{P}$ (see Section \ref{s:gP}).

There is one serious difficulty which is hidden by imagining that we only deal with exactly closed trajectories. This difficulty arises from the fact that the notion of {primitiveness} for non-closed trajectories is not a well-defined condition. One should instead think of a degree of simplicity. We will say that a point $\rho$ is $(n, \alpha)$ simple if  $d(\rho, \ms{P}^k(\rho))>\alpha$  for $0<k<n$.
Instead of working with primitive closed trajectories of length $n$ one works with trajectories that are $(n, \alpha)$ simple for some $\alpha$.

It is in fact only possible to make an $(n,\alpha)$ simple point  be $(\beta q^{-3},q)$ non-degenerate whenever $\beta$ is small depending on $\alpha$. Therefore, there are sequences $\{\beta_{k,\ell}\}_{k, \ell}$ and $\{\alpha_{k,\ell}\}_{k, \ell}$ such that  in the $\ell$-th step of the induction argument, orbits of length $n \in (2^{k-1}, 2^k)$ which are $\beta_{k,\ell}$ returning can only be made to be $(\beta_{k,\ell}q^{-3},q)$ non-degenerate for $0<q< Q_{k,\ell}$ provided they are $(n,\alpha_{k,\ell})$-simple and $\beta_{k,\ell}$ is small in terms of $\alpha_{k,\ell}$.

The induction argument is the same as the one outlined before but with each instance of the word `closed' replaced by $(n,\beta)$ returning for some parameter $\beta$ depending on $(k,\ell)$ and each instance of the word `primitive' replaced by $(n,\alpha)$ simple for some parameter $\alpha$ depending on $(k,\ell)$.

\begin{adjustwidth}{1cm}{}
\noindent{\bf A.} \emph{Hypothesis}. Assume $g_\ell$ is a metric such that for $k \leq \ell$
 \begin{itemize}
      \item[--]  $\beta_{k,\ell}$ returning trajectories of length $n \in (2^{k-1}, 2^k]$  are $\beta_{k}$ non-degenerate,
      \item[--]  $\beta_{k,\ell}$ returning trajectories that are $\alpha_{k,\ell}$ simple  and have length $n \in (2^{k-1}, 2^k]$ are $(\beta_{k,\ell}q^{-3},q)$ non-degenerate for all $0<q<Q_{k,\ell}$.
      \end{itemize}

\noindent{\bf B.} \emph{Perturbation:} 
Find a perturbation, $g_{\ell+1}$, of $g_\ell$ which satisfies
      \begin{itemize}
      \item[--] $\beta_{\ell+1,\ell+1}$ returning trajectories that are $\alpha_{\ell+1,\ell+1}$ simple and have length $n \in (2^\ell, 2^{\ell+1}]$  are {$\beta_{\ell+1}$} non-degenerate
      \item[--] $\beta_{k,\ell+1}$ returning  trajectories of  length 
      $n \in (2^{k-1},2^k]$ will remain $\beta_k$ non-degenerate for all $k\leq \ell$.
      \item[--]  $\beta_{k,\ell+1}$ returning trajectories that are $\alpha_{k,\ell+1}$ simple  and have length  $n \in (2^{k-1},2^k]$ will be $(\beta_{k,\ell+1}q^{-3},q)$ non-degenerate for $0<q<Q_{k,\ell+1}$ and $k\leq {\ell+1}$.
      \end{itemize}

\noindent{\bf C.}  \emph{Deal with non-simple trajectories:} This step becomes substantially more complicated as a result of needing to handle non-closed trajectories. One consequence of the new difficulties, is that, instead of using $2$-adic intervals, we must use $\step$-adic intervals for some $1<\step <2$. 
We include a sketch of how to handle Step C below in Section~\ref{s:stepCOutline}, in order to present a more or less complete picture of the induction. However, the reader may wish to skip the next section on first reading.
\end{adjustwidth}

\subsection{Detailed sketch of Step C}
\label{s:stepCOutline}
Suppose that $\rho$ yields a trajectory of length $n \in (2^{\ell},2^{\ell+1}]$  that is $(n,\beta_{\ell+1, \ell+1})$ returning  but not $(n,\alpha_{\ell+1,\ell+1})$ simple. 

It is in general not possible to find an $m\leq n/2$ such that $\rho$ is $(m,\alpha_{\ell+1,\ell+1})$ returning and $(m,\alpha_{\ell+1,\ell+1})$ simple. 
Instead, one only obtains $(m_1, \tilde \alpha_1)$ returning with $\tilde \alpha_1\sim C^n \alpha_{\ell+1,\ell+1}$ and $m_1 \leq 2^{-1}n$.  Therefore, in Step C of the induction, if $m_1 \in (2^{k_1-1}, 2^{k_1})$, we would need $\tilde \alpha_1 < \beta_{k_1, \ell}$ \emph{and} that the trajectory be at least $\alpha_{k_1,\ell}$ simple to apply the inductive hypothesis and obtain the $(\beta_{k_1,\ell}q^{-3},q)$ non-degeneracy of the orbit for $q<Q_{k_1,\ell}$.  In fact, the first condition that we impose on $\alpha_{\ell+1,\ell+1}$ is that $\tilde{\alpha}_1<\beta_{k_1,\ell}$.

Then, we need to ask whether the trajectory is $(m,\alpha_{k_1,\ell})$ simple. If it is, we can apply the inductive hypothesis to obtain non-degeneracy. However, if it is not, we need to repeat this process, obtaining that $\rho$ is $(m_2, \tilde{\alpha}_2)$ returning, with $\tilde{\alpha}_2\leq C^{m_1}\alpha_{k_1,\ell}$ and $m_2\leq m_1/2$ with $ m_2\in (2^{k_2}, 2^{k_2+1}]$. As in the case of treating $\rho$ as $(m_1,\tilde{\alpha_1})$ returning, we compare $\tilde{\alpha}_2$ and $\beta_{k_2,\ell-1}$ and ask whether $\rho$ is $(m_2,\alpha_{k_2,\ell-1})$ simple. The requirement that $\tilde{\alpha}_2\leq \beta_{k_2,\ell-1}$ puts an additional upper bound on $\alpha_{k_1,\ell}$. If the trajectory is not simple enough, we repeat once again finding $(m_i,k_i,\tilde{\alpha}_i)$ until $\rho$ is $(m_i, \alpha_{k_i,\ell-i+1})$ simple. Note that this \emph{will} happen, since every $1$ returning trajectory is simple and hence, if $m_i=1$, then, provided we have chosen the $\alpha_{k,\ell}$ correctly, $\rho$ is $\beta_{0,0}$ returning and $\alpha_{0,0}$ simple and hence we may apply the inductive hypothesis.

Since $m_i\leq m_{i-1}/2$, we have $m_i\leq 2^{-i}n\leq 2^{\ell+1-i}$. In some circumstances, we may be working with $n=2^{\ell+1}$ and the iteration may not terminate until $m_i=1$. If this is the case, then we want to use the inductive hypothesis on $(1,\beta_{1,1})$ returning points to obtain non-degeneracy of $\gamma$. However, this requires that every $(1,\beta_{0,0})$ returning point be $(\beta_{0,0}2^{-3(\ell+1)},2^{\ell+1})$ non-degenerate. This will eventually fail, since $Q_{1,1}<\infty$ and $\ell$ is unbounded above. This can be remedied by putting $1<\step<2$ and replacing $2$-adic intervals by $\step$-adic intervals. By doing this, we guarantee that $1\leq m_i\leq 2^{-i}\step^{\ell}$ and hence, the above iteration terminates in $i<c\ell$ steps for some $c<1$.

\subsection{Comments on the lack of a global Poincar\'e section}\label{s:gP}

In the outline above, we have worked as though one can find a single compact symplectic manifold without boundary, $\Gamma\subset \SM$, to serve as a Poincar\'e section. In general, this may not be the case and one may need to work with Poincar\'e sections with boundaries. This leads to a number of technical complications in the argument which are handled by introducing the notion of chains of symplectomporphisms and well-separated sets (see Section~\ref{s:chain}). However, these technical complications do not lead to any fundamental change in the proof and the reader may wish to imagine that there \emph{is} a global Poincar\'e section; in particular, the reader can then safely ignore the discussion about chains of symplectomorphisms and well-separated sets, instead thinking of the Poincar\'e map and its iterates.

\subsection{Comments on proving predominance}\label{s:pre}

In order to prove predominance instead of density, at each step of the induction, rather than finding a single metric, $g_\ell$, we produce a family $g_{\ell,\bm{\sigma}}$ for $\bm{\sigma}\in \prod_{j=1}^\infty B\sub{\R^\L}(0,1)$ in a set with almost full measure. This is possible since we actually showed that, at each step, all but a small measure of our perturbations have the required properties.


\section{Volume of the set of $(\beta,q)$-degenerate symplectic matrices}
\label{s:symplectic}


In this section, we do some preliminary work on volumes and covering numbers 
(see Definition \ref{d:coveringNumber}) in the space of symplectic matrices with real coefficients in $2\aleph$ dimensions, $\Sp(2\aleph)$. As explained in Section~\ref{s:outline}, it is {necessary to work with  $(\beta,q)$-non-degenerate matrices} (see \eqref{e:q-Nondegenerate} for the definition). The goal of this section is to understand the covering number of the complement of these matrices in $\Sp(2\aleph)$.

\subsection{Properties of eigenvalues of symplectic matrices}\label{s:evs}

We recall here a few properties of the spectrum of a symplectic matrix $\bm A\in \Sp(2\aleph)$.
A $2\aleph\times 2\aleph$ matrix $\bm A$ {with real coefficients} is in $\Sp(2\aleph)$ provided 
\begin{equation}
\label{e:defJ}\bm A^T  \bm J \bm A= \bm J, \qquad 
\bm J=\begin{psmallmatrix}0& I\sub{\aleph}\\-I\sub{\aleph}&0\end{psmallmatrix},
\end{equation}
where $I\sub{\aleph} \in \mathbb{M}(\aleph) $ is the identity matrix.
Symplectic matrices are invertible {with $\|\bm A\|=\|\bm A^{-1}\|$. They also have $\det \bm A=1$.}

We note that if $\lambda$ is in the spectrum of $\bm A$ and has multiplicity $m$, then  $\lambda^{-1}$, $\bar{\lambda}$ and $\bar{\lambda}^{-1}$ are also in the spectrum and have multiplicity $m$. However, $\lambda$, $\lambda^{-1}$, $\bar{\lambda}$ and $\bar{\lambda}^{-1}$ may not be distinct. First, if $\lambda\notin \mathbb{R}$ and $|\lambda|\neq 1$, then all four are distinct. Thus, each non-real eigenvalue off the unit circle comes in a quadruplet $\{\lambda, \lambda^{-1}, \bar{\lambda},  \bar{\lambda}^{-1}\}$. Next, if $\lambda\notin \mathbb{R}$ with $|\lambda|=1$, then $\lambda^{-1}=\bar{\lambda}$ and such eigenvalues come in pairs of the form $\{\lambda, \lambda^{-1}\}$. Third, if $\lambda\in \mathbb{R}$ with $|\lambda|\neq 1$, then $\lambda=\bar{\lambda}$ and the eigenvalues again come in  pairs of the form $\{\lambda, \lambda^{-1}\}$. Finally, if $\lambda=\pm1$, then $\lambda=\bar{\lambda}=\lambda^{-1}=\bar{\lambda}^{-1}$, and it is known that  $\pm1$ always occurs with even multiplicity.

We say that an invertible matrix is hyperbolic if none of its eigenvalues lie on the unit circle. Note that if  $\bm A$ is hyperbolic then the same is true for $\bm A^q$ for all $q$. Moreover, the distance from the eigenvalues to the unit circle is increasing with $q$. 

One important consequence of the facts we recalled about eigenvalues of symplectic matrices is that every small enough perturbation of a symplectic matrix with an eigenvalue $\lambda_0\in\{\lambda:\; |\lambda|=1,\; \lambda \neq \pm 1\}$ having odd multiplicity must have an eigenvalue on the unit circle {since eigenvalues off the unit circle come in quadruplets}. In other words, certain eigenvalues are `trapped' on the unit circle.  Thus, unlike general matrices, symplectic matrices \emph{cannot} be perturbed to be hyperbolic. This makes perturbing a matrix to one which is $(\beta(q),q)$-non-degenerate for all $q$ much more delicate in the symplectic category than in the category of matrices and requires the delicate induction argument carried out in Section~\ref{s:theMadness}.

\subsection{$(\beta,q)$-non-degeneracy and the location of eigenvalues}

We start by discussing the relationship between $(\beta,q)$-non-degeneracy and the location of the eigenvalues {$\{\lambda_j ({\bf A})\}_{j=1}^{2\aleph}$ of a matrix $\bf A$}.


\begin{lemma} \label{l:det} For all ${\bf {A}}\in\mathbb{M}({2\aleph})$, $q\in\mathbb{N}$,
${s\in (0,1)}$,
\[
\text{if} \;\;\min_{\substack{1\leq j\leq 2\aleph\\0\leq p\leq q-1}}|\lambda_{j}({\bf {A}})-e^{2\pi ip/q}|\geq s, \quad \text{then} \quad 
|\det(I-{\bf {A}}^{q})|>\big({\min(\tfrac{2}{5}qs\,,\,\tfrac{1}{5})}\big)^{2\aleph}.
\]
\end{lemma}

\vspace{-0.6cm}
\begin{proof} First, note that 
$
\lambda_{j}(I-{\bf {A}}^{q})=1-\lambda_{j}({\bf {A}})^{q}.
$
Therefore, 
$
|\det(I-{\bf {A}}^{q})|=\prod_{j=1}^{2\aleph}\big|(1-\lambda_{j}({\bf {A}})^{q})\big|.
$
{The lemma will follow once we prove that for all $q \in \mathbb N$, ${s\in (0,1)}$,}
\begin{equation}
\min_{j,p}|\lambda_{j}({\bf A})-e^{2\pi ip/q}|\geq s,\qquad\Rightarrow\qquad
\min_j|1-\lambda_{j}({\bf A})^{q}|\geq {\min(\tfrac{2}{5}qs,\tfrac{1}{5})}.\label{e:muchBetterClaim}
\end{equation}

\vspace{-0.4cm}
{Fix $1\leq j\leq 2\aleph,$ and set $\lambda_{j}({\bf A})=re^{i\theta}$ for some $r>0$ and $\theta \in [0, 2\pi)$.}
{First, notice that if $|r-1|\geq \frac{1}{2}$, then $|1-\lambda_{j}({\bf A})^q|\geq \frac{1}{2}$ {for all $q \in \mathbb N$}, and hence~\eqref{e:muchBetterClaim} holds in this case. Therefore, we assume that $|r-1|<\frac{1}{2}$ from now on.}

Next, {fix $q \in \mathbb N$, ${s\in (0,1)}$, and for $p \in \mathbb N$} let $\alpha_{p}=\theta-2\pi\frac{p}{q}$. We claim that 
\begin{equation}\label{e:Or}
\text{either}\;\;|\alpha_{p}|\geq\tfrac{1}{2}s \;\;\; \text{for all} \;p,
\qquad 
\text{or}\;\;\qquad 
|r-1|\geq \tfrac{1}{2}s.
\end{equation}
To prove the claim suppose there exists $p$ such that $|\alpha_{p}|<\frac{s}{2}$.
 First, note that by the assumption in \eqref{e:muchBetterClaim}
\[
{s^{2}\leq}|re^{i\theta}-e^{2\pi ip/q}|^{2}=|1-re^{i\alpha_{p}}|^{2}=(r-\cos\alpha_{p})^{2}+(\sin\alpha_{p})^{2}.
\]
Therefore, 
$
\tfrac{3s^{2}}{4}\leq (r-\cos\alpha_{p})^{2}
$
{since $|\sin\alpha_p| \leq  |\alpha_p| < s/2$. 
In particular, since  $|r-1|<\frac{1}{2}$ {and  $|1-\cos\alpha_{p}|=2|\sin^2(\tfrac{\alpha_p}{2})|\leq \tfrac{s^{2}}{8}$,}}
{\[
\tfrac{3s^{2}}{4}\leq\big(r-1+(1-\cos\alpha_{p})\big)^{2}\leq (r-1)^{2}+\tfrac{s^{2}}{4},
\]}
and hence
$
|r-1| \geq \frac{1}{\sqrt{2}}s\geq \frac{1}{2}s 
$
as claimed in \eqref{e:Or}.

With \eqref{e:Or} in place, we first consider the case $|r-1|\geq \frac{s}{2}$ and divide it into two sub-cases. 

\noindent{\bf{Case 1:}} Suppose that $r\geq1+\frac{s}{2}$. Then, 
$
|1-\lambda_{j}({\bf A})^{q}|\geq r^{q}-1\geq q\frac{s}{2}
$
as claimed in \eqref{e:muchBetterClaim}.

\noindent  {\bf{Case 2:}} Suppose $r\leq1-\frac{s}{2}$. Then,   $r^{-q}\geq1+q\frac{s}{2}(1-\frac{s}{2})^{-1}$ {since $s \in (0, 1)$}. 
Thus, 
$
r^{q}\leq
1-q\frac{s}{2}\frac{1}{1+(q-1)\frac{s}{2}}.
$
In particular, as claimed in \eqref{e:muchBetterClaim},
\[
|1-\lambda_{j}({\bf A})^{q}|\geq1-r^{q}\geq qs(2+(q-1)s)^{-1}\geq \min( qs\tfrac{2}{5},\tfrac{1}{5}).
\]

Finally, suppose that $|r-1|<\frac{s}{2}$. Then,   $|\alpha_{p}|\geq\frac{s}{2}$
for all $p$ by \eqref{e:Or} and hence
\[
|1-\lambda_{j}({\bf A})^{q}|^2=|1-r^{q}e^{iq\theta}|^2\geq {\inf_{a\geq 0} (1-a\cos (q\theta))^2+a^2\sin^2(q\theta)}=\begin{cases}
\sin^2(q\theta) &\;\; \;\inf_{k\in \mathbb{Z}}|q\theta-2\pi k|<\frac{\pi}{2},\\
1 &\;\;\; \text{else}.
\end{cases}
\]
In particular, {since $\frac{s}{2}\leq \min_p|\alpha_p|\leq \frac{\pi}{q}$}, we have
$|1-\lambda_{j}({\bf A})^{q}| \geq|\sin\big(q\tfrac{s}{2}\big)| \geq \min(\tfrac{2}{5}qs ,\tfrac{1}{5}) $.

\end{proof}

In what follows, it will be necessary to avoid symplectic matrices which have eigenvalues close to a given number. To this end, {for $\theta \in \re$ and $s>0$} let
\begin{equation}\label{e:degenerate-theta}
\mc{D}(\theta,{s}):=\{{\bf A}\in\Sp(2\aleph):\; B_{\mathbb{C}}(e^{i\theta},{s})\cap\spec({\bf {A}}){\neq}\emptyset\}.
\end{equation}

{We now record a corollary of Lemma \ref{l:det} which we use in the next section.}
It controls the {size of the inverse of ${\bf{I}}-{\bf {A}}^q$} in terms of the distance between the spectrum of $\bf {A}$ and the roots of unity $\{e^{{2\pi i p}/{q}}\}_{0\leq p<q}$.
\begin{lemma} Let {$q\in\mathbb{N}$,
 $0<{s\leq\frac{1}{2q}}$,  and {${\bm A}\in \mathbb{M}(2\aleph)$.}
 If ${\bf {A}}\notin \bigcup_{p=0}^{q-1}\mc{D}(\tfrac{2\pi p}{q},{s})$}, then, 
\[
\|({\bf {I}}-{\bf {A}}^{q})^{-1}\|\leq (\tfrac{5}{2})^{2\aleph}(q{s})^{-2\aleph}(1+\|{\bf {A}}\|^{q})^{2\aleph-1}.
\]
\end{lemma}

\begin{proof} Note that by assumption
$\min_{j,p}|\lambda_{j}({\bf {A}})-e^{2\pi ip/q}|\geq{s}$
and hence  
$
|\det({\bf {I}}-{\bf {A}}^{q})|\geq(\tfrac{2}{5}q{s})^{2\aleph}
$
by Lemma~\ref{l:det} together with $s\leq \tfrac{1}{2q}$.
In particular, the corollary follows from the fact that $\|{\bf {B}}^{-1}\|\leq {|\det({\bf {B}})|^{-1}}{\|{\bf {B}}\|^{2\aleph-1}}$ for all  ${\bf {B}}\in\mathbb{M}({2\aleph})$.
Indeed, let $\sigma_{1}^{2}\leq\sigma_{2}^{2}\leq\dots \leq\sigma_{2\aleph}^{2}$
be the eigenvalues of ${\bf {B}}^{*}{\bf {B}}$. Then, 
$\det({\bf {B}}^{*}{\bf {B}})=|\det({\bf {B}})|^{2}=\prod_{j=1}^{2\aleph}\sigma_{j}^{2}$,
$\|{\bf {B}}^{-1}\|^{2}=\sigma_{1}^{-2}$,
and
$\|{\bf {B}}\|^{2}={\sigma_{2\aleph}^{2}}.$
Therefore, as claimed,
\[
\|{\bf {B}}^{-1}\|^{2}=\sigma_{1}^{-2}\leq{\sigma_{{2\aleph}}^{2(2\aleph-1)}}{\Big(\prod_{j=1}^{2\aleph}\sigma_{j}^{2}\Big)^{-1}}={\|{\bf {B}}\|^{2(2\aleph-1)}}{|\det({\bf {B}})|^{-2}}.
\]
\end{proof}

An immediate consequence of the previous lemma is the following.
\begin{corollary}\label{c:inverse}
Let $q \in \mathbb N$ and $\beta>0$ with $2\beta<q^2$. 
If ${\bm A}$ is not $(\beta, q)$ non-degenerate, then ${\bf {A}}\in \bigcup_{p=0}^{q-1}\mc{D}(\tfrac{2\pi p}{q},\beta q^{-3})$.
\end{corollary}

{Corollary \ref{c:inverse} will be used to estimate the volume of the set of matrices which are not $(\beta, q)$ non-degenerate using covering numbers.}

\subsection{Covering numbers and $(\beta,q)$ non-degeneracy}
Throughout the article, the notion of the covering number of a set
will play an important role. Let $(Y,d)$ be a metric space. \begin{definition}
\label{d:coveringNumber} The \emph{covering number of a set $V\subset Y$
and {radius $s$}} is defined by 
\[
\mc{M}\sub{Y}(V,s):=\inf\Big\{ N\in\mathbb{N}:\;\exists \{y_{i}\}_{i=1}^{N}\subset Y\text{ such that }V\subset\bigcup_{i=1}^{N}B(y_{i},s)\Big\}.
\]
\end{definition}

\subsubsection{Covering numbers in $\mathbb{M}(2\aleph)$ and $\Sp(2\aleph)$}
{In this subsection, we will relate covering numbers in $\mathbb{M}(2\aleph)$ to those in $\Sp(2\aleph)$.} 
The goal is to compare covering numbers of a given set as measured by balls in $\Sp(2\aleph)$ and in $\mathbb{M}(2\aleph)$. The following lemma controls how $\Sp(2\aleph)$ sits inside $\mathbb{M}(2\aleph)$.

\begin{lemma}\label{l:compareMandSp}
 There exists $\e_{0}>0$
such that for all ${\bf {A}}\in\Sp(2\aleph)$ and $0<r\leq\e_{0}\|{\bf {A}}\|$,
\[
B\sub{\mathbb{M}(2\aleph)}({\bf {A}},\tfrac{r}{2\|{\bf {A}}\|^{2}})\cap\Sp(2\aleph)\subset B\sub{\Sp(2\aleph)}({\bf {A}},r).
\]
\end{lemma} 
\begin{proof} 
Let ${\bf {A}}\in\Sp(2\aleph)$
and set $\mathcal{V}:={\bf {A}}^{-1}(B\sub{\Sp(2\aleph)}({\bf {A}},r)) {\subset \Sp(2\aleph)}.$
We have that $\mc{V}$ is a neighborhood of ${\bf {I}}$ in $\Sp(2\aleph)$
and 
$
B\sub{\Sp(2\aleph)}({\bf {I}},\|{\bf {A}}\|^{-1}r)\subset\mc{V}.
$
Let $\e_{0}=\e_{0}(\aleph)$ be such that for
$0<\tilde{r}\leq\e_{0}$ 
\[
B\sub{\Sp(2\aleph)}({\bf {I}},\tfrac{1}{4}\tilde{r})\subset B\sub{\mathbb{M}(2\aleph)}({\bf {I}},\tfrac{1}{2}\tilde{r})\cap\Sp(2\aleph)\subset B\sub{\Sp(2\aleph)}({\bf {I}},\tilde{r})
\]
Therefore,  
$
B\sub{\mathbb{M}(2\aleph)}({\bf {I}},\tfrac{r}{2\|{\bf {A}}\|}){\cap \Sp(2\aleph)}\subset\mc{V}
$
for $r\leq\e_{0}\|{\bf {A}}\|$.
Hence, 
\[
\Big(B\sub{\mathbb{M}(2\aleph)}\Big({\bf {A}},\tfrac{r}{2\|{\bf {A}}\|\|{\bf {A}}^{-1}\|}\Big){\cap \Sp(2\aleph)}\Big)\subset {\bf {A}}\Big(B\sub{\mathbb{M}(2\aleph)}({\bf {I}},\tfrac{r}{2\|{\bf {A}}\|}){\cap \Sp(2\aleph)} \Big)\subset{\bf {A}}\mc{V}=B\sub{\Sp(2\aleph)}({\bf {A}},r).
\]
Since ${\bf {A}}\in\Sp(2\aleph)$, we have $\|{\bf {A}}\|=\|{\bf {A}}^{-1}\|$
and the result follows.
\end{proof}

With Lemma \ref{l:compareMandSp} in place, we are ready to study the relationship between the covering numbers of a given set as measured by balls in $\Sp(2\aleph)$ and in $\mathbb{M}(2\aleph)$. To ease notation, for $r>0$ we also introduce the ball 
\[
 {B\sub{\bf A_0}\!(r)}:=\{{\bf {A}}\in\mathbb{M}(2\aleph):\;\|{\bf {A}}-{\bf {A}_{0}}\|\leq r\}.
\]

\begin{lemma} \label{l:entropy3}
There exists $\e_{0}>0$ such that the following holds.  If $\mathcal{Z}\subset B\sub{\bf 0}\!(r)\cap\Sp(2\aleph)$ {for some $r>0$},
${\bf A_{0}}\in\Sp(2\aleph)$, $0<{s}\leq\e_{0}r$, {and $r_0>0$},
then 
\[
\mathcal{M}\sub{\Sp(2\aleph)}(\mathcal{Z}\cap B\sub{\bf A_0}\!(r_0),{s})\leq\mathcal{M}\sub{\mathbb{M}(2\aleph)}(\mathcal{Z}\cap B\sub{\bf A_0}\!(r_0),\tfrac{1}{4}{sr^{-2}}).
\]
\end{lemma} \begin{proof} 
{Fix $r_0>0$.} Let {$\tilde s>0$ and  $\{{\bf A_i}\}_{i=1}^N \subset \mathbb{M}$}
be  such
that 
$
\mathcal{Z}\cap B\sub{\bf A_0}\!(r_0)\subset\bigcup_{i=1}^{N}B\sub{{\bf A_i}}(\tilde s).
$
Without loss of generality, we may assume that $B\sub{{\bf A_i}}(\tilde s)\cap\mathcal{Z}\neq\emptyset$.
 {Let  $\{{\bf Z_i}\}_{i=1}^N \subset \mathcal{Z}$}
be such that 
$
\mathcal{Z}\cap B\sub{\bf A_0}\!(r_0)\subset \bigcup_{i=1}^{N}B\sub{\mathbb{M}(2\aleph)}({\bf {Z}}_{i},2\tilde{{s}}).
$
Then, {since $\mathcal{Z}\subset B\sub{\bf 0}\!(r)\cap\Sp(2\aleph)$}, by Lemma~\ref{l:compareMandSp} there is $\e_0>0$ such that
\[
{\mathcal{Z}}\cap B\sub{\bf A_0}\!(r_0)\subset\bigcup_{i=1}^{N}B\sub{\Sp(2\aleph)}({\bf {Z}}_{i},4\tilde{{s}}\|{\bf {Z}_{i}}\|^{2})\subset\bigcup_{i=1}^{N}B\sub{\Sp(2\aleph)}({\bf {Z}}_{i},4\tilde{{s}}r^{2}),
\]
provided  $4\tilde{{s}}r\leq\e_{0}$.
In particular, 
$
\mc{M}\sub{\Sp(2\aleph)}(\mathcal{Z}\cap B\sub{\bf A_0}\!(r_0),4\tilde{{s}}r^{2})\leq\mathcal{M}\sub{\mathbb{M}(2\aleph)}(\mathcal{Z}\cap B\sub{\bf A_0}\!(r_0),\tilde{{s}}).
$
The Lemma follows by putting ${s}=4\tilde{{s}}r^{2}$. \end{proof}

\subsubsection{Covering numbers for matrices that are not $(\beta, q)$ non-degenerate}

We now study the covering number of $\mc{D}(\theta,{s})$. Using results on semialgebraic sets from~\cite{YoCo:04}, we see below in the proof of Corollary~\ref{c:export} that it is enough to understand the volume of $\mathcal{D}(\theta,s)$. This our first goal in this subsection.

  \blue{  \begin{lemma}\label{l:volume-neighborhood-Dtheta-s-b}
There exists $C>0$ and $\rPower>0$ such that the following holds. For every
$r\ge 1$,  $\theta \in [0,2\pi]$, $r_0\in (0,1)$, and every ${\bf A}_0\in \Sp(2\aleph,\R)$, we have
\[
{\bf m}\sub{\Sp(2\aleph,\R)}\big(\mathcal D(\theta,s)\cap B_{\mathbf 0}(r)\cap B_{A_0}(r_0)\big)
\leq Cs\, r^{\rPower}\,r_0^{\aleph(2\aleph+1)-2\aleph}.
\]
\end{lemma}
\begin{proof} 
First observe that it is enough to prove the lemma when $r_0<c(1+r)^{-1}$ since, up to changing $\rPower$, we obtain the estimate for $r_0\in (c(1+r)^{-1},1)$ by applying the estimate with $r_0=c(1+r)^{-1}$ at most polynomially many in $r$ times.

\noindent \emph{Step 1: Reduction to the volume of a determinant level set.}
Suppose that $\bm{A}\in \mathcal{D}(\theta,s)$. Then there is $\mu\in B(e^{i\theta},x)$ such that $\det (\bm A-\mu \bm I)=0$. Now,  $\det (\bm A-\lambda \bm I)=P_{\bm A}(\lambda)$ where $P_{\bm A}$ is a polynomial of degree $2\aleph$ such that the coefficient of $\lambda^{2\aleph-j}$ is a polynomial of degree $j$ in the entries of $\bm A$. Since 
 $P_{\bm A}(\mu)=0$, $P_{\bm A}(\lambda)=(\lambda-\mu)\tilde{P}_{\bm A}(\lambda)$ where $\tilde{P}_{\bm A}$ is a polynomial of degree $2\aleph -1$. Hence
$$
|\det \bm A-e^{i\theta}\bm I|=|P_{\bm A}(e^{i\theta})|= |(e^{i\theta}-\mu)||\tilde{P}_{\bm A}(e^{i\theta})|\leq C_{\aleph}s(\|\bm A\|^{2\aleph -1}+1),
$$
and we have
$$
\mathcal{D}(\theta,s)\cap B_{\bm 0}(r)\subset \{ \bm A\in \Sp(2\aleph ,\mathbb{R})\cap B_{\bm 0}(r)\,:\, |\det \bm A-e^{i\theta}\bm I|\leq C_{\aleph}s(r^{2\aleph-1}+1)\}=:\tilde{D}(\theta,s,r).
$$

\noindent \emph{Step 2: Cayley transform coordinates.}
Let $\Phi:\mathrm{Sym}(2\aleph,\mathbb{R})\cap B_{\bm 0}(\frac{1}{2})\to \Sp(2\aleph,\mathbb{R})$ given by  
$$
\Phi(\bm S):=(\bm I+\bm J\bm S)(\bm I-\bm J \bm S)^{-1}. 
$$
Then, the image of $\Phi$ is a neighborhood of the identity in $\Sp(2\aleph,\mathbb{R})$ and hence for $\bm A_0\in B_{\bm 0}(r+1)$, 
$$
\Phi(B_{\bm 0}(\rho))\bm A_0
$$
is a neighborhood of $\bm A_0$ in $\Sp(2\aleph,\mathbb{R})$. More specifically, there is $c>0$ 
$$B_{\bm A_0}(c(r+1)^{-1})\subset \Phi(B_{\bm 0}(\tfrac{1}{2}))\bm A_0.$$

Now, since $\det \bm A_0=1$, 
$$
\det \big(\Phi(\bm S)\bm A_0-\lambda \bm I\big)=\det \big(\Phi(\bm S)-\lambda \bm A_0^{-1}\big)=\frac{\det (\bm I-\lambda \bm A_0^{-1}+(I+\lambda \bm A_0^{-1})\bm J\bm S)}{\det (\bm I-\bm J\bm S)}.
$$
Now, for $\|\bm S\|\leq \frac{1}{2}$,
$$
|\det (\bm I-\bm J\bm S)|\geq 2^{-2\aleph}. 
$$
and hence, since $|\det D_{\bm S} \Phi(S)\bm A_0|\leq C$, we have, for $r_0\leq c(r+1)^{-1}$,
\begin{align*}
&\vol (\mathcal D(\theta,s)\cap B_{\mathbf 0}(r)\cap B_{A_0}(r_0))\\
&\leq 
\vol(\big\{ \bm S\in \mathrm{Sym}(2\aleph)\,:\, \|\bm S\|\leq Cr_0(r+1),\, |\det (\bm I-\lambda \bm A_0^{-1}+(\bm I+\lambda \bm A_0^{-1})\bm J\bm S)|\leq Cs(r^{2\aleph-1}+1)\}).
\end{align*}
It is therefore enough to prove that  
\begin{equation} 
\label{e:volToEstimateFinally}
\begin{aligned}
&\vol(\big\{ \bm S\in \mathrm{Sym}(2\aleph)\,:\, \|\bm S\|\leq Cr_0(r+1),\, |\det (\bm I-\lambda \bm A_0^{-1}+(\bm I+\lambda \bm A_0^{-1})\bm J\bm S)|\leq Cs(r^{2\aleph-1}+1)
\big\})\\
&\leq  Cr_0^{\aleph(2\aleph+1)-2\aleph}s(r^{\aleph(2\aleph+1)-1}+1).
\end{aligned}
\end{equation}

\noindent \emph{Step 3: A normal form.}

Define 
$$
\bm F:=\begin{pmatrix} I-\lambda\bm A_0^{-1}& (\bm I+\lambda \bm A_0^{-1})\bm J\end{pmatrix}.
$$
Then we are interested in $\det \bm F\begin{pmatrix} \bm I\\\bm S\end{pmatrix}$.

Observe that $\bm F\bm F^*\geq 2\bm I_{2\aleph}$, so we may define $\tilde{\bm F}:= (\bm F \bm F^*)^{-1/2}\bm F$ so that $\tilde{\bm F}\tilde{\bm F}^*=\bm I$ and 
$$
\Big|\det \bm F\begin{pmatrix} \bm I\\ \bm S\end{pmatrix}\Big| = |\det (\bm F\bm F^*)|^{1/2}\Big|\det \tilde{\bm F}\begin{pmatrix} \bm I\\\bm S\end{pmatrix}\Big|\geq 2^{\aleph}\Big|\det \tilde{\bm F}\begin{pmatrix} \bm I\\\bm S\end{pmatrix}\Big|.
$$

Observe that since $|\lambda|=1$ and $\bm A_0\in \Sp(2\aleph,\mathbb{R})$, 
$$
(\bm F\bm F^*)^{-1/2}(\bm I-\lambda \bm A_0^{-1})[(\bm F\bm F^*)^{-1/2}(\bm I+\lambda \bm A_0^{-1})\bm J]^*=(\bm F\bm F^*)^{-1/2}(\bm I+\lambda \bm A_0^{-1})\bm J[(\bm F\bm F^*)^{-1/2}(\bm I-\lambda \bm A_0^{-1})]^*.
$$

Therefore, by Lemma~\ref{l:linearAlgebra}, there are $\bm U,\bm V$  unitary (orthogonal if $\lambda\in\{\pm 1\}$) and $\bm \Sigma_i$ diagonal, $i=1,2$ such that 
$$
(\bm F\bm F^*)^{-1/2}(\bm I-\lambda \bm A_0^{-1})= \bm U\bm\Sigma_1 \bm V^*,\qquad (\bm F\bm F^*)^{-1/2}(\bm I+\lambda \bm A_0^{-1})\bm J= \bm U\bm \Sigma_2 \bm V^*,\qquad \bm\Sigma_1^2+\bm \Sigma_2^2=\bm I.
$$

Using this decomposition, we obtain 
$$
\det \tilde{\bm F}\begin{pmatrix} \bm I\\ \bm S\end{pmatrix}= \det (\bm \Sigma_1+\bm \Sigma_2\bm V^*\bm S\bm V).
$$
Lemma~\ref{l:theVolumeEstimate} then implies that 
\begin{multline*}
\vol\Big(\big\{ \bm S\in \mathrm{Sym}(2\aleph)\,:\, \|\bm S\|\leq r_0,\, |\det (\bm I-\lambda \bm A_0^{-1}+(\bm I+\lambda \bm A_0^{-1})\bm J\bm S)|
\leq Cs(r^{2\aleph-1}+1)\big\}\Big)\\
\leq Cr_0^{\aleph(2\aleph+1)-2\aleph}s(r^{2\aleph-1}+1),
\end{multline*}
which completes the proof.
\end{proof}}

{Next, we use the results above to control  the covering numbers of $(\beta,q)$-degenerate matrices. Recall that by Corollary \ref{c:inverse}, if ${\bm A}$ is not $(\beta, q)$ non-degenerate, then ${\bf {A}}\in \bigcup_{p=0}^{q-1}\mc{D}(\tfrac{2\pi p}{q},\beta q^{-3})$. In particular, the following result bounds the covering number for such matrices.}

\blue{

\begin{corollary}\label{c:export}
There are $C,\rPower,\e_0>0$ such that for every
$r\ge 1$, $0<s<\e_0r$,  $\theta \in [0,2\pi]$, $r_0\in (0,1)$, and every ${\bf A}_0\in \Sp(2\aleph,\R)$,  
\[
\mathcal M_{Sp(2\aleph)}\Big(\mathcal D(\theta,s)\cap B_{\mathbf 0}(r)\cap B_{A_0}(r_0),s\Big)
\le Cr^{\rPower}s^{-\L+2\aleph+1}r_0^{\L-4\aleph}.
\]
\end{corollary}
\begin{proof}
    By~\cite[Theorem 3.5]{YoCo:04} 
    \begin{equation} 
    \label{e:coveringAgain}
    \mathcal{M}_{\mathbb{M}(2\aleph)}(E,\e)\leq C(\aleph)\sum_{i=0}^{(2\aleph)^2}\frac{1}{\e^i}V_i(E),
    \end{equation}
    where $V_i(E)$ is defined in~\cite[Definition 3.1]{YoCo:04}. We only use the following properties of $V_i(E)$ when $E$ is a semi-algebraic set of dimension $\ell$ contained in the ball of radius $r$:
    \begin{enumerate}
        \item By~\cite[Corollary 5.2]{YoCo:04}, 
        $$
        V_\ell(E)=\vol_{\ell}(E),
        $$
        where $\vol_{\ell}(E)$ denotes the $\ell$-dimensional Hausdorff volume. 
        \item By~\cite[Point (2) page 35]{YoCo:04} $V_{i}(E)=0$ for $i>\ell$.  
        \item 
        If $
        B_{0,j}(E):=\sup\{ \#\{ \text{connected components of }(P\cap E):\, P\subset \mathbb{R}^{(2\aleph)^2}\text{ is a }j-\text{plane}\}\}
        $, then
        $$
        V_i(E)\leq B_{0,n-i}(E)V_i(B(0,r)).
        $$
        \item By~\cite[Corollary 4.9]{YoCo:04}, for all $(n,p,j_1,\dots, j_p,(d_{ij})_{\substack{i=1,\dots,p\\j=1,\dots, j_i}})$, there is $C>0$ such that  $B_{0,k}(E)\leq C$ for all $k$ provided $E$ is defined by
        $$
        E:=\cup_{i=1}^pE_i,\qquad E_i:=\cap_{j=1}^{j_i}E_{ij},
        \qquad
        E_{ij}:=\{x\,:\, p_{ij}(x)>0\}\text{ or }\{ x\,:\, p_{ij}(x)\geq 0\},
        $$
        and $p_{ij}$ is a polynomial of degree $d_{ij}$. 
        \item By~\cite[Point (6) page 35]{YoCo:04} $V_i(B(0,r))= r^iB_i(B(0,1))$.
    \end{enumerate}

    Using these facts in~\eqref{e:coveringAgain}, we obtain that there is $C>0$ such that for all $\theta$, $s$, $r$, $r_0$
    \begin{align*}
     \mathcal{M}\sub{\mathbb{M}(2\aleph)}(E,\e)&\leq C\vol_{\L-2\aleph}(E)\e^{-\L+2\aleph}+C\sum_{i=1}^{\L-2\aleph-1} \frac{1}{\e^i}V_i(E)\\
      &\leq C\vol_{\L-2\aleph}(E)\e^{-\L+2\aleph}+C\sum_{i=1}^{\L-2\aleph-1} \frac{1}{\e^i}r_0^iV_i(B(0,1))\\
      &\leq C{\bf m}\sub{\Sp(2\aleph)}(E)\e^{-\L+2\aleph}+C\Big(1+\frac{r_0}{\e}\Big)^{\L-2\aleph -1}.
    \end{align*}

    Using Lemma~\ref{l:volume-neighborhood-Dtheta-s-b}, we obtain  
    $$
    \mathcal{M}\sub{\mathbb{M}(2\aleph)}(E,\e)\leq Cr^{\rPower}s\e^{-\L+2\aleph}r_0^{\L-4\aleph}+C\Big(1+\frac{r_0}{\e}\Big)^{\L-2\aleph -1}
    $$
    Hence, the lemma follows from Lemma~\ref{l:entropy3}. Indeed, for $0\leq \e \leq \e_0r$
    \begin{align*}
\mathcal{M}\sub{\Sp(2\aleph)}(E,\e)\leq \mathcal{M}\sub{\mathbb{M}(2\aleph)}(E,\tfrac{1}{4}\e r^{-2})\leq Cr^{\rPower+2\L -4\aleph}s\e^{-\L+2\aleph}r_0^{\L-4\aleph}+C\Big(1+\frac{r^2r_0}{\e}\Big)^{\L-2\aleph -1}
    \end{align*}
    Setting $\e=s$, 
      \begin{align*}
\mathcal{M}_{\Sp(2\aleph)}(E,s)\leq Cr^{\rPower+2\L -4\aleph}s^{-\L+2\aleph+1}r_0^{\L-4\aleph}+C\Big(1+\frac{r^2r_0}{s}\Big)^{\L-2\aleph -1}.
    \end{align*}
    Using that $s\leq \e_0r$, $r\geq 1$, and taking $\rPower$ large enough completes the proof.

\end{proof}

}

\section{Returning points, simple points, and their iterates}
\label{s:iterates}
In Section~\ref{s:returnSimple}, we introduce a version of returning points with discretized time as well as the notion of a simple point. Next, in Section~\ref{s:chain} we introduce the concept of a well separated set for the geodesic flow and the corresponding chain of symplectomorphisms. These concepts are replacements for, respectively, a global Poincar\'e section and the Poincar\'e map associated to the global section and, while they require some technical work, do not substantially change the main idea of the proof. Because of this, the reader may first wish to replace the concepts of chains of symplectomorphisms and well separated sets by the simpler notions of Poincar\'e map and section.

In the next sections of this paper, we will be varying the metric $g$. Because of this, it will be useful to have a single space on which the geodesic flow for any $g$ is defined. This space will be canonically isomorphic to $S^*\!M$ for any $g$ and will be defined as follows. Let 
\begin{equation}\label{e:tildeSM}
\tSM:=(T^*M\setminus\{0\})/\sim,\qquad (x,\lambda\xi)\sim (x,\xi)\text{ for all }\lambda>0.
\end{equation}
Then, since the geodesic flow (the Hamiltonian flow for $p(x,\xi)=|\xi|_g$), is homogeneous of degree 0 in $\xi$, the flow $\varphi_t^g$ passes naturally to the quotient $\tSM$. We also endow $\tSM$ with the distance inherited from $S^*_{g\sub{f}}M$ for some fixed reference metric $g\sub{f}$. 

\subsection{Returning points and simple points}
\label{s:returnSimple}
In this section, we define the notions of returning trajectories of
length $n$  and  simple trajectories of length $n$, and show that non-simple returning points can be seen as iterates of shorter returning trajectories. Because we will pass from the continuous time flow to a discrete time object below, these notions will be defined relative to a fixed number $\fuInj$ such that we can guarantee there are no periodic (or near periodic) trajectories of length $<\fuInj$. In Section~\ref{s:chain} we will also insist that the flow can be effectively reduced to a discrete time map by cutting in time at intervals of length $\sim\fuInj$. We will be working with metrics in small balls $B\subset \G^{3}$ and will show in Section~\ref{s:poincare} that one can take $\fuInj$ to be a small multiple of the injectivity radius for a fixed metric in $B$.

The notion of a returning point generalizes that of periodicity.
\begin{definition}\label{d:returningKappa} 
{For $\fuInj>0$, $g\in\ms{G}^{2}$, $n\in\mathbb{N}$, and $\delta>0$,
we write 
\[
\rho\in\Rec\sub{\fuInj}(n,\delta,g) \qquad \text{if}\qquad \inf_{ \substack{(n-1)\fuInj< t \leq n\fuInj\\
{t>\fuInj/4}}}d(\varphi^{g}_t(\rho),\rho)<\delta.
\]
In this case we say that $\rho$ is $\fuInj$-$(n,\delta, g)$ returning.}
\end{definition}

It will be important to have a notion of a simple returning point
of length $n$, i.e. one which is not returning for any smaller $n$.
\begin{definition}\label{d:simple}
{For $\fuInj>0$, $g\in\ms{G}^{2}$, $n\in\mathbb{N}$, and $\alpha>0$,
we write 
\[
\rho\in\Sim\sub{\!\fuInj}(n,\alpha,g) \qquad \text{if}\qquad \inf_{\frac{1}{2}\fuInj<t<(n-\frac{1}{2})\fuInj}d(\varphi^g_{t}(\rho),\rho)>\alpha.
\]
In this case we say that $\rho$ is $\fuInj$-$(n,\alpha, g)$ simple.}
\end{definition}

The importance of simple points comes from the fact
that the effect of perturbations of $g$ on simple trajectories
can be understood. This is much harder to do when considering non-simple trajectories
 because the trajectory will interact with
a given perturbation many times.

The following two lemmas are similar to \cite[Lemma 3.1]{Yo:85} and show that non-simple
returning trajectories are multiples of shorter returning trajectories.

\begin{lemma}\label{l:division}
Let $K\subset \ms{G}^3(\Gamma)$ be bounded and $\fuInj>0$. Then there is $C>0$  such that the following holds. Let  $g\in K$, $\alpha>0$, 
and  $\rho\in \widetilde{S^*\!M}$ with
$d(\varphi_{t}^g(\rho),\rho)<\alpha$. If  $ d(\varphi_{s}^g(\rho),\rho)<\alpha$ where $\frac{\fuInj}{2}\leq |s|<|t|$ and
$q\in\mathbb{Z}$ such that   $t=qs+r$ with  $ |r|<|s|$, then
$$
d(\varphi_{r}^g(\rho),\rho)<C^{|t|}\alpha.
$$
\end{lemma}

\begin{proof}
By \eqref{e:c2Bound} there is $C=C(K)\geq 1$ such that, for all $m\in \mathbb{Z}$
$$
d(\varphi_{ms}^g(\rho),\varphi_{(m-1)s}^g(\rho))\leq C^{|s|}d(\varphi_{(m-1)s}^g(\rho),\varphi_{(m-2)s}^g(\rho)).
$$
Using that  $ d(\varphi_{s}^g(\rho),\rho)<\alpha$ and $|s|\geq\frac{\fuInj}{2}>0$, we have
$
d(\varphi_{ms}^g(\rho),\rho)\leq \sum_{j=0}^{m-1} C^{|js|}\alpha\leq 2C^{|m s|}\alpha.
$
Similarly, since
$
d(\varphi_{t}^g(\rho),\rho)\leq \alpha,
$
we have
$
d(\varphi_{t-qs}^g(\rho),\varphi_{-qs}^g(\rho))\leq C^{{|qs|}}\alpha.
$
It follows that
$$
d(\varphi_{r}^g(\rho),\rho)\leq d(\varphi_{t-qs}^g(\rho),\varphi_{-qs}^g(\rho))+d(\varphi_{-qs}^g(\rho),\rho)\leq 3C^{|qs|}\alpha\leq C^{|t|}\alpha,
$$
where in the last inequality we used that $|qs|=|t-r|\leq |t|+|r| \leq 2|t|$.
\end{proof}

\begin{lemma}\label{l:iterates}
Let $K\subset \G^{3}$ be bounded and $\fuInj>0$. Suppose there are $c>0$, ${C_\fuInj}>0$ such that for all $g\in K$ and $\rho\in\widetilde{S^*\!M}$ we have
$
d(\varphi_t^g(\rho),\rho)\geq c|t|$ if $ |t|\leq {C\sub{\fuInj}}\fuInj.
$
Then there is $C>0$  such that for all $g\in K$, $\delta>0$, $t_0>\fuInj$, $\alpha\geq\delta$, and $\rho\in \widetilde{S^*\!M}$ satisfying
$$
d(\varphi_{t_0}^g(\rho),\rho)<\delta,
$$
{one of the following two statements hold}: 
\begin{enumerate}
\item $d(\varphi_{s}^g(\rho),\rho)\geq\alpha$ for all $\frac{\fuInj}{2}\leq s\leq t_0-\frac{\fuInj}{2}$.
\item there is  $s_0 \in [{C\sub{\fuInj}}\fuInj-{C^{t_0}\alpha},\frac{t_0}{2}]$ such that  $d(\varphi_{s_0}^g(\rho), \rho)\leq C^{t_0}\alpha$ 
and there is $q\in \mathbb{N}$, $q>1$, such that 
$
qs_0=t_0.
$ 
Moreover, if $t_1=q_1 s_0+\remainder\leq t_0$ for $q_1\in \mathbb{N}$ and $0\leq \remainder<s_0$, then 
\[
d(\varphi_{t_1}^g(\rho),\varphi_{\remainder}^{g}(\rho))\leq C^{t_0}\alpha.
\]
\end{enumerate}
\end{lemma}
\begin{proof}
Suppose that statement (1) is false. Then there is $s$ such that $\frac{\fuInj}{2}\leq s\leq t_0-\frac{\fuInj}{2}$ and 
$
d(\varphi_{s}^g(\rho),\rho)<\alpha.
$
Set  $t_1=s$ and let $q_1\in \mathbb{Z}$ such that $
t_{0}=q_1t_1+t_2 
$ for $| t_2|\leq \frac{s}{2}$.
Then, by Lemma~\ref{l:division},
$
d(\varphi_{t_2}^g(\rho),\rho)\leq C^{t_0}\alpha.
$
If $|t_2|\leq {C\sub{\fuInj}}\fuInj$, we put $\tilde s_0 =t_1$. 

If $|t_2|>C\sub{\fuInj}\fuInj$, we continue the process and suppose we have found $\{q_i\}_{i=1}^{m+1}\subset \mathbb{Z}$, and $\{t_i\}_{i=0}^{m+2}$ such that for $i=2,\dots, m$,
$$
t_i=q_{i+1}t_{i+1}+t_{i+2},\qquad  |t_{i+2}|\leq\tfrac{1}{2}|t_{i+1}|,\qquad |t_{i}|> {C\sub{\fuInj}}\fuInj,\qquad
d(\varphi_{t_i}^g(\rho),\rho)\leq C^{\sum_{j=0}^{i-2}|t_j|}\alpha.
$$
Then, letting $q_{m+2}\in \mathbb{Z}$, $|t_{m+3}|\leq\frac{1}{2}|t_{m+2}|$, such that 
$
t_{m+1}=q_{m+2}t_{m+2}+t_{m+3},
$
Lemma~\ref{l:division} yields
$
d(\varphi_{t_{m+3}}^g(\rho),\rho)\leq C^{\sum_{j=0}^{m+1}|t_j|}\alpha.
$
Then, if  $|t_{m+3}|\leq{C\sub{\fuInj}} \fuInj$, we set $\tilde s_0=t_{m+2}$.
In particular, since $|t_{j}|\leq \tfrac{1}{2}|t_{j-1}|$ for $j=2,\dots$, we have 
$
\sum_{j=0}^\infty|t_j|\leq 3|t_0|.
$
Therefore, we have found $\tilde s_0$ such that 
$$
d(\varphi\sub{\tilde s_0}^g(\rho),\rho)\leq C^{3|t_0|}\alpha. 
$$

Let this process terminate with  $|t_{m+3}|\leq{C\sub{\fuInj}} \fuInj$. 
We claim that 
$
|t_{m+3}|\leq C^{4|t_0|}\alpha.
$
Indeed, since $d(\varphi_t^g(\rho),\rho)\geq c|t|$ for $|t|\leq{C\sub{\fuInj}} \fuInj$ and $|t_{m+3} |\leq {C\sub{\fuInj}}\fuInj$, by Lemma~\ref{l:division}
we know
\begin{equation}
\label{e:boundOnT}
c|t_{m+3}|\leq d(\varphi_{t_{m+3}}^g(\rho),\rho)\leq C^{3|t_0|}\alpha.
\end{equation}

Finally, set $s_0=\tilde s_0 +\frac{t_{m+3}}{q_{m+2}}$. Then, after possibly modifying $C$, and using the bound on $|t_{m+3}|$ from~\eqref{e:boundOnT} we conclude 
$$
d(\varphi_{s_0}^{g}(\rho),\rho)\leq C^{|t_{m+3}|} d(\varphi_{\tilde s_0}^{g}(\rho),\rho)+d(\varphi_{\frac{t_{m+3}}{q_{m+2}}}^{g}(\rho),\rho) \leq 2C^{4t_0}\alpha
$$
as claimed, since $|q_{m+2}|\geq 1$.
Note that, since  $t_{m+1}=q_{m+2}t_{m+2}+t_{m+3}=q_{m+2}(t_{m+2}+\frac{t_{m+3}}{q_{m+2}})=q_{m+2}s_0$, we have by construction
$
t_0=qs_0
$
for some $q\in \mathbb{Z}$. In addition, 
$$
|s_0| \geq |t_{m+2}|-|\tfrac{t_{m+3}}{q_{m+2}}| \geq {C\sub{\fuInj}}\fuInj{-c^{-1}C^{4|t_0|}\alpha}.
$$

If $s_0>0$, we have the claimed properties of $s_0$. If, on the other hand, $s_0<0$, we have, modifying $C$ if necessary,
$$
d(\varphi_{-s_0}^g(\rho),\rho)\leq C^{|s_0|}d(\rho,\varphi_{s_0}^g(\rho))\leq C^{4|t_0|}\alpha.
$$

Finally, observe that for $qs_0\leq t_0$,
$$
d(\varphi_{qs_0}^g(\rho),\rho)\leq \sum_{j=0}^{q-1}d(\varphi_{js_0}^g(\rho),\varphi_{(j+1)s_0}^g(\rho))\leq \sum_{j=0}^{q-1}C^{j|s_0|}d(\varphi_{s_0}^g(\rho),\rho)\leq C^{|t_0|}\alpha.
$$
Therefore, if $t_1=qs_0+\remainder\leq t_0$ for some $q\in \mathbb{N}$, $0\leq \remainder <s_0$
$$
d(\varphi_{t_1}^g(\rho),\varphi_{\remainder}^g(\rho))\leq C^{\remainder}d(\varphi_{qs_0}^g(\rho),\rho)\leq C^{|t_0|}\alpha.
$$
\end{proof}

\subsection{Non-degeneracy and chains of symplectomorphisms}
\label{s:chain}

In general, it is not possible to find a global Poincar\'e section for the geodesic flow; i.e. a connected, closed submanifold, $\Gamma$, of co-dimension 1 in $\tSM$ that is everywhere transverse to $H\sub{|\xi|_g}$ such that every geodesic passes through $\Gamma$. Because of this, we need to replace the notion of a global Poincar\'e section and its associated Poincar\'e map with a more complicated submanifold which captures all the dynamical information for $\varphi_t^g$. To this end, we introduce below, the notion of a well separated set and its associated chains of symplectomorpshims; replacing, respectively, the global Poincar\'e section and its associated Poincar\'e map. The idea of a well-separated set will be to find, $\tilde{\Gamma}$, a disjoint union of open $2d-2$ dimensional submanifolds, and a compactly embedded open subset, $\Gamma \subset \tilde{\Gamma}$, such that 1) every geodesic passes through $\Gamma$ in a controlled time and 2) no geodesic passes in a very short time from $\tilde{\Gamma}$ to itself. These properties will guarantee that understanding geodesics which pass from $\Gamma$ to itself allows one to understand the full dynamics for the geodesic flow.


To understand the need for $\tilde{\Gamma}$ note that since the connected components of $\Gamma$ have boundaries, there are points $\rho\in\Gamma$ such that small perturbations of $\rho$ will cause large jumps in the first impact point $\ms{P}\sub\Gamma(\rho)$ of the geodesic through $\rho$ with $\Gamma$; i.e there are sequences $\rho_n\to \rho$ such that $d(\ms{P}\sub{\Gamma}(\rho_n),\ms{P}\sub{\Gamma}(\rho))\geq c>0$. This forces us to include the slightly larger $\tilde{\Gamma}$ and introduce the notion of Poincar\'e chains below. (See Figure~\ref{f:chain} for a schematic of a well-separated set and a Poincar\'e chain.)

Let $\mc{W}_i\subset \widetilde{S^*\!M}$, $i=1,\dots,N$ be disjoint, open symplectic submanifolds of dimension $2d-2=2\aleph$ and $\mc{V}_i\Subset \mc{W}_i$ open subsets. We write 
\begin{equation}\label{e:tildeGamma}
\Gamma:=\bigsqcup_{i=1}^N\mc{V}_i, \qquad \widetilde{\Gamma}:=\bigsqcup_{i=1}^N\mc{W}_i.
\end{equation}
\begin{definition}\label{d:wellSeparated}
For $\fuInj>0$, and $G\subset \ms{G}^2$ bounded, we say that $\{(\mc{W}_i,\mc{V}_i)\}_{i=1}^N$ are \emph{$\fuInj$-well-separated for $G$} if the $\mc{W}_i$ are uniformly transverse to $H\sub{|\xi|_g}$ for all $g\in \Gb$ and there are $c\sub{\Gamma}, C\sub{\Gamma}>1$ such that
\begin{equation}
\begin{gathered}\label{e:Cgammas}
\sup_{g\in \Gb}\sup_{\rho\in \widetilde{S^*\!M}}\inf \{t>0\,:\, \varphi_t^g(\rho)\in \Gamma\}<C\sub{\Gamma}\fuInj,
\qquad \qquad 
\inf_{g\in \Gb}\inf_{\rho\in \tilde{\Gamma}}\inf \{t>0\,:\, \varphi_{t}^g(\rho)\in \tilde{\Gamma}\}\geq c\sub{\Gamma}\fuInj,
\\
\inf_{g\in \Gb}\min_{i}\inf_{\rho\in \mc{W}_i}\inf \{t>0\,:\, \varphi_{t}^g(\rho)\in \overline{\mc{W}}_i\}{>} C\sub{\Gamma}\fuInj.
\end{gathered}
\end{equation}
When it will not lead to confusion, we will say that \emph{$\Gamma$ is $\fuInj$-well-separated.}
\end{definition}

In words, if $\{(\mc{W}_i,\mc{V}_i)\}_{i=1}^N$ is $\fuInj$-well-separated for $\Gb$, it means that (1) every  trajectory  hits $\Gamma$ in time less than $C\sub{\Gamma}\fuInj$; (2) trajectories starting from $\tilde{\Gamma}$ always take at least $c\sub{\Gamma}\fuInj$ time to return to  $\tilde{\Gamma}$; (3) trajectories that start within some $\mc{W}_i$ always take at least  $C\sub{\Gamma}\fuInj$ to return to  $\overline{\mc{W}_i}$ no matter the choice of $i$. 

Before we move on to the definition of chains of symplectomorphisms associated to well separated sets, we show that Lemma~\ref{l:iterates} applies when there is an $\fuInj$-well separated set.
\begin{lemma}
\label{l:united}
Let $\Gb\subset \ms{G}^2$ bounded.
Suppose that $\fuInj>0$ and $\Gamma$ is $\fuInj$-well separated for $\Gb$. Then, for all $K\subset \Gb$ bounded in $\G^{3}$, there is $c>0$ such that for all $g\in K$ and $|t|\leq {C\sub{\Gamma}}\fuInj$, 
$$
d(\varphi_t^g(\rho),\rho)\geq c|t|, \qquad \rho \in \widetilde{\SM}.
$$
\end{lemma}
\begin{remark}
    \blue{Here and below there are many places where the assumptions on the regularity of $g$ could be weakened. For instance, Lemma~\ref{l:united} holds with $K$ bounded in $\G^{2+\e}$ for any $\e>0$. However, these improvements do not result in any improvement to our main theorems since we will require boundedness in $\G^3$ e.g. in Lemma~\ref{l:NdqImpliesNd}. }
\end{remark}
\begin{proof}
First, notice that, since $K$ is bounded in $\G^{3}$, there are $c_1,c_2>0$ such that,
$
d(\varphi_t^g(\rho),\rho)\geq c_2|t|
$
for all $g\in K$, $\rho \in \tSM$, and $|t|<c_1$.  Thus, it remains only to check that there is $c>0$ such that  $d(\varphi_t^g(\rho),\rho))\geq c$  for all $g\in K$, $\rho \in \tSM$, and $c_1\leq|t|\leq C\sub{\fuInj}\fuInj$.

Suppose by contradiction there are $g_n\in K$, $\rho_n\in \tSM$, and $t_n$ with $|t_n|\in [c_1,{C\sub{\fuInj}}\fuInj]$, such that 
$$
d(\varphi_{t_n}^{g_n}(\rho_n),\rho_n)\to 0.
$$
Then, without loss of generality, we may assume $g_n\overset{\ms{G}^{2}}{\to} g\in \Gb$,  $\rho_n\to \rho\in \tSM$, and $t_n\to t$ with $|t|\in [c_1,{C\sub{\fuInj}}\fuInj]$. Now, 
$$
d(\varphi_{t}^{g}(\rho),\rho)\leq d(\varphi_{t}^{g}(\rho),\varphi_t^{g_n}(\rho))+ d(\varphi_{t}^{g_n}(\rho_n),\varphi_t^{g_n}(\rho))+d(\varphi_{t_n}^{g_n}(\rho_n),\varphi_t^{g_n}(\rho_n))+d(\varphi_{t_n}^{g_n}(\rho_n),\rho_n)+d(\rho_n,\rho).
$$
Since $g_n\to g$ in $\ms{G}^2$, the right hand side of the above inequality tends to $0$ as $n \to \infty$, and we have
$\varphi_{t}^{g}(\rho)=\rho.$
Let $T_0:=\inf\{ s>0\,:\,\varphi_s^g(\rho)\in\Gamma\}$ and $\rho\sub{\Gamma}:=\varphi_{T_0}^{g}(\rho)$ (note that $0<T_0<\infty$ by the first inequality in~\eqref{e:Cgammas}). Then, $\varphi_t^g(\rho\sub{\Gamma})=\varphi_{t+T_0}^g(\rho)=\varphi_{T_0}^g(\rho)=\rho\sub{\Gamma}$, with $c_1\leq |t|\leq {C\sub{\fuInj}} \fuInj$. This contradicts the last inequality in~\eqref{e:Cgammas}.
\end{proof}

We next use well-separated sets to reduce the continuous flow $\varphi_t^g$ to a discrete time system. Let {$\nu\geq \blue{2}$}, $\fuInj>0$, $G\subset \ms{G}^{{\nu}}$ bounded,, and  $\{(\mc{W}_i,\mc{V}_i)\}_{i=1}^N$ be \emph{$\fuInj$-well-separated for $G$}. Then define 
\begin{equation}\label{e:times}
T_{j,i}^g:\mc{W}_i\to \mathbb{R}, \qquad T_{j,i}^g(\rho):=\inf\{ t>0\,:\, \varphi_t^g(\rho)\in \mc{W}_j\},
\end{equation}
and, with $\mc{U}_{j,i}^g:=(T_{j,i}^g)^{-1}((0,C\sub{\Gamma}\fuInj))\subset \mc{W}_i$, let
$\kappa_{j,i}^g:\mc{U}^g_{j,i} \to \mc{W}_j$ be the function
$$
\kappa_{j,i}^g(\rho)=\varphi_{T_{j,i}^g(\rho)}^g(\rho).
$$

\begin{lemma}
 For all $g\in \Gb$, $j,i$, the map $\kappa_{j,i}^g:\mc{U}^g_{j,i}\to \mc{W}_j$ is a $\mc{C}^{\nu-2}$ symplectomorphism onto its image.
\end{lemma}
\begin{proof}
Once we show that $\kappa_{j,i}^g$ is $\mc{C}^{\nu-2}$, the fact that $\kappa_{j,i}^g$ is a symplectomorphism onto its image is inherited from the facts that $\varphi_t^g$ is a symplectomorphism and that $\mc{W}_j$ is transverse to $H\sub{|\xi|_g}$.

We now show that $\kappa_{j,i}^g$ is $\mc{C}^{\nu-2}$. To do this, note that $\mc{W}_j$ is open and transverse to $H\sub{|\xi|_g}$. Therefore, since $(t,\rho) \mapsto \phi_t^g(\rho)$ is $\mc{C}^{\nu-2}$ with $\nu-2\geq 1$ and $\varphi_{T_{j,i}^g(\rho)}^g(\rho)\in\mc{W}_j$, the implicit function theorem implies that there is a neighborhood $V\subset \mc{W}_i$ of $\rho$ and a $\mc{C}^{\nu-2}$ function $\tilde{T}_{j,i}^g(\rho'):V\to \mathbb{R}$ such that $\varphi_{\tilde{T}_{j,i}^g(\rho')}^g(\rho')\in\mc{W}_j$ and $\tilde{T}_{j,i}^g(\rho)=T_{j,i}^g(\rho)$. Since $\tilde{T}_{j,i}^g$ is continuous and $T_{j,i}^g(\rho)<C\sub{\Gamma}\fuInj$, shrinking $V$ if necessary, we may assume $\tilde{T}_{j,i}^g(\rho')<C\sub{\Gamma}\fuInj$.  Finally, we need to check that $\tilde{T}_{j,i}^g(\rho')=T_{j,i}^g(\rho')$. For this, it suffices to show that $\varphi_s^g(\rho')\notin \mc{W}_j$ for $0<s<\tilde{T}_{j,i}^g(\rho')$. Indeed, suppose that there is $0<s<\tilde{T}_{j,i}^g(\rho')$ such that $\varphi_s^g(\rho')\in\mc{W}_j$. Then, $\varphi_{\tilde{T}_{j,i}^g(\rho')-s}^g(\varphi_s^g(\rho'))\in\mc{W}_j$, and $0<\tilde{T}_{j,i}^g(\rho')-s<C\sub{\Gamma}\fuInj.$ This contradicts the third part of~\eqref{e:Cgammas}.
\end{proof}

\begin{definition}
\label{d:chain}
Let $\nu\geq \blue{2}$. Let $\Gb\subset \ms{G}^\nu$ be bounded in $\ms{G}^2$ and $\{(\mc{W}_i,\mc{V}_i)\}_{i=1}^N$ be $\fuInj$-well separated for $\Gb$. For $\mc{I}\in \{1,\dots ,N\}^{\mathbb{N}}$, $\mc{I}=(i_0,i_1, i_2,\dots)$, $g \in \Gb$, and $n\in \mathbb{N}$ we define the $C^{\nu-2}$ symplectomorphism 
$$
\mc{P}\sub{\mc{I}}^{(n)}[g]:\mc{D}\sub{\mc{I}}^{(n)}[g]\to \mc{W}_{i_n}, \qquad \mc{P}\sub{\mc{I}}^{(n)}[g]:=\kappa_{i_n,i_{n-1}}^g\circ\dots\kappa^g_{i_2,i_1}\circ\kappa^g_{i_1,i_0},
$$
where the $\mc{D}\sub{\mc{I}}^{(n)}[g]\subset \mc{W}_{i_0}$ is the domain of the composition. We also define $\mc{P}^{(0)}_{\mc{I}}[g]=\Id$ and define
$$
\mc{P}\sub{\mc{I}}[g]:=(\mc{P}\sub{\mc{I}}^{(0)}[g]\,,\,\mc{P}\sub{\mc{I}}^{(1)}[g]\,,\dots)
$$ 
and call it the \emph{Poincar\'e chain associated to $\mc{I}$}.
\end{definition}

\begin{definition}
\label{d:tChain}
We  define the time sequence $\bm{T}\sub{\mc{\mc{I}}}$ associated to the chain $\mc{P}\sub{\mc{\mc{I}}}$. For each $n\in \mathbb N$  and $g\in \Gb$ let  $\bm{T}\sub{\mc{I}}^{(n)}[g]:\mc{D}\sub{\mc{I}}^{(n)}[g]\to [0,\infty)$ be defined by $\bm{T}\sub{\mc{\mc{I}}}^{(0)}[g]\equiv 0$ and for all $\rho \in \mc{D}\sub{\mc{I}}^{(n)}[g]$
$$
{\bm T}\sub{\mc{I}}^{(n)}[g](\rho)= \bm{T}\sub{\mc{I}}^{(n-1)}[g](\rho)+T_{i_{n},i_{n-1}}^g(\mc{P}\sub{\mc{I}}^{(n-1)}[g](\rho)).
$$
for $T_{i_{n},i_{n-1}}^g$ as defined in \eqref{e:times}.
\end{definition}

By construction, for $\rho\in \mc{D}\sub{I}^{(n)}[g]$ we have 
\begin{equation}\label{e:thePoint}
\mc{P}\sub{\mc{I}}^{(n)}[g](\rho)=\varphi_{\bm{T}\sub{\mc{I}}^{(n)}[g](\rho)}^g(\rho).
\end{equation}

\begin{remark}
\label{r:psAndTs}

If $\rho\in \tilde{\Gamma}$ and $\varphi_{t_0}^g(\rho)\in \tilde{\Gamma}$, then there are $n$ and $\mc{I}$ such that $\rho\in \mc{D}\sub{\mc{I}}^{(n)}[g]$,  $\mc{P}\sub{\mc{I}}^{(n)}[g](\rho)=\varphi_{t_0}^g(\rho)$ and $\bm{T}\sub{\mc{I}}^{(n)}[g](\rho)=t_0$.
\end{remark}

Below, we will be varying the metric and need to control how Poincar\'e chains vary with the metric.
\begin{lemma} \label{l:PoincareDer}
Let {$\nu \geq \blue{2}$}, $0\leq k\leq \nu-{1}$, ${\Gb_0\subset \ms{G}^\nu,\Gb_1\subset \ms{G}^{\nu-1},\Gb_2\subset \ms{G}^{\nu-2}}$ bounded, and 
$${g\in \mc{C}^0(B_{\blue{\mathbb{R}^\ell}}(0,1)_\sigma; \Gb_0)\cap \mc{C}^1(B_{\blue{\mathbb{R}^\ell}}(0,1)_\sigma; \Gb_1)\cap \mc{C}^2(B_{\blue{\mathbb{R}^\ell}}(0,1)_\sigma; \Gb_2).}$$
Then there are $C>0$, $C_k>0$,  such that
$$
d_\sigma\varphi_t^g:\blue{\mathbb{R}^\ell}\to T\mc{C}^{\nu-{2}}(\tSM;\tSM),\qquad d^2_\sigma\varphi_t^g:\blue{\mathbb{R}^\ell}\times \blue{\mathbb{R}^\ell}\to T\mc{C}^{\nu-{3}}(\tSM;\tSM)
$$
are well defined and satisfy that for all $t\in \R$, $v, w \in \blue{\mathbb{R}^\ell}$, $f\in \mc{C}^k(S^*\!M;\R)$,
\begin{align*}
&\bullet \; \|d_\sigma\varphi_t^g|_{\sigma=0} v\|_{\mc{C}^{\nu-{2}}}\leq C^{|t|}\|v\|\|\partial_\sigma g|_{\sigma=0}\|_{\mc{C}^{\nu-1}},\qquad \text{if $\nu\geq \blue{2}$}\\ 
&\bullet \;\|d^2_\sigma\varphi_t^g |_{\sigma=0}(v,w)\|_{\mc{C}^{\nu-{3}}}\leq C^{|t|}\|v\|\|w\|(\|\partial_{\sigma}^2g|_{\sigma=0}\|_{\mc{C}^{\nu-2}}+\|\partial_\sigma g|_{\sigma=0}\|^2_{\mc{C}^{\nu-1}}),\qquad \text{if $\nu\geq \blue{3}$}\\
&\bullet \;\|f\circ \varphi_t^g\|_{\mc{C}^k}\leq C_k C^{|k||t|}\|f\|_{\mc{C}^k}.
\end{align*}
Moreover, for all $n \in \mathbb N$,
\begin{align*}
&\bullet \;\|d_\sigma\mc{P}\sub{\mc{I}}^{(n)}[g] v\|_{\mc{C}^{\nu-{2}}}\leq C^{n C\sub{\Gamma}\fuInj}\|v\|\|\partial_\sigma g|_{\sigma=0}\|_{\mc{C}^{\nu-1}},\qquad \text{if $\nu\geq \blue{2}$}\\ 
&\bullet \;\|d^2_\sigma\mc{P}\sub{\mc{I}}^{(n)}[g] (v,w)\|_{\mc{C}^{\nu-{3}}}\leq C^{n C\sub{\Gamma}\fuInj}\|v\|\|w\|(\|\partial_{\sigma}^2g|_{\sigma=0}\|_{\mc{C}^{\nu-2}}+\|\partial_\sigma g|_{\sigma=0}\|^2_{\mc{C}^{\nu-1}}),\qquad \text{if $\nu\geq \blue{3}$}\\
&\bullet \;\|f\circ \mc{P}\sub{\mc{I}}^{(n)}[g]\|_{\mc{C}^k}\leq C_k C^{|k|nC\sub{\Gamma}\fuInj}\|f\|_{\mc{C}^k}.
\end{align*}
\end{lemma}

\blue{Note that this result includes the bound in \eqref{e:c2Bound}.}

\begin{proof}
The proof of this lemma is a tedious calculation. We sketch here the proof.

Recall that the geodesic flow can be written as the solution $(x(t),\xi(t))$ to 
$$
\dot x^j(t)=\tfrac{1}{|\xi(t)|_{g(x(t))}} g^{ij}(x(t))\xi_i(t),\qquad \blue{\dot \xi_m}(t)=-\tfrac{1}{2|\xi(t)|_{g(x(t))}}\partial_{\blue{x^m}}g^{ij}(x(t))\xi_i(t)\xi_j(t).
$$
Letting $g_\sigma=g(\sigma)$ be a family of metrics with $g_0=g$ and differentiating with respect to $\sigma$, working always with initial data so that $|\xi_0(g_\sigma)|_{g_\sigma(x_0(g_\sigma))}=1$, we obtain equations of the form
$$
\partial_\sigma \dot x=(\partial_xg \xi) \partial_\sigma x + g\partial_\sigma \xi +\partial_\sigma g \xi,\qquad \partial_\sigma \dot\xi= (\partial_x^2g)\xi^2\partial_\sigma x+\partial_xg\xi\partial_\sigma \xi+(\partial^2_{x\sigma}g)\xi^2.
$$
In particular, provided that $\partial_\sigma g\in\mc{C}^{\nu-1}$, this is an equation with $\mc{C}^{\nu-2}$ coefficients (bounded by the $\mc{C}^{\nu}$ norm of $g$ and the $\mc{C}^{\nu-1}$ norm of $\blue{\partial_{\sigma}g}$) and hence, \blue{by repeated use the Picard Lindel\"of with continuous right-hand side} this results in a $\mc{C}^{\nu-{2}}$ flow with the claimed bounds.

Differentiating again in $g$, we obtain equations of the form
$$
\begin{aligned}\partial^2_\sigma \dot x&=(\partial^2_xg \xi) (\partial^2_\sigma x) +g \partial^2_\sigma \xi
+(\partial_x^2g \xi)(\partial_\sigma x)^2 +\partial_x g\partial_\sigma x\partial_\sigma \xi + \partial^2_{x\sigma}g\xi\partial_\sigma x + \partial_\sigma g\partial_\sigma \xi +\partial_\sigma^2g \xi\\
\partial^2_\sigma \dot \xi&= (\partial^2_xg)\xi^2\partial _\sigma^2x +\partial_xg\xi\partial^2_\sigma \xi +(\partial_x^3g)\xi^2(\partial_\sigma x)^2+ \partial^3_{\sigma xx}g\xi^2\partial_\sigma x+ \partial^2_xg\xi\partial_\sigma \xi \partial_{\sigma }x\\
&\qquad+\partial^2_{x\sigma}g\xi\partial_{\sigma}\xi +\partial_xg(\partial_\sigma \xi)^2+\partial^3_{x\sigma\sigma}g \xi^2.
\end{aligned}
$$
Now, since $\partial_{\sigma}(x,\xi)\in \mc{C}^{\nu-{2}}$ if $\partial_\sigma g\in \mc{C}^{\nu-1}$, we require $\partial^2_\sigma g\in \mc{C}^{\nu-2}$, and $\partial_\sigma g\in \mc{C}^{\nu-1}$ to obtain a solution in $\mc{C}^{\nu-{3}}$. 

The estimates on the derivative in time now follow from estimating the coefficients in these equations in $L^\infty$ by the relevant $\mc{C}^{\nu}$ norms of the metric and its derivatives.

\blue{For the estimates on $f\circ \varphi_t^g$ in $\mathcal{C}^k$ we refer the reader, for instance, to~\cite[Lemma 11.11]{Zw}. An inspection of the proof shows that estimating $k$ derivatives of the flow involves $k+1$ derivatives of the metric. }
\end{proof}

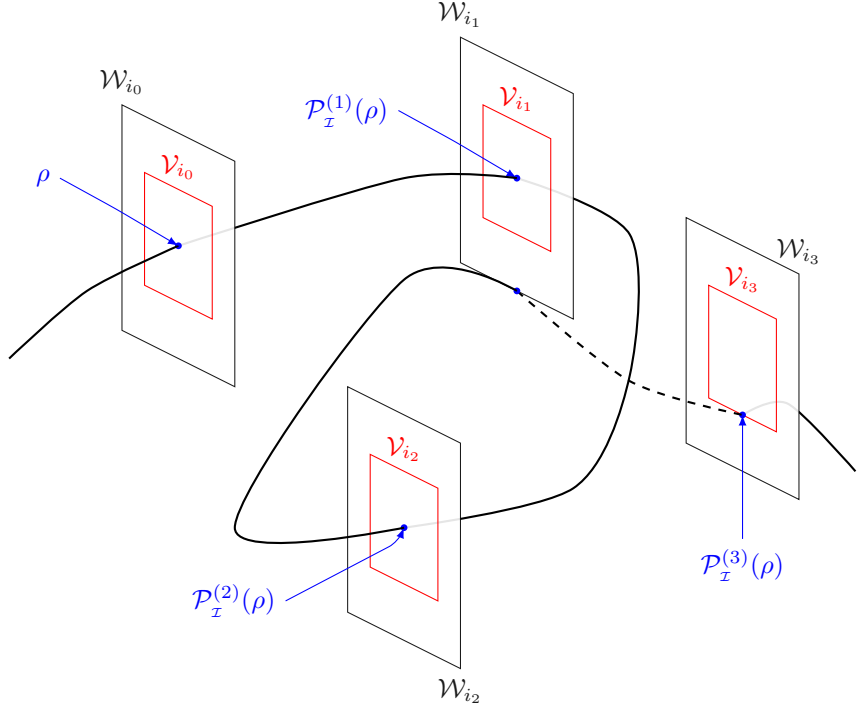
\begin{figure}
\begin{tikzpicture}

\def \w{1};
\def \h{2};
\def \shiftA{3,.6}
\def \shiftB{2,-2.5}
\def \shiftC{5,-1}

\begin{scope}[scale=1.5]

\coordinate (Aa) at  (-\w,{-.25*\w-1*\h});
\coordinate (Ab) at  (-.3\w,{-.25*\w-.7*\h});
\coordinate (B) at  (.5*\w,{-.25*\w-.5*\h});
\coordinate (Bb) at  (2.5*\w,{-.25*\w-.2*\h});
\coordinate (C) at  ($(.5*\w,{-.25*\w-.5*\h})+(\shiftA)$);
\coordinate (D) at  ($(1.5*\w,{-.25*\w-.75*\h})+(\shiftA)$);
\coordinate (E) at  ($(1.5*\w,{-.5*\w-1.25*\h})+(\shiftA)$);
\coordinate (F) at  ($(1*\w,{-.5*\w-1.75*\h})+(\shiftA)$);
\coordinate (G) at  ($(.5*\w,{-.25*\w-.5*\h})+(\shiftB)$);
\coordinate (H) at  ($(-\w,{-.25*\w-.5*\h})+(\shiftB)$);
\coordinate (I) at  ($(.5*\w,{-.25*\w+.6*\h})+(\shiftB)$);
\coordinate (J) at  ($(.5*\w,{-.25*\w-\h})+(\shiftA)$);
\coordinate (Jb) at  ($(1.5*\w,{-.25*\w-1.4*\h})+(\shiftA)$);
\coordinate (K) at  ($(.5*\w,{-.15*\w-.8*\h})+(\shiftC)$);
\coordinate (Kb) at  ($(.9*\w,{-.25*\w-.7*\h})+(\shiftC)$);
\coordinate (L) at  ($(1.5*\w,{-.25*\w-\h})+(\shiftC)$);


\draw[thick] plot[smooth] coordinates{ (K)(Kb)(L)};

\begin{scope}[shift=(\shiftC)]
\draw[fill=white,opacity=.9] (0,0)-- (\w,-1/2*\w)node[above]{$\mc{W}_{i_3}$}--(\w,-1/2*\w-\h)--(0,-\h)--cycle; 
\draw[red] (.2*\w,-.3*\h)--(.5*\w,-.3*\h-1/4*.6*\w)node[above]{$\mc{V}_{i_{3}}$}-- (.8*\w ,{-.3*\h-1/2*.6*\w})--(.8*\w,{-1/2*.6*\w-.8*\h})--(.2*\w,-.8*\h)--cycle; 
\fill[blue] (K)circle(0.03);

\draw[blue, rounded corners,->] (.5*\w,{-.25*\w-1.3*\h}) node[below]{$\mc{P}_{_{\mc{I}}}^{(3)}(\rho)$}--(K);
\end{scope}
\draw[dashed,thick] plot[smooth] coordinates{ (J)(Jb)(K)};


\draw [thick]plot[smooth]coordinates{(C) (D)(E)(F)(G)};


\begin{scope}[shift=(\shiftB)]
\draw[fill=white,opacity=.9] (0,0)-- (\w,-1/2*\w)--(\w,-1/2*\w-\h)node[below]{$\mc{W}_{i_2}$}--(0,-\h)--cycle; 
\draw[red] (.2*\w,-.3*\h)--(.5*\w,-.3*\h-1/4*.6*\w)node[above]{$\mc{V}_{i_{2}}$}-- (.8*\w ,{-.3*\h-1/2*.6*\w})--(.8*\w,{-1/2*.6*\w-.8*\h})--(.2*\w,-.8*\h)--cycle; 
\fill[blue] (G)circle(0.03);
\draw[blue, rounded corners,->] (-.55*\w,{-.5*\w-.7*\h}) node[left]{$\mc{P}_{_{\mc{I}}}^{(2)}(\rho)$}-- (.45*\w,{-.375*\w-.5*\h})--(G);

\end{scope}


\begin{scope}[shift=(\shiftA)]
\draw[fill=white,opacity=.9] (0,0)node[above]{$\mc{W}_{i_1}$}-- (\w,-1/2*\w)--(\w,-1/2*\w-\h)--(0,-\h)--cycle; 
\draw[red] (.2*\w,-.3*\h)--(.5*\w,-.3*\h-1/4*.6*\w)node[above]{$\mc{V}_{i_{1}}$}-- (.8*\w ,{-.3*\h-1/2*.6*\w})--(.8*\w,{-1/2*.6*\w-.8*\h})--(.2*\w,-.8*\h)--cycle; 
\draw[blue, rounded corners,->] (-.55*\w,{-.25*\w-.2*\h}) node[left]{$\mc{P}_{_{\mc{I}}}^{(1)}(\rho)$}-- (-.05*\w,{-.125*\w-.4*\h})--(C);
\fill[blue] (C)circle(0.03);
\fill[blue] (J)circle(0.03);

\end{scope}
\draw[thick] plot[smooth] coordinates{ (G)(H)(I)(J)};
\draw[thick] plot[smooth] coordinates{ (B) (Bb) (C)};

\draw[fill=white,opacity=.9] (0,0)node[above]{$\mc{W}_{i_0}$}-- (\w,-1/2*\w)--(\w,-1/2*\w-\h)--(0,-\h)--cycle; 
\draw[red] (.2*\w,-.3*\h)--(.5*\w,-.3*\h-1/4*.6*\w)node[above]{$\mc{V}_{i_{0}}$}-- (.8*\w ,{-.3*\h-1/2*.6*\w})--(.8*\w,{-1/2*.6*\w-.8*\h})--(.2*\w,-.8*\h)--cycle; 
\fill[blue] (B)circle(0.03);
\draw[thick] plot[smooth] coordinates{(Aa)(Ab) (B)} ;
\draw[blue, rounded corners,->] (-.55*\w,{-.25*\w-.2*\h}) node[left]{$\rho$}-- (-.05*\w,{-.125*\w-.4*\h})--(B);

\end{scope}

\end{tikzpicture}
\caption{\label{f:chain} A portion of a well-separated set and the chain of symplectomorphisms evaluated at a point $\rho\in \Gamma$. The chain pictured begins $\mc{I}=\{i_0,i_1,i_2,i_3\}$. Notice that the geodesic encounters the boundary of $\mc{W}_{i_1}$ between $\mc{W}_{i_2}$ and $\mc{W}_{i_3}$. It is because of geodesics like this that we must define a Poincar\'e chain. Furthermore, this geodesic encounters the boundary of $\mc{V}_{i_3}$ and, to make the map $\mc{P}\sub{\mc{I}}^{(3)}[g]$ smooth near $\rho$, one must refer to the subset $\tilde{\Gamma}$ in addition to $\Gamma$. }
\end{figure}

The following lemma shows that if $\rho \in \Gamma$ is a returning point, say $d(\varphi_{t_0}^g(\rho),\rho)<\delta$, then we can associate to it a Poincar\'e chain  under which  $\mc{P}\sub{\mc{I}}^{(m)}[g](\rho)=\varphi_{t_0+s}^g(\rho)$ for some $m \in \mathbb N$ and  $s$ small. {Morover, $d(\mc{P}\sub{\mc{I}}^{(m)}[g](\rho),\rho)\lesssim\delta$.}

\begin{lemma}\label{l:pLives}
Let $\Gb\subset \ms{G}^\nu$ be bounded in $\ms{G}^2$ and $\{(\mc{W}_i,\mc{V}_i)\}_{i=1}^N$ be $\fuInj$-well separated for $\Gb$. Then there are $\CG>0$ and $\dG>0$ such that for all $g\in \Gb$, $0<\delta<\dG$, $j\in \{1,\dots,N\}$, and $\rho\in \Rec\sub{\fuInj}(n,\delta,g)\cap \mc{V}_j$ the following holds.  There exist $m\in \mathbb{N}$ and  $\mc{I}$ such that 
\begin{equation}
\label{e:timesComparable}
\rho \in \mc{D}^{(m)}\sub{\mc{I}}[g], \qquad  \qquad \bm{T}\sub{\mc{I}}^{(m)}[g](\rho) \in \big[(n-1)\fuInj-\CG{\delta}\,,\, n\fuInj+\CG{\delta} \big],
\end{equation}
and $d(\mc{P}\sub{\mc{I}}^{(m)}[g](\rho),\rho)<\CG\delta$.
Moreover, for any $m_1,m_2\in \mathbb{N}$ and $\mc{I}_1,\mc{I}_2$ such that
\begin{equation}
\label{e:timesComparable2}
\rho\in \mc{D}\sub{\mc{I}_j}^{(m_j)}[g],\qquad
\bm{T}\sub{\mc{I}_j}^{(m_j)}[g](\rho) \in \big[(n-1)\fuInj-\CG\dG\,,\, n\fuInj+\CG\dG \big],\qquad j=1,2,
\end{equation}
$$
\mc{P}\sub{{\mc{I}_1}}^{({m_1})}[g](\rho)=\mc{P}\sub{{\mc{I}_2}}^{({m_2})}[g](\rho),\qquad \bm{T}\sub{{\mc{I}_1}}^{({m_1})}[g](\rho)=\bm{T}\sub{{\mc{I}_2}}^{({m_2})}[g](\rho).
$$
\end{lemma}
\begin{proof}
Since $\tilde{\Gamma}$ is uniformly transverse to $H\sub{|\xi|_g}$ for $g\in \Gb$, and for all $j$ we have $\mc{V}_j\Subset \mc{W}_j$, there are $\CG>0$ and $\dG>0$ such that for all $g\in \Gb$, $j\in \{1,\dots, N\}$,  $0<\delta<\dG$,  and $\rho \in \widetilde{S^*\!M}$ with 
$
d(\rho, \mc{V}_j)<\delta,
$
there is $|s|\leq \CG\delta$ such that 
\begin{equation}\label{e:smallS}
\varphi_{t_0+s}^g(\rho)\in \mc{W}_j.
\end{equation}

Next,  let $\rho\in \Rec\sub{\fuInj}(n,\delta,g)\cap \mc{V}_j$ with $0<\delta<\dG$. Then, there is $(n-1)\fuInj< t_0\leq n\fuInj$ such that 
$
d(\varphi_{t_0}^g(\rho),\rho)<\delta.
$
Therefore, there is $|s|\leq \CG\delta$ such that $\varphi_{t_0+s}^g(\rho)\in \mc{W}_j$. In particular, by Remark~\ref{r:psAndTs} there are $m$ and $\mc{I}$ such that~\eqref{e:timesComparable} holds and $\mc{P}\sub{\mc{I}}^{(m)}[g](\rho)=\varphi_{t_0+s}^g(\rho)$.

Note that, increasing $\CG$ if necessary (uniformly for $g \in \Gb$), 
\begin{equation}\label{e:disGa}
d(\varphi_{t_0+s}^g(\rho),\rho)\leq d(\varphi_{t_0+s}^g(\rho),\varphi_{\blue{t_0}}^g(\rho))+d(\varphi_{\blue{t_0}}^g(\rho),\rho)<\CG\delta.
\end{equation}

Therefore, it only remains to check that if $(m_1,\mc{I}_1)$ are such that~\eqref{e:timesComparable2} holds, then $\bm{T}\sub{\mc{I}_1}^{(m_1)}[g](\rho)=t_0+s.$  
Suppose this is not the case for some $(m_1,\mc{I}_1)$. Then, since 
$
\varphi_{\bm{T}\sub{\mc{I}_1}^{(m_1)}[g](\rho)}^g(\rho)\in \tilde{\Gamma},
$
by~\eqref{e:thePoint},
 the second equation in~\eqref{e:Cgammas} yields 
$
|\bm{T}\sub{\mc{I}_1}^{(m_1)}[g](\rho)-(t_0+s)|\geq c\sub{\Gamma}\fuInj.
$
In particular,  either $\bm{T}\sub{\mc{I}_1}^{(m_1)}[g](\rho)\geq (n-1)\fuInj-\CG\delta+c\sub{\Gamma}\fuInj$,  or  $\bm{T}\sub{\mc{I}_1}^{(m_1)}[g](\rho)\leq n\fuInj+\CG\delta-c\sub{\Gamma}\fuInj$.
{Since $\delta<\dG$,} this contradicts~\eqref{e:timesComparable2} provided that 
$
{2}\CG\dG<(c\sub{\Gamma}-1)\fuInj.
$
Shrinking $\dG$ if necessary completes the proof.

\end{proof}

Recall that the image of $\mc{P}\sub{\mc{I}}^{(m)}[g_0]$ lands in $\tilde \Gamma$ and that $\Gamma \Subset \widetilde{\Gamma}$. The following lemma shows that if  the iterates of $\rho$ land sufficiently close to $\Gamma$, then there is a neighborhood of $\rho$ that lies in the domain of the same Poincar\'e chain. 
\begin{lemma}\label{l:whatever}
Let $\Gb\subset \G^{\blue{2}}$ bounded and $\Gamma$ be $\fuInj$ well-separated for $\Gb$. There are $c>0$, $\delta_0>0$ such that the following holds. Let $g_0\in\Gb$, $\rho\in \Gamma$, $n\in \mathbb{N}$, and $\mc{I}$ be such that 
$$
\sup_{m\leq n} d(\mc{P}\sub{\mc{I}}^{(m)}[g_0](\rho),\Gamma)<\delta_0.
$$
Then, $B\sub{\tilde\Gamma}(\rho,c^n)\subset \mc{D}\sub{\mc{I}}^{(m)}[g]$ for all $ m\leq n$  provided $\|g-g_0\|_{\blue{\mc{C}^1}}\leq c^n$.
\end{lemma}
\begin{proof}
Let $\delta_0>0$ such that
$
\{\rho\in \tilde{\Gamma}\,:\, d(\rho,\Gamma)<3\delta_0\}\Subset \tilde{\Gamma},
$
and 
$$\sup_{g\in \Gb}\sup_{\rho\in \widetilde{S^*\!M}}\inf \{t>0\,:\, \varphi_t^g(\rho)\in \Gamma\}<C\sub{\Gamma}\fuInj-3\delta_0.$$
First, note that there is $C>0$ such that 
$
d(\varphi_t^g(\rho),\varphi_t^{g_0}(\rho))\leq  \e C^{t+1},
$
for all $g_0\in \Gb$, $\e>0$, $t\in \mathbb R$, and $g \in \G^{\blue{2}}$ with $\|g-g_0\|_{\mc{C}^\blue{2}}\leq \e$. \blue{Indeed, setting $g_\sigma:=g_0+\sigma(g-g_0)$ and using Lemma~\ref{l:PoincareDer} we have
$$
d(\varphi_t^g(\rho),\varphi_t^{g_0}(\rho))\leq \int_0^1\|\partial_\sigma \varphi_t^{g_{\sigma}}\|_{\mathcal{C}^0}d\sigma \leq C^{|t|}\|g-g_0\|_{\mathcal{C}^1}.
$$}
In particular, since $\bm{T}\sub{\mc{I}}^{(m)}[g_0](\rho)\leq mC\sub{\Gamma} \fuInj$ by \eqref{e:Cgammas},
$$
d\Big(\varphi_{\bm{T}\sub{\mc{I}}^{(m)}[g_0](\rho)}^g(\rho)\,,\,\varphi_{\bm{T}\sub{\mc{I}}^{(m)}[g_0](\rho)}^{g_0}(\rho)\Big)\leq \e C^{mC\sub{\Gamma} \fuInj+1}.
$$
Thus, since $d(\varphi^g_{\bm{T}\sub{\mc{I}}^{(m)}[g_0](\rho)}(\rho),\Gamma)<\delta_0$
and $H\sub{|\xi|_g}$ is uniformly transverse to $\tilde{\Gamma}$, for $g\in \Gb$, the implicit function theorem implies that for $\e$ small enough such that $\e C^{mC_\Gamma\fuInj+3}\ll \delta_0$, there are $t_j$, $j=1,\dots, m$ such that $|t_j-\bm{T}\sub{\mc{I}}^{(j)}[g_0](\rho)|\leq \e C^{jC_\Gamma \fuInj+2}\leq \delta_0$ and
\begin{gather*}
\varphi_{t_j}^{g_0}(\rho)\in  \tilde{\Gamma},\qquad d(\varphi_{t_j}^{g_0}(\rho),\varphi_{\bm{T}\sub{\mc{I}}^{(j)}[g_0](\rho)}^g(\rho))\leq \e C^{jC_\Gamma \fuInj+3}\leq 2\delta_0.
\end{gather*}
\blue{Here, we applied the implicit function theorem using that $G \subset \G^2$ is bounded and so the $\varphi^g_t$ flow is $\mathcal{C}^1$.}
In particular,  since $\bm{T}\sub{\mc{I}}^{(j)}[g_0](\rho)-\bm{T}\sub{\mc{I}}^{(j-1)}[g_0](\rho)<C\sub{\Gamma}\fuInj-3\delta_0$ for all $1\leq j \leq m$, we have $|t_1|< C\sub\Gamma\fuInj$, and $|t_j-t_{j-1}|<C\sub{\Gamma}\fuInj$ for $1\leq j\leq m$. Thus,  $\rho \in \mc{D}\sub{\mc{I}}^{(m)}[g]$ and
$$
|\bm{T}\sub{\mc{I}}^{(m)}[g_0](\rho)-\bm{T}\sub{\mc{I}}^{(m)}[g](\rho)|\ll 1,
$$
 provided $\e\ll  C^{-mC\sub{\Gamma}\fuInj-1}$. 
Thus, we have
$$
d(\varphi_{\bm{T}\sub{\mc{I}}^{(m)}[g](\rho)}^{g}(\rho),\Gamma)< \delta_0+\e C^{mC\sub{\Gamma}+1}<2\delta_0.
$$

Arguing as above, since there is $C>0$ such that for all $g$ with $\red{d_{\mc{C}^2}}(g,\Gb)<c$, 
$$
d(\varphi_t^{g}(\rho),\varphi_t^{g}(\rho_1))<C^td(\rho,\rho_1),
$$
we have for $d(\rho,\rho_1)<\e$, 
$$
|\bm{T}\sub{\mc{I}}^{(m)}[g](\rho_1)-\bm{T}\sub{\mc{I}}^{(m)}[g](\rho_2)|\ll 1,
$$
and hence $\rho_1\in \mc{D}\sub{\mc{I}}^{(m)}[g]$.
\end{proof}

We now define the notion of a non-degenerate returning point for
a Poincar\'e chain $\mc{P}\sub{\mc{I}}$. 
Let $\Gb\subset \ms{G}^\nu$ be bounded in $\G^{\blue{2}}$ and $\{(\mc{W}_i,\mc{V}_i)\}_{i=1}^N$ be $\fuInj$-well separated for $\Gb$.
To define non-degeneracy, we fix a finite atlas of
coordinates charts on $\Gamma$ 
\begin{equation}\label{e:atlas}
\mc{A}:=\big\{(\psi_i, U_i):\; \psi_i:U_i\to\psi_i(U_i), \; i=1, \dots, N\big\},
\end{equation}
with $U_i\subset\Gamma$ and
$\psi_i(U_i)\subset\mathbb{R}^{2\aleph}$.
Let $\CG$ and $\dG$ be as in Lemma~\ref{l:pLives} and let $\delta_0$ be as in Lemma \ref{l:whatever}. Without loss of generality we assume that $\dG$ is small enough that
\begin{equation}\label{e:Cdelta}
\begin{gathered}
\dG<\min(1, \delta_0),\\ 
\text{for every }  \rho \in \Gamma \text{ there is }  (\psi, U)\in \mc{A} \text{ such that } 
B(\rho,\CG\dG)\subset U.
\end{gathered}
\end{equation}

Next, observe that by Lemma~\ref{l:pLives}, if $\rho\in \Gamma\cap \Rec\sub{\fuInj}(n,\delta,g)$ for $0<\delta<\dG$, then there is $(\psi, U)\in \mc{A}$ such that $\rho \in U$ for every $(m,\mc{I})$ satisfying~\eqref{e:timesComparable}, $\mc{P}\sub{\mc{I}}^{(m)}[g](\rho)\in U$,
 and we may work in $\R^{2\aleph}$ by setting
\begin{equation}\label{e:subPsi}
(\mc{P}\sub{\mc{I}}^{(m)}[g])_\psi:= \psi \circ \mc{P}\sub{\mc I}^{(m)}[g] \circ \psi^{-1}.
\end{equation}
Abusing notation slightly, we write $d((\mc{P}\sub{\mc{I}}^{(n)}[g])_\psi)(\rho):=d((\mc{P}\sub{\mc{I}}^{(n)}[g])_\psi)(\psi(\rho))$.

We now define the notion of non-degeneracy of a trajectory for a chain of symplectomorphisms.

\begin{definition} Let $\CG$ and $\dG$ be as in Lemma~\ref{l:pLives}  and \eqref{e:Cdelta}.
\label{d:ND}Let 
$n\in\mathbb{N}$, and $\beta>0$. We say that $\rho\in\Rec\sub{\fuInj}(n,\dG,g)$
is $(n,\beta,g)$ non-degenerate, and write $\rho\in\ND\sub{\fuInj}(n,\beta,g)$,
if there exist $\mc{I}$ and $m$ such that $\bm{T}\sub{\mc{I}}^{(m)}[g](\rho) \in \big[(n-1)\fuInj-\CG\dG\,,\, n\fuInj+\CG\dG \big]$  and
\begin{equation}\label{e:dishes}
{\sup_{\{(\psi,U)\in\mc{A}:\,B(\rho,\CG\dG)\subset U\}}}
\big\|\big(\Id-d((\mc{P}\sub{\mc{I}}^{(m)}[g])_\psi)(\rho)\big)^{-1}\big\|<\beta^{-1}.
\end{equation}
\end{definition}

We next strengthen the notion of non-degeneracy of a trajectory to
one that can pass from a trajectory to its iterates by introducing a type of $(\beta,q)$-non-degeneracy for returning points.

\begin{definition}
\label{d:NDq}
Let $\CG$ and $\dG$ be as in Lemma~\ref{l:pLives} and \eqref{e:Cdelta}.
Let $q\in\mathbb{N}$, $n\in\mathbb{N}$ and $\beta\in(0,1)$. We say that $\rho\in\Rec\sub{\fuInj}(n,\dG,g)$
is $(n,\beta,g)$ $q$-non-degenerate, and write $\rho\in\ND_{\!\!q,_\fuInj}(n,\beta,g)$,
if there exist $\mc{I}$ and $m$ satisfying $\bm{T}\sub{\mc{I}}^{(m)}[g](\rho) \in \big[(n-1)\fuInj-\CG\dG\,,\, n\fuInj+\CG\dG \big]$ and
\begin{equation}
{\sup_{\{(\psi,U)\in\mc{A}:\,B(\rho,\CG\dG)\subset U\}}}\frac{\big\|(\Id-[(d((\mc{P}\sub{\mc{I}}^{(m)}[g])_\psi)(\rho)]^{q})^{-1}\big\|}{(1+\|d((\mc{P}\sub{\mc{I}}^{(m)}[g])_\psi)(\rho)\|^{q})^{2\aleph-1}}\leq\ANDrate(\beta,q), \qquad \ANDrate(\beta,q)=(\tfrac{5}{2} q^2\beta^{-1})^{2\aleph}.\label{e:ND_q}
\end{equation}
\end{definition}

{That is, $\rho\in\ND\sub{q,\fuInj}(n,\beta,\kappa)$ provided there are $\mc{I}$ and $m$ satisfying~\eqref{e:timesComparable2} such that the matrix ${\bf A}:=d(\psi\circ\mc{P}\sub{\mc I}^{(m)}[g] \circ\psi^{-1})(\psi(\rho))$ is $(\beta, q)$ non-degenerate (as defined in \eqref{e:q-Nondegenerate}) for every chart $(\psi,U)\in\mc{A}$ with $B(\rho,\CG\dG)\subset U$.}

\begin{remark} The reason for the choice of $\ANDrate$ comes from
Corollary~\ref{c:inverse} and the fact that we need to be able
to sum measures of bad sets over all $q$ in Lemma~\ref{l:lemma4.5}.
We will make perturbations which guarantee
certain properties of the eigenvalues of $d\mc{P}\sub{\mc I}^{(n)}[g] $ and will use
Corollary~\ref{c:inverse} to bound the inverse of $I-(d\mc{P}\sub{\mc I}^{(n)}[g] )^{q}$.
In fact, $q^{2}$ could be replaced by $q^{\gamma}$
for any $\gamma>1$.
\end{remark}

\begin{remark}
In Definitions~\ref{d:ND} and~\ref{d:NDq}, the existence of $\mc{I}$ and $m$ satisfying~\eqref{e:timesComparable} and, respectively~\eqref{e:dishes} or~\eqref{e:ND_q} could be replaced by the requirement that for all $\mc{I}$ and $m$ satisfying~\eqref{e:timesComparable2}, respectively, the estimates~\eqref{e:dishes} or~\eqref{e:ND_q} hold. To see this, we observe that by the second part of Lemma~\ref{l:pLives}, all $\mc{I}$ and $m$ satisfying~\eqref{e:timesComparable2} in fact satisfy ~\eqref{e:timesComparable} and result in the same map $\mc{P}\sub{\mc{I}}^{(m)}[g]$.
\end{remark}

The next lemma demonstrates how the notion of $(n,\beta,g)$
$q$ non-degeneracy can be passed to $(nq,\beta,g)$ non-degeneracy
with the trajectory of length $nq$ being a multiple of that of length
$n$. 

\begin{lemma} \label{l:NdqImpliesNd}
{Let $G\subset \blue{\ms{G}^{3}}$ bounded and $\Gamma$ be $\fuInj$ well separated for $G$.  let $\CG$ and $\dG$ be as in Lemma~\ref{l:pLives} and \eqref{e:Cdelta}. There is $C>0$ such that the following holds.  Let $g\in \Gb$, $k>0$, $\beta>0$,  $q_0\in \mathbb N$ and $0<\delta \leq  \min\{\dG, (2\ANDrate(\beta,q_0)C^{kq_0})^{-1}\}$. Let 
$$
\rho\in \Gamma\cap\Rec\sub{\fuInj}(k,\delta,g) \cap\ND_{q_0,_\fuInj}(k,\beta,g),
$$
and
$\mc{I}$ and $m$ such that
$\bm{T}\sub{\mc{I}}^{(m)}[g](\rho) \in \big[(k-1)\fuInj-\CG\dG\,,\, k\fuInj+\CG\dG \big]$, 
$$
\begin{gathered}
\big\{\mc{P}\sub{\mc{I}}^{(qm)}[g](\rho)\big\}_{q=0}^{q_0}\subset \mc{D}\sub{\mc{I}}^{(m)}[g],\qquad \max_{0\leq q \leq q_0}d\Big(\big(\mc{P}\sub{\mc{I}}^{(m)}[\blue{g}]\big)^q(\rho),\rho\Big)\leq\delta,\\
\mc{P}\sub{\mc{I}}^{(qm)}[g](\rho')=\big(\mc{P}\sub{\mc{I}}^{(m)}[g]\big)^q(\rho'),\qquad 0\leq q\leq q_0,\qquad\text{{ $\rho'$ in a neighborhood of $\rho$}}.
\end{gathered}
$$
Then,
$$
\rho\in \ND\sub{\fuInj}\big(n,(2\ANDrate(\beta,q_0)C^{kq_0})^{-1},g\big)
$$
for all $n\in \mathbb{N}$ such that 
$\bm{T}\sub{\mc{I}}^{(q_0m)}(\rho)\in[(n-1)\fuInj-\CG\dG, n\fuInj+\CG\dG].$}
\end{lemma} 
\begin{proof}  
For $0\leq q\leq q_0$, set $\rho_q:=\mc{P}\sub{\mc{I}}^{(qm)}[g](\rho)$ and $\Delta_q:=d\mc{P}\sub{\mc{I}}^{(m)}[g](\rho_q)-d\mc{P}\sub{\mc{I}}^{(m)}[g](\rho)$. 
 Then, {since there is $C>0$ such that $\|d^{2}\varphi_t^{g}\|\leq C^{t}$ for all $t\geq0$, by assumption} we conclude that there is $C>0$ such that 
$
\max_q\|\Delta_q\|\leq C^{k}\delta
$.
 Therefore,
\begin{align*}
d\mc{P}^{(mq_0)}\sub{\mc{I}}[g](\rho) & =(d\mc{P}\sub{\mc{I}}^{(m)}[g](\rho) +\Delta_{q_0-1})\circ\cdots\circ(d{\mc{P}\sub{\mc{I}}^{(m)}[g]}(\rho)+\Delta_{1})\circ d{\mc{P}^{(m)}\sub{\mc{I}}[g]}(\rho)\\
&=[d{\mc{P}^{(m)}\sub{\mc{I}}[g]}(\rho)]^{q_0}+\Delta'
\end{align*}
with 
$
\|\Delta'\|\leq2^{q_0}\|d{\mc{P}\sub{\mc{I}}[g]^{(m)}}(\rho)\|^{\blue{q_0-1}}C^{k}\delta
$
\blue{and we have used that $C^k\delta<1\leq \|d\mc{P}\sub{\mc{I}}^{(m)}[g](\rho)\|$.}
We conclude that there is $C_0>0$ depending only on $G$ such that 
\begin{equation}\label{e:soon}
\|d{\mc{P}^{(q_0m)}\sub{\mc{I}}[g]}(\rho)-[d{\mc{P}^{(m)}\sub{\mc{I}}[g]}(\rho)]^{q_0}\| \leq C_0^{kq_0}\delta.
\end{equation}
Since $\rho\in \ND_{q_0,_\fuInj}(k,\beta,g)$, there is $C>0$ depending only on $\Gb$ such that for all $(\psi,U)\in\mc{A}$ with $B(\rho, \CG\dG)\subset U$,
\[
\big\|\big(\Id-\big[d(\mc{P}\sub{\mc{I}}^{(m)})_\psi(\rho)\big]^{q_0}\big)^{-1}\big\|\leq C^{kq_0}\ANDrate(\beta,q_0).
\]
In particular, provided that $C_0^{2kq_0}\delta\ANDrate(\beta,q_0)<{\tfrac{1}{2}}$, \eqref{e:soon} yields
\begin{align*}
&\|[\Id-d(\mc{P}\sub{\mc{I}}^{(q_0m)})_\psi(\rho)]^{-1}\|\\
&\blue{=\Big\|\Big(\Id-\big[d(\mc{P}\sub{\mc{I}}^{(m)})_\psi(\rho)\big]^{q_0}+\big[d(\mc{P}\sub{\mc{I}}^{(m)})_\psi(\rho)\big]^{q_0}-d(\mc{P}\sub{\mc{I}}^{(q_0m)})_\psi(\rho)\Big)^{-1}\Big\|}\\
&\blue{\leq C^{kq_0}\ANDrate(\beta,q_0)\big\|\Big\|\Big(I+\big(\Id-\big[d(\mc{P}\sub{\mc{I}}^{(m)})_\psi(\rho)\big]^{q_0}\big)^{-1}\big(\big[d(\mc{P}\sub{\mc{I}}^{(m)})_\psi(\rho)\big]^{q_0}-d(\mc{P}\sub{\mc{I}}^{(q_0m)})_\psi(\rho)\big)\Big)^{-1}\Big\|}\\
&\leq C^{kq_0}\ANDrate(\beta,q_0)(1-C_0^{2kq_0}\delta\ANDrate(\beta,q_0))^{-1}\\
&\leq 2C^{kq_0}\ANDrate(\beta,q_0).
\end{align*}
The claim now follows from the definition of non-degeneracy. 
\end{proof} 

\section{Perturbing away $q$-degeneracy for simple points}
\label{s:basicPerturb}

{In this section, we first introduce general assumptions on a family of perturbations which will allow us to change the degeneracy properties of a given simple point. We then use a quantitative analog of Sard's theorem to show that the collection of perturbation parameters which produce degeneracy is small.}

\subsection{General assumptions}

Let $\ms{G}^\nu$ denote the space of $\mc{C}^\nu$ metrics on $M$ with the topology induced by the $\mc{C}^\nu$ topology on symmetric tensors, and $\iota\sub{\mc{C}^\nu}:\mc{C}^\nu\to \mc{C}^{\nu-1}$ be the natural inclusion map.



{Given a function $N:(0,1)\to \mathbb N$,} $R>0$, and $\delta>0$, we will work with  perturbation maps $Q^{R,\delta}:\ms{G}^\nu\times \Si{N(R)}\to \ms{G}^\nu$
where
\begin{equation}
\label{e:LDef}
{\Si{N(R)}:=\prod_{j=1}^{N(R)} B\sub{\R^\L}(0,1),\qquad \L:=\aleph(2\aleph+3)},
\end{equation}

\begin{remark}\label{r:mcL} Recall that $2\aleph=2d-2=\dim \Gamma,$ where $\Gamma$ is a section of $S^*\!M$.
Thus, the dimension $\L=\aleph(2\aleph+3)$ appears because we want the parameter $\sigma\in  B\sub{\R^\L}(0,1)$ to be in correspondence with points in the space $(\Gamma, \Sp(2\aleph))$, {and the latter has dimension $\L$}. Here $\Sp(2\aleph)$ will parametrize the possible derivatives of a symplectic map at a point in $\Gamma$.
\end{remark}
We write ${\bm{\sigma}}=(\sigma_{j})_{j=1}^{N(R)}$ where each $\sigma_{j}$
is a component in $B\sub{\R^\L}(0,1)$, {and we refer to $Q^{R,\delta}(g, {\bm \sigma})$ as the perturbation of $g$ by the parameter $\bm{\sigma}$}. When we construct our family of perturbations in Section~\ref{s:perturbedMetrics}, $R$ will control the scale on which our metric perturbations take place and $\delta$ will control the size of the perturbation \blue{that} takes place.
We ask that the perturbation maps satisfy the following assumptions. 

\begin{definition}[Good perturbations] \label{ass:1} 
Let $\nu_0\geq 0$ {and $N:(0,1)\to \mathbb N$}. For each $R>0, \delta>0$ consider a map $$Q^{R,\delta}:\ms{G}^{\nu_0}\times \Si{N(R)}\to \ms{G}^{\nu_0}$$
{such that for each $g\in \ms{G}^{\nu_0}$, the map $\bm\sigma \mapsto Q^{R,\delta}(g,\bm\sigma)$ is Frechet differentiable,  and for all $\bm\sigma\in \Si{N(R)}$ the map  
$
g\mapsto \iota\sub{\mc{C}^{\nu_0-1}} (Q^{R,\delta}(g,\bm\sigma))
$
is Frechet differentiable.}

We say that $\{Q^{R,\delta}\}_{R,\delta}$ is a \emph{$(\nu_0, N)$-good family of perturbations}  
if {for each {$K\subset \ms{G}^{\nu_0}$ bounded} and each $0\leq \nu\leq \nu_0$ there exist {$\vartheta_{\nu}\in\re$}, $\tilde\vartheta_{\nu}\in\re$} and   $C>0$ such that the following holds.

\noindent For  $\delta>0$,   $R>0$,
\begin{align}
&\bullet \; Q^{R,\delta}(g,0)=g, \qquad g\in\ms{G}^{\nu_0}. \label{e:Qzero}\\
&\bullet \; \|Q^{R,\delta}(g,\bm\sigma)-Q^{R,\delta}(g,\tilde{\bm\sigma})\|\sub{\mc{C}^{\nu}}\leq C\delta R^{-\vartheta_{\nu}}\|\bm{\sigma}-\tilde {\bm\sigma}\|_{\ell^{\infty}}, \qquad \bm \sigma, \tilde {\bm \sigma} \in \Si{N(R)},\,g\in K. \label{e:changeSigma}\\
&\bullet \; \|D_gQ^{R,\delta}-I\|\sub{\mc{C}^{\nu-1}\to \mc{C}^{\nu-1}}\leq C\delta R^{{-\vartheta_{\nu}}},\qquad g\in K,\bm\sigma\in \Si{N(R)}. \label{e:derivative}
\end{align}
\end{definition}

We also need an assumption which guarantees that the perturbation $Q^{R,\delta}$ explores a sufficiently large set of possible metrics.

\begin{definition}[Admissible pairs] \label{ass:2}
Let $\nu\geq \blue{3}$, $b>0$ and $\ga\geq 1$. {Let $N:(0,1)\to \mathbb N$ and $\{Q^{R,\delta}\}_{R,\delta}$ be a $(\nu, N)$-good family of perturbations}.
 
 We say that $(\Gamma,\Gb)$ is a \emph{$(\fuInj,b,\ga)$-admissible pair for $\{Q^{R,\delta}\}_{R,\delta}$} if  {$\Gb\subset \ms{G}^{\nu}$ is bounded}, $\Gamma$ is $\fuInj$-well separated for $\Gb$, and
for all $K\subset \Gb$ bounded {in $\ms{G}^{\nu}$} there are $c>0$ and $\e_0>0$
such that for all $R>0$, {$\delta>0$}, $g\in K$, $\alpha\in(0,c)$, $n\in\mathbb{N}$, {and $\rho_{0}\in\Gamma$}
{satisfying }
\[
R<c^{n}\alpha, \qquad\qquad\delta\leq R^{b}c^{n+1},\qquad \qquad
\rho_0\in\Sim\sub{\!\fuInj}(n,\alpha,g)\cap \Rec\sub{\fuInj}(n,\e_0,g),
\]
there exist {$j_0\in\{1, \dots, N(R)\}$}, $\mc{I}_0$, $m_0$ such that $B\sub{\Gamma}(\rho_0,R^{\ga})\subset \mc{D}\sub{\mc{I}_0}^{(m_0)}[Q^{R,\delta}(g,\bm\sigma)]$ for all $\bm\sigma\in \Si{N(R)}$, {$\bm{T}\sub{\mc{I}_0}^{(m_0)}[g](\rho) \in \big[(n-1)\fuInj-\CG\dG\,,\, n\fuInj+\CG\dG \big]$}, and for all $\rho\in B\sub{\Gamma}(\rho_{0},R^{\ga})$
and ${\bm{\sigma}}\in\Si{N(R)}$, 
\begin{equation}\label{e:PsiKappaBound}
|d_{\sigma_{j_0}}\Psi_{g,_{\mc{I}_0}}^{R,\delta}(m_0,{\bm{\sigma}},\rho)v|\geq\delta c^{n+1}|v|,\qquad\text{ for all }v\in T_{{\bm{\sigma}}_{j_0}}B\sub{\R^\L}(0,1),
\end{equation}
{where 
\begin{equation}\label{e:PsiKappa}
\Psi_{g,_{\mc{I}_0}}^{R,\delta}(m_0,{\bm{\sigma}},\rho)
:=\Big(\mc{P}\sub{{\mc{I}}_0}^{(m_0)}[\smalleq{Q^{R,\delta}(g,\bm\sigma)}](\rho)\;,\;d_\rho\big(\mc{P}\sub{\mc{I}_0}^{(m_0)}[\smalleq{Q^{R,\delta}(g,\bm\sigma)}]\big)(\rho)\Big).
\end{equation}}
{The constants $\CG$ and $\dG$ are the ones given in Lemma \ref{l:pLives}.}
\end{definition}

To understand why we are interested in the map $\Psi_g^{R,\delta}$, observe that knowing 
$$
(\mc{P}^{(m_0)}\sub{\mc{I}}[g](\rho),d_\rho \mc{P}\sub{\mc{I}}^{(m_0)}[g](\rho))
$$ 
is exactly the information need to understand (1) whether $\rho\in \Rec\sub{\fuInj}(n,\beta,g)$ and (2) if $\rho\in \Rec\sub{\fuInj}(n,\beta,g)$, whether $\rho\in \ND_{q,_\fuInj}(n,\beta,g).$ Thus, the dimension of $B\sub{\R^\L}(0,1)$ is chosen so that we are able to use {the $j_0$-th} factor in $\Si{N(R)}$ to locally explore an open set {in the image of the} maps $\Psi_{g,_{\mc{I}}}^{R,\delta}$ {(whose dimension is  $\L=\aleph(2\aleph+3)$) {by varying $\sigma_{j_0}\in B\sub{\R^\L}(0,1)$}. {See also Remark \ref{r:mcL}.}}

\subsection{Perturbing away $({\beta},q)$-degeneracy for simple points}
In order to control the volume of the parameters in $\Si{N(R)}$ which may produce degeneracy we apply a quantitative version of Sard's theorem
(\cite[Theorem 4.2]{Yo:85}). We recall the theorem here for convenience. Below, $\mc{M}\sub{{\R^k}}(Y,s)$ denotes the covering number of $Y \subset \R^k$ by balls in $\R^k$ of radius $s$ (see Definition \ref{d:coveringNumber}).

\begin{proposition}[\cite{Yo:85} Theorem 4.2]  \label{t:Yomdin}
Let $\Omega\subset\mathbb{R}^{\ell}$ and $L>0$
such that for all $\rho_{1},\rho_{2}\in \Omega$, there is a path in $\Omega$
of length  smaller than $L\|\rho_{1}-\rho_{2}\|$ connecting $\rho_{1}$ and
$\rho_{2}$. Let $\Psi:B\sub{\mathbb{R}^{k}}(0,1) \times \Omega\to\mathbb{R}^{k}$
be a continuously differentiable mapping such that {there exist $c_1>0$, $c_2>0$ so that for all $\rho \in \Omega, $ and $\sigma, \tilde \sigma \in B\sub{\mathbb{R}^{k}}(0,1)  $}
\[
\|d_{\rho}\Psi(\sigma,\rho)\|\leq c_{1},\qquad
  \|\Psi(\tilde \sigma,\rho)-\Psi(\sigma, \rho)\|\geq {c_{2}}\|\sigma- \tilde \sigma\|.
\]
Then, for any $s>0$ and any $X\subset \Omega$, $Y\subset\mathbb{R}^{k}$,
\[
\mc{M}\sub{\mathbb{R}^{k}}(\Delta\sub{\Psi}(X,Y),\;\tilde s)\leq\mc{M}\sub{{\R^\ell}}(X,s)\mc{M}\sub{{\R^k}}(Y,s)
\]
where $\tilde s:= 2(L c_{2}^{-1}(1+c_{1})+1)s$ and 
\begin{equation}\label{e:Delta}
\Delta\sub{\Psi}(X,Y):=\{\sigma \in B\sub{\mathbb{R}^{k}}(0,1)\,:\,\text{there is }\rho\in X\text{ such that }\Psi(\sigma,\rho)\in Y \}.
\end{equation}
\end{proposition}
Let $\mS$ be the Lebesgue measure on $\R^\L$ normalized so
that 
\begin{equation}
\mS(B\sub{\R^\L}(0,1))=1\label{e:m_1}
\end{equation}
and let {$\m\sub{\Si{N(R)}}$} be the product measure on $\Si{N(R)}$. We write $\m\sub{\Gamma}$ for the volume measure on $\Gamma$.

Next, we show that there is $c>0$ such that if $\rho_0 \in \Gamma$ is $(n, c)$ returning and $(n,\alpha)$ simple for some metric $g$, then there is only a small measure  set of $\bm{\sigma}$ that yield a perturbed metric $Q^{R,\delta}(g,\bm\sigma)$ for which there are points $\rho \in B(\rho_0, R^{\ga'})$ that are $(n, \beta q^{-3})$ returning and are \emph{not} $(n, \beta)$ $q$-non-degenerate for some $q \in \mathbb N$. 
This can be done for $\alpha<c$ and $\beta<c^n R^{\ga'}$ provided that we work with metrics that are at least \blue{$\mc{C}^3$}.

\begin{lemma}\label{l:lemma4.5}
Let $b>0$, \blue{$\nu\geq 3$}, $\ga\geq 1,$ and $N:(0,1)\to \mathbb N$. Let
$\{Q^{R,\delta}\}_{R,\delta}$ be a $(\nu, N)$-good family of perturbations and $(\Gamma,\Gb)$ be a $(\fuInj,b,\ga)$-admissible pair  for $\{Q^{R,\delta}\}_{R,\delta}$.

Let $K\subset \Gb$  bounded in \blue{$\ms{G}^{3}$}  and $\ga'\geq \ga$.  Let  $\dG$ be as in Lemma~\ref{l:pLives} and \eqref{e:Cdelta}. 
There exist $c\in(0,1)$, $C>0$ such the following holds.
Let   {$n \in \mathbb N$,}   {$\alpha\in(0, c)$,}
\[
0<R< \min(c^{n}\alpha, \tfrac{1}{2}c\dG),\qquad  0<\delta\leq \blue{R^{b}c^{n+1}}.
\]
If {$g\in K$,} $\rho_{0}\in\Gamma$, and $B\sub{\Gamma}(\rho_{0},R^{\ga'})\cap\Sim\sub{\!\fuInj}(n,\alpha,g)\cap \Rec\sub{\fuInj}(n,c,g)\neq\emptyset$, then {for all $0<\beta<\min{(\tfrac{1}{2}\dG,\blue{ R^{\ga'}})}$}
\[
\m\sub{\Si{N(R)}}(\mathcal{S}^{\rho_0,  R,\delta}_{g}(n, {\alpha,} \beta))
\leq
\m\sub{\Gamma}(B\sub{\Gamma}(\rho_{0},R^{\ga'}))\blue{C^n \delta^{-\L}\Big(R^{-\RPower}\delta+R^{\ga'}\Big)^{\L-4\aleph}\m\sub{\Gamma}(B\sub{\Gamma}(\rho_{0},R^{\ga'}))\blue{\beta}} ,
\]

where 
\begin{align*}
\mathcal{S}^{\rho_0,  R,\delta}_{g}(n,  {\alpha,} \beta)
&:=\Big\{ {\bm{\sigma}}\in\Si{N(R)}:\; \text{there are } \rho \in B\sub{\Gamma}(\rho_{0},R^{\ga'}) \text{ and } \; q\in  \mathbb N \; \;\text{such that}  \\
&\hspace{4cm}\rho\in\Rec\sub{\fuInj} \big(n,\beta q^{-3},{\smalleq{Q^{R,\delta}(g,\bm\sigma)}}\big) \big\backslash \ND_{q,_\fuInj}\big(n,\beta,\smalleq{Q^{R,\delta}(g,\bm\sigma)}\big)\Big\}.
\end{align*}
\end{lemma}

\begin{proof}  
{Let $\e_0$ as in Definition~\ref{ass:2}.} Using the estimate on $(\varphi_t^g)^*:\mc{C}^1\to \mc{C}^1$ in Lemma~\ref{l:PoincareDer}, we first note that there exists  $c=c(K)>0$ such that if 
$g\in K$, $\rho_{0}\in\Gamma$, $0<\alpha<1$, $n \in \mathbb N$,  and
 $0<R\leq c^{n}\alpha$, then
$B(\rho_0,R^{\ga'})\cap \Sim\sub{\!\fuInj}(n,\alpha,g)\cap \Rec\sub{\fuInj}(n,c,g)\neq \emptyset$
implies  $\rho_{0}\in\Sim\sub{\!\fuInj}(n,\frac{1}{3}\alpha,g)\cap \Rec\sub{\fuInj}(n,\e_0,g)$. By Definition \ref{ass:2}, this allows us to work with indices $j_0$, $m_0$, and $\mc{I}_0$ such that  the lower bound on the differential of $\Psi_{g,_{\mc{I}_0}}^{R,\delta}$ in \eqref{e:PsiKappaBound} holds.

{The result will follow once we find $c, C$ such that, under the assumptions of the lemma,
 for all $\hat{\bm{\sigma}}\in  \prod_{j=1}^{N(R)-1}\!\!B\sub{\R^\L}(0,1)$ and $0<\beta<\min(\tfrac{1}{2}\dG, \blue{R^{\ga'}})$}
\begin{equation}\label{e:goal4.5}
\mS(\mathcal{S}^{\rho_0,  R,\delta}_{g}(n,  \alpha, \beta; j_0,\hat{\bm{\sigma}}))\leq \m\sub{{\Gamma}}(B\sub{{\Gamma}}(\rho_{0},R^{\ga'}))C^n \delta^{-\L}\blue{\Big(R^{-\RPower}\delta+R^{\ga'}\Big)^{\L-4\aleph} }\m\sub{\Gamma}(B\sub{\Gamma}(\rho_{0},R^{\ga'}))\beta ,
\end{equation}
where for 
$\hat{\bm{\sigma}}\in \prod_{j=1}^{N(R)-1}\!\!B\sub{\R^\L}(0,1)$, we let $\bm{\sigma}_{j_0}^{-}:=(\hat{{\bm{\sigma}}}_{i})_{i=1}^{j_0-1}$, $\bm{\sigma}_{j_0}^{+}:=(\hat{{\bm{\sigma}}}_{i})_{i=j_0}^{N(R)-1}$, and 
\[
\mathcal{S}^{\rho_0,  R,\delta}_{g}(n, {\alpha,}  \beta; j_0,\hat{\bm{\sigma}})
=\Big\{\sigma\in B\sub{\R^\L}(0,1):\;\,\big(\bm{\sigma}_{j_0}^{-},\sigma,\bm{\sigma}_{j_0}^{+}\big)\in\mathcal{S}^{\rho_0, R,\delta}_{g}(n,  \alpha, \beta)\Big\}.
\]

{To prove the claim in \eqref{e:goal4.5} we restrict the map $\Psi_{g,_{\mc{I}_0}}^{R,\delta}$ to its $j_0^{th}$-component in $\bm{\sigma}$.}
For 
$
\hat{{\bm{\sigma}}}\in \prod_{j=1}^{N(R)-1}\!\!B\sub{\R^\L}(0,1)
$ 
define the restricted map
$
{\Psi}_{g,_{\mc{I}_0},j_0,\hat{{\bm{\sigma}}}}^{R,\delta}(n,\sigma,\rho)=\Psi_{g,_{\mc{I}_0}}^{R,\delta}\big(n, (\bm{\sigma}_{j_0}^{-},\sigma,\bm{\sigma}_{j_0}^{+}) ,\rho\big),
$
with domain $(n,\sigma,\rho)\in \mathbb{N}\times B\sub{\R^\L}(0,1) \times B\sub{\Gamma}(\rho_{0},R^{\ga'})$. 

{
{Since  $\rho_{0}\in \Gamma\cap \Sim\sub{\!\fuInj}(n,\frac{1}{3}\alpha,g)\cap\Rec\sub{\fuInj}(n,\e_0,g) $,}
assumption~\eqref{e:PsiKappaBound} \blue{and the mean value theorem  yield} that for each $\hat{\bm{\sigma}}$,
and $\sigma,\tilde \sigma\in B\sub{\R^\L}(0,1)$, 
\begin{equation}
\inf_{\rho\in B\sub{{\Gamma}}(\rho_{0},R^\ga)}\|{\Psi}_{g,_{\mc{I}_0},j_0,\hat{{\bm{\sigma}}}}^{R,\delta}(m_0,\sigma,\rho)-{\Psi}_{g,_{\mc{I}_0},j_0,\hat{{\bm{\sigma}}}}^{R,\delta}(m_0,\tilde \sigma,\rho)\| \geq \delta c^{n+1}\|\sigma- \tilde \sigma\|,\label{e:inj}
\end{equation}
 provided that 
\begin{equation}\label{e:conds}
R<{c^{n}}\alpha,\qquad\delta\leq R^{b}c^{n+1}.
\end{equation}
after possibly shrinking $c$ in a way that only depends on $(K, \Gamma, \fuInj)$.

{Note that if $\sigma \in \mathcal{S}^{\rho_0,  R,\delta}_{g}(n, \alpha, \beta; j,\hat{\bm{\sigma}})$, then by Lemma~\ref{l:pLives} there are $\rho \in B\sub{{\Gamma}}(\rho_0,  R^{\ga'})$ and $q \in \mathbb N$ such that 
\begin{equation}\label{e:rhoUf}
d \big(\rho\,,\, {\mc{P}\sub{\mc{I}_0}^{(m_0)}[\smalleq{Q^{R,\delta}(g,\bm\sigma)}]
} (\rho) \big)< \CG\beta q^{-3}, \qquad \rho \notin \ND_{q,_\fuInj}\big(n,\beta,\smalleq{Q^{R,\delta}(g,\bm\sigma)}\big).
\end{equation}
After requiring that {$R^\ga<\CG\dG$ and $\beta<\frac{\dG}{2}$}, 
 there exists a coordinate chart $(U,\psi)\in \mc{A}$  (see \eqref{e:Cdelta}) with  $B\sub{\Gamma}(\rho_0,  R^{\ga'}) \subset \psi(U)$ such that \eqref{e:ND_q} does not hold with $\mc{I}=\mc{I}_0$ and $m=m_0$.
Furthermore,   $\mc{P}\sub{\mc{I}_0}^{(m_0)}[\smalleq{Q^{R,\delta}(g,\bm\sigma)})](\rho) \in \psi(U)$ for all $\rho \in \Rec\sub{\fuInj}(n,\beta {q^{-3}},\smalleq{Q^{R,\delta}(g,\bm\sigma)})\cap B\sub{{\Gamma}}(\rho_{0},R^{\ga'})$.}

Next, consider
the map $\tilde{\Psi}_{g}^{R,\delta}:\mathbb{N}\times\Si{N(R)}\times \psi(B\sub{\tilde{\Gamma}}(\rho_{0},R^{\ga'}))\to {\psi(U)}\times\Sp(2\aleph)$
\[
\begin{gathered}\tilde{\Psi}_{g}^{R,\delta}(n,{\bm{\sigma}},x):=\Big((\mc{P}\sub{\mc{I}_0}^{(m_0)}[\smalleq{Q^{R,\delta}(g,\bm\sigma)}]_\psi(x)-x\;,\;d_x((\mc{P}\sub{\mc{I}_0}^{(m_0)}[\smalleq{Q^{R,\delta}(g,\bm\sigma)}])_\psi(x)\Big).
\end{gathered}
\]
{and for each $n \in \mathbb{N}$ consider its restriction
$$\tilde{\Psi}_{n}:B\sub{\R^\L}(0,1)\times \psi(B\sub{\Gamma}(\rho_{0},R^{\ga}))\to {\psi(U)}\times\Sp(2\aleph),
\qquad 
\tilde{\Psi}_{n}(\sigma,x):=\tilde{\Psi}_{g}^{R,\delta}\Big(n, (\bm{\sigma}_{j_0}^{-},\sigma,\bm{\sigma}_{j_0}^{+}),x\Big).
$$}
It follows from \eqref{e:rhoUf}, our choice of $\psi$, and the comment following Definition \ref{d:NDq}, that ${\bm A}:=d_x(\psi\circ(\mc{P}\sub{\mc{I}_0}^{(m_0)}[\smalleq{Q^{R,\delta}(g,\bm\sigma)}])\circ\psi^{-1})({\psi(\rho)})$ is not $(\beta, q)$ non-degenerate.

Therefore, Corollary \ref{c:inverse} applied to ${\bm A}$  combined with the first statement in \eqref{e:rhoUf} yield,
that, with $s_{q}:=\beta q^{-\blue{3}}$, $\mc{D}(\tfrac{2\pi p}{q},{s_q})$ as in \eqref{e:degenerate-theta}, and $\Delta\sub{\tilde \Psi_n}$ as in \eqref{e:Delta},
\begin{equation}\label{e:boundS}
\mathcal{S}^{\rho_0,  R,\delta}_{g}(n,  {\alpha,}\beta; j,\hat{\bm{\sigma}})\subset\Delta_{\tilde{\Psi}_n}(X,\cup_{q=1}^\infty Y_{q}),
\end{equation}
$$ 
X:=\psi(B\sub{\Gamma}(\rho_{0},R^{\ga'})), \qquad   Y_{q}:=B\sub{\mathbb{R}^{2\aleph}}(0,\CG s_q)\times \bigcup_{p=0}^{q-1}\mc{D}(\tfrac{2\pi p}{q},{s_q}){\cap B\sub{{\bf A_{0}}}\!(r_0)}\blue{{\cap B\sub{{\bf{0}}}\!(r)}}.
$$
Here, we  let ${\bf A_{0}}:=d_x(\mc{P}\sub{\mc{I}_0}^{(m_0)}[g])_\psi(\blue{\psi(\rho_0)})$, \blue{and note that there is $C_{0}>0$ depending only on $(K, \psi, U)$ such that $\|{\bf A_{0}}\|\leq C_0^n$.}
\blue{We let $r=C_0^n+r_0$ and  $r_0=C_{1}^n(R^{-{\RPower}} \delta +R^{\ga'})$ where $C_1$ is defined as follows. We claim that  there is $C_1>0$ depending only on $(K, \psi, U)$  such that 
\begin{align*}
&\|d_x((\mc{P}\sub{\mc{I}_0}^{(m_0)}[\smalleq{Q^{R,\delta}(g,\bm\sigma)}])_\psi(x)-{\bf A_{0}}\|\\
&\qquad
\leq\|d_x((\mc{P}\sub{\mc{I}_0}^{(m_0)}[\smalleq{Q^{R,\delta}(g,\bm\sigma)}])_\psi(x)-d_x((\mc{P}\sub{\mc{I}_0}^{(m_0)}[g])_\psi(x)\|+\|d_x((\mc{P}\sub{\mc{I}_0}^{(m_0)}[g])_\psi(x)-{\bf A_{0}}\|\\
&\qquad\leq  C_{1}^nR^{-{\RPower}} \delta +C_1^n R^{\ga'}.
\end{align*}} 
\blue{Indeed, the last inequality follows from combining that $\|\varphi_t^g-\varphi_t^{g_0}\|_{C^1}\leq C^t\|g-g_0\|_{C^{2}}$, Definition~\ref{ass:1} with $\nu=\blue{2}$, and the fact that $K$ is bounded in $\mathscr{G}^3$.}

With \eqref{e:boundS} in place we proceed to check that the hypothesis in Proposition \ref{t:Yomdin} hold. Note that by Lemma \ref{l:PoincareDer} there is $C>0$, {depending only on $(K, \psi, U)$}, with  
$
\sup_{x \in \psi(B\sub{\Gamma}(\rho_{0},R^{\ga'}))}\|d_{x}{\tilde{\Psi}_n}(\bm{\sigma}, x)\|\leq C^{n}.
$
In addition,  by~\eqref{e:inj}, there is $c>0$ {depending only on $(K, \psi, U)$}
such that for $\sigma, \tilde\sigma \in B\sub{\R^\L}(0,1)$
\[
\inf_{x\in\psi(B\sub{\Gamma}(\rho_{0},R^{\ga}))}\|\tilde{\Psi}_n(\sigma,x)-\tilde{\Psi}_n(\tilde\sigma,x)\| \geq \delta c^{n+1} \|\sigma- \tilde \sigma\|.
\]
It follows that, {for each $n \in \mathbb N$}, Proposition~\ref{t:Yomdin} applies to {$\tilde{\Psi}_n$
with $k:=\L=\dim({\psi(U)}\times\Sp(2\aleph))$, $\ell:=2\aleph$, $c_{1}=C^{n}$, $c_{2}=\delta c^{n+1}$, and some $L$ depending only on $(\psi, U)$.} Note
that we are abusing notation slightly here to identify $\Sp(2\aleph)$
with $\mathbb{R}^{\aleph(2\aleph+1)}$
in the codomain. 
Finally,  since $\m\sub{\R^\L}(\Delta_{\tilde{\Psi}_n}(X, Y_{q}))\leq c_{\L} \mathcal{M}\sub{\R^\L}(\Delta_{\tilde{\Psi}_n}(X, Y_{q})\,,\, s) \,s^{\L}$ for all $s>0$, 
Proposition~\ref{t:Yomdin} and \eqref{e:boundS} yield that, with $\tilde s_{q}:=2(L\delta^{-1}c^{-n-1}(1+C^{n})+1)s_{q}$,
\begin{equation}\label{e:A}
\m\sub{\R^\L}(\mathcal{S}^{\rho_0, \tilde R,\delta}_{g}(n, {\alpha,} \beta; j,\hat{\bm{\sigma}})) 
\leq \sum_{q=1}^\infty \m\sub{\R^\L}(\Delta_{\tilde{\Psi}_n}(X, Y_{q}))
\leq c_{\L} \sum_{q=1}^\infty \mathcal{M}\sub{\R^{2\aleph}}(X, s_q)\mathcal{M}\sub{\R^{\L}}(Y_{q}, s_q)\,\tilde s_q^{\L}.
\end{equation}

To control the right  hand side of \eqref{e:A} first observe that 
\begin{equation}\label{e:B}
\mc{M}\sub{\mathbb{R}^{2\aleph}}(X,s_q)\leq C\m\sub{\Gamma}(B\sub{\Gamma}(\rho_{0},R^{\ga'}))s_q^{-2\aleph}
\end{equation}
provided {$0<s_q\leq R^{\ga'}$}, where $C>0$  depends only on $(\Gamma, \psi, U)$.
Next, by Corollary~\ref{c:export} with $r=C_0^n+r_0$, there are $C,\e_{0}>0$ such
that if {$0<s_q\leq \blue{\e_{0}(C_{0}^{n}+r_0)}$} then
\begin{equation}\label{e:C}
\mathcal{M}\sub{\R^{\L}}(Y_{q}, s_q)\leq Cq\blue{(C_0^n+r_0)^{\rPower}s_q^{-\L+2\aleph+1}r_0^{\L-4\aleph}}.
\end{equation}

By possibly adjusting \blue{$C>0$ and} $c>0$ one last time, asking   \blue{$\beta\leq \min(\tfrac{\dG}{2},  \blue{R^{\ga'}})$} makes $s_q$  sufficiently small, and so
combining \eqref{e:A}, \eqref{e:B} and \eqref{e:C} yields
\begin{align*}
&\m\sub{\R^\L}(\mathcal{S}^{\rho_0, \tilde R,\delta}_{g}(n, \alpha, \beta; j,\hat{\bm{\sigma}})) \\
&\qquad \leq
\sum_{q=1}^\infty 
C\m\sub{\Gamma}(B\sub{\Gamma}(\rho_{0},R^{\ga'}))s_q^{-2\aleph}
 Cq \blue{(C_0^n+r_0)^{\rPower}s_q^{-\L+2\aleph+1}r_0^{\L-4\aleph}}
\tilde s_q^{\L}\\
&\qquad\blue{=
\sum_{q=1}^\infty 
C\m\sub{\Gamma}(B\sub{\Gamma}(\rho_{0},R^{\ga'}))
 Cq \blue{(C_0^n+r_0)^{\rPower}s_q^{1}r_0^{\L-4\aleph}}(2(L\delta^{-1}c^{-n-1}(1+C^{n})+1))^\L
}\\
&\blue{\qquad\leq C^n \delta^{-\L}\Big(R^{-\RPower}\delta+R^{\ga'}\Big)^{\L-4\aleph}\m\sub{\Gamma}(B\sub{\Gamma}(\rho_{0},R^{\ga'}))\blue{\beta} \sum_{q=1}^\infty q^{-2}}.
\end{align*}
}
This proves the claim in \eqref{e:goal4.5} provided $R, \alpha, \delta, n$ satisfy \eqref{e:conds} and $ \beta<\min(\tfrac{\dG}{2}, \blue{R^{\ga'}})$. 
\end{proof}

\begin{corollary}\label{c:fullPerturb}
Let $b>0$, $\blue{\nu\geq 3}$, $\ga\geq 1,$ and $N:(0,1)\to \mathbb N$. Let
$\{Q^{R,\delta}\}_{R,\delta}$ be a $(\nu, N)$-good family of perturbations and $(\Gamma,\Gb)$ be a $(\fuInj,b,\ga)$-admissible pair  for $\{Q^{R,\delta}\}_{R,\delta}$.

Let $K\subset \Gb$  bounded in \blue{$\ms{G}^{3}$} and $\ga'\geq \ga$.  Let  $\dG$ be as in Lemma~\ref{l:pLives} and \eqref{e:Cdelta}. 
There exist $c\in(0,1)$, $C>0$ such the following holds.
Let   {$n \in \mathbb N$,}   {$\alpha\in(0, c)$,}
\begin{equation}\label{e:condsRdelta}
0<R< \min(c^{n}\alpha, \tfrac{1}{2}c\dG),\qquad  0<\delta\leq c^{n+1}\blue{\min(R^{b}, {R^{\vartheta_1}\alpha ) }}.
\end{equation}
If {$g\in K$,}  then {for all $0<\beta<\min{(\tfrac{1}{2}\dG,\blue{R^{\ga'}})}$}
\[
\m\sub{\Si{N(R)}}(\mathcal S^{R,\delta}_{g}(n,  \alpha, \beta))\leq \blue{C^n \delta^{-\L}\Big(R^{-\RPower}\delta+R^{\ga'}\Big)^{\L-4\aleph }\blue{\beta} },
\]
where 
\begin{align*}
\mathcal S^{R,\delta}_{g}(n,  \alpha, \beta)
&:=\Big\{ {\bm{\sigma}}\in\Si{N(R)}:\; \exists q\in  \mathbb N \; \;\text{s.t.}  \\
&\hspace{0.5cm}\Gamma\cap\Sim\sub{\!\fuInj}(n,\alpha,\smalleq{Q^{R,\delta}(g,\bm\sigma)}) \cap \Rec\sub{\fuInj} \big(n,\beta q^{-3},\smalleq{Q^{R,\delta}(g,\bm\sigma)})\big) \big\backslash \ND_{q,_\fuInj}\big(n,\beta,\smalleq{Q^{R,\delta}(g,\bm\sigma)} \big)\neq \emptyset\Big\}.
\end{align*}
\end{corollary}

\begin{proof}
Let $\{\rho_{i}\}_i\subset \Gamma$ be
an $R^{\ga'}$ maximal separated set in $\Gamma$. 
We claim that 
\begin{equation}\label{e:CClaim}
\mathcal S^{R,\delta}_{g}(n,  {\alpha,} \beta)\subset \mathcal Z^{R,\delta}_{g}(n,  {\alpha,} \beta) \subset \bigcup_{i\in \mc{I}}  \mathcal S^{\rho_i, R,\delta}_{g}(n,  {\alpha,} \beta),
\end{equation}
where $\mc{I}:=\{i: B\sub{\Gamma}(\rho_{i},R^{\ga'})\cap \Sim\sub{\!\fuInj}(n,\tfrac{1}{2}\alpha,g)\cap \Rec\sub{\fuInj}(n,c,g) \neq \emptyset\}$
and 
\begin{align*}
\mathcal Z^{R,\delta}_{g}(n,  {\alpha,} \beta)
&:=\Big\{ {\bm{\sigma}}\in\Si{N(R)}:\; \exists q\in  \mathbb N \; \;\text{s.t.}  \\
&\hspace{1.6cm}\Gamma\cap\Sim\sub{\!\fuInj}(n,\tfrac{1}{2}\alpha,g)\cap \Rec\sub{\fuInj} \big(n,\beta q^{-3},\smalleq{Q^{R,\delta}(g,\bm\sigma)})\big) \big\backslash \ND_{q,_\fuInj}\big(n,\beta,\smalleq{Q^{R,\delta}(g,\bm\sigma)} \big)\neq \emptyset\Big\}.
\end{align*}
The claim in \eqref{e:CClaim} follows from observing that 
$$
\Sim\sub{\!\fuInj}(n,\alpha,Q^{R,\delta}(g,\bm\sigma))\subset \Sim\sub{\!\fuInj}(n,\tfrac{1}{2}\alpha,g),
\qquad  \Rec\sub{\fuInj}(n,\beta q^{-3},Q^{R,\delta}(g,\bm\sigma))\subset\Rec\sub{\fuInj}(n,c,g).$$
Indeed, these two inclusions follow from combining the conclusion from Lemma \ref{l:PoincareDer} that for all $t\in \R$
$$d(\varphi_t^g(\rho), \varphi_t^{Q^{R,\delta}(g,\bm\sigma)}(\rho)  )\leq C^{|t|} \|g-Q^{R,\delta}(g,\bm\sigma)\|_{\blue{\mathcal{C}^1}}\blue{\leq C^n\delta R^{-\vartheta_1}}$$ 
 with
~\eqref{e:Qzero},~\eqref{e:changeSigma}, and the requirement that $\delta \blue{R^{-\vartheta_1}}\leq c^{n}\alpha$.

Finally, combining \eqref{e:CClaim} with Lemma~\ref{l:lemma4.5} applied to each ball $B\sub{\tilde\Gamma}(\rho_{i},R^{\ga'})$ with $i \in \mc{I}$ yields the desired bound.
\end{proof}


\section{Induction on the length of orbits}\label{s:theMadness}


{Let $b>0$. Let $\nu>0$, $N:(0,1)\to \mathbb N$,
 $\{Q^{R,\delta}\}_{R,\delta}$ be a $(\nu, N)$-good family of perturbations, and let $(\Gamma,\Gb)$ be a $(\fuInj,b,\ga)$-admissible pair  for $\{Q^{R,\delta}\}_{R,\delta}$.} In this section we will prove that there is a predominant set, $G$ of metrics such that {for all $g$ in $G$, there is $C>0$ such that}
\begin{equation}
\label{e:nonDegGoal}
\Gamma \cap\Rec\sub{\fuInj}(n,\beta_n,g)\subset \ND\sub{\fuInj}(n,\beta_n,g),\qquad \beta_n= C^{-Cn^{\someLetter_\nu}},
\end{equation}
We do this using a family of probing maps generated by the $(\nu,N)$-good family of perturbations. Once this is done (see Proposition~\ref{p:thePredominantMeat}), to prove Theorem \ref{t:predominantR-ND} it will remain to show that there are $(\nu,N)$-good families of perturbations (we do this in Section~\ref{s:perturbedMetrics}). We continue to work in the setting of Section~\ref{s:basicPerturb}.

\subsection{Construction of families of probing maps}
\label{s:theProbes}

Let $\nu\geq 0$, a function $N:(0,1)\to \mathbb N$,
 and let $\{Q^{R,\delta}\}_{R,\delta}$ be a $(\nu, N)$-good family of perturbations. 
Let $\{K_s\}_{s\in (0,1)} \subset \ms G^\nu$ be a family of nested {closed} bounded subsets $K_{s_1}\subset K_{s_2}$ for $s_1>s_2$ such that $\ms G^\nu=\cup_{s\in (0,1)}K_s$.
By Definition \ref{ass:1} for every $s\in (0,1)$ bounded there is $C_s>0$ such that for all $\delta>0$, $R>0$,
\begin{align}
&\|Q^{R,\delta}(g,\bm\sigma)-Q^{R,\delta}(g,\tilde{\bm\sigma})\|\sub{\mc{C}^\nu}\leq C_s\delta R^{-\vartheta_\nu}\|\bm{\sigma}-\tilde {\bm\sigma}\|_{\ell^{\infty}},\label{e:Qdiff}\\
&\|Q^{R,\delta}(\tilde{g},\bm\sigma)-Q^{R,\delta}(g,\bm\sigma\big)\|\sub{\mc{C}^{\nu-1}}
\leq (1+C_s\delta {R^{-\vartheta_\nu}})\|g-\tilde{g}\|\sub{\mc{C}^{\nu-1}}.\label{e:changeKappa}
\end{align} 
 for  $ \bm \sigma, \tilde {\bm \sigma} \in \Si{N(R)},\,g\in K_s.$
Next, for $\e>0$ let $s(\e)>0$ be an increasing function such that $\lim_{\e \to 0^+}C_{s(\e)}\e=0$ and $C_{s(\e)}\e$ is increasing. Next, define the closed and bounded sets
\begin{equation}
\label{e:fancyPants}
\ms{G}_\e^\nu:=\{g \in K_{s(\e)}: \; d(g, \ms G^{\nu}\backslash K_{s(\e)})\geq 2C_{s(\e)}\e\}.
\end{equation}
Note that $\ms{G}_{\e_2}^\nu\subset \ms{G}_{\e_1}^\nu$ for $\e_1<\e_2$ and $\ms{G}^\nu=\cup_{\e>0}\ms{G}_\e^\nu$.

{For $\e\in (0,1)$ let  $\bm\delta_\ep=\{\delta_j(\e)\}_{j=0}^{\infty}\subset {[}0,1)$ and $\Rind_\e=\{{R_j(\e)}\}_{j=0}^{\infty}\subset (0,1)$, satisfy 
\begin{equation}
\sum_{j=0}^{\infty}\delta_{j}(\e)(\Rtemp_{j}(\e))^{-\vartheta_\nu}\leq \e.\label{e:probeControl}
\end{equation}
We now construct a family of probing maps associated to the perturbation maps $Q^{R_j(\e),\delta_j(\e)}$.
Given $J=0, \dots, N$ define the spaces
\begin{equation}\label{e:SigmaNotation}
\bm \Sigma\sub{J}(\bm{\Rind}_\e):=\prod_{j=0}^J \Si{N(\Rtemp_j(\e))}, \qquad \qquad  \bm{\Sigma}\sub{\infty}(\bm{\Rind}_\e):=\prod_{j=0}^{\infty}\Si{N(\Rtemp_{j}(\e))}.
\end{equation}
Also, given a sequence  ${\bm \sigma}=(\bm \sigma\sub{0},\bm \sigma\sub{2},\bm \sigma\sub{3}, \dots)\in \bm{\Sigma}\sub{\infty}(\bm{\Rind}_\e)$  we write  $\hat{\bm \sigma}\sub{\!J}:=(\bm\sigma\sub{\!0},\dots,\bm \sigma\sub{\!J}) \in \bm{\Sigma}\sub{J}(\bm{\Rind}_\e)$.

We are now ready to define the family of probing maps (see Definition \ref{e:probingMaps}).
We start by letting  $F^{\Rind,\bm\delta}\sub{\!0}:\ms{G}_\e^\nu\times \Si{N(\Rtemp_{0})}\to\ms{G}^{\nu}$ be the map
\begin{equation}\label{e:F_0}
F^{\Rind_\e,\bm\delta_\e}\sub{\!0}(g,{\bm{\sigma}}):=Q^{\Rtemp_{0}(\e),\delta_{0}(\e)}(g,\bm{\sigma})
\end{equation}
and $F^{\Rind_\e,\bm\delta_\e}\sub{\!J}:\ms{G}^{\nu}_\e\times\bm\Sigma\sub{J}(\bm \Rind_\e)\to\ms{G}^{\nu}$
by 
\begin{equation}
\label{e:F_J}
F^{\Rind_\e,\bm\delta_\e}\sub{\!J}(g,\hat{\bm{\sigma}}\sub{\!J}):=Q^{\Rtemp_{J}(\e),\delta_{J}(\e)}\big(F^{\Rind_\e,\bm\delta_\e}\sub{\!J-1}(g, \hat{\bm{\sigma}}\sub{\!J-1}),{\bm{\sigma}}\sub{\!J}\big).
\end{equation}
Finally, we define 
\begin{equation}
F^{\Rind_\e,\bm\delta_\e}\sub{\!\infty}:\ms{G}^\nu_\e\times\bm{\Sigma}\sub{\infty}(\bm{\Rind}_\e)\to\ms{G}^{\nu}, \qquad \qquad F^{\Rind_\e,\bm\delta_\e}\sub{\!\infty}(g,{\bm{\sigma}})=\lim_{J\to\infty}{F^{\Rind_\e,\bm\delta_\e}\sub{\!J}(g,\hat{\bm{\sigma}}\sub{\!J})}.\label{e:F_infty}
\end{equation}
}
Note that, apriori, the limit in~\eqref{e:F_infty} may not exist, but we prove in Lemma~\ref{l:probing} that it does under our assumptions on $\Rind_\e$ and $\bm\delta_\e$.
When it will not lead to confusion, we will omit the $\Rind,\bm\delta$ from the notation for $F^{\Rind,\bm\delta}\sub{J}$.

When checking the probing maps definition we work with  $\ms{G}:=\ms{G}^\nu$ and $\ms{G}':=\ms{G}^{\nu-1}$, while $\ms{B}$ and $\ms{B}'$ will be the spaces of $\mc{C}^\nu$ and $\mc{C}^{\nu-1}$ symmetric two-tensors.

\begin{lemma}
\label{l:probing}
Let $\nu\geq 0$,  $N:(0,1)\to \mathbb{N}$ and $\{Q^{R,\delta}\}_{R,\delta}$ be a $(\nu, N)$-good family of perturbations. Then for all  $\bm\delta_\ep=\{\delta_j(\e)\}_{j=0}^{\infty}\subset (0,1)$ and $\Rind_\e=\{{R_j(\e)}\}_{j=0}^{\infty}\subset (0,1)$ satisfying,
\eqref{e:probeControl}.
 $\{(F^{\Rind_{\e},\bm\delta_{\e}}\sub{\infty},\infty)\}_{\e>0}$ is a  $\ms{G}^{\nu-1}$-family of probing maps for $\ms{G}^\nu$. 
 
 Moreover, for all $K\subset \ms{G}^\nu$ bounded and $\tilde{K}\subset \ms{G}^\nu$ bounded with $d(\tilde{K}^c,K)>0$, there is $\e_1>0$ such that for all $0<\e<\e_1$ and $\bm\delta_\ep$ and $\Rind_\e$ satisfying
\eqref{e:probeControl}, $F\sub{J}^{\Rind_\e,\bm \delta_\e}(K \times \bm \Sigma\sub{J}(\Rind_\e))\subset \tilde{K}$ for all $J\in \mathbb{N}\cup\{\infty\}.$
\end{lemma}

\begin{proof}
Let $0<\e<1$ and assume~\eqref{e:probeControl} is satisfied. 
We first show that $F\sub{\!\infty}^{\Rind_{\e},\bm\delta_{\e}}$
is well defined. 

{Given $g \in \ms G^\nu_\e$ we claim that $F\sub{\!J}^{\Rind_{\e},\bm\delta_{\e}}(g,\hat{\bm{\sigma}}\sub{\!J})\in K_{s(\e)}$ for all $J$ and $\hat{\bm{\sigma}}\sub{J} \in \bm\Sigma\sub{J}(\bm \Rind_\e)$. Indeed, we claim 
\begin{equation}
\label{e:soFarAway}
d(F\sub{\!J}^{\Rind_{\e},\bm\delta_{\e}}(g,\hat{\bm{\sigma}}\sub{\!J}), \ms G^{\nu}\backslash K_{s(\e)})\geq C_{s(\e)}\Big(2\e-\sum_{j=0}^J \delta\sub{J}(\e)R\sub{J}(\e)^{-\vartheta_\nu}\Big).
\end{equation}
Note that when $J=0$ this follows from the definition of $\ms G^\nu_\e$ since \eqref{e:Qdiff} and \eqref{e:Qzero} yield $\|F^{\Rind_\e,\bm\delta_\e}\sub{\!0}(g,{\bm{\sigma}})-g \|_{\mc{C}^\nu}\leq C_{s(\e)} \delta\sub{0}(\e)R\sub{0}(\e)^{-\vartheta_\nu}$. Next, assume the claim holds for all $j \leq J-1$. Then, again by \eqref{e:Qdiff}  we obtain 
\begin{align} \label{e:cauchy}
 \|F\sub{\!J}^{\Rind_{\e},\bm\delta_{\e}}(g,\hat{\bm{\sigma}}\sub{\!J})-F\sub{\!J-1}^{\Rind_{\e},\bm\delta_{\e}}(g,\hat{\bm{\sigma}}\sub{\!J-1})\|_{\mc{C}^\nu}
 & =\big\|Q^{\Rtemp_{J}(\e),\delta_{J}(\e)}(F\sub{\!J-1}^{\Rind_{\e},\bm\delta_{\e}}(g, \hat{\bm{\sigma}}\sub{\!J-1}),\bm{\sigma}\sub{\!J})-F\sub{\!J-1}^{\Rind_{\e},\bm\delta_{\e}}(g,\hat{\bm{\sigma}}\sub{\!J-1})\big\|_{\mc{C}^{\nu}} \notag\\
 & \leq C_{s(\e)}\delta\sub{J}(\e)R\sub{J}(\e)^{-\vartheta_\nu}\|\bm{\sigma}\sub{\!J}\|_{\ell^{\infty}}. 
\end{align}}
By ~\eqref{e:cauchy} and \eqref{e:probeControl} the claim in \eqref{e:soFarAway} holds and so $F\sub{\!J}^{\Rind_{\e},\bm\delta_{\e}}(g,\hat{\bm{\sigma}}\sub{\!J})\in K_{s(\e)}$. Furthermore, ~\eqref{e:cauchy} and \eqref{e:probeControl}  guarantee that  $F\sub{\!\infty}^{\Rind_{\e},\bm\delta_{\e}}(g,\bm{\sigma})$ is well defined and $F\sub{\!\infty}^{\Rind_{\e},\bm\delta_{\e}}(g,\bm{\sigma})\in\ms{G}^{\nu}$. 
{Also, since $F\sub{J}^{\Rind_{\e},\bm\delta_{\e}}$ is continuous, so is $F\sub{\infty}^{\Rind_{\e},\bm\delta_{\e}}$.

Note that~\eqref{e:cauchy} implies the second statement in the Lemma and we need only check that $F_\infty ^{\Rind_\e,\bm\delta_\e}$ is a family of probing maps.

Next, note that \eqref{e:Qzero} implies $F\sub{\!\infty}^{\Rind_{\e},\bm\delta_{\e}}(g,0)=g$ for all $g\in \ms{G}^\nu_\e$  so \eqref{e:typicalness} holds.}
To check that \eqref{e:closeToId}, \eqref{e:bddDiff}, \eqref{e:bddDiff0} hold, let $K \subset \ms{G}^{\nu}$ be bounded. Then, there is $\e_0>0$ such that $K \subset  \ms{G}^{\nu}_\e$ for all $0<\e<\e_0$.  In particular,  \eqref{e:cauchy} implies
that for all   $0<\e<\e_0$, $g\in K$, and $\bm \sigma \in \bm \Sigma_\infty(\bm \Rind_\e)$,
\[
\|F\sub{\!\infty}^{\Rind_{\e},\bm\delta_{\e}}(g,\bm{\sigma})-g\|_{\mc{C}^{\nu}}\leq C_{s(\e)}\sum_{J=1}^\infty\delta\sub{J}R\sub{J}^{-\vartheta_\nu}\|\bm{\sigma}\sub{\!J}\|_{\ell^{\infty}}\leq C_{s(\e)}\e\|\bm{\sigma}\|_{\ell^{\infty}}.
\]
Thus, since $\lim_{t\to 0^+}{C_{s(\e)}}\e=0$,~\eqref{e:closeToId} holds.


We next prove that ${F}\sub\infty^{\Rind_{\e},\bm\delta_{\e}}$ satisfies~\eqref{e:bddDiff}. 
To ease notation, from now on we write $F\sub{J}$ in place of ${F}\sub{J}^{\Rind_{\e},\bm\delta_{\e}}$ and $\delta\sub{J}=\delta\sub{J}(\e)$, $R\sub{J}=R\sub{J}(\e)$. 
 First, note that \eqref{e:Qdiff} implies that for $g \in K_{t(\e)}$ and $\bm \sigma,\bm \mu\in \bm\Sigma_\infty(\bm \Rind)$
$$
\|F\sub0(g,\hat{\bm\sigma}_0)-F\sub{0}(g,\hat{\bm\mu}_0)\|\sub{\mc{C}^{\nu-1}}\leq C_{s(\e)}\delta_0\Rtemp_0^{-\vartheta_\blue{\nu-1}}\|\bm \sigma_0-\bm \mu_0\|_{\ell^\infty}.
$$
{Our induction hypothesis is} that for some $J\geq 1$ {and all $g \in {\ms{G}_\e^\nu}$, $\bm \sigma,\bm \mu\in \bm\Sigma_\infty(\Rind)$,}
\begin{multline*}
\|F\sub{\!J-1}(g, \hat{\bm \sigma}\sub{J-1})-F\sub{\!J-1}(g,\hat{\bm \mu}\sub{J-1})\|\sub{\mc{C}^{\nu-1}}\\
\leq C_{s(\e)}R\sub{J-1}^{-\vartheta_{\nu-1}}\delta\sub{J-1}\|\bm \sigma\sub{J-1}-\bm \mu\sub{J-1}\|_{\ell^\infty}+ \sum_{j=0}^{J-2}C_{s(\e)}\blue{\Big[}\prod_{k=j+1}^{J-1}(1+C_{s(\e)}\delta_k\Rtemp_k^{-\vartheta_\nu})\blue{\Big]}\Rtemp_j^{-\vartheta_{\nu-1}}\delta_j\|\bm \sigma_j-\bm \mu_j\|_{\ell^\infty}.
\end{multline*}
Next, note that by~\eqref{e:Qdiff}, 
$$
\|F\sub{\!J}(g, \hat{\bm \sigma}\sub{J})-F\sub{\!J}(g, (\hat{\bm \sigma}\sub{J-1},\bm\mu\sub{J}))\|\sub{\mc{C}^{\nu-1}}\leq C_{s(\e)}R\sub{J}^{\blue{-}\vartheta_{\nu-1}}\delta\sub{J}\|\bm \sigma\sub{J}-\bm \mu\sub{J}\|_{\ell^\infty}.
$$
In addition,  by~\eqref{e:changeKappa} and the induction hypothesis
$$
\begin{aligned}
\|F\sub{\!J}(g, \hat{\bm \mu}\sub{J})-F\sub{\!J}(g, (\hat{\bm \sigma}\sub{J-1},\bm\mu\sub{J}))\|\sub{\mc{C}^{\nu-1}}&\leq  (1+C_{s(\e)}\delta\sub{J}R\sub{J}^{-\vartheta_\nu})\|F\sub{\!J-1}(g,\hat{\bm\sigma}\sub{J-1})-F\sub{\!J-1}(g,\hat{\bm\mu}\sub{J-1})\|\sub{\mc{C}^{\nu-1}}\\
&\leq \sum_{j=0}^{J-1} \blue{\Big[}\prod_{k=j+1}^JC_{s(\e)}(1+C_{s(\e)}\delta_k\Rtemp_k^{-\vartheta_\nu})\blue{\Big]}R_j^{-\vartheta_{\nu-1}}\delta_j\|\bm\sigma_j-\bm\mu_j\|_{\ell^\infty}.
\end{aligned}
$$
Therefore, {for all $J=0,1, \dots, $ $g \in{ \ms{G}_{\e}^\nu}$, and $\bm \sigma,\bm \mu\in \bm\Sigma_\infty(\Rind)$,}
\begin{align*}
&\|F\sub{\!J}(g, \hat{\bm \sigma}\sub{J})-F\sub{\!J}(g, \hat{\bm \mu}\sub{J})\|\sub{\mc{C}^{\nu-1}}\\
&\qquad\qquad
\leq \|F\sub{\!J}(g, \hat{\bm \sigma}\sub{J})-F\sub{\!J}(\kappa, (\hat{\bm \sigma}\sub{J-1},\bm\mu\sub{J})\|\sub{\mc{C}^{\nu-1}}
+\|F\sub{\!J}(g, \hat{\bm \mu}\sub{J})-F\sub{\!J}(g, (\hat{\bm \sigma}\sub{J-1},\bm\mu\sub{J})\|\sub{\mc{C}^{\nu-1}}\\
&\qquad\qquad\leq C_{s(\e)}\Rtemp\sub{J}^{-\vartheta_{\nu-1}}\delta\sub{J}\|\bm \sigma\sub{J}-\bm \mu\sub{J}\|_{\ell^\infty}+ \sum_{j=0}^{J-1}C_{s(\e)}\blue{\Big[}\prod_{k=j+1}^{J}(1+C_{s(\e)}\delta_k\Rtemp_k^{-\vartheta_\nu})\blue{\Big]}\Rtemp_j^{-\vartheta_{\nu-1}}\delta_j\|\bm \sigma_j-\bm \mu_j\|_{\ell^\infty}.
\end{align*}
Sending $J\to \infty$, we conclude that the inequality in~\eqref{e:bddDiff} holds since $\lim_{t\to 0^+}{C_{s(\e)}}\e=0$ and
$$
\begin{aligned}\|F\sub{\!\infty}(g, \bm \sigma)-F\sub{\!\infty}(g, \bm \mu)\|\sub{\mc{C}^{\nu-1}}&\leq  \sum_{j=0}^{\infty}C_{s(\e)}\blue{\Big[}\prod_{k=j+1}^{\infty}(1+C_{s(\e)}\delta_k\Rtemp_k^{-\vartheta_\nu})\blue{\Big]}\Rtemp_j^{-\vartheta_{\nu-1}}\delta_j\|\bm \sigma_j-\bm \mu_j\|_{\ell^\infty}\\
&\leq  C_{s(\e)}\e\exp\big(C_{s(\e)}\e){\|\bm \sigma-\bm \mu\|_{\ell^\infty}}.
\end{aligned}
$$

The fact that $\tilde{F}\sub\infty^{\Rind_{\e},\bm\delta_{\e}}=\iota \circ {F}\sub\infty^{\Rind_{\e},\bm\delta_{\e}} :\ms{G}^\nu\times\bm{\Sigma}\sub{\infty}(\bm{\Rind})\to\ms{G}^{\nu-1}$ is Lipschitz {in both $\bm\sigma$ and }$g$  follows after checking by induction that for all $\bm \sigma \in \bm{\Sigma}\sub{\infty}(\bm{\Rind})$
$$
\|F\sub\infty(g,\bm \sigma)-F\sub\infty(\tilde g,\bm \sigma)\|\sub{\mc{C}^{\nu-1}}\leq \blue{\Big[}\prod_{j=0}^{\infty}(1+C_{s(\e)}\delta_j\Rtemp_j^{-\vartheta_\nu})\blue{\Big]} \|g-\tilde g\|\sub{\mc{C}^{\nu-1}}\leq \exp(C_{s(\e)}\e)\|g-\tilde g\|_{\mc{C}^{\nu-1}}.
$$
We proceed to prove \eqref{e:bddDiff0}. We claim that
\begin{equation}
\label{e:derLim}
D_{g}F\sub{\infty}(g,\bm\sigma)=\lim_{J\to \infty} H\sub{J}(g, \hat{\bm\sigma}\sub{J}) \dots H\sub{1}(g, \hat{\bm\sigma}\sub{1})H\sub{0}(g, \hat{\bm\sigma}\sub{0}),
\end{equation}
{where we abbreviate $H\sub{J}(g, \hat{\bm\sigma}\sub{J}):=D_g Q^{\Rtemp_J,\delta_J}(F\sub{J-1}(g,\hat{\bm\sigma}\sub{J-1}),\bm\sigma\sub{J})$ for $J\geq 1$ and set $H\sub{0}(g, \hat{\bm\sigma}\sub{0})=D_g Q^{\Rtemp_0,\delta_0}(g,\hat{\bm\sigma}_0)$.} 
Indeed, observe that for $J>K$
\begin{align}\label{e:cauchyD}
D_{g}F\sub{J}(g,\hat{\bm\sigma}\sub{J})-D_g F\sub{K}(g,\hat{\bm\sigma}\sub{K})
&=\Big(H\sub{J}(g, \hat{\bm\sigma}\sub{J})\dots H\sub{K+1}(g, \hat{\bm\sigma}\sub{K+1})-\Id\Big)D_g F\sub{K}(g,\hat{\bm\sigma}\sub{K}).
\end{align}
Now, fix $\e_0>0$ and choose $K>0$ such that 
$
{C_{s(\e)}}\sum_{j=K+1}^\infty \delta_jR_j^{-\vartheta_\nu}<\e_0.
$
Then, {using~\eqref{e:derivative}}, 
\begin{align}
&\|H\sub{J}(g, \hat{\bm\sigma}\sub{J})\dots H\sub{K+1}(g, \hat{\bm\sigma}\sub{K+1})-\Id\|\sub{\mc{C}^{\nu-1}\to \mc{C}^{\nu-1}} \notag\\
&\qquad\qquad\leq \sum_{j=K+1}^J\big\|H\sub{J}(g, \hat{\bm\sigma}\sub{J})\dots H\sub{j+1}(g, \hat{\bm\sigma}\sub{j+1})\Big(H\sub{j}(g, \hat{\bm\sigma}\sub{j})-\Id\Big)\big\|\sub{\mc{C}^{\nu-1}\to \mc{C}^{\nu-1}}\notag\\
&\qquad\qquad\leq \sum_{j=K+1}^J\blue{\Big[}\prod_{\ell=j+1}^J(1+C_{s(\e)}\delta_\ell\Rtemp_{\ell}^{-\vartheta_\nu})\blue{\Big]}C_{s(\e)}\delta_j\Rtemp_j^{-\vartheta_\nu}\leq \exp(C_{s(\e)} \e)\e_0. \label{e:nearID}
\end{align}
{Therefore, by \eqref{e:cauchyD}, the sequence $\{D_{g}F\sub{J}(g,\hat{\bm\sigma}\sub{J})\}_J$ is Cauchy.
This implies
that the limit in \eqref{e:derLim} exists} and the convergence is uniform. Since $F\sub{J}\to F\sub{\infty}$, this implies the equality~\eqref{e:derLim}. Moreover, setting $K=-1$ in~\eqref{e:nearID} and combining it with \eqref{e:derLim} we see that 
$$
\|D_gF\sub{\infty}-I\|\sub{\mc{C}^{\nu-1}\to \mc{C}^{\nu-1}}\leq C_{s(\e)}\exp(C_{s(\e)}\e)\sum_{j=0}^\infty \delta_jR_j^{-\vartheta_\nu}\leq C_{s(\e)}\e\exp(C_{s(\e)}\e)\to 0,
$$
and hence the estimate in~\eqref{e:bddDiff0} holds.

\end{proof}

\subsection{Induction on the length of orbits}

In this section, we construct the probing family for which~\eqref{e:nonDegGoal} holds for a  predominant set of metrics. The idea is to show that for any $1<\step<2$, given non-degeneracy of returning points at time $n$ for all $n\leq \ell$ {(together with some additional conditions on the eigenvalues of Poincar\'e maps)}, most perturbations have non-degeneracy of returning points for times $n$ with $ n\leq \step \ell$. An outline is given in Section \ref{s:theInduction}.

We start by obtaining non-degeneracy of orbits of length 1. In what follows, given $\delta_0(\e)$ and $R_0(\e)$ we write $F_0(g,\bm{\sigma}):=Q^{\Rtemp_{0}(\e),\delta_{0}(\e)}(g,\bm{\sigma})$ as in \eqref{e:F_0}.

\begin{lemma}\label{l:baseCase}
Let \blue{$\nu\geq 2$} and $b>0$, a function $N:(0,1)\to \mathbb N$,
 and a $(\nu, N)$-good family of perturbations $\{Q^{R,\delta}\}_{R,\delta}$.
For all $\e>0$,  there exist positive constants $\delta_0, \Rtemp_{0},\beta_{0,0}$  such that $\delta_0R_0^{-\vartheta_\nu}<\tfrac{1}{2}\e$ and the following holds. 
For every $\fuInj>0$ and $(\Gamma,\Gb)$ that is a $(\fuInj,b,\ga)$-admissible pair for $\{Q^{R,\delta}\}_{R,\delta}$,  and all
$K\subset\Gb$ bounded in $\ms{G}^\nu$, there are $C>0$ and $\e_0>0$ such that for $0<\e<\e_0$
\begin{equation}
\sup_{\bm{\sigma}\in\Sigma(\Rtemp_{0})}\|F_{0}(g,\bm{\sigma})-g\|_{\mc{C}^{\nu}}<\tfrac{1}{{4}}{C}\e,\qquad\text{for all }g\in K,\label{e:close1}
\end{equation}
and
\begin{equation}
m\sub{\bm{\Sigma}(\Rtemp_{0})}\big( \bm \sigma \in \bm{\Sigma}(\Rtemp_{0}):\; F_{0}(g, \bm \sigma) \in  \tilde L\sub{\fuInj,0}\big)\leq\tfrac{1}{2}\e,\qquad\text{for all }g\in K\label{e:measureNonDeg1}
\end{equation}
where the set $ \tilde L\sub{\fuInj,0}$ is defined by
\begin{equation}
\begin{gathered}
\tilde L\sub{\fuInj,0}:=\big\{g\in \ms{G}^{\nu}:\;\; \exists q\in\mathbb{N} \text{\;s.t.\;} \;\;\Gamma\cap\Rec\sub{\fuInj}(1,\beta_{0,0,0}q^{-3}, g)\backslash \ND_{q,_\fuInj}(1,\beta_{0,0,0},g)\neq \emptyset \big\},
\end{gathered}
\label{e:nonDeg1}
\end{equation}
with $\beta_{0,0,0}:=2\beta_{0,0}$.
\end{lemma} 

\begin{proof}
Let \blue{$\RPower$} be as in Definition \ref{ass:1}, and set
$$R_{0}=\tfrac{1}{2}\e, \qquad  
\delta_{0}={\min\big(\e^{2}\Rtemp_0^{\blue{b}},\,\e\Rtemp_0^{\blue{\vartheta_\nu}}\big)},  \qquad 
\beta_{0,0}
=\tfrac{1}{4}{R_0^{\ga+1}}\delta_0^{\L}.
$$ 

Next, fix $K\subset \Gb$ bounded in $\ms{G}^\nu$. By Definition~\ref{ass:1} there is $C_0>0$ such that 
$\|g-F_{0}(g,\bm{\sigma})\|_{\mc{C}^{\nu}}\leq C_0\delta_{0}\Rtemp_{0}^{-\vartheta_\nu}\leq \tfrac{1}{2}{C_0\e}$ for ${g\in K}$
implying~\eqref{e:close1}. 
Let  $c, C$ be as in Corollary~\ref{c:fullPerturb} (with $\ga'=\ga$). \blue{Shrinking $c$ if necessary, we may assume $c<\delta_G$ and set $\alpha=c/2$.}  Note that  the hypothesis of Corollary~\ref{c:fullPerturb} for $n=1$ are satisfied with $\delta:=\delta_0$ and $R:=R_0$ provided
 $\e<c$.
 Note that since $\Sim\sub{\!\fuInj}(1, \alpha, g)=\widetilde{S^*\!M} \supset\Gamma$ for all $g\in \Gb$ and
 $\alpha>0$, then
 $$
\big\{ \bm \sigma \in \bm{\Sigma}(\Rtemp_{0}):\; F_{0}(g, \bm \sigma) \in  \tilde L\sub{\fuInj,0}\big\}=\mathcal S^{R_0,\delta_0}_{g}(1, \alpha, 2\beta_{0,0}).
 $$
  Finally, {since $\ga\geq 1$,}  if 
  the conclusion follows from Corollary~\ref{c:fullPerturb} after asking $\e=2R_0< \min(c^2, c \dG, C^{-1})$.

\end{proof}

Now that we have obtained non-degeneracy for returning points with orbits of length $1$, we need to induct on the length of the returning orbit. In order to do this, we work with lengths between $\step^\ell\fuInj$ and $\step^{\ell+1}\fuInj$ for some $1<\step<2$ and induct on $\ell$.  Before stating our lemma, we set 
\begin{equation}
\label{e:omega}
\begin{gathered}
\gap:=1-\log_{2}\step \;\in (0,1),
\end{gathered}
\end{equation}
and for $j\geq0$ and $\e>0$ set 
\begin{equation}
\gamma_{j}(\e):=5^{-2\aleph}{\e}^{(2\aleph+1)\step^{(j+1)/\gap}}\step^{-4\aleph(j+1)/\gap}\label{e:gamma_j}.
\end{equation}

We are now ready to do the induction on the length of the orbits. At step $\ell$ in the induction, we study returning trajectories of length $t\leq \fuInj\step^{\ell+1}$, in particular, making perturbations so that the returning trajectories with $(\step^{\ell}-1)\fuInj<t\leq \step^{\ell+1}\fuInj$ are non-degenerate \emph{and} so that the non-degeneracy of shorter trajectories is maintained. See the outline of the proof in Section~\ref{s:outline} for a more detailed explanation of the argument. In what follows, we continue to use the notation $\bm{\Sigma}_{\ell}(\bm{\Rind})$  introduced in \eqref{e:SigmaNotation}.

\begin{proposition}\label{p:theInduction}
Let \blue{$\nu\geq 3$}, $\ga\geq 1$, $b>0$, $N:(0,1)\to \mathbb{N}$,
 and $\{Q^{R,\delta}\}_{R,\delta}$ be a $(\nu, N)$-good family of perturbations.
There exists $\bm{d}>0$ and for all $\e>0$ there are 
$\bm\delta_\e:=\{\delta_{\ell}\}_{\ell=0}^{\infty}$,
$\bm{\Rind}_\e:=\{\Rtemp_{\ell}\}_{\ell=0}^{\infty}$, and $\beta_{0,0}>0$,  such that {$\delta_{\ell}\Rtemp_{\ell}^{-\vartheta_\nu}<2^{-\ell-1}\e$}
and  the following holds.   For every $\fuInj>0$, $(\Gamma,\Gb)$ that is a $(\fuInj,b,\ga)$-admissible pair for $\{Q^{R,\delta}\}_{R,\delta}$,   and  $K\subset\Gb$ bounded in $\ms{G}^\nu$, there are $C>0$ and $\e_0>0$ such that  for $0<\e<\e_0$ and all $\ell=0, 1, \dots$
\begin{equation}
\sup_{\bm{\sigma}\in\bm{\Sigma}_{\ell+1}(\bm{\Rind})}\|F_{\ell+1}(g,\bm{\sigma})-F_{\ell}(g,\widehat{\bm{\sigma}}_{\ell})\|_{\mc{C}^{\nu}}\leq {C}2^{-\ell-2}\e,\qquad g\in K.\label{e:D}
\end{equation}
In addition, for $0\leq i\leq j\leq \ell$, there exist constants $\beta_{i,j}\in(0,1]$,
with $\beta_{i+1,j}\leq\beta_{i,j}$, 
\begin{equation}\label{e:betas}
\beta_{j+1,j+1}= \beta_{j,j}^{\mathfrak{c} }\e^{\bm{d}\step ^{j/\gap}},
\end{equation}
with $\mathfrak{c}:=2\aleph\max\Big(\aleph(\blue{4 m_\nu +\RPower(2\aleph-1))}\;,\,\ga\Big)$
and
\blue{$m\sub{\nu}:=\max(b,\vartheta_\nu,\vartheta_{2}+1)$} 
such that the following holds.
Let $\alpha_0={4\e^{-1}}$  and for $0\leq j\leq \ell$ set
\begin{equation}\label{e:alphas}
\begin{gathered}
\alpha_{j+1}:=\gamma_{j}\beta_{j,j}^{2\aleph},\qquad \alpha_{j,\ell}:=(1-2^{-{(\ell-j)}-1})\alpha_{j},\qquad
\beta_{i,j,\ell}:=(1+2^{-{(\ell-j)}})\beta_{i,j},\\
\tilde{\beta}_{j,\ell+1}:=s_{\ell}\tilde \beta_{j,\ell},\qquad\qquad {\tilde\beta_{j,j}}:=\beta_{j,j}^{2\aleph}\e^{2\aleph-1},\qquad \qquad s_{\ell}:=\tfrac{1+2^{-\ell-1}}{1+2^{-\ell}}
\end{gathered}
\end{equation}
with $\gamma_j$ as in~\eqref{e:gamma_j}.
Then for all $g\in K$, $\ell\geq0$, {and $0<\e<\e_0$}
\begin{equation}
m\sub{\bm{\Sigma}_{\infty}(\bm{\Rind})}\Big( \bm \sigma\in \bm{\Sigma}_{\infty}(\Rind):\; \exists \, j\leq \ell\;\; \text{s.t.}\;\;\,F_{j}(g, \widehat{\bm{\sigma}}_j)\in  L\sub{\fuInj,j}\Big)\leq (1-2^{-\ell-1}){\e},
\label{e:mask}
\end{equation}
and, 
for all ${\bm\sigma}\in\bm\Sigma_{\ell-1}(\bm\Rind)$ such that $F_{\ell-1}(g, {\bm{\sigma}})\in\ms{G}^{\nu}\setminus L\sub{\fuInj,\ell-1}$
\begin{equation}
{ m\sub{{\Sigma}(R_\ell)}\Big(\bm{\sigma}_\ell\,:\,F_{\ell}(g,( {\bm{\sigma}},\bm\sigma\sub{\ell}))\in  L\sub{\fuInj, \ell}\Big)\leq  2^{-\ell-1}{\e}.}
\label{e:mask2}
\end{equation}
Here,  $ L\sub{\fuInj,-1}=\emptyset$, $F_{-1}=\pi\sub{\ms{G}^{\nu}}$, and for $\ell\geq0$ the set $ L\sub{\fuInj,\ell}$
is defined as follows: $g\in\ms{G}^{\nu}\setminus  L\sub{\fuInj,\ell}$
\begin{enumerate}
\item For $0\leq i\leq \ell$,\; $\step^{i-1}<n\leq\step^{i}$,\; 
\begin{equation}
\Gamma\cap\Rec\sub{\fuInj}(n,\tilde{\beta}_{i,\ell},g)\subset\ND\sub{\fuInj}(n,\tilde{\beta}_{i,\ell},g).\label{e:ind1}
\end{equation}
\item For all $\gap \ell\leq j\leq \ell$,\;$0\leq i\leq j$, \;$\step^{i-1}<n\leq\step^{i}$,\;
and all $1\leq q\leq\step^{j/\gap-i{+1}}$, 
\begin{equation}
\Gamma\cap\Rec\sub{\fuInj}(n,\beta_{i,j,\ell}q^{-3},g)\cap\Sim\sub{\!\fuInj}(n,\alpha_{j,\ell},g)\subset\ND_{q,_\fuInj}(n,\beta_{i,j,\ell}, g).\label{e:ind2}
\end{equation}

\end{enumerate}
\end{proposition}

\begin{remark}
There are a large number of parameters in Proposition~\ref{p:theInduction}. One should think of these parameters as follows. Those with a $\beta$ control the degree of $q$-non-degeneracy that trajectories have while those with a $\tilde{\beta}$ control simple non-degeneracy of trajectories. Parameters with an $\alpha$ control the degree of simplicity that a trajectory must have in order to apply the inductive hypothesis. 

 The most important constants will turn out to be $\mathfrak{c}$ and $\step$.  While $1<\step<2$ is left free for the moment, and will be fixed (\blue{in fact, it will satisfy $\step ^{1/\gap}=\mathfrak{c}$}) in the proof of Proposition~\ref{p:thePredominantMeat}, $\mathfrak{c}$ controls how rapidly the non-degeneracy of trajectories decays. This constant is determined in Step 3 of the induction (Section~\ref{s:badMeasure}). 
 
 Notice that the parameter $\mathfrak{c}$ depends on $m_{\nu}$, which depends on the regularity $\nu$ in which we work. The reason for definition of $m_\nu$ is as follows: 1) $\vartheta_\nu$ appears because we want our perturbations to be close in $\ms{G}^\nu$ (see Lemma~\ref{l:Step2}), 2) and $b$ appears because we need to apply Corollary~\ref{c:fullPerturb} (see Step 3 of the induction (Section~\ref{s:badMeasure}) 

\end{remark}

Given $\ell$, the proposition yields that for all the perturbations $F_{\ell-1}^{\bm \delta_\e, \bm{\Rind}_\e}(g,  \bm \sigma)$ that satisfy both \eqref{e:ind1} and \eqref{e:ind2} the set of $\bm \sigma_{\ell}$ that would yield a new perturbation $F_{\ell}^{\bm \delta_\e, \bm{\Rind}_\e}(g, \bm \sigma, \bm \sigma_{\ell})$ not satisfying either \eqref{e:ind1} or \eqref{e:ind2} is bounded in volume by $2^{-(\ell +1)}\ep$.  This allows us to show in the proof of Theorem \ref{t:predominantR-ND}  that, except for on a measure $\e$ set of $\bm \sigma$, the perturbations $F_{\infty}^{\bm \delta_\e, \bm{\Rind}_\e}(g, \bm \sigma)$ satisfy \eqref{e:ind1} for all $n$.

We next explain the reason why we need to include \eqref{e:ind2} in the statement (this is also explained carefully in  Section~\ref{s:outline}). We prove the proposition by induction in $\ell$. Suppose  we could perturb $F_{\ell-1}^{\bm \delta_\e, \bm{\Rind}_\e}(g,  \bm \sigma)$ in such a way that orbits that start in $\Gamma$ and and are $(n,\beta_{\ell, \ell})$ returning with $n \in (\step^{\ell-1}, \step^{\ell}]$ will also be $(n, \beta_{\ell, \ell})$ non-degenerate. Then, applying \eqref{e:ind1} to get non-degeneracy of $(n,\beta_{\ell, \ell})$ returning orbits with $n \in (0, \step^{\ell-1}]$ would finish the job.
However, this cannot always be done.
The issue is that we can only arrange for most of the perturbations to yield $(n, \beta_{\ell, \ell})$ non-degenerate orbits  when the orbits are $(n, \alpha_{\ell, \ell})$ simple with $\beta_{\ell, \ell}$ small in terms of $\alpha_{\ell, \ell}$ (see Corollary \ref{c:fullPerturb}). If all $(n, \beta_{\ell, \ell})$ returning orbits were indeed $(n, \alpha_{\ell, \ell})$ simple, then  \eqref{e:ind1} applied to $\ell-1$ would guarantee that orbits that start at $\Gamma$ and are  $(n,\beta_{\ell, \ell})$ returning with $n\in [0, \step^\ell]$  will be $(n, \beta_{\ell, \ell})$ non-degenerate as desired. In reality, some returning orbits will not be $(n, \alpha_{\ell, \ell})$ simple, and so those we will view as iterates of shorter 'loops' of length $n/q\in (\step^{i-1}, \step^{i-1}]$ with $q>1$ that are $(n/q, C^n \alpha_{\ell, \ell})$ returning (see Lemma \ref{l:iterates}). If the shorter 'loop' is $(n/q, \alpha_{i,\ell})$ simple and $C^n \alpha_{\ell, \ell}<\beta_{i, \ell}q^{-3}$, then we can apply \eqref{e:ind2} to say that such an orbit is $(n/q, \beta_{i, \ell})$ $q$ non-degenerate and use that to infer that the long orbit is actually $(n, \beta_{\ell, \ell})$ non-degenerate (see Lemma \ref{l:NdqImpliesNd}). If the shorter loop is not $(n/q, \alpha_{i,\ell})$ simple, then we view it as an iterate of a shorter loop and repeat the process until it terminates. The reason for the third index in $\beta_{i,j,\ell}$ is that we will need some extra non-degeneracy which is `used up' by making additional small perturbations in future inductive steps.

We divide the proof of Proposition~\ref{p:theInduction} into steps. In Step 0 we show that the base case $\ell=0$ holds. In Step 1 we deal with the returning orbits that are not simple enough to get non-degeneracy by perturbation, instead showing that they inherit the non-degeneracy by decomposing them into shorter loops as described above. In Step 2 we prove that the non-degeneracy created up to step $\ell$ is preserved by the perturbation performed in the $\ell+1$ step.  Finally, in Step 3, we prove that the volume of the set of perturbations that yield orbits whose degeneracy we cannot control is small.

\subsubsection{Step 0: Setting up  the induction argument} 


Let $(\Gamma, \Gb)$ be a  $(\fuInj,b,\ga)$-admissible pair for $\{Q^{R,\delta}\}_{R,\delta}$,   and  $K\subset\Gb$ bounded. Let $\tilde{K}\subset \Gb$ be  bounded and satisfy $K\subset \tilde K$ with $d(K,\tilde{K}^c)>0$. By Lemma~\ref{l:probing} there is $\e_1>0$ such that $\tilde{K}\subset \ms{G}_{\e}^\nu$ for $0<\e<\e_1$, where $\ms{G}_\e^\nu$ is as in~\eqref{e:fancyPants}, and for any choice of $\bm\delta_\e$ and $\Rind_\e$ satisfying~\eqref{e:probeControl}, we have \begin{equation*}\label{e:tildeK}
F\sub{J}^{\Rind_\e,\bm\delta_\e}(K\times \bm\Sigma\sub{J}(\Rind_\e))\subset \tilde{K}, \qquad  0<\e<\e_1, \;\; J\in \mathbb{N}\cup\{\infty\}.
\end{equation*}
Since all the metrics we work with will be of the form $F_J^{\Rind_\e,\bm{\delta}_\e}(g,\hat{\bm\sigma}_J)$, for some $g\in K$, they will lie inside $\tilde{K}$ and hence have uniform estimates.

Let $\tilde{\e}_0$ be the number $\e_0$ found in Lemma~\ref{l:baseCase}. From now on we assume that $\e_0 \leq \min\{\tilde \e_0, \e_1\}$.
For convenience, in what follows we work with
\begin{equation}\label{e:constants}
C_0={2}\max\{1, C, C_1, C_2, C_3, C_4,\dG^{-1},\CG,\blue{c}^{-1}\},
\end{equation}
where  $C$ is the constant from Corollary~\ref{c:fullPerturb},  $C_1$ is the constant from Lemma \ref{l:iterates}, $C_2>0$ is such that $\|\varphi\sub{t}^g-\varphi\sub{t}^{\tilde g}\|_{\mathcal{C}^1}\leq C_2^t\|g-\tilde g\|_{\mathcal{C}^2}$ for all $g, \tilde g \in \tilde K$ and $t\in \R$,  $C_3$ is as in Lemma \ref{l:NdqImpliesNd}, and $C_4$ is the maximum of the constants $C$ in Definition \ref{ass:1} for the choices of $\nu$ in $\{2,3, \nu\}$, $\CG$, $\dG$ are from Lemma~\ref{l:pLives} and~\eqref{e:Cdelta}, \blue{and $c$ is the constant from Lemma~\ref{l:pLives}}. Note that $C_0$ depends only on $\nu, N,  \ga, b, \step, \fuInj, \Gamma, G, K$. During the induction argument we will ask that $\e_0$ be small in terms of powers of $C_0^{-1}$.

We prove the lemma by induction on $\ell$. We check the cases $0\leq\ell\leq \lceil \log_\step 2\rceil -1$. Since $\step^{\lceil \log_\step 2\rceil -1}<2$, this amounts to considering $0\leq \ell \leq\lceil \log_\step 2\rceil -1$, $n=1$, \blue{$i=0$} in~\eqref{e:ind1} and $\blue{i=0},n=1$ and $\gap \ell\leq j\leq \ell$, $0\leq i\leq j$, and $1\leq q\leq \step^{j/\gap-i+1}$ in~\eqref{e:ind2}. 

By Lemma~\ref{l:baseCase}, for all $\ep>0$ there are $\beta_{0,0}$,  $\Rtemp_{0}$, and $\delta_{0}$
such that such that~\eqref{e:close1} holds and~\eqref{e:measureNonDeg1} {holds for all $g\in K$ and $0<\e<\e_0(\Gamma,K)$. Set $\delta_\ell(\e)=0,\,R_\ell(\e)=1$ for $1\leq \ell\leq \lceil\log_\step 2\rceil-1$. That is,
\begin{equation}
\label{e:sameSame}
F^{\Rind_\e,\bm\delta_\e}_0(g,\sigma_0)=F^{\Rind_\e,\bm\delta_\e}_\ell(g,(\sigma_0,\sigma_1,\dots,\sigma_\ell)),\qquad 1\leq \ell\leq \lceil \log_\step 2\rceil -1.
\end{equation}}
Then,~\eqref{e:close1} implies~\eqref{e:D} for $\ell\leq \lceil\log_\step 2\rceil-1$. 
 
Set 
\begin{equation} 
\label{e:settingBetaHere}
\beta_{i,j}:=\beta_{i,i},\text{ for }0\leq i\leq j\leq \lceil \log_\step 2\rceil -1.
\end{equation}
We note that by \eqref{e:sameSame}, {it is enough to show that}
$$
m\sub{\Sigma(R_0)}\Big(\sigma\,:\, F_0(g,\sigma)\in \bigcup_{0\leq \ell\leq \lceil\log_\step 2\rceil-1}L\sub{\fuInj,\ell}\Big)\leq 2^{-1}\e.
$$

We next claim that for $\ell\leq \lceil \log_\step 2\rceil -1$, and $g\in K$, 
$$
\Big\{ \sigma\,:\, F^{\Rind_\e,\bm\delta_\e}_0(g,\sigma)\in \bigcup_{0\leq \ell\leq \lceil\log_\step 2\rceil-1}L\sub{\fuInj,\ell}\Big\}\subset\Big\{ \sigma\,:\, F^{\Rind_\e,\bm\delta_\e}_0(g,\sigma)\in \tilde{L}\sub{\fuInj,0}\Big\}
$$
where $ \tilde L\sub{\fuInj,0}$ is defined in Lemma~\ref{l:baseCase}.

Indeed, since $\tilde\beta_{0,0}= \beta_{0,0}^{2\aleph}\e^{2\aleph-1}$ and $\beta_{0,0,0}=2\beta_{0,0}$ for all $g\in K$ and $\e>0$ small enough depending on $K$
\begin{equation} \label{e:newInclude}
\blue{ \ND_{1,_\fuInj}(1,\beta_{0,0,0},F_0(g,\sigma)) \subset \ND_{_\fuInj}(1,\tilde{\beta}_{0,0},F_0(g,\sigma)). }
\end{equation}
\blue{(We recall Definitions~\ref{d:ND} and~\ref{d:NDq} for the definitions of  $\ND_{\fuInj}$ and$\ND_{q,_\fuInj}$.)
Indeed, if $\rho\in\ND_{\!\!1,_\fuInj}(1,\beta_{000},g_*)$,
there exist $\mc{I}$ and $m$ satisfying $\bm{T}\sub{\mc{I}}^{(m)}[g_*](\rho) \in \big[0-\CG\dG\,,\, \fuInj+\CG\dG \big]$ and
$$
{\sup_{\{(\psi,U)\in\mc{A}:\,B(\rho,\CG\dG)\subset U\}}}\frac{\big\|(\Id-[(d((\mc{P}\sub{\mc{I}}^{(m)}[g_*])_\psi)(\rho)]^{q})^{-1}\big\|}{(1+\|d((\mc{P}\sub{\mc{I}}^{(m)}[g_*])_\psi)(\rho)\|^{q})^{2\aleph-1}}\leq(\tfrac{5}{2})^{2\aleph} (2\beta_{00})^{-2\aleph}.
$$
In particular, 
$$
{\sup_{\{(\psi,U)\in\mc{A}:\,B(\rho,\CG\dG)\subset U\}}}
\big\|\big(\Id-d((\mc{P}\sub{\mc{I}}^{(m)}[g_*])_\psi)(\rho)\big)^{-1}\big\|<\tilde \beta_{00}^{-1}=\beta_{00}^{-2\aleph}\e^{1-2\aleph},
$$
provided $\e$ is small enough, and so $\rho\in \ND_{_\fuInj}(1,\tilde{\beta}_{0,0},g_*)$.}

\blue{
We claim that if  $F^{\Rind_\e,\bm\delta_\e}_0(g,\sigma)\notin \tilde{L}\sub{\fuInj,0}$, then~\eqref{e:ind1} holds for $0\leq \ell \leq \lceil\log_\step 2\rceil-1$ and $g$ replaced by $F^{\Rind_\e,\bm\delta_\e}_0(g,\sigma)$. To see this, first notice if $0<i\leq \ell$, then $1\leq \step^{i-1}$ and $\step^{i}<2$. Hence, there are no integers $n$ with $\step^{i-1}<n\leq \step^i$. Thus, we need only consider $i=0$ and $n=1$. To do this, let $\rho \in \Gamma\cap\Rec\sub{\fuInj}(
1,\tilde{\beta}_{0,\ell},F^{\Rind_\e,\bm\delta_\e}_0(g,\sigma))$. Then, since $\tilde{\beta}_{0,\ell}\leq \beta_{0,0,0}$ and $F^{\Rind_\e,\bm\delta_\e}_0(g,\sigma)\notin \tilde{L}\sub{\fuInj,0}$,~\eqref{e:nonDeg1}  and~\eqref{e:newInclude} imply
$$
\rho \in \Big(\Gamma\cap\Rec\sub{\fuInj}(
1,\beta_{0,0,0},F^{\Rind_\e,\bm\delta_\e}_0(g,\sigma))\Big)\subset \ND_{1,_\fuInj}(1,\beta_{0,0,0},F^{\Rind_\e,\bm\delta_\e}_0(g,\sigma))\subset \ND_{_\fuInj}(1,\tilde{\beta}_{0,0},F_0(g,\sigma)).
$$
Hence,~\eqref{e:ind1} follows since $\tilde{\beta}_{0,\ell}\leq \tilde{\beta}_{0,0}$. 
}

Next, since for \blue{$i=0$, $0\leq j\leq \ell$, we have using~\eqref{e:settingBetaHere} that} 
\blue{
$$
\beta_{i,j,\ell} = (1+2^{j-\ell})\beta_{0,j}=(1+2^{j-\ell})\beta_{0,0}=\frac{1+2^{j-\ell}}{2}\beta_{0,0,0}\leq \beta_{0,0,0}.$$}
Therefore, for
$\gap\ell\leq j\leq \ell$, and $0\leq \ell \leq \lceil\log_\step 2\rceil-1$, we have that 
$$ 
\big\{g\in \ms{G}^{\nu}:\;\;   \;\;\Gamma\cap\Rec\sub{\fuInj}(1,\beta_{0,0,0}q^{-3}, g)\cap \Sim\sub{\!\fuInj}(n,\alpha_{j,\ell},g)\backslash \ND_{q,_\fuInj}(1,\beta_{0,0,0},g)\neq \emptyset \big\}\subset  {\tilde L\sub{\fuInj,0}}
$$
and hence, if $F^{\Rind_\e,\bm\delta_\e}_0(g,\sigma)\notin \tilde{L}\sub{\fuInj,0}$, then~\eqref{e:ind2} holds for $0\leq \ell \leq \lceil\log_\step 2\rceil-1$ and $g$ replaced by $F^{\Rind_\e,\bm\delta_\e}_0(g,\sigma)$. 
The claim in \eqref{e:mask} then follows from \eqref{e:measureNonDeg1}. This concludes the proof of the base case in the induction argument \blue{since, for $0\leq \ell \leq \lceil \log_\step 2\rceil -1$, $L_{\digamma, \ell}=\emptyset$}. 


Note that with $\beta_{0,0}$ in place, the constants $\beta_{\ell,\ell}$ and $\alpha_\ell$ are defined as in \eqref{e:betas} and \eqref{e:alphas} for all $\ell$.
For $\e>0$ and $\ell={\lceil\log_\step2\rceil}, {\lceil\log_\step2\rceil}+1,\dots $ we set
\begin{equation}\label{e:R-delta}
{\Rtemp_{\ell}}(\e):=\tfrac{1}{8}{\e}^{\step^{\ell}+1}\alpha_{\ell},\qquad 
{\delta_{\ell}}(\e):=\min\Big({\e}^{\step^{\ell}+1}\Rtemp_{\ell}^{\blue{b}}\,,\,{\e \blue{2^{-\ell-1}}\Rtemp_{\ell}^{\blue{\max(\vartheta_\nu,\vartheta_2+1)}}}\Big).
\end{equation}
We note that $\delta_{\ell}(\e)\Rtemp_{\ell}(\e)^{-\vartheta_\nu}<2^{-\ell-1}\e$ and so
 $\bm{\delta}_\e=\{\delta_\ell(\e)\}_{\ell=0}^\infty$ and $\Rind_\e=\{R_\ell(\e)\}_{\ell=0}^\infty$ satisfy~\eqref{e:probeControl}.
We also note that with these definitions there is $\e_0$ small enough that the assumptions \eqref{e:condsRdelta} in Corollary \ref{c:fullPerturb} are satisfied with $R=R_{\ell+1}(\e)$, $\delta=\delta_{\ell+1}(\e)$, $\alpha=\tfrac{1}{2}\alpha_{\ell+1}$, and $n\leq \step^{\ell+1}$.

We now start the inductive step. Note that the base case $0\leq \ell\leq \lceil\log_\step 2\rceil-1$ only covers $n=1$ since $\step^\ell<n\leq \step^{\ell+1}$ and $\step<2$. Since to get to $n=2$ we need $2\leq \step^{\ell+1}$, we may assume that the inductive
hypotheses hold for some $\ell\geq {\lceil\log_\step 2\rceil-1}$. We split the proof of the inductive
step with $\ell+1$ into three steps. 

\subsubsection{Step 1: Nondegeneracy of returning points that are not simple}  The goal of this step is to study the set of points in $\Gamma$ that under the perturbed metric $F_{\ell}(g,\bm{\sigma})$ generate orbits  that return to $\Gamma$ at some `discrete' time $n\in (\step^{\ell},\step^{\ell+1}]$ but that are not simple enough that the $F_{\ell+1}$ perturbations would make them non-degenerate. As explained before, we decompose these returning orbits into shorter `loops' and we use the $(\beta_q,q)$ non-degeneracy of shorter orbits given by \eqref{e:ind2} to show that the original orbits were already non-degenerate.

We note that for this step in the proof the exact powers in the definition of $\beta_{j,j}$ do not play a role. Instead, we only use that $\beta_{j,j}\geq \beta_{j+1,j+1}$ and $\beta_{i+1,j}\leq \beta_{i,j}$  for all $i$ and $j$. The precise definition of $\gamma_j$ as well as the definition of $\alpha_{j+1}$ in terms of $\gamma_j$ and $\beta_{j,j}$ do, however, play a role.

\begin{lemma}\label{l:Step1}
There is $\e_0>0$ depending only on $\nu, N,  \ga, b, \step, \fuInj, \Gamma, G, K$ so that the following holds.
Suppose the conclusions of Proposition \ref{p:theInduction} hold up to the index $\ell$.  Then, for all $0<\e<\e_0$, $\step^{\ell}<n\leq\step^{\ell+1}$ and $(g, \bm{\sigma})\in K \times \bm{\Sigma}_{\ell}(\bm{\Rind}_\e)$ such that $g_\ell:=F_{\ell}(g,\bm{\sigma})\in \ms{G}^\nu \backslash  \blue{L\sub{\fuInj,\ell}}$ we have
\begin{equation}
\Gamma\cap\Rec\sub{\fuInj}(n,\alpha_{\ell+1},g_{\ell})\setminus\Sim\sub{\!\fuInj}(n,\alpha_{\ell+1},g_{\ell})\subset\ND\sub{\fuInj}(n,\alpha_{\ell+1},g_{\ell}).\label{e:Step1}
\end{equation}
\end{lemma}
\begin{proof}
Let $\step^{\ell}<n\leq\step^{\ell+1}$ and $(g, \bm{\sigma})\in K \times \bm{\Sigma}_{\ell}(\bm{\Rind})$ such that $g_\ell:=F_{\ell}(g,\bm{\sigma})\in \ms{G}^\nu \backslash  L\sub{\fuInj,\ell}$. Let 
\begin{equation}\label{e:whereIsWally}
\rho\in\Gamma\cap\Rec\sub{\fuInj}(n,\alpha_{\ell+1},g_{\ell})\setminus\Sim\sub{\!\fuInj}(n,\alpha_{\ell+1},g_{\ell}).
\end{equation}

We divide the proof into fours steps. In Step A, we decompose the non-simple orbit associated to $\rho$ into shorter $(k_m, \delta_m(\ell))$ returning orbits that are  $(k_m, \alpha_m(\ell))$ simple. The returning times $t_m$ associated to these orbits are such that $t_m$ divides $t_{m-1}$, and $k_m=\lceil \frac{t_m}{\fuInj}\rceil$ with $k_0=n$. In Step B we show that $\rho \in \ND_{q_m,_\fuInj}(k_m,\beta_m(\ell),g_\ell)$ with $q_m:=t_0/t_m \in \mathbb Z$.
In Step C we prove that there are $\mc{I}$ and $p\in \mathbb Z$ such that for all $q\leq q_m$ there is a time $ t_{q,m}$ sufficiently close to $q t_m$ with the property that  $(\mc{P}\sub{\mc{I}}^{p}[g_\ell])^q(\rho)=\mc{P}\sub{\mc{I}}^{qp}[g_\ell](\rho)=\varphi_{ t_{q,m}}^{g_\ell}(\rho)$. In particular, one can control the distance from $(\mc{P}\sub{\mc{I}}^{p}[g_\ell])^q(\rho)$ to $\rho$ for $q\leq q_m$. In Step D we check that the above guarantees that the hypothesis of Lemma \ref{l:NdqImpliesNd} are satisfied and so $\rho$ is $(n, \alpha_{\ell+1})$ non-degenerate as claimed.\\

\noindent \emph{Step A.} We start the proof by showing that there are  $1\leq m\leq \lfloor(\ell+1)\log_2\step\rfloor$, $\{t_i\}_{i=0}^{m}\subset \R$, and $\{k_i\}_{i=0}^{m}\subset \Z$, such that  $k_0=n$, $(k_i-1)\fuInj<t_i\leq k_i\fuInj$,   $t_i$ divides $t_{i-1}$  for $1\leq i\leq m$,  $t_i> \fuInj$, 
\begin{equation}\label{e:intermediateStep}
\begin{gathered}
\rho \in \Gamma\cap\Rec\sub{\fuInj}(k_m,\delta_m(\ell),g_\ell)\cap\Sim\sub{\!\fuInj}(k_m,\alpha_m(\ell),g_\ell),\qquad 
d(\varphi_{t_m}^{g_\ell}(\rho),\rho)\leq \alpha_m(\ell),
\end{gathered}
\end{equation}
and 
\begin{equation}
\label{e:intermediateStep2}
d(\varphi_t^{g_\ell}(\rho),\varphi_{\remainder}^{g_\ell}(\rho))\leq \delta_{i}(\ell),\,\qquad \,0\leq t\leq t_{i-1},\quad t=q t_i+\remainder,\,\quad\, 0\leq \remainder<t_{i}.
\end{equation}
Here,  for $i\geq 1$
$$\alpha_i(\ell):=\alpha_{\ell-i+1}, \qquad  \delta_0(\ell)=\alpha_0(\ell), \qquad   \delta_i(\ell):=C_0^{{t}_{i-1}}\alpha_{i-1}(\ell),$$ with $C_0$ as in \eqref{e:constants}.

Note that 
$\rho\in\Gamma\cap\Rec\sub{\fuInj}(k_0,\delta_{0}(\ell),g_{\ell})\setminus\Sim\sub{\!\fuInj}(k_0,\alpha_{0}(\ell),g_{\ell})$ by \eqref{e:whereIsWally}. We next show that, by splitting the orbit into shorter returning ones, we can find $\{t_i\}_{i=0}^{m-1}$ such that $t_{i+1}$ divides $t_i$ and $\rho\in\Gamma\cap\Rec\sub{\fuInj}(k_i,\delta_{i}(\ell),g_{\ell})\setminus\Sim\sub{\!\fuInj}(k_i,\alpha_{i}(\ell),g_{\ell})$ with $k_i:=\lceil\frac{t_i}{\fuInj}\rceil$, and that both \eqref{e:intermediateStep} and \eqref{e:intermediateStep2}. 

To do this, suppose we have found 
$\{t_i\}_{i=0}^{m-1}$ such that $t_{i+1}$ divides $t_i$, and, with $k_i:=\lceil\frac{t_i}{\fuInj}\rceil$,
\begin{align}\label{e:hunger}
 \rho \in \Rec\sub{\fuInj}(k_i,\delta_{i}(\ell),g_{\ell}),
 \\
 \rho\notin \Sim\sub{\!\fuInj}(k_i,\alpha_i(\ell),g_\ell)\label{e:hunger2},
\end{align}
and~\eqref{e:intermediateStep2} hold for $0\leq i\leq m-1$. We prove by induction that there is $t_m$ dividing $t_{m-1}$ such that \eqref{e:hunger} holds with $i=m$ and~\eqref{e:intermediateStep2} holds for $1\leq i\leq m$. Note that the base case $i=0$ follows from  \eqref{e:whereIsWally} after letting $t_0$ be such that $n=\lceil\frac{t_0}{\fuInj}\rceil$  since $k_0=n$.

First, note that $\delta_i(\ell)\leq \alpha_i(\ell)$ for $1\leq i\leq m-1$. Indeed,  since $\{\beta_{j,j}\}_j$ is decreasing it is straightforward to check that $\delta_i(\ell)\leq \alpha_i(\ell)$   provided $C_0^{t_{i-1}} \gamma_{\ell-i-1}\leq \gamma_{\ell-i-2}$.
Since   $\step^{-4\aleph/\gap}\leq 1$, this reduces to checking that $C_0^{t_{i-1}}\e^{(2\aleph+1)\blue{\step^{(\ell-i)/\gap}}(\step^{1/\gap}-1)} \leq 1$.
Since $t_{i-1}\leq k_{i-1}\fuInj\leq \step^{\ell+1}2^{-(i-1)}\fuInj$, the claim follows from letting $\e_0$ be small.

Next, since \eqref{e:hunger} and~\eqref{e:hunger2} hold with $i=m-1$ and $\delta_i(\ell)\leq \alpha_i(\ell)$ for $1\leq i\leq m-1$, using Lemma~\ref{l:united} together with $C\sub{\Gamma}>1$ we may apply Lemma~\ref{l:iterates} with $C\sub{\fuInj}:=C\sub{\Gamma}>1$, $\delta:=\delta_{m-1}(\ell)$, $\alpha:=\alpha_{m-1}(\ell)$, $t_0:=t_{m-1}$,  to obtain the existence of $ t_{m}\in [C\sub{\Gamma}\fuInj-C_0^{t_{m-1}}\alpha_{m-1}(\ell)\, ,\, \tfrac{1}{2}t_{m-1}]$ dividing
$t_{m-1}$ such that $d(\varphi_{t_{m}}^{g_\ell}(\rho),\rho)\leq \delta_m(\ell)$ and for $0\leq l< t_{{m-1}}$ with $l=q t_{m}+\remainder$
and $0\leq\remainder< t_{m}$ 
\begin{equation}
d(\varphi_l^{g_\ell}(\rho),\varphi^{g_\ell}_{\remainder}(\rho))
\blue{\leq} \delta_{m}(\ell).\label{e:iterate}
\end{equation}
Note that, provided $\e>0$ was chosen small enough small enough (uniformly in $K$, $\ell$), and using that $C\sub{\Gamma}>1$, we have $t_{m}>\fuInj$.
  If $\rho \notin \Sim\sub{\!\fuInj}(k_m,\alpha_m(\ell),g_\ell)$, the inductive step $i=m$ is complete. On the other hand, if $\rho \in \Sim\sub{\!\fuInj}(k_m,\alpha_m(\ell),g_\ell)$ then we have proved~\eqref{e:intermediateStep} and~\eqref{e:intermediateStep2}.

  At this point, we have proved either that there is $m\leq \lfloor (\ell+1)\log_2\step\rfloor$ such that~\eqref{e:intermediateStep} and~\eqref{e:intermediateStep2} hold or there are $\{t_i\}_{i=1}^N$, with $N >\lfloor(\ell+1)\log_2\step\rfloor$ such that $t_i$ divides $t_{i-1}$, $t_i> {\fuInj}$, and with $k_i=\lceil \frac{t_i}{\fuInj}\rceil$, \eqref{e:intermediateStep2}, ~\eqref{e:hunger},~\eqref{e:hunger2}  hold for all $i$.
   We claim that the second alternative is not possible. Indeed, note that since $t_{i}$ divides $t_{i-1}$, 
\begin{equation}\label{e:timesBound}
 {\fuInj}< t_i\leq 2^{-i}t_0\leq 2^{-i}\step^{\ell+1}\fuInj,
 \end{equation}
and this is not possible for $i>\lfloor (\ell+1)\log_2\step\rfloor$ since then $2^{-i}\step^{\ell+1}\leq 1$.
It follows that the claims~\eqref{e:intermediateStep} and~\eqref{e:intermediateStep2} are true with $m$ as stated.\\

\noindent \emph{Step B.} Let $q_m=t_0/t_m$ and $\beta_m(\ell) =\beta_{i_m,j_m}$ with $i_m=\lceil \log_{\step } k_{m}\rceil$, $j_m=\ell-m+1$. We claim
\begin{equation}\label{e:allTheIterates}
\rho\in \ND_{q_m,_\fuInj}(k_m,\beta_m(\ell),g_\ell).
\end{equation}
The objective is to show that we can apply the induction hypothesis \eqref{e:ind2} with $(n,q, i,j,\ell):=(k_m, q_m,  i_m, j_m, \ell)$. To this end, we first observe that by definition $\step^{i_m-1}<k_m \leq \step^{i_m}$. In addition, $k_m \leq 2^{-m}\step^{\ell+1}\leq \step^{\ell-m+1}$ and so $i_m \leq j_m$. Also, since $m\leq \lfloor (\ell+1)(1-\gap)\rfloor$, we have $\gap (\ell+1) \leq j_m \leq \ell$.  In particular, $\frac{t_0}{\fuInj} \leq \step^{\ell+1} \leq  \step^{\frac{j_m}{\gap}}$. Since $\frac{t_m}{\fuInj} \geq \step^{i_m-1}$, this also yields $q_m=\frac{t_0}{t_m}\leq \step^{j_m/\gap-i_m+1}$.

We first claim that  $$\delta_m(\ell)\leq \beta_m(\ell) \red{\blue{q_m^{-3}}}.$$
Indeed, using that  $\beta_{j_m,j_m}\leq \beta_{i_m,j_m}$  (since  $i_m\leq j_m$)  and  $q_m \leq \step^{\ell+1}$, this reduces to checking that $\step^{3(\ell+1)}C_0^{t_{m-1}} \alpha_{j_m+1} \leq \beta_{j_m,j_m}$.
Furthermore, since $\beta_{j_m,j_m} \leq 1$, the claim reduces to showing that $\step^{3(\ell+1)}C_0^{t_{m-1}} \gamma_{j_m} \leq 1$, and this is equivalent to
$$
5^{-2\aleph} \leq C_0^{-\step^{\ell+1}2^{-m+1}}\step^{4\aleph(\ell-m+2)/\gap-3(\ell+1)}\e^{-(2\aleph+1)\step^{(\ell-m+2)/\gap}}.
$$ 
The claim then follows from the facts that $\e_0 \leq C_0^{-1}$ and $(\ell-m+2)/\gap\geq \ell+1$.

Since $\delta_m(\ell)\leq \beta_m(\ell) \blue{q_m^{-3}}$ and $\beta_{i_m, j_m, \ell}=(1+2^{-m})\beta_{m}(\ell)$,   we conclude that 
$\Rec\sub{\fuInj}(k_m,\delta_m(\ell),g_\ell) \subset \Rec\sub{\fuInj}(k_m,  \beta_{i_m,j_m, \ell}\blue{q_m^{-3}},g_\ell)$. 
Therefore, the claim in \eqref{e:allTheIterates}  follows from combining  \eqref{e:intermediateStep} and \eqref{e:ind2} with $(n,i,j)=(k_m,i_m, j_m)$. 
This can be done since $\beta_{i_m, j_m, \ell}\geq \beta_{m}(\ell)$ and $\Sim\sub{\!\fuInj}(k_m,\alpha_m(\ell),g_\ell)\subset \Sim\sub{\!\fuInj}(k_m,\alpha_{j_m, \ell},g_\ell)$ because
$\alpha_{j_m, \ell}=(1-2^{-m})\alpha_{m}(\ell)$.\\

\noindent \emph{Step C.} Recall that $k_m=\lceil \tfrac{t_m}{\fuInj} \rceil$ and $t_0=q_m t_m$ with $k_0=n$. Let $\CG$, $\dG$ be as in Lemma \ref{l:pLives} and \eqref{e:Cdelta}. Let 
\begin{equation}\label{e:anotherDelta}
\delta_{m,\ell}:=mC_{0}^{\blue{2}\fuInj\step^{\ell+1}}\alpha_{m-1}(\ell).
\end{equation}
We next show there are $\mc{I}$ and $p$ with
$\bm{T}\sub{\mc{I}}^{(p)}[g_\ell](\rho) \in \big[(k_m-1)\fuInj-\CG\dG\,,\, k_m\fuInj+\CG\dG \big]$, 
\begin{equation}
\begin{gathered}\label{e:Piterates}
\big\{\mc{P}\sub{\mc{I}}^{(qp)}[g_\ell](\rho)\big\}_{q=0}^{q_m}\subset \mc{D}\sub{\mc{I}}^{(p)}[g_\ell],\qquad \max_{0\leq q \leq q_m}d\Big(\big(\mc{P}\sub{\mc{I}}^{(p)}[g_\ell]\big)^q(\rho),\rho\Big)\leq \CG\delta_{m,\ell},\\
\mc{P}\sub{\mc{I}}^{(qp)}[g_\ell](\rho')=\big(\mc{P}\sub{\mc{I}}^{(p)}[g_\ell]\big)^q(\rho'),\qquad 0\leq q\leq q_m,\qquad\text{{ $\rho'$ in a neighborhood of $\rho$}}.
\end{gathered}
\end{equation}

 Let $0\leq q \leq q_m$. Since $\{\alpha_j(\ell)\}_j$ is \blue{increasing}, it follows from Lemma~\ref{l:allTheIterates} below that
\begin{equation}
d(\varphi^{g_\ell}_{q t_{m}}(\rho),\rho)\leq  \delta_{m,\ell},\label{e:allTheIterates1}
\end{equation}
\blue{after using \eqref{e:timesBound} combined with the fact that $\sum_{k=j}^{m-1} t_k\leq \step^{\ell+1}\fuInj \sum_{k=0}^\infty 2^{-k}=2 \step^{\ell+1}\fuInj$ for all $0\leq j\leq m-1 $.}
Let $\e_0$ be small enough that $\delta_{m,\ell}\leq \dG$ for any choice of $m, \ell$. Then, since $\rho\in \Gamma$, and $\delta_{m,\ell}\leq \dG$ by \eqref{e:smallS} and \eqref{e:disGa} there exists $ t_{q,m}$ with 
such that 
\begin{equation}
\label{e:allTheIteratesAgain}
| t_{q,m}-qt_m|< \CG\delta_{m,\ell}, \qquad \varphi_{ t_{q,m}}^{g_\ell}(\rho)\in \tilde{\Gamma},\qquad d(\varphi_{ t_{q,m}}^{g_\ell}(\rho),\rho)<\CG \delta_{m,\ell}.
\end{equation}
By Remark~\ref{r:psAndTs} there are $p$ and $\tilde{\mc{I}}$ such that $\rho\in \mc{D}^{(p)}\sub{\tilde{\mc{I}}}[g_\ell]$,  
\begin{equation}
\label{e:pigs}\mc{P}\sub{\tilde{\mc{I}}}^{(p)}[g_\ell](\rho)=\varphi_{t_{1,m}}^g(\rho), 
\qquad 
\bm{T}\sub{\tilde{\mc{I}}}^{(p)}[g_\ell](\rho)=t_{1,m}, \qquad \max_{0\leq j \leq pq_m}d(\mc{P}\sub{\tilde{\mc{I}}}^{(j)}[g_\ell](\rho),\Gamma)<\dG.
\end{equation}
Let $c_*$ be the constants given in Lemma \ref{l:whatever}. Now, by Lemma~\ref{l:whatever}, we have $B\sub{\tilde{\Gamma}}(\rho, c_*^p)\subset \mc{D}\sub{\tilde{\mc{I}}}^{(p)}[g_\ell]$. 
Note that, since $p\leq C\step^{\ell+1}$, one can choose $\e_0$ small so that $\CG \delta_{m,\ell}\leq c_*^p$ for all $m, \ell$ and $\e<\e_0$ . Thus, 
$
[\mc{P}\sub{\tilde{\mc{I}}}^{(p)}]^q(\rho)
$ exists for $0\leq q\leq q_m$. 

Next, we claim that for all $1\leq q \leq q_m$
\begin{equation}\label{e:radish}
\big(\mc{P}\sub{\tilde{\mc{I}}}^{(p)}[g_\ell]\big)^q(\rho)=\varphi_{ t_{q,m}}^{g_\ell}(\rho).
\end{equation}
To see this, set $t_{0,m}=0$. Since the claim is true for $q=1$, we may assume that for some $1\leq q<q_m$, and every $1\leq j\leq q$, 
\begin{equation}
\label{e:horse}
T\sub{\tilde{\mc{I}}}^{(p)}[g_\ell]\Big(\big(\mc{P}\sub{\tilde{\mc{I}}}^{(p)}[g_\ell]\big)^{j-1}(\rho)\Big)=t_{j,m}-t_{j-1,m},
\end{equation}
and hence that for $1\leq j\leq q$
\begin{equation}
\label{e:cattle}
\big(\mc{P}\sub{\tilde{\mc{I}}}^{(p)}[g_\ell]\big)^{j}(\rho)= \varphi_{t_{j,m}-t_{j-1,m}}^{g_\ell} \Big(\big(\mc{P}\sub{\tilde{\mc{I}}}^{(p)}[g_\ell]\big)^{j-1}(\rho)\Big)=\varphi_{ t_{j,m}}^{g_\ell}(\rho).
\end{equation}
Now, observe that there exists $C$ depending only on $G$ such that
$$
|T\sub{\tilde{\mc{I}}}^{(p)}[g_\ell](\rho)-T\sub{\tilde{\mc{I}}}^{(p)}[g_\ell](\rho')|\leq C^{p\fuInj}d(\rho,\rho').
$$
Hence, using~\eqref{e:allTheIteratesAgain},~\eqref{e:pigs}, and~\eqref{e:cattle} 
$$
\Big|t_{1,m}-T\sub{\tilde{\mc{I}}}^{(p)}[g_\ell]\Big(\big(\mc{P}\sub{\tilde{\mc{I}}}^{(p)}[g_\ell]\big)^{q}(\rho)\Big)\Big|\leq C^{p\fuInj}\CG\delta_{m,\ell}.
$$

Thus, we have
$$
\varphi_{t_{q+1,m}-t_{q,m}}^{g_\ell}\Big(\big(\mc{P}\sub{\tilde{\mc{I}}}^{(p)}[g_\ell]\big)^{q}(\rho)\Big)\in \tilde{\Gamma},\qquad \varphi_{T\sub{\tilde{\mc{I}}}^{(p)}[g_\ell]\Big(\big(\mc{P}\sub{\tilde{\mc{I}}}^{(p)}[g_\ell]\big)^{q}(\rho)\Big)}^{g_\ell}\Big(\big(\mc{P}\sub{\tilde{\mc{I}}}^{(p)}[g_\ell]\big)^{q}(\rho)\Big)\in \tilde{\Gamma},
$$
and, choosing $\e_0$ small enough,
$$
\Big|t_{q+1,m}-t_{q,m}-T\sub{\tilde{\mc{I}}}^{(p)}[g_\ell]\Big(\big(\mc{P}\sub{\tilde{\mc{I}}}^{(p)}[g_\ell]\big)^{q}(\rho)\Big)\Big|<(C^{p\fuInj}+3)\CG\delta_{m,\ell}<c\sub{\Gamma}\fuInj.
$$
In particular, this implies that
$$t_{q+1,m}-t_{q,m}=T\sub{\tilde{\mc{I}}}^{(p)}[g_\ell]\Big(\big(\mc{P}\sub{\tilde{\mc{I}}}^{(p)}[g_\ell]\big)^{q}(\rho)\Big),$$
and hence~\eqref{e:horse} holds with $q$ replaced by $q+1$. This shows that \eqref{e:radish} holds.

Let $\tilde{\mc{I}}:=(\tilde {i}_0,\tilde{i}_1,\dots).$
Using again that $\e>0$ can be chosen small enough, this implies that, with the chain
$$
\blue{(\mc{I})_j:=\tilde{i}_{j\!\!\!\!\!\mod p},  \qquad \mathcal{I}=(\tilde {i}_0,\tilde{i}_1,\dots, \tilde{i}_{p-1},\tilde{i}_0,\dots,\tilde{i}_{p-1},\dots)},
$$
we have
\begin{equation}\label{e:Titerates}
(\lceil q t_m\rceil-1)\fuInj-\CG \dG\leq \bm{T}\sub{\mc{I}}^{(pq)}[g_\ell](\rho)\leq \lceil q t_m\rceil\fuInj+\CG \dG,
\end{equation}
and hence
$$
\mc{P}\sub{\mc{I}}^{(qp)}[g_\ell](\rho)=\varphi_{ t_{q,m}}^{g_\ell}(\rho)=\big(\mc{P}\sub{\tilde{\mc{I}}}^{(p)}[g_\ell]\big)^q(\rho).
$$
Moreover, by definition, using that the fact that $\mc{D}_{\mc{I}}^{(p)}[g_\ell]$ and $\mc{D}_{\mc{I}}^{(pq)}[g_\ell]$ are open,  we have  $\big(\mc{P}\sub{\mc{I}}^{(p)}[g_\ell]\big)^q(\rho')=\mc{P}\sub{\mc{I}}^{(qp)}[g_\ell](\rho')$ for $\rho'$ in a neighborhood of $\rho$ as claimed.\\


\noindent \emph{Step D.} We now complete the proof of the lemma. It follows from \eqref{e:intermediateStep} and \eqref{e:allTheIterates} that 
\begin{equation}\label{e:hypChecked}
\rho \in \Gamma\cap\Rec\sub{\fuInj}(k_m,\CG \delta_{m,\ell},g_\ell)\cap \ND_{q_m,_\fuInj}(k_m,\beta_m(\ell),g_\ell),
\end{equation}
 with $\delta_{m,\ell}$ as in \eqref{e:anotherDelta}. Here, we have used that $\delta_m(\ell)\leq \CG \delta_{m,\ell}$ since $t_{m-1}\leq \fuInj k_m \leq \fuInj\step^{\ell+1}$.
The goal is then  to use Lemma~\ref{l:NdqImpliesNd} with $\delta:=\CG\delta_{m,\ell}$, and $(\beta, q_0, k, m):=(\beta_m(\ell), q_m, k_m, p)$. Let $C_3$ be \blue{the constant given by} Lemma \ref{l:NdqImpliesNd} and note that $C_0>C_3$ by \eqref{e:constants}. Thus, to apply Lemma \ref{l:NdqImpliesNd} we first need to check that \begin{equation}\label{e:dCondition}
\CG\delta_{m,\ell} \leq  \min \big\{\dG\,,\, (2\ANDrate(\beta_m(\ell),q_m)C_0^{k_m q_m})^{-1} \big\}.
\end{equation}

To see this, note that by the definitions \eqref{e:anotherDelta} and \eqref{e:ND_q}
$$
\CG\delta_{m,\ell}2\ANDrate(\beta_m(\ell),q_m)C_0^{k_m q_m}= 2 \CG m C_0^{\fuInj \step^{\ell+1}+k_mq_m} (\tfrac{5}{2})^{2\aleph} q_m^{4\aleph} \alpha_{m-1}(\ell)\beta_{m}(\ell)^{-2\aleph}.
$$
In addition,   $\alpha_{m-1}(\ell)\beta_{m}(\ell)^{-2\aleph}\leq \gamma_{\ell-m+1}(\ep)$ since $\beta_{i_m,j_m}\geq \beta_{j_m, j_m}$. Therefore,   the bound in \eqref{e:dCondition} follows from  \eqref{e:gamma_j}  after noting that $k_mq_m \leq \step^{\ell+1}$ and  $(\ell-m+1)/\gap \blue{\geq } \ell+1$, and choosing $\e$ small enough.


Combining \eqref{e:dCondition}, \eqref{e:hypChecked}, and \eqref{e:Piterates} we may apply Lemma~\ref{l:NdqImpliesNd} and obtain $$
\rho \in \ND\sub{\fuInj}\big(n,(2\ANDrate(\beta_m(\ell),q_m)C_0^{n})^{-1},g_\ell\big).
$$
Here, we have used that, by~\eqref{e:Titerates} and
$
(n-1)\fuInj\leq t_0\leq n\fuInj,
$
we have
$$
(n-1)\fuInj -\CG\dG\leq \bm {T}\sub{\mc{I}}^{(q_mp)}[g](\rho)\leq n\fuInj+\CG\dG.
$$

Finally, using  that $\beta_m(\ell)\geq \beta_{j_m,j_m}\geq \beta_{\ell, \ell}$ one can check that  $C^{-k_0}\geq 2\gamma_\ell (5/2)^{2\aleph}q_m^{4\aleph}$. In particular, since $q_m \leq \step^{\ell+1}$, $\gap\in [0,1]$, making  $\e_0<C_0^{-1}$ \blue{implies}
$[2\mathfrak{f}(\beta_{m},q_{m})C_0^n]^{-1}\geq \alpha_{\ell+1}$,  and the proof  is complete. 
\end{proof}

\begin{lemma}\label{l:allTheIterates}
Let $\rho \in \Gamma$ satisfy \eqref{e:intermediateStep2} and \eqref{e:iterate}. Then for all $0\leq j\leq m-1$
and $0\leq q\leq t_{j}/ t_{m}$, 
\begin{equation}
\label{e:iterateEstimate}
d(\varphi^{g_\ell}_{q t_{m}}(\rho),\rho)\leq\sum_{i=j}^{m-1}C_{0}^{\sum_{l=i}^{m-1} t_{l}}\alpha_{i}(\ell).\end{equation}
\end{lemma}
\begin{proof} First, note that for $j=m-1$, the statement follows
from~\eqref{e:iterate}. Let $0\leq j\leq m-2$ and assume~\eqref{e:iterateEstimate} holds
for $0\leq q\leq t_{j+1}/ t_{m}$. Next, let $q< t_{j}/ t_{m}$
and write 
$
q t_{m}=\sum_{i=j+1}^{m}q_{i} t_{i}
$
with $q_{j}< t_{j-1}/ t_{j}$. Then, by \eqref{e:intermediateStep2} with $i=j+1$, $t=q_{j+1}t_{j+1}$, and $\mathfrak{r}=0$,
\[
d(\varphi^{g_\ell}_{q_{j+1} t_{j+1}}(\rho),\rho)\leq C_{0}^{ t_{j}}\alpha_{j}(\ell).
\]
Next, note that $g_\ell=F_\ell(g, \bm \sigma) \in \tilde K$ since  $g\in K$. In particular,  $\|\varphi_{\fuInj}^{g_\ell}\|_{C^{\nu-2}}\leq C\|g_\ell\|_{C^{\blue{\nu-1}}} \leq C_0$ for all $\ell$. It follows that  
\[
d\big(\varphi^{g_\ell}_{\sum_{i=j+1}^{m}q_{i} t_{i}}(\rho),\varphi^{g_\ell}_{\sum_{l=j+2}^{m}q_{l} t_{l}}(\rho)\big)
\leq C_{0}^{ t_{j}+\sum_{l=j+2}^{m}q_{l} t_{l}}\alpha_{j}(\ell)\leq C_{0}^{\sum_{l=j}^{m-1} t_{l}}\alpha_{j}(\ell).
\]
By the induction hypothesis, 
$
d\big(\varphi^{g_\ell}_{\sum_{i=j+2}^{m}q_{i} t_{i}}(\rho),\rho \big)\leq\sum_{i=j+1}^{m-1}C_{0}^{\sum_{l=i}^{m-1} t_{l}}\alpha_{i}(\ell).
$
Thus,
\[
d(\varphi^{g_\ell}_{\sum_{i=j+1}^{m+1}q_{i} t_{i}}(\rho),\rho)\leq C_{0}^{\sum_{l=j}^{m-1} t_{l}}\alpha_{j}(\ell)+\sum_{i=j+1}^{m-1}C_{0}^{\sum_{l=i}^{m-1} t_{l}}\alpha_{i}(\ell)=\sum_{i=j}^{m-1}C_{0}^{\sum_{l=i}^{m-1} t_{l}}\alpha_{i}(\ell)
\]
and the claim follows by induction.
\end{proof}

\subsubsection{Step 2: Preserving non-degeneracy under perturbation} 
In this section, we show that if the non-degeneracy properties \eqref{e:ind1} and \eqref{e:ind2} listed in Proposition \ref{p:theInduction} hold for $g_\ell:=F_{\ell}(g,\bm{\sigma})$, then they also hold for the perturbed metric $(g_\ell)_{\bm\sigma}^{R,\delta}:=Q^{R,\delta}(g_\ell,\bm\sigma)$ for appropriate $R,\delta$.

\begin{lemma}\label{l:Step2} 
There is $\e_0>0$ depending only on $\nu, N, \ga, b, \step, \fuInj, \Gamma, G, K$ so that the following holds.
Suppose that the conclusions in Proposition \ref{p:theInduction} are valid up to the index $\ell$ and let $0<\e<\e_0$. Then, for all $(g, \bm{\sigma})\in K \times \bm{\Sigma}_{\ell}(\bm{\Rind}_\e)$ such that $g_\ell:=F_{\ell}(g,\bm{\sigma})\in \ms{G}^\nu \backslash  L\sub{\blue{\fuInj},\ell}$ the following holds.
Let  $(g_\ell)_{\bm\sigma}^{R,\delta}:=Q^{R,\delta}(g_\ell,\bm\sigma),$ 
with  $R=R_{\ell+1}(\e)$ and $\delta=\delta_{\ell+1}(\e)$ as defined in \eqref{e:R-delta}. Then, with $C_0$ as in \eqref{e:constants},
\begin{equation}
\|(g_\ell)_{\bm\sigma}^{R,\delta}-g_{\ell}\|_{\mc{C}^{\nu}}\leq C_0 2^{-\ell-2}\e,\label{e:Step2a}
\end{equation}
and
\begin{itemize}
\item for $0\leq i\leq \ell$,\; $\step^{i-1}<n\leq\step^{i}$,
\begin{equation}
\Gamma \cap\Rec\sub{\fuInj}(n,\tilde{\beta}_{i,\ell+1},(g_{\ell})_{\bm{\sigma}}^{R,\delta})\subset\ND\sub{\fuInj}(n,\tilde{\beta}_{i,\ell+1},(g_{\ell})_{\bm{\sigma}}^{R,\delta}).\label{e:Step2b}
\end{equation}
\item for $\step^{\ell}<n\leq\step^{\ell+1}$ 
\begin{equation}
\Gamma \cap\Rec\sub{\fuInj}(n,\tfrac{1}{2}\alpha_{\ell+1},(g_{\ell})_{\bm{\sigma}}^{R,\delta})\setminus\Sim\sub{\!\fuInj}(n,\tfrac{1}{2}\alpha_{\ell+1},(g_{\ell})_{\bm{\sigma}}^{R,\delta})\subset\ND\sub{\fuInj}(n,\tfrac{1}{2}\alpha_{\ell+1},(g_{\ell})_{\bm{\sigma}}^{R,\delta}).\label{e:Step2c}
\end{equation}
\item for $\gap(\ell+1)\leq j\leq \ell$,\; $0\leq i\leq j$, \;$\step^{i-1}<n\leq\step^{i}$,
and  $1\leq q\leq\step^{j/\gap-i}$, 
\begin{equation}
\Gamma \cap\Rec\sub{\fuInj}(n,\beta_{i,j,\ell+1}q^{-3},(g_{\ell})_{\bm{\sigma}}^{R,\delta})\cap\Sim\sub{\!\fuInj}(n,\alpha_{j,\ell+1},(g_{\ell})_{\bm{\sigma}}^{R,\delta})\subset\ND_{q,_{\fuInj}}(n,\beta_{i,j,\ell+1},(g_{\ell})_{\bm{\sigma}}^{R,\delta}).\label{e:Step2d}
\end{equation}
\end{itemize}
\end{lemma}

\begin{proof}
 By Definition~\ref{ass:1} and  \eqref{e:constants}, the bound in \eqref{e:Step2a} holds since 
\begin{equation}\label{e:metricD}
 \|(g_\ell)_{\bm\sigma}^{R,\delta}-g_{\ell}\|_{\mc{C}^{\nu}}=\|Q^{R,\delta}(g_\ell, \bm\sigma)-Q^{R,\delta}(g_\ell, 0)\|_{\mc{C}^{\nu}} \leq{ \tfrac{1}{2}}C_0\delta R^{-\vartheta_\nu}
 \end{equation}
 and,  by definition \blue{(see~\eqref{e:R-delta})},
\begin{equation}
\delta\leq {\e R^{\vartheta_\nu}}{2^{-(\ell+1)}}.\label{e:deltaSmall-0}
\end{equation}

We next address~\eqref{e:Step2b}.  Let $0\leq i\leq \ell$,\; $\step^{i-1}<n\leq\step^{i}$.
We claim that 
\begin{equation}
\begin{gathered}\Rec\sub{\fuInj}(n,\tilde{\beta}_{i,\ell+1},(g_{\ell})_{\bm{\sigma}}^{R,\delta})\subset\Rec\sub{\fuInj}(n,\tilde{\beta}_{i,\ell},g_{\ell}),\\
\Gamma \cap\Rec\sub{\fuInj}(n,\tilde{\beta}_{i,\ell},g_{\ell})\cap\ND\sub{\fuInj}(n,\tilde{\beta}_{i,\ell},g_{\ell})\subset \ND\sub{\fuInj}(n,\tilde{\beta}_{i,\ell+1},(g_{\ell})_{\bm{\sigma}}^{R,\delta}),
\end{gathered}\label{e:soup}
\end{equation}
provided 
\begin{equation}
\delta\leq C_{0}^{-\step^{i}\fuInj-1}\blue{R^{{\vartheta_2}}}\tilde{\beta}_{i,\ell}(1-s_{\ell}),\qquad0\leq i\leq \ell.\label{e:salt}
\end{equation}
Once we prove \eqref{e:soup}, the claim in~\eqref{e:Step2b} will follow by applying the inductive assumption~\eqref{e:ind1} to $g_\ell$. \blue{The estimate \eqref{e:salt} follows from the definition of $\delta=\delta_{\ell+1}$ \eqref{e:R-delta}} after asking that $\e_0<C_0^{-\blue{\max(1,\fuInj)}}$.
\blue{Indeed, using the definition of $R_{\ell+1}$ \eqref{e:R-delta}, $\alpha_{\ell+1}$,~\eqref{e:alphas}, and $\gamma_\ell$~\eqref{e:gamma_j}, we have
$$
R_{\ell+1}=\tfrac{1}{8}5^{-2\aleph}\step^{-4\aleph(\ell+1)/\gap}\e^{\step^{\ell+1}+1+(2\aleph+1)\step^{(\ell+1)/\gap}}\beta_{\ell,\ell}^{2\aleph},
$$
and 
\begin{align}
\delta_{\ell+1}R_{\ell+1}^{-\vartheta_2}&\leq \e 2^{-\ell-1}R_{\ell+1}=\tfrac{1}{8}{2^{-\ell-1}}5^{-2\aleph}\step^{-4\aleph(\ell+1)/\gap}\e^{\step^{\ell+1}+1+(2\aleph+1)\step^{(\ell+1)/\gap}+1}\beta_{\ell,\ell}^{2\aleph}\notag\\
&\leq 2^{-\ell-1}\e^{\step^{\ell+1}+2+(2\aleph +1)\step^{(\ell+1)/\gap}}\beta_{\ell,\ell}^{2\aleph}. \label{e:RDelta}
\end{align}
On the other hand, using the definition of $\tilde{\beta}_{i,\ell}$, $s_\ell$, and $\tilde{\beta}_{i,i}$ \eqref{e:alphas}, we have
\begin{equation} 
\label{e:deltaInequalityToday}
C_{0}^{-\step^{i}\fuInj-1}\tilde{\beta}_{i,\ell}(1-s_{\ell})=C_0^{-\step^{i}\fuInj-1}\tfrac{2^{-\ell}-2^{-\ell-1}}{1+2^{-i}}\e^{2\aleph-1}\beta_{i,i}^{2\aleph}\geq C_0^{-\step^\ell\fuInj-1}2^{-\ell}\e^{2\aleph-1}\beta_{\ell,\ell}^{2\aleph}.
\end{equation}
}

To see the first part of~\eqref{e:soup} suppose that $\rho\in\Rec\sub{\fuInj}(n,\tilde{\beta}_{i,\ell+1},(g_{\ell})_{\bm{\sigma}}^{R,\delta})$.
Then, $\rho\in\Rec\sub{\fuInj}(n,\tilde{\beta}_{i,\ell},g_{\ell})$ since  
Definition~\ref{ass:1}, \eqref{e:salt}, and \eqref{e:constants} imply
\begin{equation}\label{e:kappasS}
\big\|\varphi_t^{g_\ell}-\varphi_t^{(g_{\ell})_{\bm{\sigma}}^{R,\delta}}\big\|_{\mc{C}^{\blue{0}}}\leq C_0^t\big\|g_\ell-(g_{\ell})_{\bm{\sigma}}^{R,\delta}\big\|_{\blue{\mc{C}^{1}}}\leq C_{0}^{t+1}\delta \blue{R^{-\vartheta_1}}<\tilde{\beta}_{i,\ell}(1-s_{\ell})=\tilde{\beta}_{i,\ell}-\tilde{\beta}_{i,\ell+1},
\end{equation}
provided $(n-1)\fuInj \leq t \leq n\fuInj$ and \eqref{e:salt} holds.

For the second part of~\eqref{e:soup}, note that 
with $\dG$ as in Lemma \ref{l:pLives} and \eqref{e:Cdelta},
if ${\tilde{\beta}_{i,\ell}}<\dG$,
we may work in a single coordinate chart $(U,\psi)$.
Suppose $\rho\in \Gamma \cap \Rec\sub{\fuInj}(n,\tilde{\beta}_{i,\ell},g_{\ell})\cap\ND\sub{\fuInj}(n,\tilde{\beta}_{i,\ell},g_{\ell})$.
By Definition \ref{d:ND} and Lemma~\ref{l:pLives} there exist $\mc{I}$ and $m$ such that
$\bm{T}\sub{\mc{I}}^{(m)}[g_\ell](\rho) \in \big[(n-1)\fuInj-\CG \tilde{\beta}_{i,\ell}\,,\, n\fuInj+\CG \tilde{\beta}_{i,\ell} \big]$  and
\begin{equation}\label{e:betaBound}
{\sup_{\{(\psi,U)\in\mc{A}:\,B(\rho,\CG\dG)\subset U\}}}
\big\|\big(\Id-d((\mc{P}\sub{\mc{I}}^{(m)}[g_\ell])_\psi)(\rho)\big)^{-1}\big\|<\tilde{\beta}_{i,\ell}^{-1}.
\end{equation}
In addition, {since otherwise we may modify the choice of chain}, we may assume without loss of generality that 
\begin{equation}\label{e:fine}
d(\mc{P}\sub{\mc{I}}^{(j)}[g_\ell](\rho),\Gamma)<\dG,\qquad 0\leq j\leq m.
\end{equation}
Then, by Lemma~\ref{l:whatever} there is $c>0$ depending only on $G, \Gamma, \fuInj$ such that 
$$
\rho\in \mc{D}\sub{\mc{I}}^{(m)}[(g_\ell)_{\bm{\sigma}}^{R,\delta}],
$$
as long as $\|(g_\ell)_{\bm\sigma}^{R,\delta}-g_{\ell}\|_{\mc{C}^{\blue{1}}} \leq c^m$.
Here, we use that \eqref{e:metricD} holds with \blue{$\nu=1$} together with \eqref{e:R-delta} and \blue{$\nu\geq 3$ and $\vartheta_\nu\geq \vartheta_1$} to justify that  $\e_0$ can be chosen small enough that  $ \|(g_\ell)_{\bm\sigma}^{R,\delta}-g_{\ell}\|_{\mc{C}^{\blue{1}}} \leq c^m$.
Next,
since $\bm{T}\sub{\mc{I}}^{(m)}[g_\ell](\rho) \in \big[(n-1)\fuInj-\CG \tilde{\beta}_{i,\ell}\,,\, n\fuInj+\CG \tilde{\beta}_{i,\ell} \big]$, we have
\begin{equation}\label{e:timesEll}
\bm{T}\sub{\mc{I}}^{(m)}[(g_\ell)_{\bm\sigma}^{R,\delta}](\rho)\in [(n-1)\fuInj - \CG\tilde{\beta}_{i,\ell}-C_0^{n}\delta \blue{R^{-\vartheta_2}}\, ,\, n\fuInj +\CG\tilde{\beta}_{i,\ell}+C_0^{n\fuInj+2}\delta \blue{R^{-\vartheta_2}}].
\end{equation}
Indeed, by~\eqref{e:kappasS}, the implicit function theorem, and that $H\sub{|\xi|_g}$ is uniformly transverse to $\Gamma$,
$$
|\bm{T}\sub{\mc{I}}^{(m)}[(g_\ell)_{\bm\sigma}^{R,\delta}](\rho)-\bm{T}\sub{\mc{I}}^{(m)}[g_\ell](\rho)|\leq \CG C_0^{n\fuInj+1}\delta \blue{R^{-\vartheta_2}}
$$
and~\eqref{e:timesEll} follows. Now, since for $\e$ small enough,   
$$
\CG C_0^{n\fuInj +1}\delta \blue{R^{-\vartheta_2}}+\CG\tilde{\beta}_{i,\ell}<\CG \tilde{\beta}_{i,\ell+1}<\CG\dG,
$$
we have
\begin{equation}\label{e:timesEll2}
\bm{T}\sub{\mc{I}}^{(m)}[(g_\ell)_{\bm\sigma}^{R,\delta}](\rho)\in [(n-1)\fuInj - \CG\dG\, ,\, n\fuInj +\CG\dG].
\end{equation}


Let $(\psi,U)\in\mc{A}$ be such that $B(\rho,\CG\dG)\subset U$.
Abusing notation slightly, we work with $\mc{P}\sub{\mc{I}}^{(m)}[g_\ell]$, and $\mc{P}\sub{\mc{I}}^{(m)}[(g_\ell)_{\bm\sigma}^{R,\delta}]$
for the maps induced in the $(\psi, U)$ coordinates.  Let ${\bf {A}}=\Id-d\mc{P}\sub{\mc{I}}^{(m)}[g_\ell](\rho)$
and ${\bm{\Delta}}=d\mc{P}\sub{\mc{I}}^{(m)}[g_\ell](\rho)-d\mc{P}\sub{\mc{I}}^{(m)}[(g_\ell)_{\bm\sigma}^{R,\delta}](\rho)$,
and observe that \eqref{e:kappasS} yields
$
\|{\bm{\Delta}}\|
\leq\tilde{\beta}_{i,\ell}(1-s_{\ell}).
$
In particular, by \eqref{e:betaBound},
\[
\|({\bm{A}}+{\bm{\Delta}})^{-1}\|\leq\frac{\|{\bm{A}}^{-1}\|}{1-\|{\bm{A}}^{-1}\|\|{\bm{\Delta}}\|}\leq\frac{\tilde{\beta}_{i,\ell}^{-1}}{1-\tilde{\beta}_{i,\ell}^{-1}\tilde{\beta}_{i,\ell}(1-s_{\ell})}=\tilde{\beta}_{i,\ell}^{-1}s_{\ell}^{-1}=\tilde{\beta}_{i,\ell+1}^{-1}.
\]
Combining this with \eqref{e:timesEll2} we conclude
$
\rho\in\ND\sub{\fuInj}(n,\tilde{\beta}_{i,\ell+1},(g_{\ell})_{\bm{\sigma}}^{R,\delta}),
$
and hence the second part of~\eqref{e:soup}. As noted earlier, this shows that \eqref{e:Step2b} holds.

We next prove~\eqref{e:Step2c}. Let $\step^{\ell}<n\leq\step^{\ell+1}$. We will show that
\begin{equation}
\begin{gathered}\Rec\sub{\fuInj}(n,\tfrac{1}{2}\alpha_{\ell+1},(g_{\ell})_{\bm{\sigma}}^{R,\delta})\subset\Rec\sub{\fuInj}(n,\alpha_{\ell+1},g_{\ell}),\\
\Rec\sub{\fuInj}(n,\alpha_{\ell+1},g_{\ell})\cap \ND\sub{\fuInj}(n,\alpha_{\ell+1},g_{\ell})\subset\ND\sub{\fuInj}(n,\tfrac{1}{2}\alpha_{\ell+1},(g_{\ell})_{\bm{\sigma}}^{R,\delta}),
\end{gathered}\label{e:risotto}
\end{equation}
\begin{equation}
\Sim\sub{\!\fuInj}(n,\alpha_{\ell+1},g_{\ell})\subset\Sim\sub{\!\fuInj}(n,\tfrac{1}{2}\alpha_{\ell+1},(g_{\ell})_{\bm{\sigma}}^{R,\delta}),\label{e:croutons}
\end{equation}
provided 
\begin{equation}
\delta\leq C_{0}^{-\step^{\ell+1}\fuInj -1}\blue{R^{{\vartheta_2}}}\tfrac{1}{2}\alpha_{\ell+1}.\label{e:pepper1}
\end{equation}
Note that combining \eqref{e:risotto} and \eqref{e:croutons} with
Lemma \ref{l:Step1} yields the claim in \eqref{e:Step2c}.  Also, it is immediate to check \blue{from the definition of $\delta=\delta_{\ell+1}$ \eqref{e:R-delta}} that \eqref{e:pepper1} is valid after asking $\e_0\leq C_0^{-\blue{\max(\fuInj,1)}}$ \blue{after using~\eqref{e:deltaInequalityToday} and $\alpha_{\ell+1}=\gamma_{\ell}\beta_{\ell,\ell}^{2\aleph}$ and the definition of $\gamma_\ell$~\eqref{e:gamma_j}}.

The claims in~\eqref{e:risotto} follow from the same proof that yielded~\eqref{e:soup}.
To see that~\eqref{e:croutons} holds, observe that if $\rho\notin\Sim\sub{\!\fuInj}(n,\tfrac{1}{2}\alpha_{\ell+1},(g_{\ell})_{\bm{\sigma}}^{R,\delta}))$,
then there is $t \in (\frac{1}{2}\fuInj, (n-\frac{1}{2})\fuInj)$ such that 
$
d\big(\varphi_{t}^{(g_{\ell})_{\bm{\sigma}}^{R,\delta}}(\rho),\rho\big)<\tfrac{1}{2}\alpha_{\ell+1}.
$
Thus, by \eqref{e:kappasS} and \eqref{e:pepper1}, we know
$
\|\varphi_{t}^{g_\ell}-\varphi_{t}^{(g_{\ell})_{\bm{\sigma}}^{R,\delta})}\|_{\mc{C}^{1}}\leq\tfrac{1}{2}\alpha_{\ell+1}.
$
In particular,  $d\big(\varphi_{t}^{g_\ell}(\rho),\rho\big)<\alpha_{\ell+1},$ and hence
$\rho\notin\Sim\sub\fuInj(n,\alpha_{\ell+1},\kappa_{\ell})$ as needed.

Finally, we address~\eqref{e:Step2d}. Let   $\gap(\ell+1)\leq j\leq \ell$, $0\leq i\leq j$, \;$\step^{i-1}<n\leq\step^{i}$,\;
and  $1\leq q\leq\step^{j/\gap-i}$.
We claim that, provided \eqref{e:dahl} and \eqref{e:oregano} below hold, 
\begin{equation}
\begin{gathered}\Rec\sub{\fuInj}(n,\beta_{i,j,\ell+1}q^{-3},(g_{\ell})_{\bm{\sigma}}^{R,\delta})\subset\Rec\sub{\fuInj}(n,\beta_{i,j,\ell}q^{-3},g_{\ell}), \quad
\Sim\sub{\!\fuInj}(n,\alpha_{j,\ell+1},(g_{\ell})_{\bm{\sigma}}^{R,\delta})\subset \Sim\sub{\!\fuInj}(n,\alpha_{j,\ell},g_{\ell})\label{e:spoon}
\end{gathered}
\end{equation}
\begin{equation}
\Gamma \cap\Rec\sub{\fuInj}(n,\beta_{i,j,\ell}q^{-3},g_{\ell})\cap\ND_{q,_\fuInj}(n,\beta_{i,j,\ell},g_{\ell})\subset\ND_{q,_\fuInj}(n,\beta_{i,j,\ell+1},(g_{\ell})_{\bm{\sigma}}^{R,\delta}).\label{e:pepper}
\end{equation}
Note that combining \eqref{e:spoon} and \eqref{e:pepper} with
 \eqref{e:ind2} yields the claim in \eqref{e:Step2d}.

For~\eqref{e:spoon} we argue as in the proofs of the first part of~\eqref{e:soup} and ~\eqref{e:croutons}, and use that $\alpha_{j, \ell+1} \geq \tfrac{1}{2}\alpha_{j,\ell}$. We can do this provided 
\begin{equation}
\delta<\min\Big(C_{0}^{-\step^{\ell}\fuInj-1}\blue{R^{{\vartheta_2}}}(\beta_{i,j,\ell}-\beta_{i,j,\ell+1})(\step^{j/\gap-i})^{-3}\;,\;C_{0}^{-\step^\ell\fuInj-1}\blue{R^{{\vartheta_2}}}\tfrac{1}{2}\alpha_{j,\ell}\Big).\label{e:dahl}
\end{equation}
We note \blue{the estimate \eqref{e:dahl} follows from the definition of $\delta=\delta_{\ell+1}$ \eqref{e:R-delta}}  for $\e_0$ small enough since  $\{\alpha_j\}_j$ is decreasing and $\beta_{\ell,\ell}\leq \beta_{i,j}$. \blue{Indeed, observe that 
$$
\beta_{i,j,\ell}-\beta_{i,j,\ell+1}=2^{j-(\ell+1)}\beta_{i,j}\geq 2^{j-\ell-1}\beta_{\ell,\ell}
$$
and
$$
\alpha_{j,\ell}=(1-2^{-(\ell-j)-1})\alpha_j\geq (1-2^{-(\ell-j)-1})\alpha_{\ell}=(1-2^{-(\ell-j)-1})\gamma_{\ell-1}\beta_{\ell-1,\ell-1}^{2\aleph}.
$$
combining these with the bounds in \eqref{e:RDelta} and the definition of $\gamma_{\ell-1}$~\eqref{e:gamma_j} yields \eqref{e:dahl}.}

The inclusion in~\eqref{e:pepper} will follow from Lemma~\ref{l:ndqPerturb} below with $\beta=\beta_{i,j, \ell}$.
Indeed, provided 
\begin{equation}
\label{e:necessary}
\begin{gathered}\|g_\ell-(g_\ell)_{\bm\sigma}^{R,\delta}\|_{\blue{\mc{C}^2}}\leq \min(\blue{C_0}^{-\blue{2}nq}\beta^{2\aleph}q^{-4\aleph-1},\delta_G-\beta),\\
\blue{\beta_{i,j,\ell}(1-\|g_\ell-(g_\ell)_{\bm\sigma}^{R,\delta}\|_{\blue{\mc{C}^2}}\beta_{i,j,\ell}^{-2\aleph}C_1^{nq}q^{4\aleph+1})^{1/2\aleph}\geq \beta_{i,j,\ell+1}},
\end{gathered}
\end{equation} 
we have by Lemma~\ref{l:ndqPerturb}
\begin{equation}
\label{e:empire}
\begin{aligned}
\Gamma\cap\Rec\sub{\fuInj}&(n,\beta_{i,j,\ell}q^{-3},g_{\ell})\cap \ND_{q,_\fuInj}(n,\beta_{i,j,\ell},g_{\ell})\\
&\qquad \subset \Gamma\cap\Rec\sub{\fuInj}(n,\beta_{i,j,\ell},g_{\ell})\cap \ND_{q,_\fuInj}(n,\beta_{i,j,\ell},g_{\ell})\\
&\qquad \subset \ND_{q,_\fuInj}(n,\beta_{i,j,\ell}(1-\|g_\ell-(g_\ell)_{\bm\sigma}^{R,\delta}\|_{\blue{\mc{C}^2}}\beta_{i,j,\ell}^{-2\aleph}\blue{C_0}^{nq}q^{4\aleph+1})^{1/2\aleph},(g_\ell)_{\bm\sigma}^{R,\delta})\\
&\qquad\subset \blue{\ND_{q,_\fuInj}(n,\beta_{i,j,\ell+1},(g_\ell)_{\bm\sigma}^{R,\delta})},
\end{aligned}
\end{equation}
which is~\eqref{e:pepper}.

\blue{To check that ~\eqref{e:necessary} holds, observe that
\begin{equation}
\delta<\blue{R^{{\vartheta_2}}}C_0^{-2\step^{j/\gap}-1}(\step^{j/b-i})^{-4\aleph\blue{-1}}\beta_{i,j,\ell}^{2\aleph}(1-s_{\ell-j}^{2\aleph})\label{e:oregano}
\end{equation}
(this follows in almost the same way as~\eqref{e:salt}) and hence
$$
\|g_\ell -(g_\ell)_{\bm\sigma}^{R,\delta}\|_{\blue{\mc{C}^2}}\leq \tfrac{1}{2}C_0\delta R^{-\vartheta_2}\leq \tfrac{1}{2}C_0^{-2\step^{j/\gap}}(\step^{j/b-i})^{-4\aleph-1}\beta_{i,j,\ell}^{2\aleph}(1-s_{\ell-j}^{2\aleph})\leq C_0^{-2q}\beta^{2\aleph}q^{-4\aleph-1}.
$$
In particular (up to possibly shrinking $\e_0$ in a way depending on $C_0$)
$$
\beta_{i,j,\ell}+\|g_\ell-(g_\ell)_{\bm \sigma}^{R,\delta}\|\leq \tfrac{\delta_G}{2}.
$$
Hence, the first part of~\eqref{e:necessary} holds. It remains only to check the second inequality. For this, observe that by~\eqref{e:alphas}
$$
\frac{\beta_{i,j,\ell+1}}{\beta_{i,j,\ell}}=1-\frac{2^{j-\ell-1}}{1+2^{j-\ell}} 
$$
and hence we require 
$$
\|g_\ell-(g_\ell)_{\bm\sigma}^{R,\delta}\|_{\blue{\mc{C}^2}}\beta_{i,j,\ell}^{-2\aleph}C_0^{nq}q^{4\aleph+1}\leq 1-\big(1-\tfrac{2^{j-\ell-1}}{1+2^{j-\ell}}\big)^{2\aleph}.
$$
This follows from~\eqref{e:oregano}, completing the proof of the lemma.}
\end{proof}

\begin{lemma}\label{l:ndqPerturb} {Let $K\subset \blue{\ms{G}^3}$ bounded and $C_1>0$ as above (i.e. as in~\eqref{e:constants}). Then for  $g_{1},g_2\in K$, $\beta\in(0,\dG)$, $q\in\mathbb{N}$, and $\ep>0$
satisfying 
\[
\|g_{1}-g_{2}\|_{\blue{\mc{C}^{2}}}\leq\e,\qquad\qquad\e\leq \min(\blue{C_0^{-2nq}}\beta^{2\aleph}q^{-4\aleph-1}, \dG-\beta)
\]
we have}
\[
\Gamma\cap\Rec(n,\beta,g_1)\cap\ND_{q,_\fuInj}(n,\beta,g_{1})\subset\ND_{q,_\fuInj}(n,\beta(1-\e\beta^{-2\aleph}\blue{C_0^{nq}}q^{4\aleph+1})^{1/2\aleph},g_{2}).
\]
\end{lemma} 
\begin{proof} 
Let $\rho \in \Gamma\cap\Rec(n,\beta,g_1)\cap\ND_{q,_\fuInj}(n,\beta,g_{1})$ and $(\psi,U)\in\mc{A}$ such that $B(\rho,\CG\dG)\subset U$.  Then, there are $\mc{I}$ and $m$ such that
$\bm{T}\sub{\mc{I}}^{(m)}[g_1](\rho) \in \big[(n-1)\fuInj-\CG{\beta}\,,\, n\fuInj+\CG{\beta} \big]$,
{
$$
\sup_{m\leq n}d(\mc{P}\sub{\mc{I}}^{(m)}[g_1](\rho),\Gamma)<\dG,
$$
}
and
\begin{equation}
\frac{\big\|(\Id-[(d((\mc{P}\sub{\mc{I}}^{(m)}[g_1])_\psi)(\rho)]^{q})^{-1}\big\|}{(1+\|d((\mc{P}\sub{\mc{I}}^{(m)}[g_1])_\psi)(\rho)\|^{q})^{2\aleph-1}}\leq(\tfrac{5}{2} q^2\beta^{-1})^{2\aleph}.\label{e:ND_q2}
\end{equation}
Note that
$
(1+\|d((\mc{P}\sub{\mc{I}}^{(m)}[g])_\psi)(\rho')\|^{q})^{2\aleph-1} \leq {\blue{C_0^{nq}}}
$
for all $g\in K$ and {$\rho' \in \mc{D}\sub{\mc{I}}^{(m)}[g].$}

{Now, by Lemma~\ref{l:whatever}, $\rho\in \mc{D}\sub{\mc{I}}^{(m)}[g_2]$.}
Therefore, defining
$$
{\bf {A}}_{q}=\Id-[(d((\mc{P}\sub{\mc{I}}^{(m)}[g_1])_\psi)(\rho)]^{q},
\qquad
{\bm{\Delta}}_{q}=[(d((\mc{P}\sub{\mc{I}}^{(m)}[g_1])_\psi)(\rho)]^{q}-[(d((\mc{P}\sub{\mc{I}}^{(m)}[g_2])_\psi)(\rho)]^{q},
$$
\eqref{e:ND_q2} yields 
$
\|{\bf {A}}_{q}^{-1}\|\leq\frac{5^{2\aleph}}{2^{2\aleph}}\blue{C_0^{nq}}\beta^{-2\aleph}q^{4\aleph}.
$
Next, note that 
{\begin{align*}
\|{\bm{\Delta}}_{q}\|&= \Big\|\sum_{j=\blue{0}}^{\blue{q-1}}\Big[(d((\mc{P}\sub{\mc{I}}^{(m)}[g_1])_\psi)(\rho)\Big]^{q\blue{-1}-j}\Big[d((\mc{P}\sub{\mc{I}}^{(m)}[g_1])_\psi)(\rho)-(d((\mc{P}\sub{\mc{I}}^{(m)}[g_1])_\psi)(\rho)\Big]\Big[(d((\mc{P}\sub{\mc{I}}^{(m)}[g_2])_\psi)(\rho)\Big]^{j}\Big\|\\
&\leq q\blue{C_0^{n\blue{(q-1)}}}\|\mc{P}\sub{\mc{I}}^{(m)}[g_1]-\mc{P}\sub{\mc{I}}^{(m)}[g_2]\|_{\mc{C}^{1}}\leq q\blue{C_0^{n\blue{q}}}\e,
\end{align*}}
 where we take the $\mc{C}^1$ norm on the intersection of the two domains and {the last inequality follows from Lemma~\ref{l:PoincareDer}}.
In particular, 
\[
\frac{\big\|(\Id-[(d((\mc{P}\sub{\mc{I}}^{(m)}[g_2])_\psi)(\rho)]^{q})^{-1}\big\|}{(1+\|d((\mc{P}\sub{\mc{I}}^{(m)}[g_2])_\psi)(\rho)\|^{q})^{2\aleph-1}}
\leq\|({\bm{A}}_{q}+{\bm{\Delta}}_{q})^{-1}\|\leq\frac{\|{\bm{A}}_{q}^{-1}\|}{1-\|{\bm{A}}_{q}^{-1}\|q\e C_{0}^{n(q+1)}}
\]
which yields the desired bound since $q\e C_{0}^{n\blue{q}}\leq\frac{1}{2}\frac{2^{2\aleph}}{5^{2\aleph}}\blue{C_0^{-nq}}\beta^{2\aleph}q^{-4\aleph}\leq\frac{1}{2}\|{\bm{A}}_{q}^{-1}\|^{-1}$.

{Finally, we also note that  
$$
{(n-1)\fuInj - \CG\dG\leq}(n-1)\fuInj - \CG(\beta+\e)\leq \bm{T}\sub{\mc{I}}^{(m)}[g_2](\rho)\leq n\fuInj +\CG(\beta+\e){\leq n\fuInj +\CG\dG},
$$
and hence, $\rho\in \ND_{q,_\fuInj}(n,\beta(1-\e\beta^{-2\aleph}\blue{C_0^{nq}}q^{4\aleph+1})^{1/2\aleph},g_{2})$ as claimed.}
\end{proof}

\subsubsection{Step 3: Controlling the volume of bad perturbations}\label{s:badMeasure}

This step is dedicated to finishing the induction argument that yields the proof of Proposition \ref{p:theInduction}.  We still assume the conclusions of Proposition \ref{p:theInduction} hold up to some $\ell\geq {\lceil\log_\step 2\rceil-1}$, and will show they also hold up to $\ell+1$.

In what follows, we work with
$R:={\Rtemp_{\ell+1}}$,  $\delta:={\delta_{\ell+1}}$ \blue{(see~\eqref{e:R-delta})}, and for $0\leq i\leq \ell+1$ set
\begin{equation}
\label{e:upsilon}
\beta_{i,\ell+1}=\tfrac{1}{2}\min_{\step^{i-1}<n\leq\step^{i}}\upsilon_{n}(\ell), \qquad \upsilon_{n}(\ell):=2\beta_{\ell+1,\ell+1} {\e}^{(n-\step^{\ell+1})\dim\sp(2\aleph)}.
\end{equation}
The first step is to show that for  $(g, \bm{\sigma})\in K \times \bm{\Sigma}_{\ell}(\bm{\Rind}_\e)$ such that $F_{\ell}(g,\bm{\sigma})\in \ms{G}^\nu \backslash  L\sub{\blue{\fuInj},\ell}$,
\begin{equation}\label{e:window}
\Big\{\bm{\sigma}\sub{\!\ell+1}\in\Sigma(\Rtemp_{\ell+1})\,:\,F_{\ell+1}(g,(\bm{\sigma},\bm{\sigma}\sub{\!\ell+1}))\in L\sub{\blue{\fuInj},\ell+1}\Big\} \subset \bigcup_{n=0}^{{\step^{\ell+1}}}{\mc{S}}_{F_{\ell}(g,\hat{\bm{\sigma}}_{\ell})}^{R,\delta}(n, \tfrac{1}{2}\alpha_{\ell+1}, \upsilon_n(\ell))
\end{equation}
for ${\mc{S}}_{F_{\ell}(g,\hat{\bm{\sigma}}_{\ell})}^{R,\delta}$ as defined in Corollary~\ref{c:fullPerturb}.

To prove \eqref{e:window}, first note that  if $\bm\sigma\sub{\ell+1}$ is such that $$g_{\ell+1}:=F_{\ell+1}(g,(\bm{\sigma},\bm{\sigma}\sub{\!\ell+1}))=Q^{R,\delta}(F_{\ell}(g,\bm{\sigma}),\bm{\sigma}\sub{\!\ell+1})\in L\sub{\blue{\fuInj},\ell+1},$$
by the definition of $ L\sub{\blue{\fuInj},\ell+1}$,
then either \eqref{e:ind1} or \eqref{e:ind2} does not hold with $(g,\ell)$ \blue{replaced by} $(g_{\ell+1}, \ell+1)$. We claim that in this setting, one actually has that \eqref{e:ind2} does not hold with $(g,\ell,j)$ \blue{replaced by} $(g_{\ell+1}, \ell+1, \ell+1)$. That is, there exist $0\leq i\leq \ell+1$, $\step^{i-1}<n\leq \step^i$, and $1\leq q\leq \step^{(\ell+1)/\gap-i+1}$,  such that
\begin{equation}
\Gamma\cap \Rec\sub{\fuInj}(n,\beta_{i,\ell+1,\ell+1}q^{-3},g_{\ell+1})\cap\Sim\sub{\!\fuInj}(n,\alpha_{\ell+1,\ell+1},g_{\ell+1})\nsubseteq\ND_{q,_\fuInj}(n,\beta_{i,\ell+1,\ell+1},g_{\ell+1}).\label{e:wiggle}
\end{equation}

Indeed, suppose \eqref{e:ind1} does not hold with $(g,\ell)$ \blue{replaced by} $(g_{\ell+1}, \ell+1)$.  Then, since $F_{\ell}(g,\bm{\sigma})\in \ms{G}^\nu \backslash  L\sub{\blue{\fuInj},\ell}$, by Lemma \ref{l:Step2} equation \eqref{e:Step2b}, \eqref{e:ind1} can only fail when $i=\ell+1$. Thus,
there are $\step^\ell<n\leq \step^{\ell+1}$ and  
$$
\rho \in \Gamma \cap\Rec\sub{\fuInj}(n, \tilde \beta_{\ell+1 , \ell+1}, g_{\ell+1}) \backslash \ND\sub{\fuInj}(n, \tilde \beta_{\ell+1 , \ell+1}, g_{\ell+1}).
$$

\blue{Notice that, since $\mathfrak{c}\geq 1$, for $\bm d$ large enough depending on $\step$, $\gap$, we have
\begin{align*}
    \tilde{\beta}_{\ell+1 , \ell+1}
    &=\beta_{\ell+1,\ell+1}^{2\aleph}\e^{2\aleph-1}=\beta_{\ell,\ell}^{\mathfrak{c}2\aleph}\e^{2\aleph-1+\bm{d}\step^{\ell/\gap}} \\
    &\leq 5^{-2\aleph}\e^{(2\aleph+1)\step^{(\ell+1)/\gap}}\step^{-4\aleph(\ell+1)/\gap}\beta_{\ell,\ell}^{2\aleph}= \gamma_\ell \beta_{\ell,\ell}^{2\aleph}=\tfrac{1}{2}\alpha_{\ell+1}=\alpha_{\ell+1, \ell+1}.
\end{align*} 
Therefore,} Lemma \ref{l:Step2} equation \eqref{e:Step2c} yields that $\rho \in \Sim\sub{\!\fuInj}(n, \alpha_{\ell+1, \ell+1}, g_{\ell+1})$. Furthermore, since $\tilde \beta_{\ell+1, \ell+1}\leq 2\beta_{\ell+1,\ell+1}=\beta_{\ell+1,\ell+1, \ell+1} $, we also have $\rho \in \Gamma \cap \Rec\sub{\fuInj}(n,\beta_{\ell+1,\ell+1,\ell+1},g_{\ell+1})$.
Using that $ \ANDrate(\beta_{\ell+1,\ell+1,\ell+1},1) \leq 1/\tilde \beta_{\ell+1 , \ell+1}$,
 this implies \eqref{e:wiggle} for $q=1$ and $i=\ell+1$
 since  $\ND_{1,_\fuInj}(n,\beta_{\ell+1,\ell+1,\ell+1},g_{\ell+1}) \subset \ND\sub{\fuInj}(n, \tilde \beta_{\ell+1 , \ell+1}, g_{\ell+1})$.

For any $\bm\sigma\sub{\ell+1}$ such that $g_{\ell+1}=F_{\ell+1}(g,(\bm{\sigma},\bm{\sigma}\sub{\!\ell+1}))\in L\sub{{\fuInj},\ell+1}$, we have showed that \eqref{e:ind2} does not hold with $(g,\ell)$ \blue{replaced by} $(g_{\ell+1}, \ell+1)$. Since Lemma \ref{l:Step2} equation \eqref{e:Step2d} yields that \eqref{e:ind2} holds for all $j \leq \ell$, we conclude the claim in \eqref{e:wiggle}. In particular, since  $\beta_{i,\ell+1,\ell+1} \leq \upsilon_n(\ell)$ for all $i \leq \ell+1$ and $\step^{i-1}<n\leq \step^i$, we conclude there exists $n \leq \step^{\ell+1}$ such that
 $\bm\sigma_{\ell+1} \in {\mc{S}}^{R,\delta}_{F_{\ell}(g,\hat{\bm{\sigma}}_{\ell}),\delta}(n, \alpha_{\ell+1, \ell+1}, \blue{\upsilon_n(\ell)})$.  This yields the claim in \eqref{e:window}.

 Next, we apply Corollary~\ref{c:fullPerturb} with $R=R_{\ell+1}(\e)$, $\delta=\delta_{\ell+1}(\e)$, $\alpha=\tfrac{1}{2}\alpha_{\ell+1}$, and $\beta=\upsilon_n(\ell)$. \blue{(We again observe that the definitions~\eqref{e:R-delta} imply that~\eqref{e:condsRdelta} holds.)}  Note that since $n\leq \step^{\ell+1}$ one may choose $\e_0$ so that the conditions \eqref{e:condsRdelta} on $R,\delta$ are satisfied. 
 \blue{Note that, by Corollary~\ref{c:fullPerturb}, 
 \begin{equation}\label{e:measureBd}
 \m\sub{\Sigma(R_{\ell+1})}\Big(\bigcup_{n=0}^{{\step^{\ell+1}}}{\mc{S}}_{F_{\ell}(g,\hat{\bm{\sigma}}_{\ell})}^{R,\delta}(n, \tfrac{1}{2}\alpha_{\ell+1}, \upsilon_n(\ell))\Big)\leq 
 \sum_{n=0}^\infty {C_0^n \delta^{-\L}(R^{-\vartheta_2}\delta +R^{\ga'})^{\L-4\aleph}\upsilon_n(\ell)}
 \end{equation}}
 Furthermore, 
  since
$\beta_{\ell+1,\ell+1}={\e}^{\bm{d}\sum_{j=1}^{\ell+1}\mathfrak{c}^{\ell+1-j}\step^{{j}/{\gap}}}\beta_{0,0}^{\mathfrak{c}^{\ell+1}},
$
we claim that there is $\bm{d}$ large  enough {(depending only on $\step,\gap,\aleph$)} such that 
\begin{equation}\label{e:condBeta_ell}
\blue{ \m\sub{\Sigma(R_{\ell+1})}\Big(\bigcup_{n=0}^{{\step^{\ell+1}}}{\mc{S}}_{F_{\ell}(g,\hat{\bm{\sigma}}_{\ell})}^{R,\delta}(n, \tfrac{1}{2}\alpha_{\ell+1}, \upsilon_n(\ell))\Big)} <{\e 2^{-\ell-2},}
\end{equation}
and  $\mathfrak{c}\geq 2\aleph\mathfrak B$, \blue{where we set
$$
\mathfrak B:= 4m_\nu \aleph+{\RPower}(\L-4\aleph),\qquad \tilde{\mathfrak{B}}:= \L m_\nu-\ga'(\L-4\aleph)
.
$$}
To \blue{see} this, we now estimate the \blue{right} hand side \blue{of~\eqref{e:measureBd}}. \blue{We recall $\L$ is defined in \eqref{e:LDef}.}

\blue{We start by noting that, by the definitions of $m_\nu$ and $\delta_\ell$~\eqref{e:R-delta},
\begin{align}\label{e:deltaLB}
{\delta_{\ell+1}}(\e)
\geq \min\Big({\e}^{\step^{\ell+1}+1}\,,\,{\e 2^{-\ell-1}}\Big)(R_{\ell+1}(\ep))^{m_\nu}
=\ep^{\step^{\ell+1}+1} (R_{\ell+1}(\ep))^{m_\nu},
\end{align}
 since $1<\step<2$ and so $\e^\step\leq\tfrac{1}{2}$.
Then, from the definitions of $R_{\ell+1}$~\eqref{e:R-delta}, $\alpha_{\ell+1}$~\eqref{e:alphas} and $\gamma_\ell$~\eqref{e:gamma_j}, it follows that, there is $C>0$ such that 
\begin{align}
(R_{\ell+1}(\e))^{-{\RPower}(\L-4\aleph)}(\delta_{\ell+1}(\e))^{-{4\aleph}}
&\leq \beta_{\ell,\ell}^{-2\aleph \mathfrak{B}}\e^{-C(\step^{\ell/\gap}+1)},
\label{e:this},\\
((R_{\ell+1}(\e))^{\ga'})^{\L-4\aleph}(\delta_{\ell+1}(\e))^{-{\L}}&\leq \beta_{\ell,\ell}^{-2\aleph \tilde{\mathfrak B}}{\e}^{-C(\step^{\ell/\gap}+1)}.\label{e:this2}
\end{align}
We also note that the definition of $\upsilon_n(\ell)$~\eqref{e:upsilon} yields
\begin{align}
\sum_{n=0}^\infty {C_0^n} (\upsilon_n(\ell))
&= 2(\beta_{\ell,\ell}^{\mathfrak{c} }\e^{\bm{d}\step ^{\ell/\gap}})\sum_{n=0}^\infty {C_0^n}  {\e}^{(n-\step^{\ell+1})\dim\sp(2\aleph)}\notag\\
&\leq C 2 \,\beta_{\ell,\ell}^{\mathfrak{c} } \,\e^{( \bm{d}-C)\step ^{\ell/\gap}}. \label{e:that}
\end{align}
\blue{provided $\ep$ is small enough depending $K$}. 

Combining \eqref{e:measureBd},\eqref{e:this}, \eqref{e:this2} and \eqref{e:that}, and using that $\e^\step\leq\tfrac{1}{2}$, we obtain that
there is $C>0$ such that
\begin{align*}
\blue{ \m\sub{\Sigma(R_{\ell+1})}\Big(\bigcup_{n=0}^{{\step^{\ell+1}}}{\mc{S}}_{F_{\ell}(g,\hat{\bm{\sigma}}_{\ell})}^{R,\delta}(n, \tfrac{1}{2}\alpha_{\ell+1}, \upsilon_n(\ell))\Big)} 
&\leq  \beta_{\ell,\ell}^{\mathfrak{c}-2\aleph\max({\mathfrak B},\tilde{\mathfrak B} )}\,\e^{(\bm{d}-C) \step^{\ell/\gap}}  .
\end{align*}}

Choosing $\bm d$ large enough the estimate \eqref{e:condBeta_ell} follows provided $\mathfrak{c}\geq \blue{2\aleph \mathfrak  \max({\mathfrak B},\tilde{\mathfrak B} )}$.   One can also check that $\beta=\upsilon_n(\ell)\leq \blue{R_{\ell+1}^{\ga'}}$ provided $\mathfrak{c}\geq 2\aleph \ga'$ and hence, we may apply Corollary~\ref{c:fullPerturb}. Optimizing in $\ga'\geq\ga$ we obtain \blue{that 
 $\ga'=\max(\m_\nu-\vartheta_2, \ga)$ and so $\max(\mathfrak B,\tilde{\mathfrak B})=\mathfrak B$ and $\mathfrak{c}=\max({{2\aleph \mathfrak B}},2\aleph\ga)$.}


 Corollary~\ref{c:fullPerturb} together with \blue{\eqref{e:measureBd}}, \eqref{e:condBeta_ell} and \eqref{e:window} then yield  that \eqref{e:mask2} holds up to index $\ell+1$ in place of $\ell$.

We next show that~\eqref{e:mask} is a consequence
of \eqref{e:mask2}. Given $ \bm \sigma\in \bm{\Sigma}_{\ell+1}(\Rind)$,  if 
$$ \exists j\leq \ell+1\;\;\; \text{s.t.}\;\;\,F_{j}(g,\hat{\bm{\sigma}}_j)\in  L\sub{\fuInj,j},$$
then {either} there exists $1\leq j\leq \ell+1$ such that 
$F_{j}(g,\hat{\bm{\sigma}}_j)\in L\sub{\fuInj,j}$ and  $F_{j-1}(g,\hat{\bm{\sigma}}_{j-1})\in \ms{G}^{\nu}\backslash L\sub{\fuInj,j-1}$ {or $F_0(g,\hat{\bm{\sigma}}_0)\in L\sub{\fuInj,0}$}.  
Therefore, by Fubini's theorem and the induction hypothesis \eqref{e:mask2}
\begin{align*}
 & m\sub{\bm{\Sigma}_{\ell+1}(\bm{\Rind})}\Big( \bm \sigma\in \bm{\Sigma}_{\ell+1}(\Rind):\exists  j\leq \ell+1\,\text{ s.t. }\,F_{j}(g,\widehat{\bm{\sigma}}_j)\in  L\sub{\fuInj,j}\Big)
 \leq\sum_{j=0}^{\ell+1}\e 2^{-j-1}\leq(1-2^{-(\ell+2)})\e .
\end{align*}
Now, notice that 
\begin{multline*}
m\sub{\bm{\Sigma}_{\ell}(\bm{\Rind})}\Big( \bm \sigma\in \bm{\Sigma}_{\ell}(\Rind):\exists  j\leq \ell+1\,\text{ s.t. }\,F_{j}(g,\widehat{\bm{\sigma}}_j)\in  L\sub{\fuInj,j}\Big)=\\m\sub{\bm{\Sigma}_{\infty}(\bm{\Rind})}\Big( \bm \sigma\in \bm{\Sigma}_{\infty}(\Rind):\exists  j\leq \ell\,\text{ s.t. }\,F_{j}(g,\widehat{\bm{\sigma}}_j)\in  L\sub{\fuInj,j}\Big)
\end{multline*}
and hence~\eqref{e:mask} holds. This finishes the proof of Proposition \ref{p:theInduction}.

\subsection{Predominance of quantitative non-degeneracy}

{We now turn to the proof that~\eqref{e:nonDegGoal} holds for a predominant set of $g$. The next proposition is the key technical result of the article. We will find sequences $\bm\delta_\e$ and $\Rind_\e$ such that~\eqref{e:nonDegGoal} holds for an $\ms{F}:=\{(F^{\Rind_\e,\bm\delta_\e}_\infty,\infty)\}_{\e}$ predominant set of metrics ({with $F^{\Rind_\e,\bm\delta_\e}_\infty$ as in Section~\ref{s:theProbes} built using any family perturbations $Q^{R,\delta}$ that are good in the sense of Definition \ref{ass:1}}).  More precisely, {provided we are able find a collection $\{(\Gamma,G)\}_{G\in\mc{G}}$ of $(\fuInj,b,\ga)$ admissible pairs for $\{Q^{R,\delta}\}_{R,\delta}$ such that $\bigcup_{G\in \mc{G}}G=\ms{G}^\nu$,} for each $\e>0$ {and $K\subset \ms{G}^\nu$ bounded} we will find a $C>0$  so that we can control the measure of the set  of bad parameter values $\bm\sigma\in \bm\Sigma_\infty(\Rind)$ such that $F_\infty(g,\bm\sigma)$ does not satisfy~\eqref{e:nonDegGoal} with that $C$ {and any $g\in K$}. This result will be used in Section~\ref{s:theProof} to prove Theorem~\ref{t:predominantR-ND}.  }

\begin{proposition}\label{p:thePredominantMeat} Let \blue{$\nu\geq 3$}, $\ga\geq 1$, $b>0$, $N:(0,1)\to \mathbb{N}$,
 and $\{Q^{R,\delta}\}_{R,\delta}$ be a $(\nu, N)$-good family of perturbations.  For $0<\e<1$, there are $\bm\delta_\e=\{\delta_{j}(\e)\}_{j=0}^{\infty}$, $\bm{\Rind}_\e:=\{\Rtemp_{j}(\e)\}_{j=0}^{\infty}$, such that 
$$
\sum_{j=0}^{\infty} \delta_j(\e)\Rtemp_j(\e)^{-\vartheta_\nu}\leq \e,
$$
and, for all $(\Gamma,\Gb)$ an $(\fuInj,b,\ga)$-admissible pair for $\{Q^{R,\delta}\}_{R,\delta}$, $K\subset \Gb$ bounded, there is $\e_0>0$ and $C>0$ such that, 
for all  $0<\e<\e_0$ and $g\in K$ there exists Borel set $S_{g,\e} \subset \bm \Sigma_\infty(\bm{\Rind}_\e)$ such that 
\[
\sup_{g\in K}m\sub{\bm{\Sigma}_{\infty}(\bm{\Rind})}\big(S_{g,\e}\big)\leq \e,
\]
and for all $g \in K$
\[
\big\{\bm \sigma \in \bm \Sigma_\infty(\bm{\Rind}_\e):\; F\sub{\!\infty}^{\Rind_\e, \bm\delta_\e}(g, \bm \sigma)\in  L\sub{\infty}(\e)\big\}\subset S_{g,\e}
\]
where, $F_{\infty}^{\bm\Rind_\e,\bm\delta\e}$ is defined in Lemma~\ref{l:probing}, 
\begin{equation}\label{e:rain}
 L\sub{\infty}(\e)=\Big\{g\in\ms{G}^{\nu}\,:\,\text{\ensuremath{\exists n} such that }\Gamma\cap \Rec\sub{\fuInj}(n,\beta_n(\e),g)\nsubseteq\ND\sub{\fuInj}(n,\beta_n(\e),g)\Big\},
\end{equation}
and 
\begin{equation}\label{e:gamma0}
\beta_n(\e):=\e^{ C n^{\gamma}}C^{-n^\gamma}n^{-C\log \e^{-1} n^{\gamma}},
\end{equation}
with $\gamma:=1+\log_{2}\blue{\Big[2\aleph\max\Big(\aleph(\blue{4 m_\nu +\RPower(2\aleph-1))}\;,\,\ga\Big)\Big]}$ and $m_\nu=\blue{\max(b,\vartheta_\nu,\vartheta_2+1)}$. In addition, $ L\sub{\infty}(\e)$ is Borel.
\end{proposition}

{\begin{proof}
  From
Proposition~\ref{p:theInduction}, for all $1<\step <2$ there are $\{\delta_{j}(\e)\}_{j=0}^{\infty}$
and $\{\Rtemp_{j}(\e)\}_{j=0}^{\infty}$ such that \eqref{e:probeControl}, \eqref{e:mask} and~\eqref{e:D}
hold.  
Let $(\Gamma,\Gb)$ satisfying Definition~\ref{ass:2} and $K\subset \Gb$ bounded.
 Suppose that $g\in K$, and $\bm{\sigma}\in\bm{\Sigma}_{\infty}(\bm{\Rind})$.
Then, by Lemma~\ref{l:probing}, $F_{\infty}(g,\bm{\sigma})\in\ms{G}^{\nu}$
is well defined.
We claim that the statement holds with 
\[
S_{g, \e}:=\bigcup_{\ell=1}^\infty{S}_{g, \e}(\ell), \qquad {S}_{g, \ep}(\ell):= \{\bm{\sigma}\in\bm{\Sigma}_{\infty}(\Rind):\; \exists j \leq \ell \;\; \text{s.t.}\;\; F_j(g, \widehat{\bm \sigma}_j) \in  L\sub{\fuInj, j}\},
\]
with $L\sub{\fuInj, j}$ as defined in Proposition~\ref{p:theInduction}.
Note that by Proposition \ref{p:theInduction}, there exists $\e_0>0$ such that for all $0<\ep<\ep_0$ and $g\in K$,
$
m\sub{\bm{\Sigma}_{\infty}(\bm{\Rind})}\big(S_{g,\e}(\ell)\big)
\leq  (1-2^{-\ell-1})\e,
$
and so, since \blue{${S}_{g, \e}(\ell) \subset {S}_{g, \e}(\ell+1)$,}
 for all $0<\ep<\ep_0$ and $g\in K$,
$
m\sub{\bm{\Sigma}_{\infty}(\bm{\Rind})}\big(S_{g,\e}\big)
\leq \e,
$
as claimed.

Next, we claim that $\big\{\bm \sigma \in \bm \Sigma_\infty(\bm{\Rind}_\e):\; F\sub{\!\infty}^{\Rind_\e, \bm\delta_\e}(g, \bm \sigma)\in  L\sub{\infty}(\e)\big\}\subset S_{g,\e}$. For $\bm{\sigma}\in\bm{\Sigma}_{\infty}(\bm{\Rind})\setminus S_{g,\e}$ we will show that $F\sub{\!\infty}^{\Rind_\e, \bm\delta_\e}(g, \bm \sigma)\notin  L\sub{\infty}(\e)$.  Writing $g_{\ell}=F_{\ell}(g,\hat{\bm{\sigma}}_{\ell})$, we have that $g_\ell \in \ms{G}^{\nu}\backslash L\sub{\fuInj, \ell}$ for all $\ell$. In particular, for all $\ell$, {$0\leq i\leq \ell$}, and $\step^{{i}-1}<n\leq \step^{{i}}$,
$$
\Gamma \cap \Rec\sub{\fuInj}(n,\tilde{\beta}_{{i},\ell}, g_{\ell}) \subset\ND\sub{\fuInj}(n,\tilde{\beta}_{{i},\ell},{g_{\ell}}).
$$
By definition, $\tilde{\beta}_{i,\ell}=\frac{1+2^{-\ell}}{1+2^{-i}}\tilde{\beta}_{i,i}$. Therefore, for $\ell$, {$0\leq i\leq \ell$}, and $\step^{{i}-1}<n\leq \step^{{i}}$, we have
$$
\Gamma \cap \Rec\sub{\fuInj}(n,\tfrac{1}{2}\tilde{\beta}_{{i},i}, g_{\ell}) \subset\ND\sub{\fuInj}(n,\tfrac{1}{2}\tilde{\beta}_{{i},i},{g_{\ell}}).
$$

Now, by~\eqref{e:D}, we have
$\|g_{\ell}-g_{\ell-1}\|_{\mc{C}^{\nu}}\leq C2^{-\ell-1}\e.$
Therefore, $g_{\ell}\to F_{\infty}(g,\bm{\sigma})=:g_\infty\in\mc{C}^{\nu}$
with $\|g_\infty-g\|_{\mc{C}^{\nu}}\leq\e.$

Fix $\step^{i-1}< n\leq \step^i$. We claim that 
\begin{equation}
\label{e:gInfinity}
\Gamma \cap \Rec\sub{\fuInj}(n,\tfrac{1}{4}\tilde{\beta}_{i,i}, g_{\infty}) \subset\ND\sub{\fuInj}(n,\tfrac{1}{4}\tilde{\beta}_{i,i},{g_{\infty}}).
\end{equation}
To see this, suppose that $\rho\in \Gamma \cap \Rec\sub{\fuInj}(n,\tfrac{1}{4}\tilde{\beta}_{i,i}, g_{\infty})$.
Then, for $\ell$ large enough, 
\[
\rho\in \Gamma \cap\Rec\sub{\fuInj}(n,\tfrac{1}{2}\tilde{\beta}_{i,i},g_\ell)\subset\ND\sub{\fuInj}(n,\tfrac{1}{2}\tilde{\beta}_{i,i}, g_\ell)
\subset\blue{\ND\sub{\fuInj}(n,\tfrac{1}{4}\tilde{\beta}_{i,i},g_\infty)}
\]

Finally, we choose $1<\step<2$ (and hence find $\{\Rind_\e\}$ and $\{\bm\delta_\e\}$) such that~\eqref{e:gInfinity} implies that for all $n\geq 1$,
\begin{equation}
\label{e:gInfinityFinal}
\Gamma \cap\Rec\sub{\fuInj}(n,\beta_n(\e),g_\infty)\subset\ND\sub{\fuInj}(n,\beta_n(\e), g_\infty),
\end{equation}
and hence that $g_\infty\notin L_{\infty}.$ To check~\eqref{e:gInfinityFinal}, we need only show that for $\step^{\ell-1}<n\leq \step^\ell$, we have $ \tfrac{1}{4}\tilde{\beta}_{\ell,\ell}\geq \beta_n(\e)$. 

By definition, 
$\beta_{\ell, \ell}=\beta_{0,0}^{\mathfrak{c}^\ell}\e^{{\bm d \mathfrak{c}^{\ell-1}} \sum_{j=0}^{\ell-1}(\mathfrak{c}^{-1}\step^{\frac{1}{\gap}})^j}$.
In addition, we set
\begin{equation}
\label{e:itsA}
\step :=2^{\frac{\log \mathfrak{c}}{\log 2+\log \mathfrak{c}}} \qquad \Rightarrow \qquad
\mathfrak{c}\step^{-{1}/{\gap}}=1,
\end{equation}
where \blue{ $\mathfrak{c}:=2\aleph\max\Big(\aleph(\blue{4 m_\nu +\RPower(2\aleph-1))}\;,\,\ga\Big)$
and
$m\sub{\nu}:=\max(b,\vartheta_\nu,\vartheta_2+1).$}
Therefore, 
\begin{align*}
\tilde{\beta}_{\ell,\ell}=\e^{2\aleph-1}\beta_{\ell, \ell}^{2\aleph}
&=\beta_{0,0}^{2\aleph\mathfrak{c}^\ell}\e^{2\aleph{\bm d \mathfrak{c}^{\ell-1}}\ell} \e^{2\aleph-1}\\
&\geq \beta_{0,0}^{2\aleph\mathfrak{c}^\ell}\e^{2\aleph{\bm d \mathfrak{c}^{\ell-1}}\ell+2\aleph-1}.
\end{align*}
Moreover,~\eqref{e:itsA} implies
$\frac{\log_2\mathfrak{c}}{\log_2 \step}=1+\log_2\mathfrak{c}=\gamma$. 

Notice that $\log_{\step}n<\ell\leq1+\log_{\step}n$ since $\step^{\ell-1}<n\leq\step^{\ell}$ and hence $\mathfrak{c}^\ell \leq c n^{\gamma}$ which, using that $\beta_{0,0}(\e)\geq c \e^N$ for some $c>0$ and $N\geq 0$, yields, for $C>0$ large enough (in the definition of $\beta_n(\e)$) depending on $\mathfrak{c}$, $\aleph$, and $\bm{d}$, yields
$$
\tilde{\beta}_{\ell,\ell}\geq [\beta_{0,0}(\e)]^{2\aleph \mathfrak{c}n^{\gamma}}n^{-2\aleph \bm{d}\gamma\frac{\log_2\e^{-1}}{\log_2\mathfrak{c}}n^{\gamma}}\e^{2\aleph \bm{d}n^{\gamma}+2\aleph-1}\geq  {4}\beta_n(\e).
$$
 \end{proof}}


\section{Nice perturbations in the space of metrics }
\label{s:perturbedMetrics}



In this section, we define families of perturbations that satisfy Definition~\ref{ass:1} and~\ref{ass:2}. The basic perturbation takes place on a small ball through which a geodesic passes exactly once and is inspired by those in~\cite{An:82,Kl:78}. We then build a large family of these perturbations in order to be able to handle any possible geodesic.

\subsection{The elementary family of perturbations}

Let $\nu\geq 0$. We fix a reference metric, $g\sub{f}\in\ms{G}^{\nu+{3}}$. Let $\rho_0\in \tSM$ and $e_1,\dots, e_d\in T_{\pi\sub M(\rho_0)}M$ an orthonormal set of vectors for $g\sub{f}$. \blue{(Recall the definition of $\tSM$ from~\eqref{e:tildeSM}.)}

For $g_{\star}\in\ms{G}^{\nu}$,  consider the
geodesic $$\gamma_{\rho_{0}}^{g_\star}(t):=\pi\sub{M}(\varphi_{t}^{g_{\star}}({\rho_{0}})).$$
We choose a basis $E^{g_\star}_{1},\dots,E^{g_\star}_{d}\in T_{\gamma_{\rho_{0}}^{g_\star}(0)}M$,
orthonormal with respect to $g_{\star}$ by applying Gram-Schmidt process to $(\dot{\gamma}_{\rho_{0}}(0),e_1,\dots, e_d)$ so that, in particular,  $E_{d}^{g_\star}=\dot{\gamma}_{\rho_{0}}(0)$. 
Let $E_{i}^{g_\star}(t)\in T_{\gamma_{\rho_{0}}^{g_\star}(t)}M$ be the parallel transport
(again with respect to $g_{\star}$) of $E_{i}^{g_\star}$ along $\gamma_{\rho_{0}}^{g_\star}$. 
{For $\e\sub{\!f}>0$, define the map}
\begin{equation}\label{e:picky}
\Phi_{{\rho_{0}}}^{g_{\star}}:B_{\mathbb{R}^{d-1}}(0,{\e\sub{\!f}})\times\mathbb{R}\to M,\qquad\Phi_{{\rho_{0}}}^{g_{\star}}(u,t)=\exp_{\gamma_{\rho_{0}}^{g_\star}(t)}^{g\sub{f}}\Big(\sum_{i=1}^{d-1}u^{i}E_{i}^{g_\star}(t)\Big),
\end{equation}
where $\exp_{\gamma_{{\rho_{0}}}(t)}^{g\sub{f}}$ denotes the exponential
map for the metric $g\sub{f}$ with base point $\gamma_{\rho_{0}}(t)$.
Then for $\e\sub{\!f}>0$ small enough depending only on $g\sub{f}$, and any small enough (depending on $g_\star$ and $g\sub{f}$) interval $I\subset\mathbb{R}$, $\Phi_{{\rho_0}}^{g_\star}|_{I\times B(0,\e\sub{\!f})}$  is a diffeomorphsim onto its image.

The reason for using $g\sub{f}$ above rather than $g_{\star}$ itself
is that, in our perturbation argument, $g_{\star}$ will vary, remaining
bounded in $\ms{G}^{\nu}$, but not necessarily in $\ms{G}^{\nu'}$
for $\nu'>\nu$. It will therefore be important for us to control
the $\mc{C}^{\nu}$ norms of $\Phi_{\rho_{0}}^{g_{\star}}$ using
only the $\ms{G}^{\nu}$ norms of $g_{\star}$ (See Lemma~\ref{l:regularity}).
This is \emph{not} possible if one replaces $g\sub{f}$ by $g_{\star}$ \blue{since the exponential map is only $\mathcal{C}^{\nu-1}$}.

\blue{
\begin{lemma} \label{l:regularity} Let $\nu\geq2$ and $K\subset\ms{G}^{\nu}$
be bounded. Then there is $C>0$ such that for all $g_{\star}\in K$
and ${\rho_{0}}\in\tSM$, 
\[
\|\Phi_{\rho_{0}}^{g_{\star}}\|_{\mc{C}^{\nu}}\leq C.
\]
\end{lemma} 
\begin{proof} First, recall that in any coordinate system
$x$ on $M$ with dual coordinates $\xi$, if ${\rho_{0}}=(0,(0,\dots,0,1))$, then
$\gamma_{\rho_{0}}(s)$ is given by the solution $x(s)$ to 
\[
\dot{x}^{\alpha}=\frac{g_{\star}^{\alpha\beta}\xi_{\beta}}{|\xi|_{g}},\qquad\dot{\xi}_{\alpha}=-\frac{\partial_{x_{\alpha}}g_{\star}^{\mu\beta}\xi_{\mu}\xi_{\beta}}{|\xi|_{g}},\qquad(x(0),\xi(0))=(0,(0,\dots,0,1)).
\]
Therefore, standard regularity theory for ODEs shows that $\gamma_{\rho_{0}}(t)\in C^{\nu+1}(\mathbb{R};M)$
and there exists $C\sub{K}>0$ with $\|\gamma_{\rho_{0}}\|_{\mc{C}^{\nu+1}}\leq C\sub{K}.$
In particular, $\|\dot{\gamma}_{\rho_{0}}\|_{\mc{C}^{\nu}}\leq C\sub{K}.$
Next, recall that for each $i$ the equations for parallel transport
take the form $E_{i}(t)=e^{\alpha}(t)\partial_{x^{\alpha}}$ where
$e^{\alpha}(t)$ solves an equation of the form $\dot{e}^{\alpha}(t)=\Gamma e(t)\dot{\gamma}_{\rho_{0}}^{g_\star}(t)$
with $\Gamma$ representing the Christoffel symbol for $g$. In particular,
$\|e^{\alpha}\|_{\mc{C}^{\nu}}\leq C\sub{K}$ and hence $\|E_{i}\|_{\mc{C}^{\nu}}\leq C\sub{K}.$

Now, since $g\sub{f}\in\ms{G}^{\nu+2}$ is independent of $K$, the
exponential map $(p,V)\mapsto\exp_{p}^{g\sub{f}}(V)$ is in $C^{\nu}(TM;M)$
with bounds independent of $K$. Therefore, $\Phi_{\rho_{0}}^{g_{\star}}$
satisfies the desired estimates. \end{proof}}

\begin{remark} Although $\Phi_{{\rho_{0}}}^{g_{\star}}$ is not a
global diffeomorphism, it is locally a diffeomorphism, and we will
often ignore this fact in order to simplify the presentation. When
defining our perturbations below, we will work in a $(u,t)$ neighborhood
where the map is a genuine diffeomorphism. \end{remark}

\begin{lemma} \label{l:metricCoordinate}
In the $(u,t)$ coordinates $g_{\star}$ satisfies 
\[
g_{\star}^{ij}(0,t)=\delta^{ij},\qquad g_{\star}^{dd}(0,t)=1,\qquad g_{\star}^{jd}(0,t)=0,\qquad\partial_{u_{j}}g_{\star}^{dd}(0,t)=0.
\]
where, here and below, we write $i,j,k,\ell$ for
indices in $\{1,\dots,d-1\}$. \end{lemma} \begin{proof} The first
three equalities follow from the fact that $\{E_{i}^{g_\star}(t)\}_{i=1}^{d}$
is an orthonormal frame with respect to $g_{\star}$. Since $(0,t)$
is a geodesics for $g_{\star}$, letting $x(s)=(0,s)\in\mathbb{R}_{u}^{d-1}\times\mathbb{R}_{t}$,
we have 
\[
\partial_{s}^{2}x^{\beta}+\frac{1}{2}g_{\star}^{\mu\beta}(\partial_{\alpha}g_{\star,\mu\nu}+\partial_{\nu}g_{\star,\mu\alpha}-\partial_{\mu}g_{\star,\alpha\nu})\partial_{s}x^{\alpha}\partial_{s}x^{\nu}=0.
\]
Since $x^{i}=\partial_{s}x^{i}=\partial_{s}^{2}x^{i}=0$ for $i=1, \dots, d-1$, $\partial_{s}x^{d}=1$,
and $\partial_{s}^{2}x^{d}=0$, we have that for $i=1, \dots, d-1$
\[
\frac{1}{2}g_{\star}^{\mu i}(\partial_{d}g_{\star,\mu d}+\partial_{d}g_{\star,\mu d}-\partial_{\mu}g_{\star,dd})=-\frac{1}{2}\partial_{u_i}g_{\star,dd}=0.
\]
Since $g_{\star, ij}(0,t)$ is the identity, $\partial_{u_i}g_{\star}^{dd}=0$.
\end{proof}

{In order to define the perturbation of a metric $g\in \ms{G}^\nu$, let 
$$\ms{E}(x,v):=\exp_x^{g\sub{\!f}}(v), \qquad V(u,t):=\sum_{i=1}^{d-1}u^iE_i^{g}(t).$$
Note that $\Phi_{\rho_0}^{g}(u,t)=\ms{E}(\gamma_{\rho_0}^{g}(t),V(u,t))$  is a $C^\nu$ map. Indeed, this follows from noting that $g\sub{\!f}\in\mc{C}^{\nu+{3}}$, $\ms{E}\in C^{\nu+2}$, $g\in\mc{C}^{\nu}$, $\dot{\gamma}_{\rho_0}^{g}\in \mc{C}^\nu$,  and $E_i^{g}\in \mc{C}^\nu$ for $i=1, \dots, d-1$. To see these last two, observe that the geodesic equation is of the form $\ddot{\gamma}^g=\Gamma \dot \gamma^g \dot \gamma^g$, where $\Gamma$ involves the Christoffel symbols for $g$ and hence is in $\mc{C}^{\nu-1}$ and the equations of parallel transport take the form $\dot E^g= \Gamma E^g \dot \gamma^g$. Using the standard theory of ODEs one obtains $\gamma^g\in\mc{C}^{\nu+1}$ and $E^g\in\mc{C}^\nu$.  

Next, for each $j=1,\dots,d-1$ introduce the follow $\mc{C}^{\nu}$ vector fields
\begin{align}
\mc{U}^{g}_j(\Phi_{\rho_0}^{g}(u,t))
&:=D_v\ms{E}|_{(\gamma_{\rho_0}^{g}(t),V(u,t) )} E_j^{g}(t)\label{e:U}\\
\mc{X}^{g}(\Phi_{\rho_0}^{g}(u,t))
&:= D_{x} \ms{E}|_{(\gamma_{\rho_0}^{g}(t),V(u,t) )}\dot{\gamma}_{\rho_0}^{g}(t). \label{e:X}
\end{align}
The regularity of these vector fields follows from the facts that $\mathscr{E}\in\mc{C}^{\nu+\blue{2}}$, $\gamma_{\rho_0}^{g}\in \mc{C}^{\nu+1}$, and $E_{i}^g\in\mc{C}^\nu$.

\begin{remark}
Note that 
\begin{equation}
\label{e:xIsNice}
(\Phi_{\rho_0}^{g})_*\partial_{t}=\mc{X}^g +D_v\ms{E}|_{(\gamma_{\rho_0}^{g}(t),V(u,t) )}\sum_{i=1}^{d-1}u^i\dot{E}_i^g(t).
\end{equation}
may not be in $\mc{C}^\nu$ since $E_{i}^g$ may only be $\mc{C}^{\nu}$. Therefore, to perturb within the class of $\mc{C}^\nu$ metrics, we write our perturbation in terms of its action on $\mc{X}^g$ rather on $(\Phi_{\rho_0}^{g})_*\partial_{t}$.  Observe that $\mc{U}^{g}_j=(\Phi_{\rho_0}^{g})_*\partial_{u^j}$.
Equation~\eqref{e:xIsNice} also implies that
$\mc{X}^{g}({\gamma_{\rho_0}^{g}(t)})=\dot{\gamma}_{\rho_0}^{g}(t),$
since $\gamma_{\rho_0}^{g}(t)=\Phi_{\rho_0}^{g}(0,t)$. This fact will result in convenient structure in the perturbation theory for geodesics. 
\end{remark}
}

\begin{lemma}
\label{l:Frechet1}
The maps $F:\ms{G}^\nu\to \mc{C}^{\nu}(\mathbb{R}^d; M)$ {and  $\mc{V}_j:\mathscr{G}^\nu\to \mc{C}^{\nu}(TM)$ with $j=1,\dots, d$} defined as  
$$F(g):=\Phi_{\rho_0}^{g}, \qquad{ \mc{V}_d(g):=\mc{X}^g, \qquad \mc{V}_j(g):=\mc{U}^g_j,}$$   are continuous. As maps $F:\ms{G}^\nu\to \mc{C}^{\nu-1}(\mathbb{R}^d; M)$ and  $\mc{V}_j:\mathscr{G}^\nu\to \mc{C}^{\nu-1}(TM)$, they are Frechet differentiable and, moreover, for any bounded subset $\Gb\subset \ms{G}^\nu$ there is $C>0$ such that for $g\in \Gb$,
\begin{align*}
\| DF|_{g}\delta_g\|_{\mc{C}^{\nu-1}}&\leq C\|\delta_g\|_{\mc{C}^{\nu-1}},\\
{\| D\mc{V}_j|_{g}\delta_g\|_{\mc{C}^{\nu-1}}}&{\leq C\|\delta_g\|_{\mc{C}^{\nu-1}},\quad j=1,\dots, d.}
\end{align*}
\end{lemma}
\begin{proof}
{In the discussion above, we have checked that the maps are continuous from $\mathscr{G}^\nu$ to $\mc{C}^\nu$. Therefore, it remains only to prove Frechet differentiability.}
To prove the lemma, we work locally in a coordinate chart so that we may identify vectors along a curve with $\mathbb{R}^d$. 

Define the map $\mc{E}_0:\ms{G}^\nu\to (\mathbb{R}^d)^d$ by $\mc{E}_0(g):=(E_1^g(0), \dots, E_d^g(0))$ with $E_i^g$ defined as above.

Let  $\mc{E}:\ms{G}^\nu \times \mc{C}^{\nu+1}([0,1]; \mathbb{R}^d)\times (\mathbb{R}^d)^d\to C^{\nu}([0,1];(\mathbb{R}^d)^d)$ be the map $\mc{E}(g, \gamma, (e_i)_{i=1}^{d}):= (e^g_i)_{i=1}^d,$ where $e^g_i\in C^{\nu}([0,1];\mathbb{R}^d)$ is a parallel frame along $\gamma$ (with respect to $g$) with initial conditions $(e_i^g(0))_{i=1}^{d}=(e_i)_{i=1}^{d}$. 

In addition, define 
$\exp^{g\sub{f}}:\mathbb{R}^d\times \mathbb{R}^d\to \mathbb{R}^d$ by $\exp^{g\sub{f}}(x, v):=\exp_x^{g\sub{f}}(v)$
and let $\Psi:\widetilde{S^*\!M}\times \ms{G}^\nu\to \mc{C}^{\nu+1}([0,1];\mathbb{R}^d)$ be the map $\Psi(\rho, g):=\gamma_\rho^g$ where $\gamma_\rho^g $ is the unit speed geodesic for the metric $g$ with initial conditions given by $\rho$.
Let 
$Y: \mc{C}^{\nu+1}([0,1];\mathbb{R}^d)\times C^{\nu}([0,1];(\mathbb{R}^d)^d)\to C^{\nu}([0,1]\times \mathbb{R}^{d-1};\mathbb{R}^d)$,
$$
\big[Y(\gamma,(E_i)_{i=1}^d)\big](t,u)=\ms{E}\Big(\gamma(t), \sum_{i=1}^{{d-1}}u^iE_i(t)\Big).
$$
{
Finally, for $j=1,\dots, d$ let $Y_j:\mc{C}^{\nu+1}([0,1];\mathbb{R}^d)\times C^{\nu}([0,1];(\mathbb{R}^d)^d)\to C^{\nu}([0,1]\times \mathbb{R}^{d-1};\mathbb{R}^d)$ be 
\begin{align*}
[Y_{\!j}(\gamma, (E_i)_{i=1}^d)](t,u)&:=D_v \ms{E}|_{(\gamma(t),\sum_{i=1}^{d-1} u^iE_i(t))}E_j(t),\quad j=1,\dots, d-1,\\
[Y_d(\gamma,(E_i)_{i=1}^d](t,u)&:=D_x\ms{E}|_{(\gamma(t),\sum_{i=1}^{d-1} u^iE_i(t))}\dot{\gamma}(t).
\end{align*}}
We are interested in studying the compositions
\begin{align*}
F( g)&= Y\Big(\Psi(\rho_0, g )\,,\, \mc{E}\big( g , \Psi(\rho_0, g), \mc{E}_0(g) \big)\Big),\\
{\mc{V}_j( g)}&={Y_j\Big(\Psi(\rho_0, g )\,,\, \mc{E}\big( g , \Psi(\rho_0, g), \mc{E}_0(g) \big)\Big),\quad j=1,\dots,d.}
\end{align*}

To clarify the exposition, in what follows we write $\ms{S}^\nu$ for the space of symmetric tensors endowed with the $\mc{C}^\nu$ topology.

Note that $g \mapsto D_g\mc{E}_0$ is bounded and continuous in the $\ms{G}^\nu$ topology. By writing the ordinary differential equation satisfied by a parallel vector field and using that $\dot \gamma \in \mc{C}^{\nu-1}([0,1]; \mathbb{R}^d)$, $\partial_xg\in \ms{S}^{\nu-1}$, one can check that 
{$D_g\mc{E}:{\ms{S}^{\nu-1}\to \mc{C}^{\nu-1}([0,1]; (\mathbb{R}^d)^d)}$}
and the map $(g,\gamma)\mapsto D_{g}\mc{E}|_{(g,\gamma)}$ is continuous in the $\ms{G}^\nu\times \mc{C}^\nu$ topology {with codomain the set of bounded linear maps from $\ms{S}^{\nu-1}$ to $\mc{C}^{\nu-1}([0,1]\times (\mathbb{R}^d)^d)$.}

\blue{To see this, recall that parallel transport is given by
\[
\frac{d(e_i^{g})^k}{ds} +\Gamma^k_{ij}(\gamma(s))\dot \gamma^i(s)(e_i^{g})^j(s)=0,\qquad
\Gamma^k_{ij}
=
\frac12 g^{k\ell}
\bigl(
\partial_i g_{j\ell}
+
\partial_j g_{i\ell}
-
\partial_\ell g_{ij}
\bigr).
\]
Setting $g=g(\delta)=g+\delta g_0$ and differentiating in $\delta$ at $\delta=0$ we obtain 
\[
\frac{d(\partial_{\delta}e_i^{g})^k}{ds} +(\partial_{\delta}\Gamma^k_{ij})(\gamma(s))\dot \gamma^i(s)(e_i^{g})^j(s)+\Gamma^{i}_{ij}(\gamma(s))\dot \gamma^i(s)(\partial_{\delta}e_i^g)^j(s)=0,
\]
which yields $\partial_{\delta}e_i^g\in \mathcal{C}^{\nu-1}$. Continuity as $g$ and $\gamma$ vary in $\mathcal{C}^\nu$ follows since this amounts to changing the vector field continuously in $\mathcal{C}^{\nu-1}$.  
}

Next, again using the parallel transport equation and $\dot\gamma\in \mc{C}^{\nu-1}$, $g\in \ms{G}^\nu$, the map $D\sub{\gamma}\mc{E}:{\ms{S}^{\nu-1}\to \mc{C}^{\nu-1}([0,1]; (\mathbb{R}^d)^d)}$ is bounded and the map $(g,\gamma)\mapsto D\sub{\gamma}\mc{E}|_{(g,\gamma)}$ continuous in the $\ms{G}^\nu\times \mc{C}^\nu$ topology {with codomain the set of bounded linear maps from $\ms{S}^{\nu-1}$ to $\mc{C}^{\nu-1}([0,1]; (\mathbb{R}^d)^d)$.}

{Furthermore, since $g\sub{f}\in \mc{C}^{\nu+{3}}$, the maps
$$
{D\sub{(\gamma,E)}Y,\,D\sub{(\gamma,E)} Y_j}:{\mc{C}^{\nu}([0,1]; \mathbb{R}^d)\times \mc{C}^{\nu-1}([0,1]; (\mathbb{R}^d)^d)\to \mc{C}^{\nu-1}([0,1]\times \mathbb{R}^{d-1};\mathbb{R}^d)}
$$
are bounded and the maps $(\gamma,E)\mapsto D\sub{(\gamma,E)} Y$ {and $(\gamma,E)\mapsto D\sub{(\gamma,E)} Y_j$} are  continuous in the $\mc{C}^{\nu}([0,1];\mathbb{R}^d)\times \mc{C}^{\nu-1}([0,1];(\mathbb{R}^d)^d)$ topology with codomain the set of bounded linear maps from $\mc{C}^{\nu}([0,1]; \mathbb{R}^d)\times \mc{C}^{\nu-1}([0,1]; (\mathbb{R}^d)^d)$ to $ \mc{C}^{\nu-1}([0,1]\times \mathbb{R}^{d-1};\mathbb{R}^d)$.
}

Finally, using the geodesic equation \blue{as we did the parallel transport equation} together with $g\in\ms{G}^\nu$ and that geodesics for a $\ms{G}^\nu$ metric lie in $\mc{C}^{\nu+1}$, we have
$
D_g\Psi:{\ms{S}^{\nu-1}\to \mc{C}^{\nu}([0,1];\mathbb{R}^d)}
$
is bounded and the map $g_0\mapsto D_g\Psi|_{g_0}$ is continuous as a function of $g_0$ in the $\ms{G}^\nu$ topology {with with codomain the set of bounded linear maps from $\ms{S}^{\nu-1}$ to $ \mc{C}^{\nu}([0,1];\mathbb{R}^d)$}.
Indeed, 
\blue{writing $g(\delta)=g+\delta g_0$ as above with $g_0\in \mathscr{S}^{\nu-1}$,
$$
\tfrac{d^2 \partial_{\delta}\gamma^k}{ds^2}+(\partial_{\delta}\Gamma_{ij}^k)(\gamma(s))\tfrac{d\gamma^i}{ds}\tfrac{d\gamma^j}{ds}+\partial_{x^\ell} \Gamma_{ij}^k \partial_{\delta}\gamma^\ell\tfrac{d\gamma^i}{ds}\tfrac{d\gamma^j}{ds}+\Gamma_{ij}^k(\gamma(s))\tfrac{d\partial_{\delta}\gamma^i}{ds}\tfrac{d\gamma^j}{ds}+ \Gamma_{ij}^k(\gamma(s))\tfrac{d\gamma^i}{ds}\tfrac{d\partial_{\delta}\gamma^j}{ds}=0.
$$}
With these estimates in place it is now easy to see that $D_gF:{\ms{S}^{\nu-1}\to \mc{C}^{\nu-1}([0,1]\times \mathbb{R}^{d-1};\mathbb{R}^d)}$ is bounded and $g_0\mapsto D_gF|_{g_0}$ is continuous in the $\ms{G}^\nu$ topology {with with codomain the set of bounded linear maps from $\ms{S}^{\nu-1}$ to $ \mc{C}^{\nu-1}([0,1]\times \mathbb{R}^{d-1};\mathbb{R}^d)$}.
\end{proof}

 For $s\in \R$  and $\ep>0$ small let
\begin{equation}\label{e:H}
\mc{Z}_{s}:=\Phi_{\rho_0}^{g_\star}(B(0,{\e\sub{\!f}})\times\{s\}),\qquad \mc{H}_{s}:=\widetilde{S^*\sub{\mc{Z}_{s}}M}\cap B\sub{{\tSM}}(\varphi_{s}^{g_\star}(\rho_0),\e).
\end{equation}
Here, we use the metric $g\sub{f}$ to define the ball in $\tSM$.

Note that $\mc{H}_{s}$ is
a smooth hypersurface such that $\rho_{0}\in\mc{H}_{0}$, and
$\mc{H}_{s}$ is transverse to $H\sub{|\xi|_{g_{\star}}}$; 
in particular, satisfying, 
\[
|H\sub{|\xi|_{g_{\star}}}(\rho)f|\geq\frac{1}{2}|df|,\qquad\text{for }f\text{ defining }\mc{H}_{s} \text{ and all }\rho\in \mc{H}_{s}.
\]
Despite the fact that we define $\mc{H}_s$ as a subset of $\widetilde{S^*\!M}$, it will be convenient in the rest of this section to identify $\widetilde{S^*\!M}$ with $S^*_gM$ for each $g$ i.e. the unit sphere bundle with respect to the metric $g$. To do this, we define the canonical isomorphism from  $\widetilde{S^*\!M}\to S^*_gM$ to by $(x,[\xi])\mapsto (x,\frac{\xi}{|\xi|_g})$. Here, we have used that $[\xi]=\big[\tfrac{\xi}{|\xi|_g}\big]$ in $\widetilde{S^*\!M}$.

We now define a family of perturbations $g_{\sigma}\in\ms{G}^{\blue{3}}$
of a metric $g_{0}\in\ms{G}^{\blue{3}}$ close to $g_{\star}$ as follows:
Let 
\[
\Sigma_1:=\mathbb{R}^{d-1}\times\mathbb{R}^{d-1},\qquad \Sigma_2:=\Sym(d-1)\times\mathbb{M}(d-1)\times\Sym(d-1),\qquad \Sigma:=\Sigma_1\times \Sigma_2
\]
and denote the elements of $\Sigma$ by 
\begin{equation}\label{e:sigma}
{\sigma}:=(\sigma_1,\sigma_2)=\big(({\bf {A}},{\bf {B}}),({\bf {C}},{\bf {D}},{\bf {E}})\big)\in\Sigma.
\end{equation}
The closeness of $g_0$ to $g_\star$ will be crucial in Lemma~\ref{l:finalNondegenerate}.

{
Next, we continue to work with the vector fields $\mc{X}^{g_\star}$ and $\mc{U}^{g_\star}_j$ as defined in  \eqref{e:U} and \eqref{e:X}. Since $\mc{U}^{g_\star}_j=(\Phi_{\rho_0}^{g_\star})_*(\partial_{u_j})$, we have $du_j=(\Phi_{\rho_0}^{g_\star})^*(\mc{U}_j^{g_{\star}, \flat})$.
Defining  $$\varkappa:=(\Phi_{\rho_0}^{g_\star})^*(\mc{X}^{g_{\star, \flat}}),$$ we observe from~\eqref{e:xIsNice} that there are a matrix valued function 
\begin{equation}\label{e:largeX}
(u,t)\mapsto \bm X(u,t) \in \mathbb{M}(d-1),
\end{equation}
with entries $(\bm X)_{ij}:=x_{ij}$ for $i,j=1,\dots, d-1$, and a vector valued function \begin{equation}\label{e:largeXd}
(u,t)\mapsto  \bm X_d(u,t) \in \R^{d-1}
\end{equation}
with entries $(X_d)_i=x_{id}$ for $i=1,\dots, d-1$,
such that 
$$
\varkappa(u,t)=\Big(1+\sum_{i=1}^{d-1}u^ix_{id}(u,t)\Big)dt +\sum_{i,j=1}^{d-1}u^ix_{ij}(u,t)du^j.
$$

For $g_0 \in \ms{G}^\nu$, $\rho_0\in \widetilde{S^*\!M}$, $\sigma=\big({\bf {A}},{\bf {B}},{\bf {C}},{\bf {D}},{\bf {E}}\big)\in\Sigma$, $R>0$, and $t_{\star}\in\re$, we define the
family of perturbations $$g_{\sigma}=g(\rho_0,t_{\star}, R,g_{0}, \sigma)$$
in terms of its action on $\varkappa$ and the $du_j$:
\begin{equation*}
\begin{aligned}g_{\sigma}(u,t)(\varkappa,\varkappa) & :=g_{0}(u,t)(\varkappa,\varkappa){-}\big(\langle{\bf A},u\rangle+\tfrac{1}{2}\langle{\bf C}u,u\rangle\big)\chi\sub{R}(u,t-t_{\star}),\\
g_{\sigma}(u,t)(du^j,\varkappa) & :=g_{0}(u,t)(du^j,\varkappa)+\tfrac{1}{2}\big({\bf B}^{j}+({\bf D}u)^{j}\big)\chi\sub{R}(u,t-t_{\star}),\\
g_{\sigma}(u,t)(du^i,du^j) & :=g_{0}(u,t)(du^i,du^j)+{\tfrac{1}{2}}{\bf E}^{ij}\chi\sub{R}(u,t-t_{\star}),
\end{aligned}
\end{equation*}
where $\chi\sub{R}(u,t)=\tfrac{1}{R}\chi(\tfrac{1}{R}t)\chi(\tfrac{1}{R}|u|)$ and $\chi\in C_{c}^{\infty}((-\sqrt{2},\sqrt{2});[0,1])$
even with $\chi\equiv1$ on $[-{\tfrac{1}{4}},{\tfrac{1}{4}}]$ and $\int\chi=1$. 
We then find that in the $(u,t)$ coordinates the co-metric entries are
\begin{equation}
\label{e:perturbationDef}
\begin{aligned}
g_{\sigma}^{dd}&= g_{0}^{dd}
{-}\Big(\langle{\bf A},u\rangle+\tfrac{1}{2}\langle{\bf C}u,u\rangle+{\bf B}^{j}u^lx_{lj}+({\bf D}u)^{j}u^lx_{lj}
-\tfrac{1}{2}\mathbf{E}^{lm} u^iu^jx_{il}x_{jm}\Big)\frac{\chi\sub{R}(u,t-t_{\star})}{(1+u_ix^{id})^2}\\
g_{\sigma}^{jd}&=g_{0}^{jd}+\tfrac{1}{2}\big({\bf B}^{j}+({\bf D}u)^{j}-\mathbf{E}^{jl}u^ix_{il}\big)\frac{\chi\sub{R}(u,t-t_{\star})}{1+u^ix_{id}}\\
g_{\sigma}^{ij}&=g_{0}^{ij}+{\tfrac{1}{2}}{\bf E}^{ij}\chi\sub{R}(u,t-t_{\star}).
\end{aligned}
\end{equation}}

In order to identify $g_\sigma$ with a metric on $M$, we let $K\subset \ms{G}^3$ be bounded and work with $g_\star \in K$. 
In addition, we require that $0<R<R_0$ where $R_0=R_0(K)$ is chosen such that for all $0<R<R_0$, $g_\star\in K$, and $t\in \mathbb{R}$, we have
\begin{equation}
\label{e:rSmall}
\Phi_{\rho_0}^{g_\star}\Big(B_{\mathbb{R}^{d-1}}(0,\sqrt{2}R)\times [t-\sqrt{2}R,t+\sqrt{2}R]\big)\subset B_{g_\star}(\gamma_{\rho_0}^{g_\star}(t),3R)
\end{equation}
with $\Phi_{\rho_0}^{g_\star}$ as in~\eqref{e:picky}. This is possible since, by Lemma \ref{l:metricCoordinate}, $g^{ij}_\star(0,t)=\delta^{ij}$, for all $t\in \R$ and $i,j=1,\dots, d$.  In addition, \blue{we work with $T_0>0$ satisfying}
$$
\Big((\Phi_{\rho_0}^{g_\star})^{-1}(B(\gamma_{\rho_{0}}^{g_\star}(t_{\star}),3R))\cap B_{\mathbb{R}^{d-1}}(0,\e_f)\times [0,T_0] \Big) \subset B_{\mathbb{R}^{d-1}}(0,\e_f)\times [t_\star-3R,\blue{t_\star}+3R],
$$
\blue{and $0<t_\star-R_0\sqrt{2}<t_\star+R_0\sqrt{2}<T_0$.}
That is, the geodesic $\gamma_{\rho_0}^{g_\star}$ passes near $\gamma_{\rho_0}^{g_\star}(t_\star)$ only once between time $0$ and $T_0$.

\begin{remark}
Observe that $\chi\sub{R}$ is chosen so that the integral along $(0,t)$ is 1. 
\end{remark}

Define the map $\ze:B_{\mathbb{R}^{d-1}}(0,{\e\sub{\!f}}) \times B_{\mathbb{R}^{d-1}}(0,{\e\sub{\!f}}) \times \ms{G}^{\blue{3}}\times \Sigma \to \mc{H}_0$
\begin{equation}
\ze(z_0, g_0, \sigma):=(u_{0},0,\omega_{0},\tau_{0}(z_0, g_0, \sigma)), \qquad {z_0:=(u_0,\omega_0)} \label{e:zetaZero}
\end{equation}
where  $\tau_{0}>0$ is defined so that $|(\omega_{0},\tau_{0})|_{g_{\sigma}(u_{0},0)}=1$
and $\ze(z_0, g_0, \sigma)\in\mc{H}_{{0}}$. In particular,
\begin{equation}\label{e:zeta0}
\rho_0=\ze(0,g_0,0)=(0,0,0,1) \in \mc{H}_0.
\end{equation}
Let ${\bm{u}}, {\bm{t}},{\bm{\omega}},{\bm{\tau}}$ be functions defined on  $\R\times B_{\mathbb{R}^{d-1}}(0,{\e\sub{\!f}}) \times B_{\mathbb{R}^{d-1}}(0,{\e\sub{\!f}}) \times \ms{G}^4 \times\Sigma$ such that
\begin{align}\label{e:varphi}
\varphi_s^{g_\sigma}(\ze(z_0, g_0, \sigma))
&=\big({\bm{u}}, {\bm{t}},{\bm{\omega}},{\bm{\tau}}\big)(s, z_0,g_{0},\sigma).
\end{align}
Note that by the definition of $\ze$, for all $s$ and $\sigma$
$$ | (\bm \omega, \bm \tau)|_{g_{\sigma}(\bm u, \bm t)}=1.$$
Also,  ${\bm{t}}(s,{0},g_{\star},\blue{0})=s$.
Next, define the matrices
\begin{equation}
\label{e:deltaSigma}
\begin{gathered}
\Delta_{\sigma}(s,z_0,g_{0}):=\begin{pmatrix}\partial_{{u_{0}}}{\bm{u}}(s,z_0,g_{0},\sigma) & \partial_{{\omega_{0}}}{\bm{u}}(s,z_0,g_{0},\sigma)\\
\partial_{u_{0}}{\bm{\omega}}(s,z_0,g_{0},\sigma) & \partial_{{\omega_{0}}}{\bm{\omega}}(s,z_0,g_{0},\sigma)
\end{pmatrix},\\
\tilde{\Delta}_\sigma:=\Delta_\sigma-{[\partial_s\bm t]^{-1}}\begin{pmatrix}\partial_s\bm u\\\partial_s\bm \omega\end{pmatrix}\begin{pmatrix}\partial_{u_0}\bm t&\partial_{\omega_0}\bm t\end{pmatrix}.\end{gathered}
\end{equation}
We note that $\tilde{\Delta}_\sigma \in \Sp(2(d-1))$. Indeed, for each fixed $s$, observe that $\tilde{\Delta}_{\sigma}$ is the derivative of the following Poincar\'e map (a symplectomorphism)
$$
\mc{ P}_\sigma:\mc{H}_{0}\to\mc{H}_{\bm{t}(s,z_0, g_0, \sigma)}, \qquad 
\mc{P}_\sigma(z_0)=({\bm{u}}( S,z_0,g_{0},\sigma),{\bm{\omega}}( S,z_0,g_{0},\sigma)),
$$
where $S=S(s, z_0, g_0,\sigma)$ is defined such that
$
\varphi_{S}^{g_\sigma}(\ze(z_0, \sigma))\in \mc{H}_{\bm{t}(s,z_0, g_0, {\sigma})}.
$
Here, we are identifying $\mc{H}_0$ with $B(0,\e\sub{f})\times B(0,\e\sub{f})$ via $z_0=(u_0, \omega_0)\mapsto \ze(z_0, \sigma)$.
 Note also that the symplectic form induced on $\mc{H}_s$ from $T^*M$ is the standard symplectic form in $(u,\omega)$ coordinates.

Let $T_0 \in \R$, $g_0\in \ms{G}^4$, and $z_0 \in B_{\mathbb{R}^{d-1}}(0,{\e\sub{\!f}}) \times B_{\mathbb{R}^{d-1}}(0,{\e\sub{\!f}})$. Define the map $\cd{T}{z_0}{g_{0}}:\Sigma\to\mathbb{R}^{d-1}\times\mathbb{R}^{d-1}\times\Sp(2(d-1))$
by 
\begin{equation}\label{e:PsiMetric}
\cd{T(\zeta_{0},\sigma)}{z_0}{g_{0}, T_0}(\sigma)
:=(\Xi_1( z_0,g_0,\sigma),\Xi_2(z_0,g_0,\sigma))
\end{equation}
$$
\Xi_1( z_0,g_0,\sigma):=\big(\bm{u}(T,z_0,g_{0},\sigma),{\bm{\omega}}(T,z_0,g_{0},\sigma)\big), \qquad \Xi_2( z_0,g_0,\sigma):=\tilde{\Delta}_{\sigma}\tilde{\Delta}_{0}^{-1}(T,z_0,g_{0}),
$$
where $T=T(z_0,g_0,\sigma, T_0)$ is defined by $$\varphi\sub{T(z_0,g_0,\sigma, T_0)}^{g_{\sigma}}(\ze(z_0, g_0, \sigma))\in\mc{H}\sub{T_0}.$$
Note that $T(0,g_\star,0,T_0)=T_{0}$,  i.e., $\varphi\sub{T_{0}}^{g_{\star}}(\rho_0)\in\mc{H}\sub{T_0}$.

\begin{remark}
The choice of $\tilde{\Delta}_\sigma\tilde{\Delta}_0^{-1}$ as the last entry of $\Psi_{z_0}^{g_0}(\sigma)$ is motivated by the fact that we intend to write an ODE (in the $s$ variable) which $\partial_\delta \tilde{\Delta}_{\de\sigma}|_{\delta=0} \tilde{\Delta}_{0}^{-1}(s)$ solves. Because $\tilde{\Delta}_{\de\sigma}(s)\in \Sp(2d-1)$, $\partial_\delta\tilde{\Delta}_{\de\sigma}(s)|_{\delta=0}\in T_{\tilde{\Delta}_0(s)}\Sp(2d-1)$, while $\partial_{\delta}\tilde{\Delta}_{\de\sigma}(s)|_{\delta=0}\tilde{\Delta}_0^{-1}\in T\sub{\Id}\Sp(2d-1)=\sp(2d-1)$. Because this is a linear subspace of $\mathbb M(2(d-1))$ which is independent of $s$, ODEs posed in this space are much simpler to work with.
\end{remark}

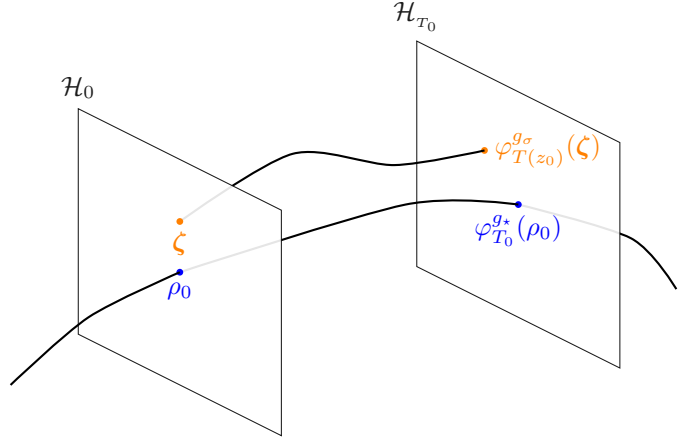
\begin{figure}
\begin{tikzpicture}
\def \w{1};
\def \h{2};
\def \shiftA{3,.6}
\def \shiftB{2,-2.5}
\def \shiftC{5,-1}

\begin{scope}[scale=1.5]

\coordinate (Aa) at  (-\w,{-.25*\w-1*\h});
\coordinate (Ab) at  (-.3\w,{-.25*\w-.7*\h});
\coordinate (B) at  (.5*\w,{-.25*\w-.5*\h});
\coordinate (Bb) at  (2.5*\w,{-.25*\w-.2*\h});
\coordinate (C) at  ($(.5*\w,{-.25*\w-.5*\h})+(\shiftA)$);
\coordinate (D) at  ($(1.5*\w,{-.25*\w-.65*\h})+(\shiftA)$);
\coordinate (E) at  ($(1.9*\w,{-.5*\w-.75*\h})+(\shiftA)$);

\coordinate (F) at  (.5*\w,{-.2*\w-.3*\h});
\coordinate (G) at  ($(-1.5*\w,{-.4*\h})+(\shiftA)$);
\coordinate (H) at  ($(-.6*\w,{-.45*\h})+(\shiftA)$);
\coordinate (I) at  ($(.2*\w,{-.35*\w-.4\h})+(\shiftA)$);






\draw [thick]plot[smooth]coordinates{(C) (D)(E)};


\begin{scope}[shift=(\shiftA)]
\draw[fill=white,opacity=.9] (-.4*\w,.4*1/2*\w)node[above]{$\mc{H}\sub{T_0}$}-- (1.4*\w,-1.4*1/2*\w)--(1.4*\w,-1.4*1/2*\w-\h)--(-.4*\w,{.4*1/2*\w-\h})--cycle; 
\fill[orange] (I)circle(0.03);
\fill[blue] (C)circle(0.03);

\end{scope}
\draw[thick] plot[smooth] coordinates{ (F)(G)(H)(I)};
\draw[thick] plot[smooth] coordinates{ (B) (Bb) (C)};
\draw[blue] (C)node[below]{$\varphi_{T_0}^{g_\star}(\rho_0)$};
\draw[orange] (I)node[right]{$\varphi^{g_\sigma}_{T(z_0)}(\ze)$};

\draw[fill=white,opacity=.9] (-.4*\w,.4*1/2*\w)node[above]{$\mc{H}_{0}$}-- (1.4*\w,-1.4*1/2*\w)--(1.4*\w,-1.4*1/2*\w-\h)--(-.4*\w,{.4*1/2*\w-\h})--cycle; 
\fill[orange] (F)circle(0.03);
\fill[blue] (B)circle(0.03);
\draw[thick] plot[smooth] coordinates{(Aa)(Ab) (B)} ;
\draw[blue] (B)node[below]{$\rho_0$};
\draw[orange] (F)node[below]{$\ze$};

\end{scope}
\end{tikzpicture}
\caption{The figure shows the setup for the definition of the functions $\bm{t}$ and $T$. Here, we abbreviate $\ze=\ze(z_0,g_0, \sigma)$ and $T(z_0)=T(z_0,g_0, \sigma)$.  The hypersurface $\mc{H}_0$ parametrized by $z_0=(u_0, \omega_0)$  via $z_0 \mapsto \ze(z_0, \sigma)$, while the corresponding points in $\mc{H}\sub{T_0}$ are $(\bm u, \bm \omega)=\Xi_1(z_0, g_0, \sigma)$. }
\end{figure}

For the purposes of the next calculations, we take $g_{0}=g_{\star}$.
In Lemma~\ref{l:bijectiveDiff} we show how to handle $g_{0}\neq g_{\star}$.
Note that, with $g_{0}=g_{\star}$, we have 
$(\bm{u}, \bm{t}, \bm{\omega}, \bm{\tau})|_{\sigma=0}=(0, s, 0,1)$ when $z_0=(0,0)$. 
Let
\begin{equation}\label{e:usANDws}
\bm{u}_{\sigma}(s):=\bm{u}(s,0,g_\star,\sigma), \qquad \bm{\omega}_{\sigma}(s)=\bm{\omega}(s,0,g_\star,\sigma), \qquad \bm{t}_{\sigma}(s)=\bm{t}(s,0,g_\star,\sigma).
\end{equation}
We continue to write ${\sigma}=\big({\bf {A}},{\bf {B}},{\bf {C}},{\bf {D}},{\bf {E}}\big)$
and define 
\begin{equation*}
L_{\sigma}:=\begin{pmatrix}\partial_{u^{k}}g_{\sigma}^{dj} & {g_{\sigma}^{jk}}\\
-\frac{1}{2}\partial_{u^{j}u^{k}}^{2}g_{\sigma}^{dd} & -\partial_{u^{j}}g_{\sigma}^{dk}
\end{pmatrix}\in\sp(2(d-1)).
\end{equation*}

Next, we calculate how derivatives of $\Delta_{\de\sigma}$ and $\Xi_1(z_0,g_0,\de\sigma)$ with respect to $\de$ behave.
\begin{lemma} 
{Let $K\subset \ms{G}^3$ bounded. Then there is $R_0>0$ such that for all $g_{\star}\in K $  the following holds.}
Let ${\rho_{0}} \in \widetilde{S^*\!M}$, $T_{0}\in \R$, $0<R<R_0$ and
$t_{\star}\in[3R,T_{0}-3R]$ be such that the set 
\begin{equation}
\label{e:setConnected}
\{t\in[0,T_{0}]:\;\gamma_{\rho_{0}}^{g_\star}(t)\cap B(\gamma_{\rho_{0}}^{g_\star}(t_{\star}),3R)\}
\end{equation}
is connected.

Let $g_{0}=g_{\star}$, $\sigma=\big({\bf {A}},{\bf {B}},{\bf {C}},{\bf {D}},{\bf {E}}\big)\in\Sigma$   and $\delta>0$.
Then, with $\bm{u}_\sigma$ and $\bm \omega_\sigma$ as defined in \eqref{e:usANDws}, {for $s\in [0,T_0]$}
\begin{equation}
\partial_{s}\begin{pmatrix}\partial_{\delta}\bm{u}_{\de\sigma}|_{\delta=0}\\
\partial_{\delta}\bm{\omega}_{\de\sigma}|_{\delta=0}
\end{pmatrix} 
=L_{0}\begin{pmatrix}\partial_{\delta}\bm{u}_{\de\sigma}|_{\delta=0}\\
\partial_{\delta}\bm{\omega}_{\de\sigma}|_{\delta=0}
\end{pmatrix}
+\tfrac{1}{2}\begin{pmatrix}\bm{B}\\
\bm{A}+\bm{X}\bm{B}
\end{pmatrix}\chi\sub{R}(0,s-t_{\star}),\label{e:firstDerivative}
\end{equation}
{where $\bm X$ is as in \eqref{e:largeX}}. In addition,
\begin{multline}\label{e:diffEqn}
\partial_{s}\big(\partial_{\delta}\Delta_{\de\sigma}\big|_{\delta=0}\Delta_{0}^{-1}\big)=\tfrac{1}{2}\begin{pmatrix}\bm{D}{-\bm{X}\bm{E}} & \bm{E}\\
\bm{C}{+\bm{X}\bm{D}+\bm{D}^t\bm{X}^t} & -\bm{D}^t{+\bm{E}^t\bm{X}^t}
\end{pmatrix}\chi\sub{R}(0,s-t_{\star})\\
+[L_{0},\partial_{\delta}\Delta_{\de\sigma}\big|_{\delta=0}\Delta_{0}^{-1}]+F\sub{R}(\bm{A},\bm{B},s),
\end{multline}
where $L_{\sigma}$ is given by~\eqref{e:lSig}, $F\sub{R}(\bm{A},\bm{B},s)\in\sp(2(d-1))$
and 
\[
|F\sub{R}(\bm{A},\bm{B},s)|\leq Ce^{C|s{-t_{\star}}|}\|(\bm{A},\bm{B})\|
\]
with $C=C(\|g_{\star}\|_{\mc{C}^{3}})$.
\end{lemma}

\begin{proof}
{There is $R_0=R_0(K)>0$ such that~\eqref{e:rSmall} holds and we can work with coordinates $(u,t)$ on the ball $B(\gamma_{\rho_{0}}^{g_\star}(t_{\star}),3R_0)$. Moreover, by the connectedness of the set in \eqref{e:setConnected}
$$
\Big((\Phi_{\rho_0}^{g_\star})^{-1}(B(\gamma_{\rho_{0}}^{g_\star}(t_{\star}),3R_0))\cap B_{\mathbb{R}^{d-1}}(0,\e_f)\times [0,T_0] \Big) \subset B_{\mathbb{R}^{d-1}}(0,\e_f)\times [t_\star-3R_0,T_\star+3R_0].
$$
Since the support of the perturbation is inside $B(\gamma_{\rho_{0}}^{g_\star}(t_{\star}),3R_0)$, this allows us to identify the perturbation $g_\sigma$ with a genuine metric perturbation on $M$ and to treat $(u,t)$ as though they were global coordinates for the purposes of the calculations in this lemma.}
By \eqref{e:varphi}, the Hamiltonian flow, $\varphi_s^{g_\sigma}$, for $|\xi|_{g_{\sigma}}$ is defined by 
\begin{equation}
\begin{aligned}\partial_{s}\bm{u}^{j} & ={g_{\sigma}^{ij}\bm{\omega}_{i}}+g_{\sigma}^{dj}\bm{\tau},\quad & \partial_{s}\bm{\omega}_{j} & =-(\tfrac{1}{2}\partial_{u^{j}}g_{\sigma}^{i\ell}\bm{\omega}_{i}\bm{\omega}_{\ell}+\partial_{u^{j}}g_{\sigma}^{d\ell}\bm{\tau}\bm{\omega}_{\ell}+\tfrac{1}{2}\partial_{u^{j}}g_{\sigma}^{dd}\bm{\tau}^{2}),\\
\partial_{s}\bm{t} & =g_{\sigma}^{dj}\bm{\omega}_{j}+g_{\sigma}^{dd}\bm{\tau},\quad & \partial_{s}\bm{\tau} & =-(\tfrac{1}{2}\partial_{t}g_{\sigma}^{i\ell}\bm{\omega}_{i}\bm{\omega}_{\ell}+\partial_{t}g_{\sigma}^{d\ell}\bm{\tau}\bm{\omega}_{\ell}+\tfrac{1}{2}\partial_{t}g_{\sigma}^{dd}\bm{\tau}^{2}),
\end{aligned}
\label{e:hamiltonEqns}
\end{equation}
where $\xi=(\bm{\omega},\bm{\tau})$.

\blue{Differentiating the Hamiltonian equations in $\delta$ and setting $\delta=0$, we obtain
\begin{equation}
    \begin{aligned}
\partial_s\partial_\delta \bm u^j|_{\delta=0}
&=
g_{\star}^{ji}\partial_\delta \bm\omega_i|_{\delta=0}
+
\partial_{u^k}g_{\star}^{dj}
\partial_\delta \bm u^k|_{\delta=0}
+
\partial_{\delta} g_{\star}^{dj}|_{\delta=0},
\\[0.5em]
\partial_s\partial_\delta \bm\omega_j|_{\delta=0}
&=
-
\partial_{u^j}g_{\star}^{d\ell}
\partial_\delta \bm\omega_\ell|_{\delta=0}
-
\frac12
\partial^2_{u^ju^k}g_{\star}^{dd}
\partial_\delta \bm u^k|_{\delta=0}
-
\frac12
\partial_{u^j}\partial_\delta g_{\star}^{dd}|_{\delta=0}.
\end{aligned}
\end{equation}}
Combining the Hamilton
equations \eqref{e:hamiltonEqns} together with \eqref{e:perturbationDef} and  Lemma \ref{l:metricCoordinate} yields \eqref{e:firstDerivative}.
Here, we are using that $\partial_{t}g_{\star}^{jd}(0,t)=g_{\star}^{jd}(0,t)=\partial_{t}\partial_{u^{j}}g_{\star}^{dd}=\partial_{u^{j}}g_{\star}^{dd}=0$,
$\partial_{\delta}\bm{u}_{\de\sigma}|_{s=0}\equiv0$, and $\partial_{\delta}\bm{\omega}_{\de\sigma}|_{s=0}\equiv0$. The last two equalities follow from the facts that $\partial_{\delta}\bm{u}_\de|_{s=0}=\partial_{\delta}(\bm{u}_\de|_{s=0})$ and $\bm{u}_\de|_{s=0}\equiv u_0$, and the analogue for $\bm{\omega}_\de$.

The same computation yields that at  $z_0=(0,0)$,
\begin{equation}\label{e:derDelta}
\partial_{s}\Delta_{0}=L_{0}\Delta_{0},\qquad\Delta_{0}|_{s=0}=I
\end{equation}
\blue{and}
\blue{that $\partial_{z_0}\bm t|_{\delta=z_0=0}=\partial_{z_0}\bm\tau|_{\delta=z_0=0} =0$.}
Differentiating~\eqref{e:hamiltonEqns} in $(u_{0},\omega_{0})$, then $\delta$, using that
$\Delta_{\sigma}|_{s=0}=\Id$,  and that $\partial_{\delta}\bm u_{\delta\sigma}|_{\delta=0}$ and $\partial_{\delta}\bm \omega _{\delta\sigma}|_{\delta=0}$ depend only on $(\bm{A},\bm{B})$ and satisfy
\begin{equation}
\label{e:deltaUEst}
\|(\partial_{\delta}\bm u_{\delta\sigma}|_{\delta=0}(s), \partial_{\delta}\bm \omega _{\delta\sigma}|_{\delta=0}(s)\|\leq Ce^{C|s-t_\star|},
\end{equation}
we find there is $\tilde{F}\sub{R}(\bm{A},\bm{B},s)$
such that at $z_0=(0,0)$,  
\[
\partial_{s}\partial_{\delta}\Delta_{\de\sigma}|_{\delta=0}=L_{0}\partial_{\delta}\Delta_{\de\sigma}|_{\delta=0}+\partial_{\delta}L_{\de\sigma}|_{\delta=0}\Delta_{0}+\tilde{F}\sub{R}(\bm{A},\bm{B},s)\Delta_{0},\qquad\partial_{\delta}\Delta_{\de\sigma}|_{s=0}=0.
\]
\blue{Indeed, using that $\partial_{\delta}\bm\tau|_{\delta=z_0=0}=\partial_{\delta}\bm t|_{\delta=z_0=0}$ (in addition to the vanishing of the various items above), we obtain
\[
\begin{aligned}
\partial_s \partial_{z_0}\partial_\delta \bm u^j\big|_{\delta=0}
&=
g_{\star}^{ji}\partial_{z_0}\partial_\delta \bm\omega_i\big|_{\delta=0}
+
\partial_{u^k}g_{\star}^{dj}
\partial_{z_0}\partial_\delta \bm u^k\big|_{\delta=0}
+
\partial_\delta g_{\star}^{ij}\big|_{\delta=0}
\partial_{z_0}\bm\omega_i\big|_{\delta=0}+
\partial_{u^k}\partial_\delta g_{\star}^{dj}\big|_{\delta=0}
\partial_{z_0}\bm u^k\big|_{\delta=0}\\
&
+
\partial_{u^k}g_{\star}^{ij}
\bigg(
\partial_{z_0}\bm u^k\big|_{\delta=0}
\partial_\delta\bm\omega_i\big|_{\delta=0}
+
\partial_\delta\bm u^k\big|_{\delta=0}
\partial_{z_0}\bm\omega_i\big|_{\delta=0}
\bigg)
+
\partial^2_{u^\ell u^k}g_{\star}^{dj}
\partial_{z_0}\bm u^\ell\big|_{\delta=0}
\partial_\delta\bm u^k\big|_{\delta=0}
\end{aligned}
\]
and
\[
\begin{aligned}
&\partial_s \partial_{z_0}\partial_\delta \bm\omega_j\big|_{\delta=0}
\\&=
-
\partial_{u^j}g_{\star}^{d\ell}
\partial_{z_0}\partial_\delta \bm\omega_\ell\big|_{\delta=0}
-
\frac12
\partial^2_{u^ju^k}g_{\star}^{dd}
\partial_{z_0}\partial_\delta \bm u^k\big|_{\delta=0}
-
\partial_{u^j}\partial_\delta g_{\star}^{d\ell}\big|_{\delta=0}
\partial_{z_0}\bm\omega_\ell\big|_{\delta=0}
-\frac12
\partial^2_{u^ju^k}\partial_\delta g_{\star}^{dd}\big|_{\delta=0}
\partial_{z_0}\bm u^k\big|_{\delta=0}\\
&\quad
-\frac12
\partial_{u^j}g_{\star}^{i\ell}
\bigg(
\partial_{z_0}\bm\omega_i\big|_{\delta=0}
\partial_\delta\bm\omega_\ell\big|_{\delta=0}
+
\partial_\delta\bm\omega_i\big|_{\delta=0}
\partial_{z_0}\bm\omega_\ell\big|_{\delta=0}
\bigg)
-
\partial^2_{u^ju^k}g_{\star}^{d\ell}
\partial_{z_0}\bm u^k\big|_{\delta=0}
\partial_\delta\bm\omega_\ell\big|_{\delta=0}
\\
&\quad
-
\partial^2_{u^ju^k}g_{\star}^{d\ell}
\partial_\delta\bm u^k\big|_{\delta=0}
\partial_{z_0}\bm\omega_\ell\big|_{\delta=0}
-\frac12
\partial^3_{u^ru^ju^k}g_{\star}^{dd}
\partial_{z_0}\bm u^r\big|_{\delta=0}
\partial_\delta\bm u^k\big|_{\delta=0}.
\end{aligned}
\]
In the derivative computations, the second and third lines consist of the terms accounted for in $\tilde{F}\sub{R}(\bm{A},\bm{B},s)\Delta_0$.}

Next, note that 
\begin{equation}\label{e:lSig}
\partial_{\delta}L_{\de\sigma}|_{\delta=0}
=\tfrac{1}{2}\begin{pmatrix} \bm{D} {-\bm{X}\bm{E}-\bm{X}_d\bm{B}^t}& \bm{E}\\
\bm{C}{+\bm{G}} & -(\bm{D}^t {-\bm{E}^t\bm{X}^t-\bm{B}\bm{X}_d^t})
\end{pmatrix}\chi\sub{R}(0,s-t_{\star}),
\end{equation}
with
$$
{\bm{G}:=\bm{X}\bm{D}+\bm{D}^t\bm{X}^t-\bm{X}\bm{E}\bm{X}^t-\bm{X}_d\bm{A}^t-\bm{A}\bm{X}_d^t-\bm{X}_d(\bm{X}\bm{B})^t-\bm{X}\bm{B}\bm{X}_d^t.}
$$
{where $\bm{X}_d$ is defined in \eqref{e:largeXd}.}
Therefore, at $z_0=(0,0)$, by \eqref{e:derDelta} we have $(\partial_{\delta}\Delta_{\de\sigma}|_{\delta=0}\Delta_{0}^{-1})|_{s=0}=0$
and conclude \eqref{e:diffEqn}.
We claim that~\eqref{e:diffEqn} implies $$F\sub{R}(\bm A,\bm B,s)\in \sp(2(d-1)).$$
Indeed, this follows from the fact that 1) $\Delta_0 \in \Sp(2(d-1))$ and 2) $\partial_\delta\Delta_{\de\sigma}|_{\delta=0}\Delta_0^{-1}\in T\sub{\Id}\Sp(2(d-1))= \sp(2(d-1))$. Claim 1) follows from the fact that $\partial_s(\bm{u},\bm{\omega})|_{(u_0,\omega_0)=0}=0$ and hence that $\Delta_0=\tilde{\Delta}_0\in \Sp(2(d-1))$.  Claim 2) The claim is not obvious since we do not know that $\Delta_{\sigma}\in \Sp(2(d-1))$ for all $\sigma \neq 0$. However, using that $\partial_{u_0}\bm t|_{\delta=0}=\partial_{\omega_0}\bm t|_{\delta=0}=0$ and $\partial_s\bm u|_{\delta=0}=\partial_s\bm \omega|_{\delta=0}=0$, we have $\partial_\delta \Delta_{\de\sigma}|_{\delta=0}=\partial_\delta \tilde{\Delta}_{\delta\sigma}|_{\delta=0}$, and hence, since $\tilde{\Delta}_{\sigma}\in \Sp(2(d-1))$, $\partial_{\delta}(\Delta_{\de\sigma}^T J \Delta_{\de\sigma})|_{\delta=0}=0$. In particular, this implies that 
$$
(\partial_\delta\Delta_{\de\sigma}|_{\delta=0}\Delta_0^{-1})^TJ\partial_\delta\Delta_{\de\sigma}|_{\delta=0}\Delta_0^{-1}= (\Delta_0^T)^{-1}(\partial_\delta (\Delta_{\delta\sigma}^TJ\Delta_{\delta\sigma})|_{\delta=0})\Delta_0^{-1}=0,
$$
and hence that  $\partial_\delta\Delta_{\de\sigma}|_{\delta=0}\Delta_0^{-1}\in \sp(2(d-1)).$


The estimate on $F_R(\bm A,\bm B,s)$, now follows from~\eqref{e:deltaUEst}.
\end{proof}

In what follows we prove that $\partial_{\sigma}\cd{T}{0}{g_{\star}}|_{\sigma=0}$ is bijective and bound its inverse. We write $\sigma=(\sigma_1, \sigma_2)$ as introduced in \eqref{e:sigma}. We note that $\Psi_{z_{0}}^{g_{0}, T_0}(\sigma)$ is defined using the Hamiltonian flow for $g_\sigma=g_{\sigma}(\rho_0,t_{\star}, R,g_{0}, \sigma)$, where $g_\sigma$ is a perturbation of the metric $g_0$ done in a neighborhood of the point $\gamma_{\rho_{0}}^{g_\star}(t_\star)$. Indeed, in the $(u,t)$ coordinates,  the perturbation is supported where $\chi\sub{R}(u, t-t_\star)$ is, i.e. for $|u|\leq \sqrt{2}R$ and $|t-t_\star|\leq \sqrt{2}R$. 

\begin{lemma} \label{l:bijLocal} 
Let $K\subset \ms{G}^3$ bounded. Then there is $C>0$ and $R_0>0$ such that for all $g_{\star}\in K$ the following holds.
Let ${\rho_{0}} \in \widetilde{S^*\!M}$, $T_{0}\in \R$, $0<R<R_0$ and
$t_{\star}\in[3R,T_{0}-3R]$ be such that the set 
\begin{equation}
\label{e:setConnected2}
\{t\in[0,T_{0}]:\;\gamma_{\rho_{0}}^{g_\star}(t)\cap B(\gamma_{\rho_{0}}^{g_\star}(t_{\star}),3R)\}
\end{equation}
is connected.
For $0<R<R_0$ and $j=1,2,$ let $\Xi_j(\sigma):=\Xi_j(0, g_\star,\sigma)$ be the maps $\Xi_j$ introduced in \eqref{e:PsiMetric}, which are defined using the perturbations $g_\sigma=g_{\sigma}(\rho_0,t_{\star}, R,g_\star, \sigma)$ introduced in \eqref{e:perturbationDef}.

Then, the map $\partial_{\sigma_1}\Xi_1(0)|\sub{\Sigma_1}:\Sigma_1\to \mathbb{R}^{2d-2}$ is bijective with inverse bounded by $Ce^{C|T_0-t_\star|}$ and $\partial_{\sigma_2}\Xi_1(0)|\sub{\Sigma_1}=0$. Also, for every fixed $\sigma_1 \in \Sigma_1$, the map $\partial_{\sigma_2}\Xi_2(\sigma_1,0)|\sub{\Sigma_2}: \Sigma_2\to \sp(2(d-1))$
is bijective with inverse bounded by $Ce^{C|T_{0}-t_{\star}|}$, and the map $\partial_{\sigma_1}\Xi_2(0)|\sub{\Sigma_1}:\Sigma_1\to \sp(2(d-1))$ is bounded by $Ce^{C|T_{0}-t_{\star}|}$.
In particular, 
$$\partial_{\sigma}\cd{T}{0}{g_{\star},T_0}|_{\sigma=0} \text{
is bijective with inverse bounded by} \;\; Ce^{C|T_{0}-t_{\star}|}.$$ 
\end{lemma}

\begin{proof}First, note that 
$\partial_{\delta} \Xi_1( \delta\sigma)|_{\delta=0}=(\partial_{\delta}\bm{u}_{\delta\sigma}(T_0),\partial_{\delta}\bm{\omega}_{\delta\sigma}(T_0))\big|_{\delta=0}$ since $\partial_{s}(\bm{u},\bm\omega)(T_{0},z_0, g_\star,0)=0.$ 
 Fix $\mathbf{v}\in\mathbb{R}^{2d-2}$. 
By the connectedness of the set~\eqref{e:setConnected2}, and the fact that $t_\star\in[3R,T_0-3R]$, the right hand side of the equation~\eqref{e:firstDerivative} is the same as that in Lemma~\ref{l:control}. In particular, taking $R_0$ small enough, Lemma~\ref{l:control} together with~\eqref{e:firstDerivative},
there are $R_0,C_{1}>0$ depending on $\|g\|_{C^{2}}$ such that
for all $R<R_0$, there are $\sigma_1=({\bf {A}},{\bf {B}})\in\mathbb{R}^{2d-2}$
satisfying 
\[
(\partial_{\delta}\bm{u}_{\delta(\sigma_1, 0)}(T_0),\partial_{\delta}\bm{\omega}_{\delta(\sigma_1, 0)}(T_0))=\mathbf{v}, \qquad \|\sigma_1\|\leq C_{1}e^{C_{1}|T_{0}-t_{\star}|}\|{\bf v}\|.
\]
Thus, the map $\sigma_1\mapsto (\partial_{\delta}\bm{u}_{\delta(\sigma_1, 0)}(T_0),\partial_{\delta}\bm{\omega}_{\delta(\sigma_1, 0)}(T_0))$, which equals $\partial_\delta \Xi_1( 0)|\sub{\Sigma_1}$,
is bijective with inverse bounded by $C_{1}e^{C_{1}|T_{0}-t_{\star}|}$.

Checking that $\partial_{\delta} \Xi_1( \delta\sigma)|_{\delta=0}=0$  when $\sigma_1=0$ is straightforward. It follows that $\partial_{\sigma_2}\Xi_1(0)|_{\Sigma_1}=0$.

Next, since $\Xi_2( \delta \sigma)=[\tilde{\Delta}_{\delta\sigma}\tilde {\Delta}_{0}^{-1}](T(z_0, g_\star, \delta\sigma, T_0), z_0, g_\star)$,  by \eqref{e:deltaSigma}
\begin{align}
\partial_{\delta}\Xi_2(\delta \sigma)\Big|_{\delta=0}
&=\Big(\partial_{\delta}\Delta_{\de\sigma} - \partial_{\delta}\Big[[\partial_s\bm{t}_{\de\sigma}]^{-1}
\begin{pmatrix}
\partial_s \bm u_{\de\sigma}\notag\\
\partial_s\bm \omega_{\de\sigma}
\end{pmatrix}
\begin{pmatrix}
\partial_{u_0}\bm t_{\de\sigma}
&\partial_{\omega_0}\bm t_{\de\sigma}
\end{pmatrix}\Big] \Big)\Big|_{\substack{s=T_0\\\delta=0}}\Delta_0^{-1}(T_0)
\\
&=\Big(\partial_\delta\Delta_{\delta\sigma}\big|_{\delta=0}\Delta_0^{-1}\Big){\Big|_{s=T_0}}.\label{e:computeTheDerivative}
\end{align}

Where to obtain the final equality, we observe that 
$$
\begin{pmatrix}\partial_{u_0}\bm t&\partial_{\omega_0}\bm t\end{pmatrix}\Big|_{\substack{s=T_0\\z_0=0}}=0,\qquad\text{and}\qquad\begin{pmatrix}\partial_{s}\bm u\\\partial_{s}\bm \omega\end{pmatrix}\Big|_{\substack{s=T_0\\z_0=0}}=0.
$$

We now write $Q_{0}$ for the solution of 
\begin{equation}\label{e:Q0}
\partial_{s}Q_{0}=[L_{0},Q_{0}]+F\sub{R}(\bm{A},\bm{B},s),\qquad Q_{0}(0)=0.
\end{equation}
Note that since $L_0(s)$, $F\sub{R}(\bm{A},\bm{B},s)\in \sp(2(d-1))$ we have $Q_0(s)\in \sp(2(d-1))$.
Observe that, working with $\sigma=(\sigma_1, 0)$,  we have $\partial_{\delta}\Xi_2(\delta \sigma)\Big|_{\delta=0}=\partial_{\sigma_1}\Xi_2(0)\sigma_1$ and so \eqref{e:diffEqn} yields
$$
\big\|\partial_{\sigma_1}\Xi_2(0)\sigma_1\big\|=\|Q_{0}(T_0)\|\leq Ce^{C|T_{0}-t_{\star}|}\|\sigma_1\|,$$
where $C=C(\|g\|_{\mc{C}^{3}})$, and hence 
$
\|\partial_{\sigma_1}\Xi_2(0)|\sub{\Sigma_1}\|\leq Ce^{C|T_{0}-t_{\star}|},
$
as claimed. 

Next, we show that $\partial_{\sigma_2}\Xi_2(0)|\sub{\Sigma_2}$ is invertible.  Fix $\mathbf{L}\in\mathfrak{sp}(2(d-1))$.  Then, we use Lemma~\ref{l:control}
to find $\sigma_2=(\bm{C},\bm{D},\bm{E})$ such that the solution $Q_{1}$ to
\begin{equation}
\partial_{s}Q_{1}=[L_{0},Q_{1}]+\frac{1}{2}\begin{pmatrix}\bm{D}{-\bm{X}\bm{E}} & \bm{E}\\
\bm{C}{+\bm{X}\bm{D}+\bm{D}^t\bm{X}^t} & -\bm{D}^t+{\bm{E}^t\bm{X}^t}
\end{pmatrix}\chi\sub{R}(0,s-t_{\star}),\qquad Q_{1}(0)=0,\label{e:lieGp}
\end{equation}
{where $\bm X$ is as in \eqref{e:largeX}},
satisfies $Q_{1}(T_{0})=\mathbf{L}$ and such that 
\[
\partial_{\sigma_2}\Xi_2(0)\sigma_2=Q_1(T_0)=\bm {L}, \qquad \|\sigma_2\|\leq Ce^{C|T_{0}-t_{\star}|}\|\bm{L}\|,
\]
and hence $\partial_{\sigma_2}\Xi_2(0)$ is invertible as claimed.

The application Lemma~\ref{l:control} requires a brief explanation.
Observe that, {again using the connectedness of~\eqref{e:setConnected2},}~\eqref{e:lieGp} is a differential equation on $\sp(2(d-1))$
of the form 
\[
\partial_{s}Q_{1}=\bm{R}Q_{1}+\sum_{i}(\bm{Y}\bm{a})_if_{i,t_{\star}}^{R}
\]
where $\bm{R}(s),{\bm{Y}(s)}$ are linear maps on $\sp(2(d-1))$ with \blue{$\bm Y$ invertible,} $\|\bm{R}\|_{L^\infty}{+\|\bm{Y}\|_{\mc{C}^{1}}}\leq C\|g\|_{\mc{C}^{\blue{3}
}}$ and {$\bm{a}\in\sp(2(d-1))$ is arbitrary}. \blue{In fact, $\bm Y$ is the map
$$
\begin{pmatrix}\bm D&\bm E\\\bm C&-\bm D^t\end{pmatrix}\mapsto \tfrac{1}{2}\begin{pmatrix}\bm D-\bm X\bm E&\bm E\\\bm C+\bm X\bm D+\bm D^t\bm X^t&-\bm D^t+\bm E^t\bm X^t\end{pmatrix}.
$$
}

Finally, we show that $\partial_{\sigma}\cd{T}{0}{g_{\star},T_0}|_{\sigma=0}$ has the required properties. Notice that, as a map from $\mathbb{R}^{2(d-1)}\times \sp(2(d-1))$ to itself, 
$$
\partial_{\sigma}\cd{T}{0}{g_{\star},T_0}|_{\sigma=0}=\begin{pmatrix} \partial_{\sigma_1}\Xi_1&\partial_{\sigma_2}\Xi_1\\ \partial_{\sigma_1}\Xi_2&\partial_{\sigma_2}\Xi_2\end{pmatrix}=\begin{pmatrix} \partial_{\sigma_1}\Xi_1&0\\ \partial_{\sigma_1}\Xi_2&\partial_{\sigma_2}\Xi_2\end{pmatrix}.
$$
\begin{remark}
Here it is crucial that $\partial_{\sigma_1}\Xi_2:\mathbb{R}^{2d-2}\to \sp(2(d-1))$ so that we may think of the map $\partial_{\sigma}\cd{T}{0}{g_{\star},T_0}|_{\sigma=0}$ acting on $\mathbb{R}^{2d-2}\times \sp(2(d-1))$. This follows from the fact that $F\sub{R}(\bm A,\bm B,s)\in \sp(2(d-1))$.
\end{remark}
In particular, we have
$$
[\partial_{\sigma}\cd{T}{0}{g_{\star},T_0}|_{\sigma=0}]^{-1}=\begin{pmatrix} (\partial_{\sigma_1}\Xi_1)^{-1}&0\\ -[\partial_{\sigma_2}\Xi_2]^{-1}\partial_{\sigma_1}\Xi_2(\partial_{\sigma_1}\Xi_1)^{-1}&[\partial_{\sigma_2}\Xi_2]^{-1}\end{pmatrix},
$$
from which the estimates on $[\partial_{\sigma}\cd{T}{0}{g_{\star},T_0}|_{\sigma=0}]^{-1}$ easily follow.
\end{proof}

Before we prove that the estimates in Lemma~\ref{l:bijLocal} are stable under small changes of the metric or initial position, we need to control how much these changes affect $\partial_{\sigma}\cd{T}{0}{g_{\star},T_0}|_{\sigma=0}$. This is done in our next lemma.

\blue{
\begin{lemma} \label{l:bijectiveDiff} 
Given $K\subset {\ms{G}^{3}}$ bounded,
there is $C>0$ such that the following holds.

Let $g_\star\in K$ and  $R_0$ be as in Lemma \ref{l:bijLocal}. Let $\rho_{0} \in \widetilde{S^*\!M}$, $T_{0}\in \R$, $0<R<R_0$, and
$t_{\star}\in[3R,T_{0}-3R]$ be such that the set 
$
\{t\in[0,T_{0}]:\;\gamma_{\rho_{0}}^{g_\star}(t)\cap B(\gamma_{\rho_{0}}^{g_\star}(t_{\star}),3R)\}
$
is connected. Then, for all   $g_0 \in {\ms{G}^{3}}$ with $\|g_0-g_\star\|_{\mc{C}^3}\leq 1$,  $z_0\in B_{\mathbb{R}^{d-1}}(0,{\e\sub{\!f}}) \times B_{\mathbb{R}^{d-1}}(0,{\e\sub{\!f}})$,   we have
\begin{equation}
\|d_{\sigma}(\Psi_{z_{0}}^{g_{0}, T_0}-\Psi_{0}^{g_{\star}, T_0})|_{\sigma=0}\|_{\mc{C}^{0}}\leq C\big(\|g_{0}-g_{\star}\|_{\mc{C}^{3}}+|z_0|\big)e^{C|T_{0}|},\label{e:thePainClaim}
\end{equation}
whenever $\|g_0-g_\star\|_{\mathcal{C}^1}+R\|g_0-g_\star\|_{\mathcal{C}^2}+|z_0|\leq Re^{-C(T_0+1)} $,
where $\Psi_{z_{0}}^{g_{0}}$ is the map introduced in \eqref{e:PsiMetric}, defined using the perturbations $g_\sigma=g_{\sigma}(\rho_0,t_{\star}, R,g_0, \sigma)$ introduced in \eqref{e:perturbationDef}.
\end{lemma}}
\blue{
\begin{proof} 

In this proof $C_j$ will denote a constant that may change from line to line but depends only $\|g_0\|_{\mathcal{C}^j}$ and $\|g_\star\|_{\mathcal{C}^j}$. In addition, $C$ denotes a constant that does not depend on $g_0$ or $g_\star$.

First, observe that $\|H\sub{|\xi|_{g_{0}}}-H\sub{|\xi|_{g_{\star}}}\|_{\mc{C}^{k-1}}\leq\|g_{0}-g_{\star}\|_{\mc{C}^{k}}.$
In particular, defining $\rho_0(t)$ and $\rho_\star(t)$ such that $\partial_{t}\rho_{0}(s)=H\sub{|\xi|_{g_{0}}}(\rho_{0}(s))$,
$\partial_{t}\rho_{\star}(s)=H\sub{|\xi|_{g_{\star}}}(\rho_{\star}(s))$,
and $\rho_{0}(0)=\rho_{\star}(0)$, we obtain for all $s\in \R$
\begin{align*}
|\rho_{0}(t)-\rho_{\star}(t)| & \leq\int_{0}^{t}|H\sub{|\xi|_{g_{0}}}(\rho_{0}(s))-H\sub{|\xi|_{g_{\star}}}(\rho_{\star}(s))|ds\\
 & \leq\int_{0}^{t}C_2|\rho_{0}(s)-\rho_{\star}(s)|ds+t\|g_{0}-g_{\star}\|_{\mc{C}^{1}}.
\end{align*}
Therefore, using Gronwall's inequality, we have
\[
d(\varphi_s^{g_0}(\rho),\varphi_s^{g_\star}(\rho))\leq C_2\|g_{0}-g_{\star}\|_{\mc{C}^{1}}e^{C_2|s|},\qquad\text{for all }\rho\in T^{*}M\setminus\{0\}.
\]
In addition, recall that for every metric $g$ there is $C=C(\|g\|_{\mc{C}^{2}})$
such that for all $\rho_{1},\rho_{2}\in T^{*}M\setminus\{0\}$ 
\[
d(\varphi_s^{g}(\rho_{1}),\varphi_s^{g}(\rho_{2}))\leq Cd(\rho_{1},\rho_{2})e^{C|s|}.
\]
In particular,  for $\|g_{\star}-g_{0}\|_{\mc{C}^{2}}\leq1$, and $\rho_{0},\rho_{\star}\in T^{*}M\setminus\{0\}$, $t\in \R$,
\begin{equation}
d(\varphi_s^{g_0}(\rho_{0}),\varphi_s^{g_\star}(\rho_{\star}))\leq C_2(\|g_{\star}-g_{0}\|_{\mc{C}^{1}}+d(\rho_{0},\rho_{\star}))e^{C_2|s|}.\label{e:perturbedFlows}
\end{equation}

It is convenient below to rewrite the equations~\eqref{e:hamiltonEqns} with $t$ as the independent variable and to solve for $\tau$ using the Hamiltonian structure. Doing this, we arrive at
\begin{equation}
    \label{e:reducedEquations}
\begin{aligned}
\partial_t\bm{u}^j&=(g_{\delta \sigma}^{ij}\bm\omega_i+g_{\delta \sigma}^{dj}\bm \tau(\bm u, t,\bm \omega))\big/(g_{\delta \sigma}^{dj}\bm\omega_j +g_{\delta \sigma}^{dd}\bm \tau (\bm u,t,\bm \omega)),\\
\partial_t\bm{\omega}_j&=-(\tfrac{1}{2}\partial_{u^j}g_{\delta \sigma}^{i\ell}\bm\omega_i\bm \omega_\ell+\partial_{u^j}g_{\delta \sigma}^{d\ell}\bm\tau(\bm u, t,\bm \omega)\bm\omega_{\ell}+\tfrac{1}{2}g_{\delta \sigma}^{dd}\bm \tau^2(\bm u, t,\bm \omega))\big/(g_{\delta \sigma}^{dj}\bm\omega_j +g_{\delta \sigma}^{dd}\bm \tau (\bm u,t,\bm \omega)).
\end{aligned}
\end{equation}
Note that $\rho_0(t)$ is the solution of these equations with metric $g_{0,\delta \sigma}$ and initial condition $\rho_0$, and $\rho_{\star}(t)$ is that with metric $g_{\star,\delta \sigma}$ and $\rho_{\star}$.
The change of variables is possible since~\eqref{e:perturbedFlows} implies that $|\bm \omega|\leq R$ and $|\bm \tau -1|\leq R$ and $g^{dd}_{\delta \sigma}\geq \tfrac{1}{2}$. Moreover, this implies that $\bm\tau$ is a smooth function of $\bm u, t,\bm \omega$ and the initial data given by solving
$$
g_{\delta \sigma}^{dd}\bm\tau^2 +g_{\delta \sigma}^{dj}\bm \tau \bm \omega_j +g_{\delta \sigma}^{ij}\bm \omega_i\bm\omega_j= |\xi(0)|_{g_{\delta \sigma}(x(0))}^2,\qquad \rho(0)=:(x(0),\xi(0)).
$$

Notice that, the definition of $\Psi_{z_0}^{g_0,T_0}(\sigma)$ (see~\eqref{e:PsiMetric}) implies that 
$$\Psi_{z_0}^{g_0,T_0}(\delta \sigma)=\Big(\bm u(T_0,\delta;\sigma)\,,\,\bm \omega (T_0,\delta;\sigma),  [\partial_{z_0}(\bm u,\bm \omega)(T_0,\delta;\sigma)] [\partial_{z_0}(\bm u,\bm \omega)(T_0,0;0)]^{-1}\Big)$$
and similarly for $\Psi_{0}^{g_\star,T_0}$. 
In particular, recall that in the $(u,t)$ coordinates, $\mathcal{H}_{T}=\{ (u,T)\,:u\in B(0,\e_{f})\}$.

\noindent\emph {Step 1.} Our first goal is to estimate $|\partial_{\delta}(\bm u_{\star},\bm \omega_{\star})-\partial_{\delta}(\bm u_{0},\bm \omega_{0})|$.
Let 
$$
\tilde{\rho}_{\bullet}(t,\delta;\sigma):=(\bm u_{\bullet},\bm \omega_{\bullet})(t,\delta;\sigma).
$$
Observe that~\eqref{e:firstDerivative} implies
\begin{align}\label{e:Ddelta-rho}
    |\partial_{\delta}\tilde{\rho}_\star(t)|&\leq C_2 e^{C_2 t}\|\sigma\|.
    \end{align}
In addition, \eqref{e:perturbedFlows} implies
\begin{equation}
d(\tilde{\rho}_{\star}(t,0; \sigma),\tilde{\rho}_0(t,0;\sigma))\leq C_2(\|g_{\star}-g_{0}\|_{\mc{C}^{1}}+d(\rho_{0},\rho_{\star}))e^{C_2|t|}\leq \tfrac{1}{4}R.\label{e:perturbedFlowsb}
\end{equation}
Observe that
\begin{equation} 
\label{e:tildeRhoEquationOne}
\partial_t\tilde{\rho}_\bullet=\mc{V}\big(g_{\bullet}(\tilde{\rho}_{\bullet}), \partial_u g_{\bullet}(\tilde{\rho}_{\bullet}),\tilde{\rho}_{\bullet}\big),
\end{equation}
where $\mathcal{V}(g,G, \rho)$ is obtained by replacing $g^{ij}_{\delta\sigma}(\bm u, t)$ by $g^{ij}$ and $\partial_{u^\ell}g^{ij}_{\delta \sigma}(\bm u,t)$ by $G^{ij}_{\ell}$ in the right-hand side of~\eqref{e:reducedEquations}. 
Differentiating in $\delta$, we obtain
\begin{equation} 
\label{e:tildeRhoEquation}
\begin{gathered}
\partial_t(\partial_{\delta}\tilde{\rho}_\bullet)=\partial_g\mc{V}\partial_{\delta}g_{\bullet}+\partial_G\mc{V}\partial_{\delta}\partial_u g_{\bullet}(\tilde{\rho}_{\bullet})+L\big(g_{\bullet}(\tilde{\rho}_{\bullet}),\partial_u g(\tilde{\rho}_\bullet),\partial^2_ug(\tilde{\rho}_\bullet),\tilde{\rho}_\bullet\big)\partial_{\delta}\tilde{\rho}_{\bullet},\\
L(g,G,\mathcal{G}, \tilde{\rho}):=\partial_{\tilde{\rho
}} \mathcal{V}(g,G,\tilde{\rho})+\partial_g \mathcal{V}(g,G,\tilde{\rho})G+\partial_{G}\mathcal{V}(g,G,\tilde{\rho})\mathcal{G} .
\end{gathered}
\end{equation}
Hence,
\begin{equation}
\label{e:differenceFlow1}
\begin{aligned}
\partial_{s}(\partial_{\delta}(\tilde{\rho}_0-\tilde{\rho}_{\star}))
&=L(g_{0},\partial_u g_0,\partial_u^2g_0,\tilde{\rho}_0)(\partial_{\delta}(\tilde{\rho}_0-\tilde{\rho}_\star)) \\
&\;\;\;+\big(L(g_0,\partial _ug_0,\partial^2_ug_0,\tilde{\rho}_0)-L(g_{\star},
\partial_u g_\star,\partial^2_ug_\star,\tilde{\rho}_{\star})\big)\partial_{\delta}\tilde{\rho}_{\star}\\
&\;\;\; +\partial_g\mathcal{V}(g_{0},\partial_u g_0,\tilde{\rho}_0)\big(\partial_{\delta}g_0(\tilde{\rho}_0)-\partial_{\delta}g_0(\tilde{\rho}_\star)\big)\\
&\;\;\;+\big(\partial_g\mathcal{V}(g_0,\partial_u g_0,\tilde{\rho}_0)-\partial_g\mathcal{V}(g_{\star},\partial_u g_{\star},\tilde{\rho}_{\star})\big)\partial_{\delta}g_{\star}(\tilde{\rho}_\star)\\
&\;\;\; +\partial_G\mathcal{V}(g_{0},\partial_u g_0,\tilde{\rho}_0)\big(\partial_{\delta}\partial_u g_0(\tilde{\rho}_0)-\partial_{\delta}\partial_u g_0(\tilde{\rho}_\star)\big)\\
&\;\;\;+\big(\partial_G\mathcal{V}(g_0,\partial_u g_0,\tilde{\rho}_0)-\partial_G\mathcal{V}(g_{\star},\partial_u g_{\star},\tilde{\rho}_{\star})\big)\partial_{\delta}\partial_u g_{\star}(\tilde{\rho}_\star)\\
\end{aligned}
\end{equation}
Notice that  
\begin{gather*}
    |L(g,G,\mathcal{G},\tilde{\rho})|\leq C(\|g\|,\|G\|,\|\mathcal{G}\|),\qquad  |\partial_{g}\mathcal{V}(g,G,\tilde{\rho})|\leq C(\|g\|,\|G\|),\qquad
|\partial_{G}\mathcal{V}(g,G,\tilde{\rho})|\leq C(\|g\|),\\
\begin{aligned}
|(L(g_0,\partial_u g_0, \partial^2_ug_0,\tilde{\rho}_0)-L(g_{\star},\partial _ug_\star,\partial^2_ug_\star, \tilde{\rho}_{\star}))|&\leq C_2\|g_0-g_\star\|_{\mathcal{C}^2}+C_3 d(\tilde{\rho}_0,\tilde{\rho}_\star),\\
|\big(\partial_g\mathcal{V}(g_0,\partial_u g_0,\tilde{\rho}_0)-\partial_g\mathcal{V}(g_{\star},\partial_u g_\star, \tilde{\rho}_{\star})\big)|&\leq C_1\|g_0-g_\star\|_{\mathcal{C}^2}+C_2 d(\tilde{\rho}_0,\tilde{\rho}_\star) ,\\
|\big(\partial_G\mathcal{V}(g_0,\partial_u g_0,\tilde{\rho}_0)-\partial_G\mathcal{V}(g_{\star},\partial_u g_\star, \tilde{\rho}_{\star})\big)|&\leq C_1\|g_0-g_\star\|_{\mathcal{C}^2}+C_2 d(\tilde{\rho}_0,\tilde{\rho}_\star).
\end{aligned}
\end{gather*}

Next, observe that for $|u|\leq \tfrac{1}{4}R$, 
$$
\partial_{\delta}g^{ij}_{\bullet}(u, t)= P^{ij}(u) \chi\sub{R}(t-t_\star), 
$$
where $P^{ij}=P^{ij}(\sigma)$ is polynomial of degree 2 whose coefficients are bounded by $C\|\sigma\|$. 

Using this together with the fact that $\|\bm{u}_0\|\leq \tfrac{1}{4}R$ and $\bm u_\star=0$, we have 
\begin{align*}
    |(\partial_{\delta}g_\star)({\tilde{\rho}}_\star,t)|+|(\partial_{\delta}\partial_u g_\star)({\tilde{\rho}}_\star,t)|&\leq C\|\sigma\|\chi\sub{R}(t-t_\star),\\
    |(\partial_{\delta}\partial_u g_0)(\tilde{\rho}_0,t)-(\partial_{\delta}\partial_u g_0)(\tilde{\rho}_\star,t)|
    &\leq C\|\sigma\|\chi\sub{R}(t-t_\star)\sup_{s\in[0,t]} d(\tilde{\rho}_0(s), \tilde{\rho}_\star(s)),\\
    |(\partial_{\delta}g_0)({\tilde{\rho}}_0,t)-(\partial_{\delta}g_0)({\tilde{\rho}}_\star,t)|&\leq C\|\sigma\|\chi\sub{R}(t-t_\star)\sup_{s\in[0,t]}  d(\tilde{\rho}_0(s), \tilde{\rho}_\star(s)).
\end{align*}
Using this in~\eqref{e:differenceFlow1}, the estimate~\eqref{e:perturbedFlows}, and the fact that $\partial_{\delta}\tilde{\rho}_{\bullet}(0)=0$, we obtain 
\begin{align*}
&|\partial_\delta ((\bm u_{0}-\bm u_{\star},\bm \omega_0-\bm u_{\star})|(t, 0, \sigma)
=|\partial_\delta (\tilde{\rho}_0-\tilde{\rho}_{\star})|(t,0, \sigma)\\
&\hspace{4cm}\leq C_3\|\sigma\|e^{C_2t}\big(\|g_\star-g_0\|_{\mathcal{C}^2}+d(\rho_0(0),\rho_\star(0))\big). 
\end{align*}
Here, the constant $C_2$ in the exponential comes from the estimate on $L$, via Gronwall's inequality, \eqref{e:Ddelta-rho}, and~\eqref{e:perturbedFlows}.

\noindent\emph {Step 2.} Our second goal is to estimate $|\partial_{\delta}\partial_{(u_0,\omega_0)}(\bm u_{\star},\bm \omega_{\star})-\partial_{\delta}\partial_{(u_0,\omega_0)}(\bm u_{0},\bm \omega_{0})|$.

Differentiating~\eqref{e:tildeRhoEquation} with respect to $(u_0,\omega_0)$ and estimating similarly we obtain
\begin{align*}
&|\partial_\delta \partial_{(u_0,\omega_0)}((\bm u_{0}-\bm u_{\star},\bm \omega_0-\bm u_{\star})|(t, 0, \sigma)
=|\partial_\delta (\tilde{\rho}_0-\tilde{\rho}_{\star})|(t,0, \sigma)\\
&\hspace{4cm}\leq C_4\|\sigma\|e^{C_2t}\big(\|g_\star-g_0\|_{\mathcal{C}^3}+d(\rho_0(0),\rho_\star(0))\big). 
\end{align*}

\begin{remark}
    We have estimated any first order partial derivative in $\sigma$ of $(\bm{u},\bm \omega, \partial_{z_0}(\bm u,\bm \omega))$ at $\sigma=0$. This is sufficient to estimate the derivative since these functions are $\mathcal{C}^1$. 
\end{remark}

\noindent \emph{Step 3.} We now estimate 
$$
(\partial_{z_0}(\bm u_0,\bm \omega_0)(T_0,0;0))-(\partial_{z_0}(\bm u_\star,\bm \omega_\star)(T_0,0;0)).
$$
We do this differentiating~\eqref{e:tildeRhoEquationOne} in $z_0$ and taking the difference leading to an analog of~\eqref{e:differenceFlow1} with $\delta$ derivatives replaced by $z_0$ derivatives:
\begin{align*}
\partial_s(\partial_{z_0}(\tilde{\rho}_0-\tilde{\rho}_\star))
&=L(g_{0},\partial_u g_0,\partial_u^2g_0,\tilde{\rho}_0)(\partial_{z_0}(\tilde{\rho}_0-\tilde{\rho}_\star)) \\
&\;\;\;+\big(L(g_0,\partial _ug_0,\partial^2_ug_0,\tilde{\rho}_0)-L(g_{\star},
\partial_u g_\star,\partial^2_ug_\star,\tilde{\rho}_{\star})\big)\partial_{z_0}\tilde{\rho}_{\star}.
\end{align*}
Hence, 
\begin{equation} 
\label{e:1moreDerivative}
|(\partial_{z_0}(\bm u_0,\bm \omega_0)(T_0,0;0))-(\partial_{z_0}(\bm u_\star,\bm \omega_\star)(T_0,0;0))|\leq  C_3e^{C_2T_0}(\|g_\star -g_0\|_{\mathcal{C}^2}+d(\rho_0(0),\rho_\star(0))).
\end{equation}
To finish the proof, observe that by~\eqref{e:derDelta} $\partial_{z_0}(\bm u_\star,\bm \omega_\star)$ is a symplectic matrix with norm bounded by $C_2e^{C_2t}$ and hence inverse bounded by $C_2e^{C_2 t}$. Using the identity
$$
A^{-1}-B^{-1}= -A^{-1}(B-A)(I+A^{-1}(B-A))^{-1}A^{-1}
$$
with $A=\partial_{z_0}(\bm u_\star,\bm \omega_\star)$ and $B= \partial_{z_0}(\bm u_0,\bm \omega_0)$ together with the fact that by~\eqref{e:1moreDerivative}, and this $A$ and $B$,
and, since $\|g_0-g_\star\|_{\mathcal{C}^2}+d(\rho_0(0),\rho_\star(0))\leq e^{-C_2(T_0+1)}$, $\|A^{-1}\|\|B-A\|\leq \frac{1}{2}$, we obtain
$$
\|(\partial_{z_0}(\bm u_0,\bm \omega_0)(T_0,0;0))^{-1}-(\partial_{z_0}(\bm u_\star,\bm \omega_\star)(T_0,0;0))^{-1}\|\leq C_3e^{C_2T_0}(\|g_\star -g_0\|_{\mathcal{C}^2}+d(\rho_0(0),\rho_\star(0))).
$$

Using that~\eqref{e:diffEqn} implies
$$
\|\partial_{\sigma} \partial_{z_0}(\bm u_\star ,\bm \omega_{\star})\|\leq C_3e^{C_3T_0}
$$
then finishes the proof.

 \end{proof}
}

\begin{figure}
\begin{tikzpicture}
\def \w{3.5};
\def \h{2};
\def \shiftA{.5}
\def \shiftB{.5}

\fill[light-gray] (.9,.1)circle(1.2);
\draw[->] (.9,1.8)node[above]{$\supp (g_0-g_\sigma)$}--(.9,1.325);
\draw[name path= H0,dashed] plot [smooth] coordinates {(-\w,\h) (-\w+\shiftA,0) (-\w,-\h)};
\draw (-\w,\h) node[left]{$\mc{H}_0$};
\draw [name path= H1,dashed] plot [smooth] coordinates {(\w,\h) (\w+\shiftB,0) (\w,-\h)};
\draw (\w,\h) node[left]{$\mc{H}_{\tilde{T}_0}$};

\draw plot[smooth] coordinates {(-\w,0) (-.5,.1) (\w+\shiftB,0)};
\draw (-.5,.1) node[below]{$\gamma^{g_\star}_\rho$};
\draw (\w+\shiftB,0)node[right]{$\varphi_{T_0}^{g_\star}(\rho)$};
\fill[blue](\w+\shiftB,0) circle(.07);

\draw [name path= Gamma0] plot [smooth] coordinates {(-\w+1.2,\h+.2) (-\w,0) (-\w-.5,-\h)};
\draw (-\w+1.2,\h+.2) node[right]{$\mc{V}_{i_0}$};

\draw (-\w,0)node[left]{$\rho$};
\fill[blue] (-\w,0)circle(.07);

\draw [name path= Gamma1] plot [smooth] coordinates {(\w+1.3,\h+.2) (\w,0) (\w-1,-\h-.1)};
\draw (\w+1.3,\h+.2) node[right]{$\mc{W}_{i_1}$};

\draw[name intersections={of=H0 and Gamma0, by=rhoT}]  (rhoT)node[left]{$\tilde{\rho}$};
\fill[fill=blue]  (rhoT) circle (.07);

\draw[name intersections={of=H1 and Gamma1, by=rhoTF}]  plot [smooth] coordinates{(rhoT) (-\w/2,\h/3) (-.2,3*\h/7)(\w/2,\h/3) (rhoTF)};
\draw (rhoTF)node[right]{$\tilde{\rho}_1$};
\fill[fill=blue]  (rhoTF) circle (.07);
\draw (-.2,3*\h/7) node[above]{$\gamma^{g_0}_{\tilde{\rho}}$};
\draw (-\w+\shiftA,-.2)node[right]{$\rho_0$};
\fill[blue] (-\w+\shiftA,0)circle(.07);

\end{tikzpicture}
\caption{The figures shows the geometric setup used in Lemma~\ref{l:finalNondegenerate} to pass from the special hypersurfaces defined using the coordinates $\Phi_\rho^{g_\star}$ to a given well separated set $(\Gamma,\tilde{\Gamma})$.}
\end{figure}
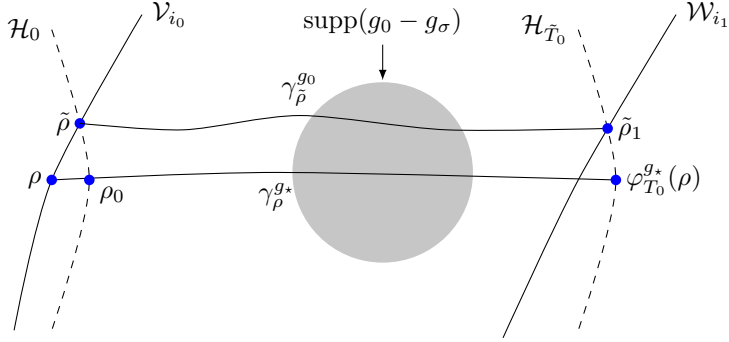

In what follows we use the notion of a well-separated set introduced in Definition \ref{d:wellSeparated}.

\begin{lemma}
\label{l:finalNondegenerate}
Let $G\subset \ms{G}^{\blue{3}}$ bounded, $\fuInj>0$, and  $\{(\mc{W}_i,\mc{V}_i)\}_{i=1}^N$ be $\fuInj$-well separated for $\Gb$. There exist $\delta>0$, $R_{0}>0$, $C>0$, and $C_{1}>0$, such that the following holds. Let $g_\star\in \Gb$, $T_0>0$, $t_{\star}\in[3R_{0},T_{0}-3R_{0}]$, $\rho\in \mc{V}_{i_0}$, and $i_1 \in \{1, \dots, N\}$ such that
$$
\varphi_{T_0}^{g_\star}(\rho)\in \mc{W}_{i_1},\qquad d(\varphi_{T_0}^{g_\star}(\rho),\mc{V}_{i_1})<\delta,
$$
and the set
\begin{equation}
\label{e:onlyOnce}
\{t\in[0,T_{0}]:\;\gamma^{g_\star}_{\rho}(t)\cap B(\gamma^{g_\star}_{\rho}(t_{\star}),3R_{0})\neq \emptyset\}
\end{equation}
is connected. Then for all 
\[
0<R<R_{0},\qquad0<\e\leq \blue{R}e^{-C_{1}|T_{0}|}/C_{1},
\]
$g_{0}\in \blue{\ms{G}^{3}}$ and $\tilde{\rho}\in \mc{V}_{i_0}$
with 
\[
\|g_{\star}-g_{0}\|_{\ms{G}^{\blue{3}}}\leq\e,\qquad\|d(\tilde{\rho},\rho)\|\leq\e,
\]
the map 
$$
d_\sigma(\ms{P}_{\sigma}|_{q=\tilde{\rho}}, d_q\ms{P}_\sigma|_{q=\tilde{\rho}})|_{\sigma=0}:\Sigma\to T_{\ms{P}_{0}(\tilde{\rho})}\mc{W}_{i_0}\times T_{d_{q}\ms{P}_0|_{q=\tilde{\rho}}}\Sp(T_{\tilde{\rho}}\mc{W}_{i_0};T_{\ms{P}_0(\tilde{\rho})}\mc{W}_{i_1}),
$$
is
bijective with inverse bounded by $Ce^{C|T_{0}|}$, where $\ms{P}_{\sigma}:B_{\mc{V}_{i_0}}(\tilde{\rho}, e^{-C_1|T_0|}/C_1)\subset \mc{V}_{i_0}\to\mc{W}_{i_1}$ is the map
$$
\ms{P}_{\sigma}(q):=\varphi_{T^{g_\sigma}(q)}^{g_\sigma}(q),\qquad T^{g_\star}(\tilde{\rho})=T_0,\qquad \varphi_{T^{g_\sigma}(q)}^{g_\sigma}(q)\in \mc{W}_{i_1},
$$
with $T^{g_\sigma}(q)$ smooth in $\sigma$ and $q$, 
$g_\sigma=g_\sigma(\rho, t_\star, R, g_0, \sigma)$ is the perturbed metric defined in \eqref{e:perturbationDef}, and $\Sp(T_{\tilde{\rho}}\mc{W}_{i_0};T_{\ms{P}^0(\tilde{\rho})}\mc{W}_{i_1})$ denotes the symplectic linear maps from $T_{\tilde{\rho}}\mc{W}_i$ to $T_{\ms{P}_0(\tilde{\rho})}\mc{W}_{i_1}$.
\end{lemma}

\begin{proof} Let $(u_0,t_0)$ be such that $\Phi_{\rho}^{g_\star}(u_0,t_0)=\pi\sub{M}(\tilde{\rho})$. Next, let  $\rho_0:=\gamma_{\rho}^{g_\star}(t_0)$. From now on we work with the system of coordinates $(u,t)$ induced by $(u, t)\mapsto \Phi_{\rho_0}^{g_\star}(u,t)$ and with $\mc{H}_s$ as defined in \eqref{e:H}. Note that 
$
\tilde{\rho}\in \mc{V}_{i_0}\cap \mc{H}_{0}.
$

Let $\tilde{\rho}_1:=\ms{P}\sub{0}(\tilde{\rho})=\varphi_{T^{g_0}(\tilde \rho)}^{g_0}(\tilde \rho) \in \mc{W}_{i_1}$, and $\tilde T_0$ be such that  $\tilde{\rho}_1\in \mc{W}_{i_1}\cap \mc{H}\sub{\tilde{T}_0}$
Define the maps 
$$
\mc{Q}_{0,\sigma}:\mc{W}_{i_0}\to \mc{H}_{0},\qquad \mc{Q}_{1,\sigma}:\mc{H}\sub{\tilde{T}_0} \to \mc{W}_{i_1},
$$
by 
$$
\mc{Q}_{0,\sigma}(q)=\varphi^{g_\sigma}\sub{S_0(q)}(q),\qquad \mc{Q}_{1,\sigma}(q)=\varphi^{g_\sigma}\sub{S_1(q)}(q),
$$
where $S_0(q)$ is the real number with smallest absolute value such that $\varphi^{g_\sigma}\sub{S_0(q)}(q)\in \mc{H}_{t_0}$, and $S_1(q)$ is the real number with smallest absolute value such that $\varphi^{g_\sigma}\sub{S_1(q)}(q)\in \mc{W}_{i_1}$.

Note that, by construction, the perturbation is supported on $(u,t)$ such that $|t-t_\star|\leq \sqrt{2}R$. Hence, since $t_\star>3R_0$ and $R<R_0$, the perturbation is supported away from $\mc{W}_{i_0}$ and $\mc{W}_{i_1}$. Therefore, choosing $R_0$ small enough, $\mc{Q}_{0,\sigma}(q)=\mc{Q}_{0,0}(q)$ for all $q$ with $d(q,\rho)\ll R_0$ and $\mc{Q}_{1,\sigma}(q)=\mc{Q}_{1,0}(q)$ for all $q$ with $d(q,\rho_1)\ll R_0$. 

In the $(u,t)$ coordinate system, with $\omega$ dual to $u$, we write $\mc{Z}_{0,\sigma}$ and $\mc{Z}_{1,\sigma}$ for the $(u, \omega)$ representation of $\mc{Q}_{0,\sigma}$ and $\mc{Q}_{1,\sigma}$ respectively.
Next, consider the map 
$$\cd{T(\zeta_{0},\sigma)}{z_0}{g_{0}, \tilde T_0}(\sigma)
:=(\Xi_1( z_0,g_0,\sigma),\Xi_2(z_0,g_0,\sigma))$$
defined in \eqref{e:PsiMetric} with $\tilde T_0$ in place of $T_0$. 
Using that the perturbation is supported away from $\mc{W}_{i_0}$ and $\mc{W}_{i_1}$, we have $\mc{Z}_{j,\sigma}(z_0)=\mc{Z}_{j,0}(z_0)$ for $|z_0|\ll R_0$. 

Now, abbreviating $\Xi_{j,\sigma}(z_0):=\Xi_j(z_0, g_0, \sigma)$, and putting $\tilde{\Xi}_{2,\sigma}(z_0):= \Xi_{2,\sigma}(z_0)\tilde{\Delta}_0(T,z_0,g_0)$, with $\tilde{\Delta}_0$ as in~\eqref{e:deltaSigma} and $T=T(z_0,g_0,\sigma,\tilde{T}_0)$ as in~\eqref{e:PsiMetric} with $T_0$ replaced by $\tilde{T}_0$, we have (in the $(u_0,\omega_0)$ coordinates),
\begin{equation}
\label{e:chai}
\ms{P}_{\sigma}=\mc{Z}_{1,\sigma}\circ \Xi_{1, \sigma}\circ\mc{Z}_{0,\sigma},\qquad
d_\rho \ms{P}_{\sigma}= 
d_\rho\mc{Z}_{1,\sigma}\circ \tilde{\Xi}_{2,\sigma} \circ d_\rho\mc{Z}_{0,\sigma}.
\end{equation}
Thus, since for $|z_0|<\e$ with $C_1$ to be determined later, and $\|\sigma\|\leq 1$, we have  $|\Xi_{1,\sigma}(\mc{Z}_{0,\sigma}(z_0))|\ll R_0$ and $|z_0|\ll R_0$, and hence we may replace $\mc{Z}_{j,\sigma}$ by $\mc{Z}_{j,0}$ in~\eqref{e:chai} to obtain
\begin{align*}
&d_{\sigma}((\ms{P}_{\sigma}(z_0),d_\rho \ms{P}_{\sigma}(z_0)))|_{\sigma=0}\\
&=\Big(d_\rho\mc{Z}_{1,0}|_{\Xi_{1,0}(\mc{Z}_{0,0}(z_0))}d_{\sigma}\Xi_{1,\sigma}(\mc{Z}_{0,0}(z_0))|
_{\sigma=0}\;,\;
d^2_\rho\mc{Z}_{1,0}|_{\Xi_{1,0}(\mc{Z}_{0,0}(z_0))}\tilde{\Xi}_{2,0}d_{\rho}\mc{Z}_{0,0}\\
&\hspace{7.5cm}+d_\rho\mc{Z}_{1,0}|_{\Xi_1(\mc{Z}_{0,0}(z_0)}d_\sigma\tilde{\Xi}_{2,\sigma}(\mc{Z}_{0,0}(z_0))|_{\sigma=0}d_\rho\mc{Z}_{0,0}|_{z_0}\Big).
\end{align*}

Next, observe that \blue{for $C_1>0$ large enough}, by Lemmas~\ref{l:bijLocal} and~\ref{l:bijectiveDiff}
$$
\begin{gathered}
\|d_{\sigma_2} \Xi_{1,\sigma}(\mc{Z}_{0,0}(z_0))|_{\sigma=0}\|\leq \blue{C\e e^{C|\tilde{T}_0|}},\\ 
\|d_{\sigma_1} \Xi_{1,\sigma}(\mc{Z}_{0,0}(z_0))|_{\sigma=0}\|+\|d_{\sigma_1} \Xi_{2,\sigma}(\mc{Z}_{0,0}(z_0))|_{\sigma=0}\|\leq Ce^{C|\tilde{T}_0|},
\end{gathered}
$$
and, provided that $C_1$ is chosen large enough,
$$
\|(d_{\sigma_1} \Xi_{1,\sigma}(\mc{Z}_{0,0}(z_0))|_{\sigma=0})^{-1}\|+\|(d_{\sigma_2} \Xi_{2,\sigma}(\mc{Z}_{0,0}(z_0))|_{\sigma=0})^{-1}\|\leq Ce^{C|\tilde{T}_1|}.
$$
Together with the fact that 
$$
\|d^2_{\rho}\mc{Z}_{k,0}\|+\|d_{\rho}\mc{Z}_{k,0}\|+\|(d_{\rho}\mc{Z}_{k,0})^{-1}\|\leq C,\qquad \|\Delta_0\|\leq Ce^{C|\tilde{T}_0|},
$$
this implies the lemma.
To see this, observe that, taking $C_1$ large enough, the estimates above imply that $d_{\sigma}((\ms{P}_{\sigma}(q),d_\rho \ms{P}_{\sigma}(q)))|_{\sigma=0}$ takes the form
$
\begin{psmallmatrix}
A&R\\B&D
\end{psmallmatrix},
$
where $\|BA^{-1}R\|<\tfrac{1}{2}\|D^{-1}\|.$
\end{proof}

\subsection{Probing families of metrics}

\label{s:completePerturb}

Before defining our family of probing metrics, we need one more auxiliary lemma which will allow us to control the size of the perturbation, $g_\sigma-g_0$, in the $\mc{C}^\nu$ norm on symmetric tensors. Here, we recall that, to make sense of $C^\nu$ norms, we use the fixed norm as in Remark~\ref{r:fixedNorm}. Note that, in the $(u,t)$ coordinates from~\eqref{e:perturbationDef}, controlling the size of the perturbation is trivial. However, we need to estimate how the coordinate change, $\Phi_{\rho}^{g_\star}$ affects these norms. This is the purpose of our next lemma.

\begin{lemma} \label{l:regularityPerturbation} Let $\nu\geq2$ and
$K\subset\ms{G}^{\nu}$ be bounded. Then there are $C_{\nu}>0$ and $R_0>0$ such that for all $g_{\star}\in K$, $t_0\in \R$, $0<R<R_0$  and $v\in  \mc{C}^\nu(\mathbb{R}^n)$ supported in the ball $B((0,t_0),2R)$,   we have 
\[
\|[(\Phi_{\rho}^{g_\star})^{-1}]^*v\|_{\mc{C}^{\nu}}\leq C_{\nu}\|v\|_{\mc{C}^{\nu}},
\]

\end{lemma} \begin{proof} 

Once we understand the $\mc{C}^{\nu}$ norms of $(\Phi_{\rho}^{g_{\star}})^{-1}$,
it will remain to apply the Faà di Bruno formula. To do this, we first
choose $R_0\ll1$ small enough such that for all $g_\star$ in $K$, $\Phi_{\rho}^{g_{\star}}$
is a diffeomorphism on $B((0,t_0),2R)$ for any $t_0$ and $R<R_0$.

To estimate the $C^\nu$ norms of $(\Phi_\rho^{g_\star})^{-1}$, it is enough to estimate $(d\Phi_{\rho}^{g_{\star}}|_{(0,t)})^{-1}$
from below and $\|\Phi_{\rho}^{g_{\star}}\|_{\mc{C}^{\nu}}$ from above. To see that estimating the inverse at $u=0$ is sufficient, observe that 
$$
\|d\Phi_{\rho}^{g_{\star}}|_{(0,t)}-d\Phi_{\rho}^{g_{\star}}|_{(\tilde{u},t)}\|\leq \|\Phi_{\rho}^{g_\star}\|_{\mc{C}^2}|\tilde{u}|\leq R_0 \|\Phi_{\rho}^{g_\star}\|_{\mc{C}^2},
$$
for $(\tilde{u},t)\in \supp v$. Therefore, provided $\|(d\Phi_{\rho}^{g_{\star}}|_{(0,t)})^{-1}\|<\frac{1}{2R_0\|\Phi_\rho^{g_\star}\|_{\mc{C}^2}}$, we have $d\Phi_{\rho}^{g_{\star}} $ is also invertible in the support of $v$.

Observe that  there is $C>0$
such that for all $g \in K$ and $V\in TM$
\begin{equation}
C^{-1}|V|_{g}\leq|V|\sub{\tSM}\leq C|V|_{g}.\label{e:compareG}
\end{equation}
Therefore, since 
\[
d\Phi_{\rho}^{g_{\star}}|_{(0,t_{0})}(E^i(0)\delta_{u^i}+E^d(0)\delta_{t})=E^{d}(t_{0})\delta_{t}+E^{i}(t_{0})\delta_{u^{i}},
\]
we have
$$
\begin{aligned}C|d\Phi_{\rho}^{g_{\star}}|_{(0,t_{0})}(E^i(0)\delta_{u^i}+E^d(0)\delta_{t})|_{\tSM}^2&\geq |d\Phi_{\rho}^{g_{\star}}|_{(0,t_{0})}(E^i(0)\delta_{u^i}+E^d(0)\delta_{t})|^2_{g_\star}\\
&= |(E^i(0)\delta_{u^i}+E^d(0)\delta_{t}|_{g_\star}^2\geq C^{-1}|E^d(0)\delta_t+E^i(0)\delta_{u^i})|_{\tSM}^2.
\end{aligned}
$$
and hence 
$
\|(d\Phi_{\rho}^{g_{\star}}|_{(0,t)})^{-1}\|\leq C^2
$
\blue{and hence, for $R_0$ small enough$
\|(d\Phi_{\rho}^{g_{\star}}|_{(\tilde{u},t)})^{-1}\|\leq 2C^2
$.}

Next, Lemma~\ref{l:regularity} implies that there is $C$, depending
only on $K$, such that $\|\Phi_{\rho}^{g_{\star}}\|_{\mc{C}^{\nu}}\leq C.$
In particular, we obtain that there is $c_{0}>0$ depending on $K$
such that for $R\leq c_{0}$ there is $C_{\nu}$ depending on $K$ and $\nu$ such that $\|(\Phi_{\rho}^{g_{\star}})^{-1}\|_{\mc{C}^{\nu}}\leq C_{\nu}$
on $\Phi_{\rho}^{g_\star}(\supp v)$.
Using this in the Faà di Bruno formula then completes the proof of the lemma. \end{proof}

Let $r>0$ and $\{\rho_{i}(r)\}_{i=1}^{N(r)}$ be a maximal $r$ separated
set in $\tSM$ so that 
\begin{equation}\label{e:marker}
\tSM\subset\bigcup_{i=1}^{N(r)}B(\rho_{i},r).
\end{equation}

We associate to each $\rho_{i}(r)$, a family of perturbations defined by iterative application of the construction~\eqref{e:perturbationDef}. In particular, let 
\[
\param(r):=\big(B\sub{\R^\L}(0,1)\big)^{N(r)}
\]
We define for $\delta,R>0$,  ${\bm{\sigma}}\in\param(r)$, and $g_\star\in\ms{G}^\nu$,
\[
g_{{\bm{\sigma}},g_{\star}}^{r,R,\delta}=g_{\star}+\sum_{j=1}^{N(r)}(g_{\delta\sigma_{j}}(\rho_j,0, R,g_{\star}, \delta\sigma_j)-g_{\star}),
\]
where $g_{\sigma}=g_{\sigma}(\rho_j,0, R,g_{\star}, \sigma)$ is the metric perturbation of $g_\star$ defined along $\gamma_{\rho_j}^{g_\star}(t)$ in $R$-neighborhood of $\rho_j=\gamma_{\rho_j}^{g_\star}(0)$  (see~\eqref{e:perturbationDef}).

In what follows, we write $\ms{S}^\nu$ for the space of $\mc{C}^\nu$ symmetric tensors on $M$.

For $R,r,\delta>0$ define the map $Q:\ms{G}^\nu \times \bm \Sigma(R)\to \ms{S}^{\nu}$ by
\begin{equation}
\label{e:defMetricQ}
Q^{r,R,\delta}(g,\bm\sigma):=g^{r,R,\delta}_{\bm \sigma, g}.
\end{equation}

\begin{lemma} \label{l:regularityTotalPerturbation} 
Let $\blue{\nu\geq 2}$, $\iota:\ms{G}^\nu\to \ms{G}^{\nu-1}$ be the natural embedding. Then, for all $K\subset \ms{G}^\nu$ bounded, there is $\e_0>0$ such that if 
$$
\delta R^{-1}r^{-2d+1}\max(r,R)^d<\e_0,
$$
then 
\begin{equation}
\label{e:finishHim}
Q^{r,R,\delta}:K\times \bm \Sigma(R)\to \ms{G}^\nu.
\end{equation}
Moreover, with
$\tilde{Q}^{r,R,\delta}:=\iota\circ Q^{r,R,\delta}:\ms{G}^\nu \times \bm \Sigma(R)\to \ms{G}^{\nu}$ is Frechet differentiable, the map $g\to D_{(g,\bm\sigma)}\tilde{Q}^{r,R,\delta}$ is continuous,  and
$$
 D_g\tilde{Q}^{r,R,\delta}:\ms{S}^{\nu-1}\to \ms{S}^{\nu-1}
$$
is bijective, where we have extended $D_g\tilde{Q}^{r,R,\delta}$ to $ \ms{S}^{\nu-1}$ by density.
Moreover, for all $K\subset \ms{G}^\nu$ bounded, there is $C>0$ such that for all $\delta,r, R>0$, $g \in K$, and $\bm \sigma \in \bm \Sigma(R)$,
\begin{align}
\label{e:chapel}
&\|D_{\bm \sigma} Q^{r,R,\delta}\|_{\ell^\infty\to\mc{C}^{\blue{\nu'}}}\leq C\delta R^{-1-\blue{\nu'}}r^{-2d+1}\max(r,R)^{d},\qquad \blue{0\leq \nu'\leq \nu,}\\
\label{e:hill}
&\|D_g Q^{r,R,\delta}(g,\bm\sigma)-I\|_{\mc{C}^{\nu-1}\to \mc{C}^{\nu-1}}\leq C\delta R^{-\nu}r^{-2d+1}\max(r,R)^d,\\
\label{e:london}
&\|\partial_{\sigma_j}Q^{r,R,\delta}(g,\bm\sigma)\|_{B\sub{\R^\L}(0,1)\to \mc{C}^\nu}\leq C \delta R^{-1-\nu},\\
\label{e:gatwick}
&\partial^{\alpha}_{\sigma_j}Q^{r,R,\delta}(g,\bm\sigma)=0,\qquad |\alpha|\geq 2.
\end{align}


\end{lemma} 

{
\begin{proof}
First observe that~\eqref{e:chapel} implies~\eqref{e:finishHim}, so we only check the estimates~\eqref{e:chapel} through~\eqref{e:gatwick}.
Let $\rho \in \{\rho_i\}_{i=1}^{N(r)}$ and $\sigma \in B\sub{\R^\L}(0,1)$, and set
$
g_{\sigma}:=g_{\delta\sigma}(\rho, 0, R, g, \sigma).
$
{Let $\{\mathcal V_j\}_{j=1}^d$ be as defined in Lemma \ref{l:Frechet1}.}
Observe that,
$$
(\Phi_{\rho}^g)^*(g_{\delta\sigma}({\mc{V}_i(g),\mc{V}_j(g)}))= (\Phi_{\rho}^g)^*(g{(\mc{V}_i(g),\mc{V}_j(g))})+\delta h_{\sigma,{ij}}
$$
where $h_{\sigma}$ does not depend on $g$ and
$$
|\partial^\alpha h_{\sigma,{ij}}|\leq C_\alpha R^{-|\alpha|}\|\sigma\|,\qquad |\partial^\alpha D_\sigma h_{\sigma,{ij}}|\leq C_\alpha R^{-|\alpha|},\qquad \partial^\alpha_\sigma h_{\sigma,{ij}}=0,\quad |\alpha|\geq 2
$$
and $h_\sigma$ is supported in a ball of radius $2R$ around $(\Phi_{\rho}^g)^{-1}\pi\sub{M}(\rho)$.
Therefore, using Lemma~\ref{l:regularityPerturbation},
$$
\|D_\sigma g_{\delta\sigma}{(\mc{V}_i(g),\mc{V}_j(g))}\|_{\mc{C}^\nu}=\|\delta D_\sigma ((\Phi_{\rho}^g)^{-1})^*h_{\sigma,{ij}}\|_{\mc{C}^\nu}\leq C_\nu R^{{-1}-\nu}\delta,\qquad D_{\sigma}^{\alpha}g_{\delta\sigma}=0,\qquad |\alpha|\geq 2.
$$
This proves~\eqref{e:london} and~\eqref{e:gatwick}.

Next, let $g_i\in \ms{G}^\nu$, $i=0,1$ and note
\begin{align*}
&\big((g_0)_{\delta\sigma}-g_0\big){(\mc{V}_i(g_0),\mc{V}_j(g_0))}-\big[(g_1)_{\delta\sigma}-g_1\big]{(\mc{V}_i(g_1),\mc{V}_j(g_1))}\\
&=\big(([\Phi_{\rho}^{g_0}]^{-1})^*(g_0)_{\delta\sigma}-g_0\big){(\mc{V}_i(g_0),\mc{V}_j(g_0))}+\big(- ([\Phi_{\rho}^{g_1}]^{-1})^*(g_1)_{\delta\sigma}+g_1\big){(\mc{V}_i(g_1),\mc{V}_j(g_1))}\\
&= \big[([\Phi_{\rho}^{g_0}]^{-1})^* -([\Phi_{\rho}^{g_1}]^{-1})^*\big]\delta h_{\sigma,{ij}}.
\end{align*}
Moreover, {by Lemma~\ref{l:Frechet1}  the map $\mc{V}_i(g)$
satisfies }{$D_g\mc{V}_i(g):\mathscr{S}^{\nu-1}\to \mc{C}^{\nu-1}(TM)$
and $g\mapsto D_g\tilde{\mc{E}}$ is continuous in the $\mathscr{G}^\nu$ topology.} 
Therefore, by Lemmas~\ref{l:Frechet1} and~\ref{l:regularityPerturbation},
\begin{equation}
\label{e:luton}
\|D_g[(g)_{\delta\sigma}-g]\|_{\mc{C}^{\nu-1}\to \mc{C}^{\nu-1}}\leq C_\nu \delta \|\sigma\|R^{-\nu},
\end{equation}
the map $(g,\sigma)\mapsto D_g[g_{\delta\sigma}-g]$ is continuous in the $\ms{G}^\nu\times B\sub{\mathbb{R}^\L}(0,1)$ topology, and the range of the derivative is supported in a ball of radius $CR$ around $\pi\sub{M}(\rho)$ where $C>0$ depends only on $K$.

Using these estimates together with the definition of $Q^{r,R,\delta}$, the lemma follows after a counting argument. In particular, it is enough to bound 
\[
\sup_{j}\#\{i \in \{1,\dots, N(r)\}:\; B(\pi\sub{M}(\rho_{j}),2R)\cap B(\pi\sub{M}(\rho_{i}),2R)\neq \emptyset\}.
\]
To do this, note that since $\rho_{j}$ are a maximal $r$ separated
set, there is a constant $\mathfrak{D}$ depending only on $d$ and $C$
such that there are $\{\mc{J}_{\ell}\}_{\ell=1}^{\mathfrak{D}}$,
$\mc{J}_{\ell}\subset\{1,\dots N(r)\}$ such that 
\[
\{1,\dots, N(r)\}=\bigcup_{\ell}\mc{J}_{\ell},\qquad B(\rho_{j},Cr)\cap B(\rho_{i},Cr)=\emptyset,\quad i\neq j,\,i,j\in\mc{J}_{\ell}.
\]
In particular, for each fixed $j$,
\begin{align*}
 & \#\{i \in \{1,\dots, N(r)\}:\; B(\pi\sub{M}(\rho_{j}),2R)\cap B(\pi\sub{M}(\rho_{i}),2R)\neq \emptyset\}\\
  & \hspace{4cm}\leq \#\{i\in \{1,\dots, N(r)\}\,:\, \pi\sub{M}(\rho_i)\in B\sub{M}(\pi\sub{M}(\rho_j),4R)\}\\
   & \hspace{4cm}\leq \#\{i\in \{1,\dots, N(r)\}\,:\, \pi\sub{M}(B\sub{\tSM}(\rho_i,2r))\subset B\sub{M}(\pi\sub{M}(\rho_j),4R+2r)\}\\
 & \hspace{4cm}\leq {C_d}\mathfrak{D}\vol\Big(\rho\in\tSM:\;\pi\sub{M}(\rho)\in B\sub{M}(\pi\sub{M}(\rho_{j}),2r+4R)\Big)r^{-2d+1}\\
 & \hspace{4cm} \leq {C_d}\tilde{\mathfrak{D}}(\max(r,R))^{d}r^{-2d+1}
\end{align*}
Combining this bound with~\eqref{e:london} implies~\eqref{e:chapel}, and with~\eqref{e:luton} implies~\eqref{e:hill}.
\end{proof}}


\section{Construction of well separated sets}
\label{s:poincare}


In this section, we construct a well-separated set for a small neighborhood of a metric $g\sub{\dagger}\in\G^{3}$.
We then show that these well-separated sets are admissible in the sense of Definition~\ref{ass:2} for the perturbation from Section~\ref{s:perturbedMetrics}. 

\subsection{Construction of the well-separated set}


Define for $g\in\G^{3}$ 
\[
\Psi^{g}:\mathbb{R}\times\widetilde{S^*\!M}\to\widetilde{S^{*}M},\qquad\qquad\Psi^{g}(t,q)=\varphi_{t}^{g}(q),
\]
and for $T>0$, $V\subset \widetilde{S^*\!M}$, consider the geodesic tube 
\[
\mathcal{T}\sub{V}^{g}(T):=\Psi^g\big((-T,T)\times V\big).
\]

\begin{lemma} \label{l:disjointTubes} Let $g\sub{\dagger}\in \G^{3}$. There exist $N>0$, $\delta>0$, \blue{hypersurfaces} $\mc{Z}_i\subset M$, $i=1,\dots, N$, and open subsets $V_i\Subset W_i\subset \widetilde{S^*_{\mc{Z}_i}M}$ such that $H\sub{|\xi|_{g\sub{\dagger}}}$ is transverse to $W_i$ and \blue{there is a neighborhood $\tilde{G}$ of $g\sub{\dagger}$ in $\G^2$} such that for any bounded $G\subset \G^3$, and all $g\in G\cap \tilde{G}$, 
\begin{equation}
\begin{gathered}\mathcal{T}_{W_i}^{g}(\delta)\cap\mathcal{T}^{g}_{W_j}(\delta)=\emptyset,\quad i\neq j,\qquad\qquad\text{and}\qquad\qquad\widetilde{S^{*}M}\subset\bigcup_{i=1}^{N}\mathcal{T}_{V_i}^g(\tfrac{1}{3}T\sub{\dagger}),\end{gathered}
\label{e:tubeCondition}
\end{equation}
where $T\sub{\dagger}=\inj_{g\sub{\dagger}}(M)$. Moreover, $\Psi^g:(-\frac{T\sub{\dagger}}{3},\frac{T\sub{\dagger}}{3})\times W_i\to \widetilde{S^*\!M}$ is a $\mc{C}^1$-diffeomorphism onto its image. \end{lemma}

\begin{proof} 
For each $q\in \widetilde{S^*\!M}$, let $\mc{Z}_q\subset M$ be a hypersurface with $q\in N^*\mc{Z}_q$ and such that we can work with Fermi normal coordinates  associated to $\mc{Z}_q$. That is, we assume that there is $\e_q>0$ such that   $(x_1, x')$ are well defined coordinates for $|(x_1, x')|<2\e_q$ with the property that $(0,0)=\pi\sub{M}(q)$ and  $\mc{Z}_q=\{(x_1,x')\,:\, |x'|<\e_q,\,x_1=0\}$. Let
$$
\mc{H}^{\e_q}_q:=\{ (0,x',\xi_1,\xi')\in \widetilde{S^*\!M}\,:\, |\xi_1|>\tfrac{1}{2}|\xi|_{g\sub{\dagger}},\,|x'|<\e_q\},
$$
where we write $(\xi_1,\xi')$ for the dual coordinates so that we have $q=(0,0,1,0)$.
Shrinking $\e_q$ if necessary, we may further assume that 
$
\Psi^{g\sub{\dagger}}:(-\tfrac{1}{2}T\sub{\dagger},\tfrac{1}{2}T\sub{\dagger})\times \mc{H}^{3\e_q}_q\to \widetilde{S^*\!M}
$
is a $\mc{C}^1$-diffeomorphism onto its image.

Next, for each $q\in \widetilde{S^*\!M}$ and  $x_1 \in (-\e_q, \e_q)$ let
$
\mc{Z}_{q}(x_1):=\{ (x_1,x')\,:\, |x'|<{3}\e_q\}.
$
Observe that there exist $0<\delta_q<\e_q$ and $C_q>0$ such that for $x_1 \in (0, \delta_q)$ and $\rho\in \mc{H}_q^{2\e_q}$, 
$$
T_{x_1}(\rho):=\inf \Big\{t>0: \pi\sub{M}(\varphi_t^{g\sub{\dagger}}(\rho))\in \mc{Z}_{q}(x_1)\Big\}<C_q|x_1|,
$$
and the map
$
\rho\mapsto \varphi^{g\sub\dagger}_{T_{x_1}(\rho)}(\rho)
$
is a $\mc{C}^1$-diffeomorphism from $\mc{H}_q^{2\e_q}$ onto its image.
In particular, choosing $\delta_q$ small enough, we have for $x_1 \in (0, \delta_q)$, $V\subset \mc{H}_q^{2\e_q}$ open, 
we have 
\begin{equation}
\label{e:shiftStart}
\mc{T}\sub{V}^{g\sub\dagger}(\tfrac{1}{5}T\sub{\dagger})\subset\mc{T}\sub{V[x_1]}^{g\sub\dagger}(\tfrac{1}{4}T\sub{\dagger}), \qquad V[x_1]:=\mc{T}\sub{V}^{g\sub{\dagger}}(\tfrac{1}{2}T\sub{\dagger})\cap  \widetilde{S^*\mc{Z}_{q}(x_1)},
\end{equation}
and $V[x_1]$ is $\mc{C}^1$-diffeomorphic to $V$.
Therefore, we can shift slightly the initial position of our tubes without changing covering properties. Moreover, after possibly shrinking $\delta_q$, we may assume that
\begin{equation}
\label{e:diameter}
\sup_{x_1 \in [0, \delta_q)} \operatorname{diam}\big(V[x_1]\big)\leq 4\operatorname{diam}(V).
\end{equation}

By compactness, there are $\{q_i\}_{i=1}^L \subset \widetilde{S^*\!M}$ such that 
$
\widetilde{S^*\!M}\subset\bigcup_{i=1}^L \mc{T}\sub{\mc{H}^{\e_{q_i}}_{q_i}}(\tfrac{1}{5}T\sub{\dagger}).
$
Fix $r>0$. For each $i=1, \dots, L$, we let $\{\rho_{ij}\}_{j=1}^{N_i(r)}$ be a maximal $r/2$ separated set on $\mc{H}^{2\e_{q_i}}_{q_i}$ so that 
$$
\mc{H}_{q_i}^{\e_{q_i}}\subset \bigcup_{j=1}^{N_i(r)} B\sub{\mc{H}^{2\e_{q_i}}_{q_i}}(\rho_{ij},r/2),
$$
and there is $\mathfrak{D}>0$, independent of $r>0$, and a partition $\{1,\dots, N_i(r)\}=\cup_{k=1}^{\mathfrak{D}}\mc{J}_k$ such that for any $k\in 1,\dots ,\mathfrak{D}$ and any $j_1\neq j_2$, $j_1,j_2\in \mc{J}_k$
$$
B\sub{\mc{H}^{2\e_{q_i}}_{q_i}}(\rho_{ij_1},4r)\cap B\sub{\mc{H}^{2\e_{q_i}}_{q_i}}(\rho_{ij_2},4r)=\emptyset.
$$

In particular, since $L$ is independent of $r$,  there is $\tilde{\mathfrak{D}}>0$ independent of $r$ such that the tubes
$$
\mc{T}_{ij}(4r,\tfrac{1}{2}T\sub{\dagger}):=\Psi^{g\sub\dagger}((-\tfrac{1}{2}T\sub{\dagger},\tfrac{1}{2}T\sub{\dagger})\times B\sub{\mc{H}_{q_i}^{2\e_{q_i}}}(\rho_{ij},4r)),
$$
can be divided into $\tilde{\mathfrak{D}}$ collections of disjoint tubes, and 
\begin{equation}
\label{e:coverOne}
\widetilde{S^*\!M}\subset\bigcup_{i=1}^L \mc{T}\sub{\mc{H}^{\e_{q_i}}_{q_i}}(\tfrac{1}{5}T\sub{\dagger})\subset \bigcup_{i=1}^L\bigcup_{j=1}^{N_i(r)}\mc{T}_{ij}(\tfrac{1}{2}r,\tfrac{1}{5}T\sub{\dagger}).
\end{equation}

We now fix $0<r<\frac{\delta_0}{100\tilde{\mathfrak{D}}}$, where $\delta_0=\min_{i}\delta_{q_i}$. By~\eqref{e:diameter}, for $|x_1|<\delta_0$,
$$\max_{i,j}\operatorname{diam}\big(\mc{B}\sub{H\sub{q_i}^{2\e_{q_i}}}(\rho_{ij},4r)\big)[x_1]\leq 32r.$$
Therefore, we may find $|x_1(i,j)|<\delta_0$ for all $i,j$ such that for $(i,j)\neq (k,\ell)$,
\begin{equation}
\label{e:disjointness}
W_{ij}\cap W_{kl}=\emptyset,
\qquad
W_{ij}:=\big(\mc{B}\sub{H\sub{q_i}^{2\e_{q_i}}}(\rho_{i,j},4r)\big)[x_1(i,j)].
\end{equation}
(This is possible, for instance, using a greedy algorithm where we simply select $x_1(i,j)$ iteratively for each intersecting tube.)

Thus, by~\eqref{e:shiftStart} and~\eqref{e:coverOne}, 
$$
\widetilde{S^*\!M}\subset \bigcup_{i=1}^L\bigcup_{j=1}^{N_i(r)}\mc{T}^{g\sub{\dagger}}_{V_{ij}}(\tfrac{1}{4}T\sub{\dagger}), \qquad
V_{ij}:=\big(\mc{B}\sub{\mc{H}_{q_i}^{2\e_{q_i}}}(\rho_{ij},\tfrac{1}{2}r)\big)[x_1(i,j)].
$$
Moreover, by~\eqref{e:disjointness} there is $\delta>0$ such that for $(i,j)\neq (k,\ell)$,
$$
\mc{T}^{g\sub{\dagger}}_{W_{ij}}(2\delta)\cap \mc{T}^{g\sub{\dagger}}_{W_{kl}}(2\delta)=\emptyset.
$$

\blue{Finally, since $G$ is bounded in $\G^3$, there is a neighborhood $\tilde{G}$ in $\G^2$ such that for all $g\in G\cap \tilde{G}$  $\Psi^{g}:(-\tfrac{1}{2}T\sub{\dagger},\tfrac{1}{2}T\sub{\dagger})\times \mc{H}^{3\e_q}_q\to \widetilde{S^*\!M}$ is a $\mc{C}^1$ close to $\Psi^{g_{\dagger}}$ (see e.g. the proof of step 3 in Lemma~\ref{l:bijectiveDiff}) and hence is a $\mathcal{C}^1$-diffeomorphism. This proves the claim.}
\end{proof}

\begin{lemma} \label{l:comparableTime} 
Let $g\sub{\dagger}\in \G^{3}$ and $\{(W_i,V_i)\}_{i=1}^N$ as in Lemma~\ref{e:tubeCondition}. Let $\mc{V}_i=V_i$, $V_i\Subset \mc{W}_i\Subset W_i$. Then, \blue{for any $G$ bounded in $\G^3$ there is a neighborhood
$\tilde{G}\subset\G^{2}$} of $g\sub{\dagger}$ and $\fuInj>0$ such that $\{(\mc{W}_i,\mc{V}_i)\}_{i=1}^N$  is a $\fuInj$-well separated set for $G\blue{\cap \tilde{G}}$. 
\end{lemma} \begin{proof} Let $G$,\blue{$\tilde{G}$}, $\delta$, and $T\sub{\dagger}$ as in Lemma~\ref{l:disjointTubes}. Shrinking $\blue{\tilde{G}}$ if necessary, we may assume that $H\sub{|\xi|_g}$ is uniformly transverse to $\tilde{\Gamma}$ for $g\in G\blue{\cap\tilde{G}}$.

We need to check the conditions~\eqref{e:Cgammas}. Let $\fuInj= \tfrac{1}{2}\delta$. Then, by the second condition in~\eqref{e:tubeCondition}, the first condition in~\eqref{e:Cgammas} follows with $C\sub{\Gamma}=\frac{4T\sub{\dagger}}{3  \delta}.$ 

Next, by the first condition in~\eqref{e:tubeCondition}, the second condition in~\eqref{e:Cgammas} holds with $c\sub{\Gamma}=2$. Finally, since for all $g\in G\blue{\cap \tilde{G}}$, and $i$,
$
\Psi^g(-\frac{T\sub{\dagger}}{3},\frac{T\sub{\dagger}}{3})\times W_i\to \tSM
$
is a $\mc{C}^1$-diffeomorphism onto its image, and $\mc{W}_i\Subset W_i$, 
$$
\inf_{\rho\in \mc{W}_i}\inf\{t>0\,:\,\varphi_t^g(\rho)\in \overline{\mc{W}}_i\}\geq\frac{2T\sub{\dagger}}{3}=\frac{4T\sub{\dagger}}{3\delta }\fuInj\geq C\sub{\Gamma}\fuInj,
$$
and hence, the last condition in~\eqref{e:Cgammas} holds.
\end{proof}

\subsection{Admissibility of the family of perturbations from Section~\ref{s:perturbedMetrics}}

In order to show that the well separated-sets constructed above are admissible for the perturbations from Section~\ref{s:perturbedMetrics}, we study these perturbations near simple points. 

We start by showing that one can find a ball through which the geodesic emanating from a simple point passes only once.

\begin{lemma} \label{l:simpleBalls} Let $K\subset \G^{3}$ bounded, $\fuInj<\inf_{g\in K}\inj_g(M)$. Then there are $c_{0}>0$ and $c_{1}>0$
such that if $0<\alpha<c_{0}$, $n\in \mathbb N$, and $\rho\in\Sim\sub{\!\fuInj}(n,\alpha,g)$, then
for every $R>0$ and any maximal $R$ separated set $\{x_{j}\}_{j=1}^{J(R)}$ on $M$, and any $g\in K$
\[
\#\{j:\; I_g(x_j, R)\neq \emptyset\text{ and }I_g(x_j, 4R)\text{ is disconnected}\}\leq\frac{n^2\fuInj^2}{\alpha c_{1}^{n}},
\]
where 
$
I_g(x_j, R):=\{t\in[0,(n-\frac{1}{2})\fuInj]:\;\pi\sub{M}(\varphi^g_{t}(\rho))\blue{\in} B(x_{j},R)\}.
$

\end{lemma}
\begin{proof}
We claim that since $\rho\in\Sim\sub{\!\fuInj}(n,\alpha,g)$, there
is $\e>0$, $0<c_{1}<1$ such 
that if $g\in K$, $\alpha<\e$ , $0\leq t,s\leq (n-\frac{1}{2})\fuInj$, and $|t-s|>c_1^{n-1}\alpha$,
then 
\begin{equation}\label{e:82a}
d(\varphi^g_{s}(\rho),\varphi^g_{t}(\rho))\geq c_{1}^n\alpha.
\end{equation}

To prove the claim in \eqref{e:82a}, recall that there is $C>0$ such that 
$
d(\varphi_t^g(\rho_1),\varphi_t^g(\rho_2))\leq C^{|t|}d(\rho_1,\rho_2),
$
 for $g\in K$ and all $\rho_1, \rho_2\in \widetilde{S^*\!M}$.
In addition, by Lemma~\ref{l:united}, for all $g\in K$ and $\rho \in \widetilde{S^*\!M}$ we have 
\begin{equation}
\label{e:shortTime}
d(\varphi_t^g(\rho),\rho)\geq |t|/C,\quad |t|\leq \fuInj.
\end{equation}
By contradiction, suppose there are $0\leq t\leq s-c_1^{n-1}\alpha$ and $s\leq (n-\frac{1}{2})\fuInj$ such that 
$
d(\varphi^g_s(\rho),\varphi^g_t(\rho))\leq c_1^n\alpha.
$
First, by~\eqref{e:shortTime}, if $|t-s|\leq \fuInj$, we have
$$
C^{-1}|s-t|\leq d(\varphi^g_{s-t}(\varphi^g_t(\rho)),\varphi^g_t(\rho))\leq c_1^{n}\alpha.
$$
This contradicts $t\leq s-c_1^{n-1}\alpha$ for $c_1$ chosen small enough. In particular, we have $|s-t|\geq \fuInj$. 
Next, applying the flow for time $-t$,
$$
d(\varphi^g_{s-t}(\rho),\rho)\leq c_1^n C^t\alpha\leq c_1^nC^{n\fuInj}\alpha.
$$
Since $|s-t|\geq \frac{1}{2}\fuInj$, this contradicts the fact that $\rho \in \Sim\sub{\! \fuInj}(n,\alpha,g)$ and hence proves the claim in \eqref{e:82a}.

We next claim that there is $c_{0}>0$ 
such that if $g\in K$, $\rho\in\Sim\sub{\!\fuInj}(n,\alpha,g)$, $t_0, s_0 \in [0,(n-\frac{1}{2})\digamma]$, $|s_0-t_0|\geq c_1^n\alpha$, and
\[
d(\pi\sub{M}(\varphi^g_{t_{0}}(\rho)),\pi\sub{M}(\varphi^g_{s_{0}}(\rho)))<4R,
\]
then for $|t-t_{0}|<c_{0}$ and $|s-s_{0}|<c_{0}$, 
\begin{equation}\label{e:82}
d(\pi\sub{M}(\varphi^g_{t}(\rho)),\pi\sub{M}(\varphi^g_{s}(\rho))>c_0(|t-t_{0}|+|s-s_{0}|)c_1^n\alpha-8R.
\end{equation}

We now prove~\eqref{e:82}. Without loss of generality, we may assume $c_1^n\alpha\geq 8R$ since otherwise~\eqref{e:82} is trivial. Since we are working locally near $\varphi_{s_0}^g(\rho)$, we may work in geodesic normal coordinates such that $\varphi_{s_0}^g(\rho)=(0,e_1)$ where $e_1:=(1,0,\dots,0)$ and $\varphi_{t_0}^g(\rho)=(x_1,\xi_1)$. Observe that by~\eqref{e:82a},
$$
c_1^n\alpha\leq |(x_1,\xi_1)-(0,e_1)|\leq 4R+|\xi_1-e_1|.
$$
Thus, $|\xi_1-e_1|\geq \frac{1}{2}c_1^n\alpha$.
Then, define
$$
(x(t,(x_0,\xi_0)),\xi(t,(x_0,\xi_0))):=\varphi_t^g(x_0,\xi_0).
$$
We first claim that 
\begin{equation}
\label{e:82c}
|\xi(t_0-s_0,(0,e_1))+e_1|>\tfrac{1}{2}c_1^n\alpha.
\end{equation}
Indeed, suppose not. Then,
$$
|x(t_0-s_0,(0,e_1))|+|\xi(t_0-s_0,(0,e_1))+e_1|\leq 4R+\tfrac{1}{2}c_1^n\alpha<c_1^n\alpha.
$$
In particular, applying the backward flow, we have 
\begin{equation}
\label{e:82b}
d\Big(\big(x(t_0-s_0-t,(0,e_1)),\xi(t_0-s_0-t,(0,e_1))\big), \big(x(-t,(0,-e_1)),\xi(-t,(0,-e_1))\big)\Big)<C^{|t|}c_1^n\alpha.
\end{equation}
Next, observe that 
$$
(x(-t,(0,-e_1)), \xi(-t,(0,-e_1)))=(x(t,(0,e_1)),-\xi(t,(0,e_1)).
$$
Therefore, putting $t=\frac{t_0-s_0}{2}$ in~\eqref{e:82b}, we have
$$
d\Big(\big(x(\tfrac{t_0-s_0}{2},(0,e_1)),\xi(\tfrac{t_0-s_0}{2},(0,e_1))\big),\big(x(\tfrac{t_0-s_0}{2},(0,e_1)),-\xi(\tfrac{t_0-s_0}{2},(0,e_1)\big)\Big)
<C^{n\fuInj}c_1^n\alpha.
$$
Choosing $c_1$ small enough (uniformly in $K$), this is a contradiction since $|\xi(\tfrac{t_0-s_0}{2},(0,e_1))|_g=1$. In particular,~\eqref{e:82c} holds.

Now, we prove~\eqref{e:82}. Since we work in geodesic normal coordinates at $x(s_0,(0,e_1))$, we have
$$
\begin{aligned}
&|x(s-s_0,(0,e_1))-x(t-t_0,(x_1,\xi_1))|\\
&\qquad\qquad\geq |x(s-s_0,(0,e_1))-x(t-t_0,(0,\xi_1))|-|x(t-t_0,(0,\xi_1))-x(t-t_0,(x_1,\xi_1))|\\
&\qquad\qquad\geq |(s-s_0)e_1-(t-t_0)\xi_1|-(1+C|t-t_0|)|x_1|\\
&\qquad \qquad\geq c_0c_1^n\alpha(|s-s_0|+|t-t_0|)- 8R,
\end{aligned}
$$
where in the last line we have used that $|\xi_1|_{g(x_1)}=|e_1|=1$, $\min(|\xi_1-e_1|,|\xi_1+e_1|)>c_0c_1^n\alpha$, and that $C|t-t_0|<Cc_0\leq 1$.  This proves the claim in \eqref{e:82}.

We now use~\eqref{e:82} to control the length of time for which the (projection to the base of the) segment of the geodesic near $t_0$ can be close to that near $s_0$. Let $\Omega:=B\Big(\bigcup_{|s-s_{0}|<c_{0}}\pi\sub{M}(\varphi^g_{s}(\rho)),8R\Big)$. We claim that
\begin{equation}
\label{e:spreading}
\begin{gathered}
\Big|\Big\{t:\;|t-t_{0}|<c_{0},\,B(\pi\sub{M}(\varphi^g_{t}(\rho)),4R)\cap \Omega\neq\emptyset\Big\}\Big|\leq \frac{CR}{c_0 c_1^n\alpha},
\end{gathered}
\end{equation}
where $C$ depends only on $K$. To see~\eqref{e:spreading}, note that if $t$ lies in the left hand side, then there is $s$ such that $|s-s_0|<c_0$ and, by~\eqref{e:82}, we have
$$
12R >d(\pi\sub{M}(\varphi_t(\rho)),\pi\sub{M}(\varphi_s(\rho)))\geq c_0c_1^n\alpha|t-t_0|-8R.
$$
In particular,~\eqref{e:spreading} holds.

Now, let $I_{\ell}:=[\ell c_{0},(\ell+2)c_{0})$, $\ell=1,\dots,  L-1$, $L=\lfloor\frac{\fuInj (n-\frac{1}{2})-2c_0}{c_0}\rfloor+1$, $I\sub{L}:=[Lc_0,\fuInj (n-\frac{1}{2})]$, and 
\[
\Omega_{\ell}:=B\Big(\bigcup_{t\in I_{\ell}}\pi\sub{M}(\varphi^g_{t}(\rho)),8R\Big).
\]

Since for $|\ell-k|>1$, $\inf_{t\in I_\ell,s\in I_k}|t-s|\geq 2c_0$,  using~\eqref{e:spreading} for each pair $\ell,k$ with $|\ell-k|>1$, 
\[
\begin{aligned}\Big|\{t:\;\exists k, \ell\;\text{s.t.}\;\;|k-\ell|>1,\; t\in I_{\ell},\;\;B(\pi\sub{M}(\varphi^g_{t}(\rho)),4R)\cap \Omega_{k}\neq\emptyset\} \Big|
 & \leq \frac{CR}{c_0 c_1^n\alpha}\frac{n ^{2}\fuInj^2}{c_{0}^{2}}.
\end{aligned}
\]
Next, note that  there exists $c>0$ such that the following holds. Let $g\in K$ and $j$ such that $I_g({x_j,R})\neq \emptyset$ and $I_g({x_j,4R})$ is disconnected there is an interval of length $\geq c R$ inside the set
$$\{t:\;\exists k, \ell \;\text{s.t.}\;\;|\ell-k|>1,\; t\in I_{\ell},\quad B(\pi\sub{M}(\varphi^g_{t}(\rho)),4R)\cap \Omega_{k}\neq\emptyset\},$$
and therefore the lemma follows.

Indeed, to see the last claim, shrinking $c_0$ so that $4c_0<\inf_{g \in K}\inj_g(M)$ if necessary, we can find $t_1\in I_{k_1}$, $t_2\in I_{k_2}$ such that 
$
\pi\sub{M}(\varphi_{t_1}^g(\rho))\in B(x_j,R)$ and 
$
\pi\sub{M}(\varphi_{t_2}^g(\rho))\in B(x_j,4R),
$
with $|t_1-t_2|>4c_0$ and hence $|k_1-k_2|>1$.
Notice that $B(x_j,4R)\subset \Omega_{k_2}$ and hence, since $\pi\sub{M}(\varphi_{t_1}^g(\rho))\in B(x_j,R)$, we have $\pi\sub{M}(\varphi_{t}^g(\rho))\in \Omega_{k_2}$ for $|t-t_1|\leq cR$ (here, again $c>0$ is a constant depending only on $K$). 
\end{proof}

Let $Q^{r,R,\delta}$ be the map that defines the perturbation of the metric as introduced in~\eqref{e:defMetricQ}. We now check that for any $\nu_0$ and $\ga>2$, there is $b$ such that $(\Gamma,G)$ is a $(\fuInj,b,\ga)$ admissible pair for $\{Q^{R^\ga,R,\delta}\}_{R,\delta}$.  (see Definition \ref{ass:2}).

\begin{lemma}
\label{l:admissible} Let \blue{$\nu_0\geq 3$} and $g\sub{\dagger}\in \G^{3}$.
Let $\fuInj>0$, $G\subset\G^{3}$ be a bounded neighborhood of $\blue{g\sub{\dagger}}$, and $\{(\mc{W}_i,\mc{V}_i)\}_{i=1}^N$  be the $\fuInj$-well separated set for $G$ given by Lemma \ref{l:comparableTime}. 
Let 
$\Gamma:=\bigsqcup_{i=1}^{N}\mc{V}_i$.
Then, for $\ga>\blue{1}$,
$$(\Gamma,G)\; \text{is a}\; (\fuInj,b,\ga)-\text{admissible pair for}\; \{Q^{R,\delta}\}_{R,\delta},$$
with $b:=\blue{5}-\ga+d(2\ga-1)$  and $Q^{R,\delta}:=Q^{R^\ga,R,\delta}$ as defined in~\eqref{e:defMetricQ}.
\end{lemma}
\begin{proof}
Let $K\subset G$ bounded in $\ms{G}^{\nu_0}$. We will show that there are $c>0$, $\e>0$ such that for all $\delta>0$, $g\in K$, $\alpha\in(0,c)$, $n \in \mathbb N$, $0<R<c^n \alpha$, $0<\delta\leq R^b c^{n+1}$, and 
$$
\rho_0\in\Gamma\cap\Sim\sub{\!\fuInj}(n,\alpha,g)\cap \Rec\sub{\fuInj}(n,\e,g),
$$
then there are $i_0\in\{1, \dots, N(R)\}$, $\mc{I}_0$, and $m_0$ such that
 $B\sub{\Gamma}(\rho_0,R^{\ga})\subset \mc{D}\sub{\mc{I}_0}^{(m_0)}[Q^{R,\delta}(g,\bm\sigma)]$ for all $\bm\sigma\in \Si{N(R)}$, {$\bm{T}\sub{\mc{I}_0}^{(m_0)}[g](\rho) \in \big[(n-1)\fuInj-\CG\dG\,,\, n\fuInj+\CG\dG \big]$}, and for all $\rho\in B\sub{\Gamma}(\rho_{0},R^{\ga})$
and ${\bm{\sigma}}\in\Si{N(R)}$,
the lower bound in~\eqref{e:PsiKappaBound} holds.

By Lemma~\ref{l:pLives}, for $c,R, \e$ small enough there are $\mc{I}$ and $m$ such that $\rho_0\in \mc{D}\sub{\mc{I}}^{(m)}[g]$, $\bm{T}\sub{\mc{I}}^{(m)}[g](\rho) \in \big[(n-1)\fuInj-\CG\e\,,\, n\fuInj+\CG\e \big]$, and $d(\mc{P}\sub{\mc{I}}^{(m)}[g](\rho_0), \rho_0)<\CG \e$ holds.

Let $\delta_0$ be as in Lemma~\ref{l:whatever}. Next, since 
$$
\sup_{g\in G}\sup_{\rho\in \tilde{\Gamma}}\inf\{t>0\,:\,\varphi_t^g(\rho)\in \Gamma\}<C\sub{\Gamma}\fuInj,
$$
by removing elements from $\mc{I}$, there is $C>0$ depending only on $G$ such that for any $\delta_0>0$, we find $(m_0,\mc{I}_0)$  such that $\bm{T}\sub{\mc{I}_0}^{(m_0)}[g](\rho_0) \in \big[(n-1)\fuInj-\CG\e\,,\, n\fuInj+\CG\e \big]$  and
$$
\sup_{m\leq m_0}d(\mc{P}\sub{\mc{I}_0}^{(m)}[g](\rho_0),\Gamma)\leq \max(\CG\e,\delta_0).
$$
In particular, for $c$ small enough (depending only on $G$) and
\begin{equation}\label{e:cond1}
\blue{\delta<  R^{5-\ga +d(2\ga-1)} c^{n+1},}
\qquad 0<\tilde{R}<{c^{n+1}},
\end{equation}
we have by Lemma~\ref{l:whatever} that $B_{\blue{\tilde{\Gamma}}}(\rho_0, \tilde{R})\subset \mc{D}_{\mc{I}_0}^{(m)}[Q^{R,\delta}(g,\bm\sigma)]$, for all $\bm\sigma\in \bm\Sigma\sub{N(R)}$ and $m\leq m_0$. Indeed, since  { 
$$C\sub\Gamma \fuInj m_0\geq\bm{T}\sub{\mc{I}_0}^{(m_0)}[g](\rho_0)\geq c_\Gamma\fuInj m_0>\fuInj m_0,$$
we have $n\leq C\sub{\Gamma}m_0+2$ and $m_0\leq (n+1)$.} Thus, {$\tilde{R}<c^n \leq c^{m_0}$} and the condition on $\delta$ yields 
\blue{
$$
\|Q^{R,\delta}(g,\bm\sigma)-g\|_{\mc{C}^1} \leq C_0 \delta R^{-2}R^{\ga(-2d+1)}\max(R^\ga, R)^d=C_0 \delta R^{-2+\ga -d(2\ga-1)}  <c^{n+1}\leq c^{m_0}
$$
}
as needed, after applying Lemma \ref{l:regularityTotalPerturbation} with $(R^\ga, R, \delta, 3)$ in place of $(r,R,\delta, \nu)$ and use~\eqref{e:chapel} with $\nu'=1$.

We now study the action of the perturbation $Q^{R,\delta}$ on $\Psi_{g,_{\mc{I}_0}}^{R,\delta}$ as defined in~\eqref{e:PsiKappa}. To do this, we will use Lemma~\ref{l:finalNondegenerate} with $g_\star=g$, and some appropriate choice of $\rho$, $g_0$, and $T_0$.

Note that for any maximal $R$ separated set $\{x_{j}\}_{j=1}^{J(R)}\subset M$, since $\fuInj <\inj_g(M)$,
\begin{equation*}
\#\big\{j\,:\,\text{ there exists }t\in[\fuInj/5,\fuInj/4]\text{ satisfying }\pi\sub{M}(\varphi_{t}^{g}(\rho_0))\in B(x_{j},R)\big\}\geq c\fuInj R^{-1}.
\end{equation*}
Therefore, by Lemma~\ref{l:simpleBalls} if 
\begin{equation}\label{e:cond2}
\alpha<c_0 \qquad \text{and} \qquad \frac{n^2\fuInj^2}{\alpha c_1^n}< c \fuInj R^{-1},
\end{equation}
then
there is an index $j_0$ such that 
\begin{equation}
\label{e:conn}
\{t\in[0,(n-\tfrac{1}{2})\fuInj]:\;\pi\sub{M}(\varphi_{t}^{g}(\rho_0))\cap B(x_{j_0},4R)\neq\emptyset\}\;\;\text{ is connected}
\end{equation}
and there is $\blue{t_0}\in [\fuInj/5,\fuInj/4]$ satisfying $\pi\sub{M}(\varphi_{\blue{t_0}}^g(\rho_0))\in B(x_{j_0},R).$

\blue{
Let $\{\rho_i\}_{i=1}^{N(R^{\blue{\ga}})}$ as in \eqref{e:marker}. Then, there are $i_0$ and $t\in [\fuInj/5,\fuInj/4]$ such that $d(\rho_{i_0},\varphi_{t_0}^g(\rho_0))<R^{\ga}$ and hence $\pi\sub{M}(\rho_{i_0})\in B(x_{j_0},R+R^\ga)$. Define
$$
t_{\star}:=\inf\{t>0\,:\, \varphi_{-t}^{g_0}(\rho_{i_0})\in \Gamma\}
$$
and notice that, since $d(\varphi_{-t}^{g_0}(\rho_{i_0}), \varphi_{t_0-t}^{g_0}(\rho_0))\leq R^{\ga}C_G^{|t|}$, for 
\begin{equation} 
\label{e:cond0}
R^{\ga}<C_G^{-\fuInj}\min(\fuInj/30,\tilde{R}),
\end{equation} 
$t_\star\in (1/6\fuInj,\fuInj/3)$ and $\rho_\star:=\varphi_{-t_\star}^{g}(\rho_{i_0})$ satisfies $d(\rho_0,\rho_{\star})<\tilde{R}$.

Moreover, if 
\begin{equation} \label{e:cond3}
\tilde{R}<c^nR.
\end{equation}
for some $c$ depending $G$, we claim that, with $x_{i_0}:=\pi\sub{M}(\rho_{i_0})$, 
\begin{equation}
\label{e:conn2}
\{t\in[0,(n-\tfrac{1}{2})\fuInj]:\;\pi\sub{M}(\varphi_{t}^{g}({\rho}_\star))\in B(x_{i_0},\tfrac{7}{2}R)\}\;\;\text{ is connected}.
\end{equation}
Suppose not. Then, there are $0\leq t_1<t_2<t_3\leq (n-\tfrac{1}{2})\fuInj$
$$
\varphi_{t_1}^{g_0}(\rho_\star),\varphi_{t_3}^{g_0}(\rho_{\star})\in B(x_{i_0},\tfrac{7}{2}R),\qquad \varphi_{t_2}^{g_0}(\rho_\star)\notin B(x_{j_0},\tfrac{7}{2}R).
$$

We first show that $|t_i-t_0|<5R$. Indeed, for $|t-t_0|>5R$, $t\in[0, (n-\frac{1}{2})\fuInj]$, we have $\pi_M(\varphi_t^{g_0}(\rho_0)\notin B(x_{j_0},4R)$.
Hence, since
$$
d(\varphi_t^{g_0}(\rho_0),\varphi_t^{g_0}(\rho_{\star}))< C_G^{|t|}d(\rho_0,\rho_{\star})<\tfrac{R}{100},
$$
 we obtain 
 $$
 d(\pi\sub{M}(\varphi_{t}^{g_0}(\rho_{\star})),x_{i_0})\geq  d(\pi \sub{M}(\varphi_{t}^{g_0}(\rho_0)),x_{j_0})-\tfrac{R}{100}-d(x_{i_0},x_{j_0})\geq 4R-\tfrac{R}{100}-(R+R^\ga)\geq \tfrac{7}{2}R$$
 In particular, $\varphi_t^{g_0}(\rho_\star)\notin B(x_{i_0},\tfrac{7}{2}R)$.  This implies $|t_i-t_0|<5R$ for $i=1,2,3$. But this is impossible for $R$ small enough, since geodesics are length minimizing for short enough times.
 
}


To obtain the lower bound in \eqref{e:PsiKappaBound}, we aim to apply Lemma ~\ref{l:finalNondegenerate}.
To do this, let $T_0=\bm{T}\sub{\mc{I}_0}^{(m_0)}[g]({\rho}_\star)$, since $\bm{T}\sub{\mc{I}_0}^{(m_0)}[g](\rho_0) \in \big[(n-1)\fuInj-\CG\e\,,\, n\fuInj+\CG\e \big]$,
$$
(n-1)\fuInj-\CG\e -C^n\tilde{R} \leq T_0\leq n\fuInj+\CG\e+C^n\tilde{R}. 
$$
Now, we claim that
\begin{equation}
\label{e:onlyOnce2}
\{t\in[0,T_{0}]:\;\pi\sub{M}(\varphi_{t}^{g}({\rho}_\star))\cap B(\pi\sub{M}(\rho_{i_0}),\tfrac{7}{2}R)\neq \emptyset\} \quad \text{is connected}.
\end{equation}
Indeed, if $T_0\leq (n-\frac{1}{2})\fuInj$, then this is true by~\eqref{e:conn2}. Suppose instead that $T_0>(n-\frac{1}{2})\fuInj$. We need to show that there is no $t\in [(n-\frac{1}{2})\fuInj,T_0]$ such that $\pi\sub{M}(\varphi_t^g(\rho_\star))\in B(\blue{\pi \sub{M}(\rho_{i_0})},\frac{7}{2}R)$.  To do this, recall that ${\rho}_\star\in \mc{W}_l\subset \tilde{\Gamma}$, $\mc{W}_l\subset \widetilde{S^*_{\mc{Z}}M}$ for some hypersurface $\mc{Z}\subset M$, $H\sub{|\xi|_g}$ is transverse to $\mc{W}_l$ and $\fuInj$ is chosen small enough that there are local coordinates $\Omega\ni (y_1,y')\overset{\psi}{\mapsto} V\subset M$ with $\mc{Z}= \psi(\{y_1=0\}\cap\Omega)$, $\pi\sub{M}(\varphi_t(\mc{W}_l))\subset V$ for $|t|\leq \fuInj$, and  $H\sub{|\xi|_g}(\psi^{-1})^*y_1>c>0$ on $\varphi_t(\mc{W}_l)$ for $|t|\leq \fuInj$. Therefore, $y_1(\pi\sub{M}(\varphi_t(\rho_\star)))\geq c t$. In particular, since $\pi\sub{M}(\varphi_{t_\star}^{g}({\rho}_\star))\in B(x_{j_0},R+R^\ga)$  and $t_\star\in(1/6\fuInj,\fuInj/3)$,
this implies for $R$ small enough (depending only on $G$ and $\fuInj$), we have $B(x_{j_0},\tfrac{7}{2}R)\subset \{y_1>0\}.$

We claim that $\varphi_{T_0}^g(\rho_\star)\in \mc{W}_l$. To see this, note that if $\tilde R$ is small enough, then $\rho_0 \in \mc{W}_l$. This implies that $\mc{P}\sub{\mc{I}_0}^{(m_0)}[g](\rho_0)\in\mc{W}_l$, and hence, since $B(\rho_0,\tilde{R})\in \mc{D}\sub{\mc{I}_0}^{(m_0)}[g]$, $\mc{P}\sub{\mc{I}_0}^{(m_0)}[g](\rho_\star)\in \mc{W}_l$. Therefore, the claim holds. Now, since $\varphi_{T_0}^g(\rho_\star)\in \mc{W}_l$, we have $y_1(\pi\sub{M}(\varphi_{T_0-s}^g(\rho_\star))\leq -cs$ for $0\leq s<\fuInj$. In particular, since $T_0\leq n\fuInj+\CG\e+C^n\tilde{R}<(n+\frac{1}{2})\fuInj$ (for $\e$ and $\tilde{R}$ chosen small enough as above), for $t\in[(n-\frac{1}{2})\digamma,T_0]$, $\pi\sub{M}(\varphi_t(\rho_\star))\in \{y_1\leq 0\}$, and hence $\pi\sub{M}(\varphi_t^g(\rho_\star))\notin B(x_{j_0},\tfrac{7}{2}R)$, as claimed. This proves \eqref{e:onlyOnce2}.

\blue{To finish the proof, we aim to apply Lemma~\ref{l:finalNondegenerate} with $\rho=\rho_\star$, $g_\star= g$, $g_0=Q^{R,\delta}(g,\bm \sigma)$, $R_0=\frac{7}{6}R$, $\e= \max(R^\ga, \|g_\star-g_0\|_{\mathcal{G}^3})$.} We postpone checking the hypotheses of Lemma~\ref{l:finalNondegenerate} momentarily, instead first finishing the proof of the current lemma, assuming they hold.

Notice that, since $\varphi_{t_\star}^{g_\star}(\rho_\star)=\rho_{i_0}$, $$Q^{R,\delta}(g,\bm\sigma+\delta \sigma_{i_0})=g_{\delta \sigma_{i_0}}(\rho_{i_0}, 0, R, g_0,  \delta\sigma_{i_0})=g_{\delta \sigma_{i_0}}(\rho_{\star}, t_\star, R, g_0,  \delta\sigma_{i_0}).$$
In particular, with 
$\ms{P}_{\sigma}:B\sub{\Gamma}(\rho_\star, R^\ga)\subset \mc{W}_{i_0} \to \mc{W}_{i_{m_0}}$ as defined in Lemma~\ref{l:finalNondegenerate} with $T_0=\bm{T}\sub{\mc{I}_0}^{(m_0)}[g]({\rho}_\star)$ we have
\begin{align*}
\Psi_{g,_{\mc{I}_0}}^{R,\delta}(m_0,\bm\sigma+\delta\sigma_{i_0},\rho)
&=\Big(\mc{P}\sub{{\mc{I}}_0}^{(m_0)}[g_{\delta \sigma_{i_0}}](\rho)\;,\;d_\rho\big(\mc{P}\sub{\mc{I}_0}^{(m_0)}[g_{\delta \sigma_{i_0}}]\big)(\rho)\Big)\\
&=(\ms{P}_{\delta\sigma_{i_0}}(\rho),d_\rho \ms{P}_{\delta\sigma_{i_0}}(\rho)).
\end{align*}
Lemma~\ref{l:finalNondegenerate} then implies that for all $\rho$ with $d(\rho,\rho_{\star})<R^{\ga}$, 
$$
\partial_{\sigma_{i_0}}\Psi_{g,_{\mc{I}_0}}^{R,\delta}(m_0,\bm\sigma+\delta\sigma_{i_0},\rho)
$$
has a left inverse bounded by $e^{-C_1|T_0|/C_1}\geq c^{n+1}$ and hence~\eqref{e:PsiKappaBound} holds as desired.

\blue{We now check the hypotheses of Lemma~\ref{l:finalNondegenerate}. We first use~\eqref{e:conn2} (and hence $R<\alpha c^{n+1}$) to obtain~\eqref{e:onlyOnce}.  Next, we check that 
$$
\max(R^\ga, \|Q^{R,\delta}(g,\bm \sigma)-g\|_{\mathcal{C}^3})\leq Re^{-C_1|T_0|/C_1}. 
$$
Since $\ga >1$, choosing $c$ small enough and using that $R<c^{n+1}$, we obtain 
\begin{equation}
\label{e:heresTheSquare}
R^\ga \leq Re^{-C_1|T_0|/C_1}. 
\end{equation}
Next, by~\eqref{e:chapel}
\begin{equation}
\label{e:heresTheSquare2}
\|Q^{R,\delta}(g,\bm\sigma)-g\|_{\blue{\ms{G}^3}}<C\delta R^{-\blue{4}-\ga(2d-1)+d},
\end{equation}
Hence, if 
$
\delta <R^{b}c^{n+1}
$
with $b=\blue{5}-\ga +d(2\ga-1)$, then 
$$
\|Q^{R,\delta}(g,\bm\sigma)-g\|_{\blue{\ms{G}^3}}<Rc^{n+1}<Re^{-C_1|T_0|/C_1}.
$$}



\end{proof}

\section{Proof of Theorem~\ref{t:predominantR-ND}}
\label{s:theProof}
We now use Proposition~\ref{p:thePredominantMeat} to prove Theorem~\ref{t:predominantR-ND}. Let $\nu \geq \blue{3}$ and $\someLetter>\someLetter_\nu$. 

Let \blue{$\ga>1$} {to be chosen close to $\blue{1}$ later}, and $\tilde{Q}^{\Rind,\delta}(g,\sigma):=Q^{\Rind^\ga,\Rind,\delta}(g,\sigma)$ with $Q$ as in~\eqref{e:defMetricQ}. Let  $\ms{F}:=\{(F^{\bm\Rind_\e,\bm\delta_\e}_\infty,\infty)\}_{\e}$ be the $\ms{G}^{\nu-1}$ family of probing maps for $\ms{G}^\nu$ constructed as in Proposition~\ref{p:thePredominantMeat} (see also Lemma~\ref{l:probing}).

Let $K_n\subset \ms{G}^\nu$ bounded such that $K_n\subset K_{n+1}$ and $\bigcup_{n}K_n=\ms{G}^\nu$. 
Then for each $g\in K_n$, by Lemmas~\ref{l:comparableTime} and~\ref{l:admissible}, there are $r_g>0$, $\fuInj_{\!g}>0$, and a symplectic manifold $\Gamma_g$ such that $(\Gamma_g,G_g)$ is $\fuInj_{\!g}$ well separated with $G_g:=B\sub{\G^{\blue{2}}}(g,2r_g)\cap \ms{G}^\nu\blue{\cap K_n}$. Moreover, $(\Gamma_g,G_g)$ is $(\fuInj_g,b,\ga)$ admissible for $\tilde{Q}^{\Rind,\delta}$ with $b=\blue{5}-\ga+d(2\ga-1)$.

 Since $\nu\geq 3$, $K_n$ is compact in $\G^{\blue{2}}$, there are $\{g_{i,n}\}_{i=1}^{N_n}\blue{\subset K_n}$ such that $K_n\subset \bigcup_{i=1}^{N_n} B\sub{\G^{\blue{2}}}(g_{i,n},r_{g_{i,n}}).$ 
We relabel $\{g_i\}_{i=1}^\infty=\cup_{n}\cup_{i=1}^{N_n}\{g_{i,n}\}$. 
  Next, fix $K\subset \ms{G}^\nu$ bounded and let $N\sub{K}$ be such that 
   $$
   K\subset \bigcup_{i=1}^{N\sub{K}}B\sub{\ms{G}^{\blue{2}}}(g_i,r_{g_i}).
   $$
 To each $i\in\{1,\dots,N\sub{K}\}$ apply  Proposition~\ref{p:thePredominantMeat} to $(\Gamma_{g_i},G_{g_i}\cap K)$ in place of $(\Gamma,K)$. Let $\e_i$, $C_i$ be the constants $\e_0, C$ given by the proposition.  Let $\e\sub{K,1}=\min_{1\leq i\leq N\sub{K}}\e_i$. 
   
    Next, let  
 $$
 \e\sub{K,2}=\sup\{ \e>0\,:\, \| F^{\bm{\Rind}_\e,\bm\delta_\e}_\infty(g,\bm\sigma)-g\|_{\ms{G}^\nu}< \min_{1\leq i \leq N_K}r_{g_i}\;\; \text{for all }\; g\in K,\,\bm\sigma\in\bm\Sigma_\infty\},
 $$
and set $\e\sub{K}:=\min(\e\sub{K,1},\e\sub{K,2})>0$. By Proposition~\ref{p:thePredominantMeat}, for each $i\in\{1,\dots, N\sub{K}\}$ and $0<\e<\e\sub{K}$ there exists $S_{g,\e}(g_i) \subset \bm{\Sigma}\sub{\infty}(\bm{\Rind}_\ep)$ an $m\sub{\bm \Sigma_\infty(\bm{\Rind})}$ measurable set such that for all $g\in K\cap G_{g_i}$
 \begin{equation}
 \label{e:orMaybeNextWeek}
\big\{\bm \sigma \in \bm \Sigma_\infty(\bm{\Rind}_\e):\; F\sub{\!\infty}^{\Rind_\e, \bm\delta_\e}(g, \bm \sigma)\in  L\sub{\infty}
(\e,g_i)\big\}\subset S_{g,\e}(g_i), \qquad \sup_{g\in K\cap G_{g_i}}m\sub{\bm{\Sigma}_{\infty}(\bm{\Rind})}\big(S_{g,\e}(g_i)\big)\leq \e,
\end{equation}
where $ L\sub{\infty}(\e,g_i):=\Big\{g\in\ms{G}^{\nu}\,:\,\text{\ensuremath{\exists n} such that }\Gamma_{g_i}\cap \Rec\sub{\fuInj\sub{\!g_i}}(n,\beta_n(\e,g_i),g)\nsubseteq\ND\sub{\fuInj\sub{\!g_i}}(n,\beta_n(\e,g_i),g)\Big\},$ with
$$\beta_n(\e,g_i):=\e^{ C_i n^{\gamma}} C_i^{-n^\gamma}n^{- C_i\log \e^{-1} n^{\gamma}},$$
where, as in~\eqref{e:gamma0}, $\gamma=\gamma_\nu(\ga)$ is given by
\begin{gather*}
\gamma_\nu(\ga):=1+\log_{2}\blue{\Big[\max\Big(2\aleph^2(4\max(4,\nu)+(2\aleph+1)(2\aleph+3)\ga +\aleph-2-2\aleph^2)\;,\,2\aleph\ga\Big)\Big]}, 
\end{gather*}
\blue{Here, we have used $\aleph=d-1$, $m_\nu(\ga)=\max(b,\vartheta_\nu,\vartheta_{2}+1)$}
and  $\vartheta_\nu=\nu+(2\aleph+1)\ga-\aleph$ by \eqref{e:chapel} with $r=R^\ga$. 

 Then, let
\begin{equation}
\label{e:tempND}
\tilde{G}_i:=\bigg\{ g\in G_{g_i}\,\Big|\;\; \begin{gathered}\text{there exists } C>0\text{ such that for all $n>0$}\\
\Gamma_{g_i}\cap \Rec\sub{\fuInj\sub{\!g_i}}\!(n,(Cn)^{-Cn^\gamma},g)\subset \ND\sub{\fuInj\sub{\!g_i}}\!(n,(Cn)^{-Cn^\gamma},g)\end{gathered}\bigg\}.
\end{equation}

We claim that $\tilde{G}:=\bigcup_i \tilde{G}_i$ is $\ms{F}$-predominant. To see this, fix $\e>0$. We claim that for all $g\in K$, there is an $m\sub{\bm \Sigma_\infty(\bm{\Rind})}$-measurable set  $S_{g,\e}$ such that
\begin{equation}
\label{e:thisWeek}
\big\{\bm \sigma \in \bm \Sigma_\infty(\bm{\Rind}_\e):\; F\sub{\!\infty}^{\Rind_\e, \bm\delta_\e}(g, \bm \sigma)\in  \tilde{G}^c\}\subset S_{g,\e},\qquad m\sub{\bm \Sigma_\infty(\bm{\Rind})}(S_{g,\e})\leq \e.
\end{equation}

To obtain~\eqref{e:thisWeek}, fix $g\in K$. Then there is $i$ such that $g\in K\cap G\sub{\ms{G}^3}(g_i,r_{g_i})$ and $F\sub{\!\infty}^{\Rind_\e, \bm\delta_\e}(g, \bm \sigma)\in G_{g_i}$ for all $\bm \sigma \in \bm \Sigma\sub{\infty}(\bm{\Rind}_\e)$ (by our choice of $\e\sub{K,2}$). 
Now, since for all $i$ and $\e>0$ there is $C>0$ such that $(Cn)^{-Cn^\gamma} \leq \beta_n (\e,g_i)$ for all $n$, we have 
$$
\tilde{G}^c\cap G_{g_i}\subset \bigcap_{s>0}L\sub{\infty}(s,g_i)\cap G_{g_i}.
$$
Therefore, 
\begin{align*}
\big\{\bm \sigma \in \bm \Sigma_\infty(\bm{\Rind}_\e):\; F\sub{\!\infty}^{\Rind_\e, \bm\delta_\e}(g, \bm \sigma)\in  \tilde{G}^c\}&
\subset\big\{\bm \sigma \in \bm \Sigma_\infty(\bm{\Rind}_\e):\; F\sub{\!\infty}^{\Rind_\e, \bm\delta_\e}(g, \bm \sigma)\in  \bigcap_{s>0} L\sub{\infty}(s,g_i)\cap G_{g_i}\big\}\\
&\subset\big\{\bm \sigma \in \bm \Sigma_\infty(\bm{\Rind}_\e):\; F\sub{\!\infty}^{\Rind_\e, \bm\delta_\e}(g, \bm \sigma)\in   L\sub{\infty}(\e,g_i)\cap G_{g_i}\big\}\subset S_{g,\e}(g_i).
\end{align*}
In particular,~\eqref{e:thisWeek} follows from~\eqref{e:orMaybeNextWeek}.

We next show that if $\someLetter>\someLetter_\nu$, {we may choose $\ga>2$ such that if} $g\in \tilde{G}$, and $\mc{W}=\{\mc{W}^{U_i}\}_i$ is a family of transition maps for $g$, then  there is $C\sub{\someLetter}>0$ such that for all ${t>\fuInj{\sub{\!g_i}}/2}$
\begin{equation}\label{e:clouds}
\Rec\Big(t, C\sub{\someLetter}^{-C\sub{\someLetter} (t+1)^{\someLetter}-1},g\Big)\subset \ND\Big(t, C\sub{\someLetter}^{-C\sub{\someLetter} (t+1)^{\someLetter}-1},(g,\mc{W})\Big).
\end{equation}

To see \eqref{e:clouds}, let $g\in \tilde G$ and $i$  such that $g\in \tilde G_i$. Let $C_g>0$ be such that for all $n>0$
\begin{equation}\label{e:assumption}
\Gamma_{g_i}\cap \Rec\sub{\fuInj\sub{\!g_i}}\!\big(n,(C_gn)^{-C_gn^\gamma},g\big)\subset \ND\sub{\fuInj\sub{\!g_i}}\!\big(n,(C_gn)^{-C_gn^\gamma},g\big).
\end{equation}

Suppose that {$t>\fuInj\sub{\!g_i}/2$} and
$
\rho\in \Rec(t,(\Lambda (t+1))^{-\Lambda {(t+1)}^\gamma},g)
$
for some $\Lambda>0$. 
Define 
$$
s(\rho):=\inf\{ s\geq 0\,:\,\varphi^g_{ s}(\rho)\in\Gamma_{g_i}\},\qquad \tilde{s}_{\pm}(\rho):=\inf\{ s\geq 0\,:\,\varphi^g_{\pm s}(\rho)\in\tilde{\Gamma}_{g_i}\}
$$
where $\tilde{\Gamma}_{g_i}$ is defined as in \eqref{e:tildeGamma}. In what follows, $C,c$ are two positive constants that depend only on $K$. Note also that $\sup_{\rho\in S^*\!M}s(\rho)<C\inj_{g_i}(M)$.  Define $\rho_+:=\varphi_{s(\rho)}^{g}(\rho)$ and observe that there is a choice of $\pm$ such that 
$$
\begin{gathered}d(\rho_+, \rho_t)< C(\Lambda (t+1))^{-\Lambda {(t+1)^\gamma}},\qquad \rho_t:=\varphi_{\pm \tilde{s}_{\pm}(\varphi_t^g(\rho))}(\varphi_t^g(\rho)),\\
|s(\rho)- \tilde{s}_\pm (\varphi^g_t(\rho))|< C(\Lambda (t+1))^{-\Lambda {(t+1)^\gamma}}, \qquad\text{ or }\qquad |s(\rho)+ \tilde{s}_\pm (\varphi^g_t(\rho))|< C(\Lambda (t+1))^{-\Lambda {(t+1)^\gamma}}.
\end{gathered}
$$
In particular, {by Remark~\ref{r:psAndTs}}, there are $\mc{I}$ and $m$ such that 
$$
\mc{P}\sub{\mc{I}}^{(m)}[g](\rho_+)=\rho_t,\qquad |\bm{T}\sub{\mc{I}}^{(m)}[g](\rho_+)-t|<C(\Lambda (t+1))^{-\Lambda {(t+1)^\gamma}},
$$
and so, {choosing $\Lambda>0$ large enough,} there is $n$ with {$\frac{1}{4}<\frac{t}{\fuInj\sub{\!g_i}}-\frac{1}{4}<n
\leq \frac{t}{\fuInj\sub{\!g_i}}+\frac{3}{4}$}, such that  
$
\rho_+\in \Gamma_{g_i}\cap\Rec\sub{\fuInj\sub{\!g_i}}\!(n,(c \Lambda n)^{-c\Lambda n^\gamma},g).
$
Furthermore, choosing $\Lambda$ large enough, we have $ c\Lambda>C_g$ and hence, by \eqref{e:assumption},
$$
\rho_+\in \ND\sub{\fuInj\sub{\!g_i}}(n,(C_g n)^{-C_gn^\gamma},g). 
$$

Now, let $v\in T_\rho (S^*\!M)_g$. Then, there is $w\in \mathbb{R}H\sub{|\xi|_g}(\rho)$ such that  
$$
d\varphi_{s(\rho)}\tilde{v}\in T_{\rho_+}\Gamma_{g_i}, \qquad \tilde{v}:=w+v.
$$ 
Therefore, choosing coordinates on $\Gamma_{g_i}$ near $\rho_+$ to identify $T_{\rho_+}\Gamma_{g_i}$ and $T_{\rho_t}\Gamma_{g_i}$, and writing the identification of tangent spaces as $\tilde{\mc{W}}_{\rho_t,\rho_+}:T_{\rho_+}\Gamma_{g_i} \to T_{\rho_t}\Gamma_{g_i}$,
$$
\|(d\mc{P}^{(m)}\sub{\mc{I}}[g]-\tilde{\mc{W}}_{\rho_t,\rho_+})d\varphi_{s(\rho)}^g\tilde{v}\|\geq (C_gn)^{-C_gn^\gamma}\|\tilde{v}\|.
$$
In particular, 
$$
\|(d\varphi_{-s(\rho)}^g d\mc{P}^{(m)}\sub{\mc{I}}[g]d\varphi_{s(\rho)}^g-d\varphi_{-s(\rho)}^g\tilde{\mc{W}}_{\rho_t,\rho_+}d\varphi_{s(\rho)}^g)\tilde{v}\|\geq c(C_gn)^{-C_gn^\gamma}\|\tilde{v}\|.
$$
Now, there is $\tilde{t}$ with $|\tilde{t}-t|<C( \Lambda (t+1))^{-\Lambda {(t+1)^\gamma}}$, such that 
$$
 d\varphi_{-s(\rho)}^gd\mc{P}\sub{\mc{I}}[g]^{(m)}d\varphi_{s(\rho)}^g= d\varphi^g_{\tilde{t}},
$$
and hence
$$
\|(d\varphi_{\tilde{t}}^g-d\varphi_{-s(\rho)}^g\tilde{\mc{W}}_{\rho_t,\rho_+}d\varphi_{s(\rho)}^g)\tilde{v}\|\geq c(C_g n)^{-C_g n^\gamma}\|\tilde{v}\|.
$$
Applying $d\varphi^g_{t-\tilde{t}}$ on the left, we have
\begin{equation}\label{e:rose}
\|(d\varphi_t^g-d\varphi^g_{t-\tilde{t}}d\varphi_{-s(\rho)}^g\tilde{\mc{W}}_{\rho_t,\rho_+}d\varphi_{s(\rho)}^g)\tilde{v}\|\geq c( C_gn)^{- C_g n^\gamma}\|\tilde{v}\|.
\end{equation}
Now, the map
$
d\varphi^g_{t-\tilde{t}}d\varphi_{-s(\rho)}^g\tilde{\mc{W}}_{\rho_t,\rho_+}d\varphi_{s(\rho)}^g
$
identifies $T_{\rho}S^*\!M/\mathbb{R}H_p$ with $T_{\varphi_t^g(\rho)}S^*\!M/\mathbb{R}H_p$ and has uniformly bounded derivatives in $t$. Suppose $\mc{W}:=\mc{W}^{U_k}$ is a transition map such that $(\rho, \varphi_t^g(\rho))\in U_k$. Then, there is $C\sub{\mc{W}}>0$
\begin{equation}\label{e:tulip}
\|\mc{W}_{\varphi_t^g(\rho),\rho}-d\varphi^g_{t-\tilde{t}}d\varphi_{-s(\rho)}^g\tilde{\mc{W}}_{\rho_t,\rho_+}d\varphi_{s(\rho)}^g\|\leq C\sub{\mc{W}} d(\varphi_t^g(\rho), \rho) \leq C\sub{\mc{W}} (\Lambda (t+1))^{-\Lambda {(t+1)^\gamma}}.
\end{equation}
Hence, letting $\Lambda$ be large enough, there is ${C\sub{\mc{W},g}}>0$, depending on $g$ and $\mc{W}$, such that
\begin{equation}
\label{e:lillies}
\|(d\varphi_t^g-\mc{W}_{\varphi_t^g(\rho),\rho})\tilde{v}\|\geq ({C\sub{\mc{W},g}}(t+1))^{-{C\sub{\mc{W},g}}{(t+1)^\gamma}}\|\tilde{v}\|.
\end{equation}
Next, we claim that
\begin{equation}\label{e:sunflower}
\|(d\varphi_t^g-\mc{W}_{\varphi_t^g(\rho),\rho})v+\mathbb{R}H_p\|\geq ({C\sub{\mc{W},g}}(t+1))^{-{C\sub{\mc{W},g}}{(t+1)^\gamma}}\|v+\mathbb{R}H_p\|.
\end{equation}
Note that then \eqref{e:sunflower} implies $\rho \in \ND(t,({C\sub{\mc{W},g}}(t+1))^{-{C\sub{\mc{W},g}}{(t+1)^\gamma}},g).$
Thus, increasing $\Lambda$ again if necessary, we would conclude that there is ${C\sub{\mc{W},g}}={C\sub{\mc{W},g}}(g, \mc{W})>0$ such that, {for $t>\fuInj\sub{\!g_i}/2$,}
\begin{equation}\label{e:sunflower2}
\Rec(t,({C\sub{\mc{W},g}}(t+1))^{-{C\sub{\mc{W},g}}{(t+1)^\gamma}},g)\subset \ND(t,({C\sub{\mc{W},g}}(t+1))^{-{C\sub{\mc{W},g}}{(t+1)^\gamma}},g).
\end{equation}

To prove~\eqref{e:sunflower}
recall the smooth decomposition $T_\rho\tSM=\ms{H}(\rho)\oplus \mathbb{R}H_p(\rho)$ and the fact that, by  Definition~\eqref{e:transition}, $\mc{W}$ preserves this splitting. Therefore, letting $v= v\sub{\ms{H}}+tH_p(\rho)$ with $v\sub{\ms{H}}\in \ms{H}(\rho)$, 
$$
 (d\varphi_t^g-\mc{W}_{\varphi_t^g(\rho),\rho})v=(d\varphi_t^g-\mc{W}_{\varphi_t^g(\rho),\rho})\tilde{v}= (d\varphi_t^g-\mc{W}_{\varphi_t^g(\rho),\rho})v\sub{\ms{H}}\in \ms{H}(\varphi_t^g(\rho)).
$$
In particular, since $\ms{H}(\rho)$ is uniformly transverse to $H_p(\rho)$, there is $c>0$ such that 
\begin{equation}
\label{e:roseb}
\begin{aligned}
\|(d\varphi_t^g-\mc{W}_{\varphi_t^g(\rho),\rho})v +\mathbb{R}H_p(\varphi_t^g(\rho))\|&=\inf_{s\in\mathbb{R}}\|(d\varphi_t^g-\mc{W}_{\varphi_t^g(\rho),\rho})v +sH_p(\varphi_t^g(\rho))\|\\
&\geq c\|(d\varphi_t^g-\mc{W}_{\varphi_t^g(\rho),\rho}\tilde{v}\|.
\end{aligned}
\end{equation}
Next, since $v=\tilde{v}+w$, with $w\in \mathbb{R}H_p$, we have 
\begin{equation}
\label{e:tulipb}
\|\tilde{v}\|\geq \inf_{s\in\mathbb{R}}\|v+sH_p\|=\|v+\mathbb{R}H_p\|.
\end{equation}
Combining~\eqref{e:lillies}, ~\eqref{e:roseb}, and~\eqref{e:tulipb} implies~\eqref{e:sunflower}.

Finally, to finish the proof of Theorem~\ref{t:predominantR-ND}, we observe that $\ga \mapsto \gamma(\ga)$ is an increasing function and $\Omega_\nu=\gamma_\nu(\blue{1})$. Therefore, for every $\Omega>\Omega_\nu$ there is $\ga>\blue{1}$ such that $\gamma(\ga)<\Omega$.
In particular, there is $C\sub{\someLetter}>0$ such that for $t>{\fuInj\sub{\!g_i}/2}$,
$$
 C\sub{\someLetter}^{-C_\someLetter (t+1)^{\someLetter}-1}<(C(t+1))^{-C{(t+1)^\gamma}}, 
$$
and hence~\eqref{e:sunflower2} implies that   \eqref{e:clouds} holds for $g\in \tilde{G}_i$ and $t>{\fuInj\sub{\!g_i}}/2$.
Thus, since $\cup_i {\tilde{G}_i}$ is predominant, Theorem~\ref{t:predominantR-ND} follows {after recalling that Lemma~\ref{l:united} implies that for all $i$ there is $c>0$ such that for $g\in G_i$,
$$
d(\varphi_t^g(\rho),\rho)\geq c|t|,\qquad |t|<\fuInj\sub{\!g_i}.
$$
\qed}


\appendix

\section{Elementary Control theory for ODEs}
\label{s:control}
Let  $\aleph\in \mathbb N$, $R>0$, $t_\star\in \R$, and for $1\leq i\leq \aleph$ and $t\in \R$ define
\[
{f}^{R}_{i,t_{\star}}(t):={\tfrac{1}{R}}\chi(\tfrac{1}{R}(t-t_{\star})){\bf {e}}_{i}=:\chi_R(t-t_\star)\mathbf{e}_i,\qquad{\bf {e}}_{i}:=(\underset{i-1}{\underbrace{0,\dots,0}},1,\underset{\aleph-i}{\underbrace{0,\dots,0}})^{t},
\]
where $\chi\in C_{c}^{\infty}((-2,2);[0,1])$ with $\int\chi=1$.
\begin{lemma}[Observation] \label{l:observe} Let ${\bf {L}}\in L^{\infty}(\mathbb{R};\mathbb{M}(\aleph))$. Then there are $C,R_0>0$ such that for all $0<R<R_0$, $t_\star,T \in \mathbb{R}$ \blue{with $2R<t_\star< T-2R$}, ${\bf {u_{0}}}\in\mathbb{R}^{\aleph}$, and ${\bf {u}}$ solving
\[
\dot{{\bf {u}}}+{\bf {L^{*}}}{\bf {u}}=0,\qquad{\bf {u}}(0)={\bf {u_{0}}},
\]
we have
\[
\|{\bf {u}}(T)\|^{2}\leq Ce^{C(|T-t_{\star}|)}\sum_{i=1}^{\aleph}|\langle{\bf {u}},f_{i,t_{\star}}^{R}\rangle_{L_{t}^{2}}|^{2}.
\]
\end{lemma} \begin{proof} Note that 
\begin{align*}
\langle{\bf {u}},f_{i,t_{\star}}^{R}\rangle-{\bf {u}}_{i}(\blue{t_{\star}}) & =\langle{\bf {u}}(\blue{t_{\star}}),f_{i,t_{\star}}^{R}\rangle-{\bf {u}}_{i}(\blue{t_{\star}})-\iint_{0}^{t}{\bf {L^{*}}}(s){\bf {u}}(s)\overline{f_{i,t_{\star}}^{R}(t)}dsdt\\
 & =-\iint_{0}^{t}{\bf {L^{*}}}(s){\bf {u}}(s)\overline{f_{i,t_{\star}}^{R}(t)}dsdt.
\end{align*}
Then, since there is $C>0$ depending only on $\|{\bf {L^{*}}}\|_{L^\infty}$
such that $\|{\bf {u}}(s)\|\leq C\|{\bf {u}}(\blue{t_{\star}})\|$ for $|s-\blue{t_{\star}}|\leq1$, 
we obtain for $R<\frac{1}{2}$, 
\[
\Big|\iint_{0}^{t}{\bf {L^{*}}}(s){\bf {u}}(s)\overline{{ {f}}_{i,t_{\star}}^{R}(t)}dsdt\Big|\leq C\|{\bf {u}}(\blue{t_{\star}})\|\int_{-2R}^{2R}\|{\bf {L^{*}}}(s)\|ds\leq CR\|{\bf {u}}(\blue{t_{\star}})\|\|{\bf {L^{*}}}\|_{L^\infty}.
\]

In particular,  $\big|\langle{\bf {u}},f_{i,t_{\star}}^{R}\rangle-{\bf {u}}_{i}(\blue{t_{\star}})\big|\leq CR\|{\bf {L^{*}}}\|_{L^\infty}\|{\bf {u}}(\blue{t_{\star}})\|$ for $R<\frac{1}{2}$.
Therefore, 
\[
\sum_{i=1}^{\aleph}\big|\langle{\bf {u}},f_{i,t_{\star}}^{R}\rangle\big|^{2}\geq\|{\bf {u}}(\blue{t_{\star}})\|^{2}(\frac{1}{2}-CR^{2}\|{\bf {L^{*}}}\|_{L^\infty}^{2}).
\]
In particular, taking $R<\frac{1}{2C\|{\bf {L^{*}}}\|_{L^\infty}}$,
we have 
\[
\sum_{i=1}^{\aleph}\big|\langle{\bf {u}},f_{i,t_{\star}}^{R}\rangle\big|^{2}\geq\frac{1}{4}\|{\bf {u}}(\blue{t_{\star}})\|^{2}.
\]
The claim now follows from standard estimates on first order ODE's.
\end{proof}

\begin{lemma}[Control] \label{l:control}  Let ${\bf {L}}\in L^{\infty}(\mathbb{R};\mathbb{M}(\aleph))$ and \textcolor{blue}{$\mathbf{Y}\in \mathcal{C}^1(\mathbb{R};\mathbb{M}(\aleph))$ with  ${\mathbf{Y}}^{-1}\in L^\infty(\mathbb{R};\mathbb{M}(\aleph))$}. There are $C,R_{0}>0$,
depending only on $\|{\bf {L}}\|_{L^\infty}$, such that for all $0<R<R_{0}$,
${\bf {v}}\in\mathbb{R}^{\aleph}$, $t_\star>2R$ and $T>t_\star+2R$, there is
${\bf {a}}\in\mathbb{R}^{\aleph}$ such that the solution,
${\bf {w}}$ to 
\begin{equation}
\dot{{\bf {w}}}-{\bf {L}}{\bf {w}}=\sum_{i=1}^{\aleph}{\bf ({\mathbf{Y}}{a})}_{i}f_{i,t_{\star}}^{R},\qquad{\bf {w}}(0)=0\label{e:toControl}
\end{equation}
satisfies ${\bf {w}}(T)={\bf {v}}$ and, moreover, 
$\|{\bf {a}}\|\leq Ce^{C|T-t_{\star}|}\|{\bf {v}}\|.$
\end{lemma}
\begin{proof} {We first reduce to the case $\mathbf{Y}=I$.  Suppose the statement holds with $\mathbf{Y}=I$ and let $\mathbf{a}\in\mathbb{R}^\alpha$ be such that the solution $\mathbf{z}$ to 
$$
\dot{\mathbf{z}}+(\mathbf{Y}^{-1}\dot{\mathbf{Y}}-\mathbf{Y}^{-1}\mathbf{L}\mathbf{Y})\mathbf{z}=\sum_{i=1}^\aleph \mathbf{a}_if_{i,t_\star}^R,\qquad \mathbf{z}(0)=0
$$
satisfies $\mathbf{z}(T)=\mathbf{Y}^{-1}(T)\mathbf{v}$. Set $\mathbf{w}:=\mathbf{Y}\mathbf{z}$. Then, $\mathbf{w}$ satisfies $\mathbf{w}(0)=0$, $\mathbf{w}(T)=\mathbf{v}$ and 
\begin{align*}
\dot{\mathbf{w}}-\mathbf{L}\mathbf{w}&= \mathbf{Y}\dot{\mathbf{z}}+\dot{\mathbf{Y}}\mathbf{z}-\mathbf{L}\mathbf{Y}\mathbf{z}\\
&=\mathbf{Y}\big(\dot{\mathbf{z}}+\mathbf{Y}^{-1}\dot{\mathbf{Y}}\mathbf{z}-\mathbf{Y}^{-1}\mathbf{L}\mathbf{Y}\mathbf{z})\\
&=\mathbf{Y}\sum_{i=1}^\aleph \mathbf{a}_if_{i,t_\star}^R=\mathbf{Y} \mathbf{a}\chi\sub{R}(\cdot-t_*)=\sum_{i=1}^\aleph (\mathbf{Y} \mathbf{a})_if_{i,t_\star}^R.
\end{align*}
Hence, $\mathbf{w}$ satisfies the required properties.

We now show that the lemma holds with $\mathbf{Y}=I$.} 
Define a map $K:\mathbb{R}^{\aleph}\to\mathbb{R}^{\aleph}$
by 
\[
K({\bf {u_{0}}})=\sum_{i=1}^\aleph \Big(\int_0^T\langle{\bf {u}},f_{i,t_{\star}}^{R}\rangle(s)ds\Big){\bf {e}}_{i}
\]
where ${\bf {u}}$ solves 
\begin{equation}
\dot{{\bf {u}}}+{\bf {L^{*}}}{\bf {u}}=0,\qquad{\bf {u}}(0)={\bf {u_{0}}}.\label{e:dual}
\end{equation}
Next, fix ${\bf {v}}\in\mathbb{R}^{\aleph}$ and define
the linear map $ \ell_{{\bf {v}}}:K(\mathbb{R}^{\aleph})\to\mathbb{R}$
by 
\[
\ell_{{\bf {v}}}(K({\bf {u_{0}}}))=\langle{\bf {u}}(T),{\bf {v}}\rangle
\]
where again ${\bf {u}}$ is as in~\eqref{e:dual}. Then, by Lemma~\ref{l:observe},
there are $R_{0},C>0$ such that for $0<R<R_{0}$, 
\begin{align*}
\|{\bf {u}}(T)\|^2&\leq Ce^{C|T-t_{\star}|}\sum_{i=1}^\aleph |\langle\mathbf{u},f_{i,t_\star}^R\rangle_{L^2_t}|^2\leq |K(\mathbf{u}_0)|_{\ell^2(\mathbb{R}^\aleph)}^2.
\end{align*} 
In particular, $\ell_{{\bf {v}}}$ is a bounded linear function on
the subspace $K(\mathbb{R}^{\aleph})$ with $\|\ell_{{\bf {v}}}\|\leq Ce^{C|T-t_{\star}|}\|{\bf {v}}\|$.
Thus, by the Hahn--Banach theorem, we can extend $\ell$ to a linear
functional on $\mathbb{R}^{\aleph}$ with the same norm
and there is ${\bf {a}}\in\mathbb{R}^{\aleph}$ with
\[
\|{\bf {a}}\|\leq Ce^{C|T-t_{\star}|}\|{\bf {v}}\|
\]
such that 
$\ell({\bf {b}})=\langle{\bf {b}},{\bf {a}}\rangle,$ for all ${\bf {b}}\in\mathbb{R}^{\aleph}.$
Now, suppose that ${\bf {w}}$ solves 
\[
\dot{{\bf {w}}}-{\bf {L}}{\bf {w}}=\sum_{i}{\bf {a}}_{i}f_{i,t_{\star}}^{R},\qquad{\bf {w}}(T)={\bf {v}.}
\]
Then, integrating by parts, we have for ${\bf {u}}$ solving~\eqref{e:dual},
\begin{align*}
\int_{0}^{T}-\langle\sum_{i}{\bf {a}}_{i}f_{i,t_{\star}}^{R},{\bf {u}}\rangle(s)ds & =\int_{0}^{T}\langle{\bf {w}}(s),\dot{{\bf {u}}}+{\bf {L^{*}}}{\bf {u}}\rangle(s)ds+\langle{\bf {w}}(0),{\bf {u}}(0)\rangle-\langle{\bf {w}}(T),{\bf {u}}(T)\rangle\\
 & =\langle{\bf {w}}(0),{\bf {u_{0}}}\rangle-\langle{\bf {v}},{\bf {u}}(T)\rangle.
\end{align*}
But, by construction of ${\bf {a}}$,  and that $T>t_\star+2R>4R$,
\begin{align*}
\langle{\bf {u}}(T),{\bf {v}}\rangle
&=\int_0^T\langle \sum_i {\bf a}_if_{i,t_\star}^R(s),\mathbf{u}\rangle (s)ds.
\end{align*}
In particular, 
$
\langle{\bf {w}}(0),{\bf {u_{0}}}\rangle=0
$
for all ${\bf {u_{0}}}\in\mathbb{R}^{\aleph}$ and hence
${\bf {w}}(0)=0$ and the lemma is proved. \end{proof}

\section{Linear algebra used to prove Lemma \ref{l:volume-neighborhood-Dtheta-s-b}}
\label{a:linear}

\blue{
The first lemma gives a normal form for two matrices. This is sometimes called the cosine sine decomposition (see also~\cite[Theorems 2.1, 2.2, and Proposition 2.3]{FaMaSc:04} and references therein). 
\begin{lemma}
\label{l:linearAlgebra}
    Let $n\in\mathbb{Z}_+$ and suppose that $\bm A,\bm B\in \mathbb{M}(n\times n)$ satisfy
    $$
    \bm A\bm A^*+\bm B\bm B^*=\bm I,\qquad \bm A\bm B^*=\bm B\bm A^*.
    $$
    Then, there are $\bm U,\bm V\in\mathbb{M}(n\times n)$ unitary and $\bm \Sigma_i$, $i=1,2$ diagonal matrices such that 
    $$
    \bm A=\bm U\bm \Sigma_1\bm V^*,\qquad \bm B=\bm U\bm \Sigma _2\bm V^*.
    $$
    Moreover, if $\bm A$ and $\bm B$ are real valued, then $\bm U$ and $\bm V$ are orthogonal.
\end{lemma}
\begin{proof}
Consider the relation 
    $$
    \Gamma:=\{ (\bm A^*x,\bm B^*x)\,:\, x\in\mathbb{R}^{n}\}
    $$
    Then, observe that for $(w_1,z_1)$, $(w_2,z_2)\in\Gamma$, $w_i=\bm A^*x_i$ and $z_i=\bm B^*x_i$ for some $x_i\in\mathbb{R}^n$. Therefore,
    $$
    \langle z_1,w_2\rangle -\langle w_1,z_2\rangle =\langle \bm B^*x_1,\bm A^*x_2\rangle -\langle \bm A^*x_1,\bm B^*x_2\rangle =0
    $$
    since $\bm B\bm A^*=\bm A\bm B^*$.
    Hence, $\Gamma$ is isotropic. Moreover, since 
    $$
    \|\bm A^*x\|^2+\|\bm B^*x\|^2=\|x\|^2,
    $$
    $\dim \Gamma=n$ and hence $\Gamma$ is Lagrangian.

    Let $D:=\pi_L(\Gamma)=\operatorname{ran}(\bm A^*)$, where $\Pi_L:\Gamma\to \mathbb{R}^n$ is defined by $\pi_L(x,y):=x$. Then, notice that if $(0,\eta)\in\Gamma$, since $\Gamma$ is Lagrangian, we have for any $(x,y)\in\Gamma$
    $$
    \langle \eta,x\rangle =\langle 0,y\rangle=0.
    $$
    In particular, $\eta\in D^{\perp}$ and hence 
    $$
    \pi_{L}^{-1}(0)\cap \Gamma\subset \{0\}\times D^\perp.
    $$
    In particular, $\bm B^*(\ker (\bm A^*))\subset (\operatorname{ran}(\bm A^*))^\perp$. 
    Now, notice that since $\bm B\bm B^*|_{\ker (\bm A^*)}=\bm I|_{\ker (\bm A^*)}$ and hence $\dim \bm B^*(\ker(\bm A^*))=\dim \ker \bm A^*=n-\dim \operatorname{ran}(\bm A^*)=\dim(\operatorname{ran}(\bm A^*))^\perp. $
    Therefore, $\bm B^*(\ker (\bm A^*))=(\operatorname{ran}(\bm A^*))^\perp.$

    Now, for each $x\in D$, let $(x,z_x)\in\Gamma$. Then, define an operator $\bm T:D\to D$ by 
    $$
    \bm Tx:=\pi_D z_x
    $$
    where $\pi_D:\mathbb{R}^n\to D$ is the orthogonal projector. (Note that $\bm T$ is well defined since if $(x,z_1)$ and $(x,z_2)\in \Gamma$, then $(0,z_1-z_2)\in \Gamma$ and hence $z_1-z_2\in D^{\perp}$ so that $\pi_D(z_1-z_2)=0$. )
    Note that $\bm T$ is self-adjoint since the fact tha $\Gamma$ is Lagrangian implies
    $$
    \langle \bm Tx_1,x_2\rangle =\langle \Pi_Dz_{x_1},x_2\rangle =\langle z_{x_1},x_2\rangle =\langle x_1,z_{x_2}\rangle =\langle x_1,\bm Tx_2\rangle.
    $$
    In particular, there is an orthonoraml basis $\{v_i\}_{i=1}^n$ such that $\{v_i\}_{i=1}^k$ is an orthonormal basis for $D$ and
    $$
    \bm Tv_i = \lambda_i v_i,\qquad i=1,\dots k.
    $$
    Hence, letting $\theta_i\in (-\pi/2,\pi/2]$, $i=1\dots k$ such that $\sqrt{1+\lambda_i^2}\cos \theta_i=1$, $\sqrt{1+\lambda_i^2}\sin \theta_i=\lambda_i$, and $\theta_i=\frac{\pi}{2}$, $i=k+1,\dots, n$
    $$
    \Gamma=\operatorname{span} \{ (\cos \theta_i v_i,\sin\theta_i v_i)\}_{i=1}^k\oplus \{0\}\times D^\perp= \operatorname{span} \{ (\cos \theta_i v_i,\sin\theta_i v_i)\}_{i=1}^k\oplus\operatorname{span}\{(\cos \theta_i v_i,\sin \theta_i v_i)\}_{i=k+1}^n.
    $$
    In particular, there are $\{u_i\}_{i=1}^n$ such that 
    $$
    \bm A^*u_i=\cos \theta_i v_i,\quad  \bm B^* u_i=\sin \theta_i v_i
    $$
    and 
    \begin{align*}
   \langle u_i,u_j\rangle &= \langle (\bm A \bm A^*+\bm B\bm B^* )u_i,u_j\rangle \\
   &= \langle \bm A^*u_i,\bm A^*u_j\rangle +\langle \bm B^*u_i,\bm B^*u_j\rangle\\
   &=\cos \theta_i\cos \theta_j\delta_{ij}+\sin \theta_i\sin\theta_j\delta_{ij}=\delta_{ij},
    \end{align*}
    so that $\{u_i\}_{i=1}^n$ is an orthonormal basis of $\mathbb{R}^n$. Setting $\bm V= \begin{pmatrix} v_1&\dots&v_n\end{pmatrix}$ and $\bm U=\begin{pmatrix} u_1&\dots &u_n\end{pmatrix}$, we have
    $$
    \bm A^*\bm U=\bm V\bm \Sigma_1 ,\qquad \bm B^*\bm U=\bm V\bm \Sigma_2,
    $$
    which completes the proof. 
\end{proof}

We next record some estimates on the volume of determinant level sets arising in the computation of the volume of $(\beta,q)$-degenerate matrices.

We start with a relatively soft volume estimate for (shifted) determinant level sets where constants are not explicit in norms of the matrices involved.
\begin{lemma}
\label{l:smallK}
    Let $M>0$, $n\geq 1$. Then there is $C>0$ such that for all $\bm K\in\mathbb{M}(n\times n;\mathbb{C})$ with  $\bm K=\bm K^*$, $\|\bm K\|\leq M$ and all $t>0$ 
    \begin{equation} 
    \label{e:toProveVolume1}
    \vol(\{ \bm S\in \mathbb{M}(n\times n;\mathbb{R})\,:\, \bm S=\bm S^t,\, \|\bm S\|\leq M,\, |\det(\bm S+\bm K)|\leq t\})\leq Ct.
    \end{equation}
\end{lemma}
\begin{proof}

By rescaling $\bm S\mapsto M^{-1}\bm S$ and $\bm K\mapsto M^{-1}\bm K$, it is enough to prove the bound with $M=1$.

    We proceed by induction on the dimension of $\bm S$. Indeed, for $\bm S\in\mathbb{R}$, the statement is immediate. Therefore, suppose that~\eqref{e:toProveVolume1} holds with $n\in \{ 1,\dots, \ell-1\}$. 

    Observe that 
    $$
    \bm K=\bm A+i\bm B,
    $$
    with $\bm A=\bm A^t$, $\bm B=-\bm B^t$, $\|\bm A\|,\|\bm B\|\leq 1$.  
    Shifting $\bm S$, it is enough  to obtain the estimate
    \begin{equation}
    \label{e:claimVolume}
    \vol(D_{\bm B}(t))\leq Ct,\quad D_{\bm B}(t):=\{ \bm S\,:\, \bm S=\bm S^t,\, \|\bm S+i\bm B\|\leq 3,\, |\det(\bm S+i\bm B)|\leq t\}.
    \end{equation}

    Now, we prove the estimate for $\bm S\in \mathbb{M}(\ell\times \ell)$. Observe that 
    $$
    D_{\bm B}(t)=\cup_{m=0}^{\infty}D_{\bm B,m}(t),\qquad D_{\bm B,m}(t):=D_{\bm B}(t)\cap \{ 2^{-m-1}3\leq \| \bm S+i\bm B\|\leq 2^{-m}3\}.
    $$

    We start by estimate the volume of $D_{\bm B,0}(t)$. 

    Let $\bm S_0,\bm B_0$ such that $\frac{3}{2}\leq \|\bm S_0+i\bm B_0\|\leq 3$. We claim there is are are neighborhoods, $U,V$ of $\bm S_0$ and $\bm B_0$ respectively and $C_{(\bm S_0,\bm B_0)}$ such that
    \begin{equation} 
    \label{e:intermediateClaim}
    \sup_{\bm B\in V} \vol (D_{\bm B}(t))\cap U\leq C_{(\bm S_0,\bm B_0)}t.
    \end{equation}

    To prove this, let $E:=\ker \bm S_0+i\bm B_0$ and $r:=\dim E$. Notice that $r\leq \ell-1$.

    If $r=0$, then there there are neighborhoods  $U$ and $V$ as claimed such that $\sup_{\bm B\in V}\vol (D_{\bm B}(t) U)=0$ for $t$ small enough and hence~\eqref{e:intermediateClaim} holds.

    Now suppose that $1\leq r\leq \ell-1$. Then, up to permuting the coordinates, we can assume that, with $\bm H_0=\bm S_0+i\bm B_0$,  $\Pi_{1}\bm H_0\Pi^t_{1}$ is invertible, where 
    $$
    \Pi_{1}:=\begin{pmatrix} \bm I_{n-r}&\bm 0_{r}\end{pmatrix},\qquad \Pi_{2}:=\begin{pmatrix} \bm 0_{n-r}&\bm I_{r}\end{pmatrix}
    $$
    In particular, for $\|\bm S-\bm S_0\|$ and $\|\bm B-\bm B_0\|$ small enough,
    $$
    \|(\Pi_1(\bm S+i\bm B)\Pi_1^t)^{-1}\|\leq 2\|(\Pi_1 \bm H_0\Pi_1^t)^{-1}\|
    $$
    Therefore, writing $\tilde{\bm H}=\bm S+i\bm B-\bm H_0$ and $\bm A_{ij}=\Pi_i\bm A\Pi_j^t$, we have, using the Schur complement formula
    \begin{align*}
    \det(\bm S+i\bm B)&= \det \begin{pmatrix}\tilde{\bm H}_{11}+(\bm H_0)_{11}&
    \tilde{\bm H}_{12}+(\bm H_0)_{12}\\\tilde{\bm H}_{21}+(\bm H_0)_{21}&\tilde{\bm H}_{22}+(\bm H_0)_{22}\end{pmatrix}\\
    &= \det(\tilde{\bm H}_{22}+(\bm H_0)_{22s})\\
    &\qquad\det (\tilde{\bm H}_{22}+(\bm H_0)_{22}-(\tilde{\bm H}_{21}+(\bm H_0)_{21})(\tilde{\bm H}_{11}+(\bm H_0)_{11})^{-1}(\tilde{\bm H}_{12}+(\bm H_0)_{12})).
    \end{align*}
    Notice that $\tilde{\bm H}_{22}=\bm S_{22}+(-\bm S_0+i(\bm B-\bm B_0))_{22}$, and $\bm S_{22}-\bm S_0$ is an arbitrary, real symmetric matrix in dimension $r\leq m-1$. Furthermore, notice that $\operatorname{rank}(\bm H_0)=\operatorname{rank}((\bm H_0)_{22})$ and hence, 
    $$
    (\bm H_0)_{22}-((\bm H_0)_{21})((\bm H_0)_{11})^{-1}((\bm H_0)_{12})=0.
    $$
    Therefore, by shrinking $U$ and $V$ we may assume 
    $$
    (\tilde{\bm H}_{22}+(\bm H_0)_{22}-(\tilde{\bm H}_{21}+(\bm H_0)_{21})(\tilde{\bm H}_{11}+(\bm H_0)_{11})^{-1}(\tilde{\bm H}_{12}+(\bm H_0)_{12})= \bm{\tilde{S}}_{22}+\bm K_{22},
    $$
    where $\bm K_{22}$ is self-adjoint, $\|\bm K_{22}\|\leq 1$ and $\bm{\tilde{S}}_{22}=\bm S_{22}-(\bm S_{0})_{22}$
    applying  the inductive claim to see that~\eqref{e:toProveVolume1}  holds in dimension $<\ell$ (applied with $\bm S$ replaced by $\bm S_{22}-(\bm S_0)_{22}$).
    we have
    $$
    \vol (D_{\bm B,0}(t)\cap U)\leq C_{\bm B_0,\bm S_0} t.
    $$
    Now, by compactness of 
    $$
    \{ (\bm B,\bm S)\in \mathbb{M}(n\times n;\mathbb{R})\,:\, \bm S=\bm S^t,\, \bm B=-\bm B^t,\, \tfrac{3}{2}\leq \|\bm S+i\bm B\|\leq 3\},$$ 
    there is $C>0$ such that
    $$
    \sup_{\bm B}\vol (D_{\bm B,0}(t))\leq C t.
    $$

    Next, notice that 
    $$
    \vol (D_{\bm B,m}(t))\leq 2^{-m(\frac{\ell(\ell+1)}{2})}\vol (D_{ 2^{m}\bm B,0}(2^{\ell m}t))\leq C2^{-m(\frac{\ell(\ell+1)}{2}-\ell)}t 
    $$
    Since $\ell\geq 2$, $\ell(\ell+1)/2-\ell\geq 1$ this completes the proof of the inductive claim after summing in $m$. 
    
\end{proof}

We now determine how constants in Lemma~\ref{l:smallK} depedn on $\bm K$. 
\begin{lemma}
\label{l:volMatFull}
There is $C>0$ such that for all $t>0$ and $\bm K=\bm K^*$, 
       \begin{equation} 
    \label{e:volMat1}
    \vol \{ \bm S\in \mathbb{M}(n\times n;\mathbb{R})\,:\, \bm S=\bm S^T, \, \|\bm S\|\leq 1,\,|\det(\bm S+\bm K)|\leq t \Big\}\leq Ct  \det (\bm I+\bm K^*\bm K)^{-1/2}.
    \end{equation}
\end{lemma}
\begin{proof}
Let $C_0>0$ to be chosen large enough later.  $\|\bm K\|\leq C_0$, then~\eqref{e:volMat1} follows from Lemma~\ref{l:smallK}. Therefore, we assume that $\|\bm K\|>C_0$. 

For $0<r<1$, and $\bm X\in \mathbb{M}(n\times n;\mathbb{C})$ with $\|\bm X\|\leq 1$, define 
$$
\Phi_r(\bm X):= (\bm X+r\bm I)(\bm I+r\bm X)^{-1}. 
$$
Then, 
$$
\Phi_r^{-1}(\bm Y)=(\bm Y-r\bm I)(\bm I-r\bm Y)^{-1},
$$
and
$$
\Phi_r(\{ \bm S\in \mathbb{M}(n\times n;\mathbb{R})\,:\,\|\bm S\|\leq 1, \bm S^T=\bm S\})=\{ \bm S\in \mathbb{M}(n\times n;\mathbb{R})\,:\,\|\bm S\|\leq 1, \bm S^T=\bm S\}.
$$

There is $\eta>0$ depending only on $n$ such that for all $\lambda_1\leq \lambda_2\leq \dots \leq \lambda_n$, there is $r\in [\frac{1}{2},\frac{3}{4}]$ with 
\begin{equation}
\label{e:lambdas}
|1+r\lambda_i|\geq \eta\sqrt{1+\lambda_i^2},\qquad i=1,\dots,n.
\end{equation}
(To see this, observe that for a fixed $\lambda$ the interval in  $r$ where this fails is given by $[-\lambda^{-1}-\sqrt{\eta(1+\lambda^{-2})},-\lambda^{-1}+\sqrt{\eta(1+\lambda^{-2})}]$. Hence, the interval either does not intersect $[\frac{1}{2},\frac{3}{4}]$ or has width bounded by $C\sqrt{\eta}$ so that, taking $\eta$ small depending on $n$, one can find the desired $r$.)

We now fix such an $r$ where $\lambda_1\leq \lambda_2\leq \dots \leq \lambda_n$ are eigenvalues of $\bm K$.

By direct computation, one has 
$$
\bm S+\bm K=(\bm I+r\bm K)(\Phi_r^{-1}(\bm S)+\Phi_r(\bm K))(\bm I+r\Phi_r^{-1}(\bm S))^{-1},
$$
therefore, 
$$
\det (\bm S+\bm K)=\det (\bm I+r\bm K)\det (\Phi_r^{-1}(\bm S)+\Phi_r(\bm K))\det (I+r\Phi_r^{-1}(\bm S)).
$$
The estimate~\eqref{e:lambdas} implies that 
$$
(\det (\bm I+r\bm K))^2\geq \eta^2\det (\bm I+\bm K^2).
$$
Moreover, since $\frac{1}{2}\leq r\leq \frac{3}{4}$ and $\|\Phi_r^{-1}(\bm S)\|\leq 1$, 
$$
|\det (I+r\Phi_r^{-1}(\bm S))|\geq 4^{-n}.
$$
Therefore, it is enough to show that 
\begin{equation} 
\label{e:changedVol}
\vol (\{\bm S\,:\, |\det (\Phi_r^{-1}(\bm S)+\Phi_r(\bm K))|\leq t\})\leq Ct.
\end{equation}

To do this, notice that $\Phi_r(\bm K)$ is self-adjoint and, using~\eqref{e:lambdas},
$$
\|\Phi_r(\bm K)\|=\sup_{i}\Big|\frac{\lambda_i +r}{1+r\lambda_i}\Big|\leq \eta^{-1}\sup_{i}\frac{|\lambda_i +r|}{\sqrt{1+\lambda_i^2}}\leq \eta^{-1}.
$$
Hence, Lemma~\ref{l:smallK} implies that 
$$
\vol (\{\bm X\in \mathbb{M}(n\times n;\mathbb{R}),:\, \bm X=\bm X^t,\,\|\bm X\|\leq 1,\, \|\det (\bm X+\Phi_r(\bm K))|\leq t\})\leq Ct.
$$
This, together the fact that 
$$
\|D_{\bm S}\Phi_r\bm H\|=\|(1-r^2)(I+r\bm S)^{-1}\bm H(I+r\bm S)^{-1}\|\leq 12\|\bm H\|
$$
implies~\eqref{e:changedVol}.
\end{proof}

Finally, we use Lemma~\ref{l:volMatFull} to prove a volume estimate on determinant level sets necessary for our analysis of $(\beta,q)$ non-degenerate symplectic matrices.
\begin{lemma}
\label{l:theVolumeEstimate}
There is $C>0$ such that the following holds.
    Suppose that $\bm \Sigma_1,\bm \Sigma_2\in\mathbb{M}(n\times n;\mathbb{R})$ are diagonal with $\bm \Sigma_1^2+\bm \Sigma_2^2=\bm I$ and $\bm V\in\mathbb{M}(n\times n)$ is unitary. Then,
    $$
    \vol\Big\{ \bm S\in \mathbb{M}(n\times n;\mathbb{R})\,:\, \bm S=\bm S^T, \|\bm S\|\leq r_0,\,|\det(\bm \Sigma_1+\bm \Sigma_2 \bm V \bm S\bm V^*)|\leq \e\Big\}\leq Cr_0^{\frac{n(n+1)}{2}-n}\e.
    $$
\end{lemma}
\begin{proof}
    First, suppose that $\bm \Sigma_2$ is invertible. Then, 
    $$
    \det(\bm \Sigma_1+\bm \Sigma_2 \bm V \bm S\bm V^*)=\det (\bm \Sigma_2)\det (\bm S+\bm V^* \bm \Sigma_2^{-1}\bm \Sigma_1\bm V).
    $$
     Therefore,
    $$
     \det(\bm \Sigma_1+\bm \Sigma_2 \bm V \bm S\bm V^*)\leq \e \quad\Rightarrow\quad \det (\bm S+\bm V^* \bm \Sigma_2^{-1}\bm \Sigma_1\bm V)\leq \det (\bm \Sigma_2)^{-1}\e 
    $$
    Hence, setting $\bm K:=\bm V^* \bm \Sigma_2^{-1}\bm \Sigma_1\bm V$,  it is enough to estimate 
    \begin{equation}
    \label{e:rescaleMatrix}
            \begin{aligned}
    &\vol\Big\{ \bm S\in \mathbb{M}(n\times n;\mathbb{R})\,:\, \bm S=\bm S^T, \|\bm S\|\leq r_0,\,|\det(\bm S+\bm K)|\leq  \det (\bm \Sigma_2)^{-1}\e\Big\}\\
     &=r_0^{\frac{n(n+1)}{2}}\vol\Big\{ \tilde{\bm S}\in \mathbb{M}(n\times n;\mathbb{R})\,:\, \tilde{\bm S}=\tilde{\bm S}^T, \|\tilde{\bm S}\|\leq 1,\,|\det(\tilde{\bm S}+r_0^{-1}\bm K)|\leq  r_0^{-n}\det (\bm \Sigma_2)^{-1}\e\Big\}\\
    \end{aligned}
    \end{equation}
    Observe that by Lemma~\ref{l:volMatFull}
        \begin{equation} 
    \label{e:volMat2}
    \begin{aligned}
    \vol \{ \tilde{\bm S}\in \mathbb{M}(n\times n;\mathbb{R})\,:\, \tilde{\bm S}=\tilde{\bm S}^T, \, \|\tilde{\bm S}\|&\leq 1,\,|\det(\tilde{\bm S}+r_0^{-1}\bm K)|\leq r_0^{-n}t \Big\}\\&\leq Cr_0^{-n}\det (\bm \Sigma_2^2)^{-1/2}\e \det (\bm I+r_0^{-2}\bm K^*\bm K)^{-1/2}\\
    &\leq Cr_0^{-n}\e \det (\bm \Sigma_2^2+r_0^{-2}\bm\Sigma_1^2)^{-1/2}\leq Cr_0^{-n}\e.
    \end{aligned}
    \end{equation}
Using this in~\eqref{e:rescaleMatrix} completes the proof when $\bm \Sigma_2$ is invertible. 

If $\bm \Sigma_2$ is not invertible, let $\bm\Sigma_1(\delta)^2+\bm\Sigma_2(\delta)^2=\bm I$ with $\bm\Sigma_2(\delta)$ invertible and $\lim_{\delta \to 0}^+\bm \Sigma_2(\delta)=\bm \Sigma_2$. Observe that 
\begin{align*}
&\Big\{ \bm S\in \mathbb{M}(n\times n;\mathbb{R})\,:\, \bm S=\bm S^T, \|\bm S\|\leq r_0,\,|\det(\bm \Sigma_1+\bm \Sigma_2 \bm V \bm S\bm V^*)|\leq \e\Big\}\\
     &\subset \Big\{ \bm S\in \mathbb{M}(n\times n;\mathbb{R})\,:\, \bm S=\bm S^T, \|\bm S\|\leq r_0,\,|\det(\bm \Sigma_1(\delta)+\bm \Sigma_2(\delta) \bm V \bm S\bm V^*)|\leq 2\e\Big\}
     \end{align*}
     for $\delta$ small enough. and hence the lemma follows after increasing the constant if necessary.
\end{proof}}

\bibliography{biblio}
 \bibliographystyle{alpha}
\end{document}